\documentclass[10pt,a4paper]{smfart}

\usepackage[utf8]{inputenc}
\usepackage[french, english]{babel}
\usepackage[T1]{fontenc}
\usepackage{lmodern}

\usepackage[hidelinks,pdfencoding=unicode]{hyperref}
\usepackage[shortcuts]{extdash}
\usepackage{xpatch}

\usepackage{amsmath}
\usepackage{amssymb}
\usepackage{stmaryrd}
\usepackage{mathtools} 

\usepackage[all,2cell]{xy}
\UseAllTwocells
\SilentMatrices
\SelectTips{cm}{10}

\usepackage{version}

\usepackage{enumitem}
\setlist{nosep}
\setlist[enumerate, 1]{label={\rm (}\emph{\alph*}{\rm )}}
\setlist[enumerate, 2]{label={\rm (}\emph{\alph{enumi}.\arabic*}{\rm )}}

\makeatletter

\newcommand\mysubsection{%
 \@startsection{subsection}{1}
    {\z@}%
    {\bigskipamount}%
    {\medskipamount}%
    {\noindent\textbf}}

\patchcmd\@begintheorem
  {\ignorespaces}
  {\hypertarget{\csname @currentHref\endcsname}{}\ignorespaces}
  {}{}

\patchcmd\@maketitlehook
  {\ifx\@empty\@keywords\else\@footnotetext{\@setkeywords}\fi}
  {\ifx\@empty\@keywords\else\@footnotetext{\@setkeywords}\fi
   \ifx\@empty\@altkeywords\else\@footnotetext{\@setaltkeywords}\fi}
  {}{}

\makeatother

\theoremstyle{plain}
\newtheorem{thm}{Th\'eor\`eme}[section]
\newtheorem*{thm*}{Th\'eor\`eme}
\newtheorem{prop}[thm]{Proposition}
\newtheorem{lemme}[thm]{Lemme}
\newtheorem{coro}[thm]{Corollaire}
\theoremstyle{remark}

\theoremstyle{definition}

\newtheorem{paragr}[thm]{}

\theoremstyle{plain}
\swapnumbers

\swapnumbers

\numberwithin{equation}{thm}

\makeatletter
\newif\ifsection
\preto\section{\sectiontrue}
\preto\subsection{\sectionfalse}
\xapptocmd\@sect{%
  \ifsection
    \numberwithin{thm}{section}
  \else
    \numberwithin{thm}{subsection}
  \fi
  \setcounter{thm}{0}\relax}
  {}{}
\makeatother

\newcommand\letenv[2]{%
\expandafter\expandafter\expandafter\let\expandafter\expandafter
\csname #1\endcsname\csname #2\endcsname
\expandafter\expandafter\expandafter\let\expandafter\expandafter
\csname end#1\endcsname\csname end#2\endcsname
}

\letenv{paragraph}{paragr}
\letenv{proposition}{prop}
\letenv{corollary}{coro}
\letenv{theorem}{thm}
\letenv{theorem*}{thm*}
\letenv{remark}{rem}
\letenv{example}{exem}

\let\forlang\emph
\let\ndef\emph
\let\nbd\nobreakdash

\def\xpoint{\futurelet\@let@token\@xpoint}
\def\@xpoint{%
  \ifx\@let@token.\else
    .%
  \fi
  \xspace}

\newcommand\zbox[1]{\makebox[0pt][l]{#1}}
\newcommand\pbox[1]{\zbox{\quad#1}}

\newcommand\quadtext[1]{\quad\text{#1}\quad}

\newcommand\quadet{\quadtext{et}}

\renewcommand\le\leqslant
\renewcommand\ge\geqslant
\renewcommand\epsilon\varepsilon
\renewcommand\phi\varphi

\let\hookto\hookrightarrow
\let\xto\xrightarrow
\let\xot\xleftarrow
\newcommand\tod\Rightarrow
\newcommand\tot\Rrightarrow

\newcommand\e{\epsilon}

\let\var\bullet

\newcommand\leN{\le_\N}

\newcommand{\sauf}{\mathchoice{\raise 1.8pt\hbox{${\scriptstyle\kern
2.5pt\smallsetminus\kern 2.5pt}$}}{\raise 1.8pt\hbox{${\scriptstyle\kern
2.5pt\smallsetminus\kern 2.5pt}$}}{\raise
1.8pt\hbox{${\scriptscriptstyle\kern 1.5pt\smallsetminus\kern
1.5pt}$}}{\raise 1.8pt\hbox{${\scriptscriptstyle\kern
1.5pt\smallsetminus\kern 1.5pt}$}}}

\newcommand\HomOpLax{\Homi_{\mathrm{oplax}}}
\newcommand\HomLax{\Homi_{\mathrm{lax}}}
\newcommand\HomStr{\Homi_{\mathrm{str}}}
\newcommand\ooCatOpLax{\infty\hbox{\protect\nbd-}\mathcal{C}\mathsf{at}^{}_\mathrm{oplax}}
\newcommand\ooCatLax{\infty\hbox{\protect\nbd-}\mathcal{C}\mathsf{at}^{}_\mathrm{lax}}
\newcommand\ooCatOpLaxGray{\infty\hbox{\protect\nbd-}\mathbb{C}\mathsf{at}^{}_\mathrm{oplax}}

\let\joint\star
\newcommand\vide{\varnothing}

\newcommand\pdfoo{\texorpdfstring{$\infty$}{\unichar{"221E}}}

\newcommand\Z{\mathbb{Z}}
\newcommand\N{\mathbb{N}}

\let\limind\varinjlim

\let\C\relax
\newcommand\C{\mathcal{C}}

\newcommand{\Hom}{\operatorname{\mathsf{Hom}}}
\newcommand{\Homi}{\operatorname{\kern.5truept\underline{\kern-.5truept\mathsf{Hom}\kern-.5truept}\kern1truept}}

\newcommand{\pref}[1]{{\widehat{ #1 }}}
\newcommand\id[1]{1_{#1}}
\renewcommand\o\circ
\newcommand\op{\mathrm{op}}
\newcommand\co{\mathrm{co}}
\newcommand\oloc[1]{{}^t{#1}{}^\o}
\newcommand{\Ob}{\operatorname{\mathsf{Ob}}}

\newcommand{\Ens}{{\mathcal{E}\mspace{-2.mu}\it{ns}}}

\newcommand{\Cat}{{\mathcal{C}\mspace{-2.mu}\it{at}}}
\newcommand{\nCat}[1]{{#1}\hbox{\protect\nbd-}\kern1pt\Cat}
\newcommand{\dCat}{\nCat{2}}
\newcommand{\ooCat}{\nCat{\infty}}
\newcommand{\coCat}{{\it{co}\mathcal{C}\mspace{-2.mu}\it{at}}}
\newcommand{\ncoCat}[1]{{#1}\hbox{\protect\nbd-}\kern1pt\coCat}

\newcommand\oo[1]{$\infty$\=/}

\newcommand\trans[1]{{}^t{#1}}

\let\comp\ast

\newcommand\Dn[1]{\mathrm{D}_{#1}}

\newcommand\cDelta{\mathbf{\Delta}}
\newcommand\cDeltaAug{\mathbf{\Delta}_{+}}
\newcommand\Deltan[1]{\varDelta_{#1}}
\newcommand\EnsSimp{\pref{\cDelta}}

\newcommand\DDelta{D_{\!\cDelta}}
\newcommand\DDeltaAug{D_{\!\cDeltaAug}}

\newcommand\cO{\mathcal{O}}
\newcommand\cOAug{\mathcal{O}_+}
\newcommand\On[1]{\mathcal{O}_{#1}}

\newcommand{\tr}[2]{\mathchoice
  {#1\raise -1.8pt\vbox{\hbox{$\kern -.8pt/#2$}}}
  {#1\raise -1.8pt\vbox{\hbox{$\kern -.8pt/#2$}}\kern .8pt}
  {#1\raise -1.8pt\vbox{\hbox{$\scriptstyle\kern -.8pt /#2$}}}
  {#1\raise -1.8pt\vbox{\hbox{$\scriptscriptstyle\kern -.8pt /#2$}}}}

\newcommand{\trm}[2]{\mathchoice
  {#1\raise -1.8pt\vbox{\hbox{$\kern -.8pt\!\stackrel{\,\rm co}{/}\!\!#2$}}}
  {#1\raise -1.8pt\vbox{\hbox{$\kern -.8pt\!\stackrel{\,\rm co}{/}\!\!#2$}}\kern .8pt}
  {#1\raise -1.8pt\vbox{\hbox{$\scriptstyle\kern -.8pt\!\stackrel{\,\,\rm co}{/}\!\!#2$}}\kern .8pt}
  {TODO}}

\newcommand{\trto}[2]{\mathchoice
  {#1\raise -1.8pt\vbox{\hbox{$\kern -.8pt\!\stackrel{\,\,t \rm o}{/}\!\!\!#2$}}}
  {#1\raise -1.8pt\vbox{\hbox{$\kern -.8pt\!\stackrel{\,\,t \rm o}{/}\!\!\!#2$}}\kern .8pt}
  {#1\raise -1.8pt\vbox{\hbox{$\scriptstyle\kern -.8pt\!\stackrel{\,\,t \rm o}{/}\!\!\!#2$}}\kern .8pt}
  {TODO}}

\newcommand{\cotr}[2]{\mathchoice
  {\raise -1.8pt\vbox{\hbox{$#2\backslash$}}#1}
  {\raise -1.8pt\vbox{\hbox{$#2\backslash$}}#1}
  {\raise -1.8pt\vbox{\hbox{$\scriptstyle#2\backslash$}}#1}
  {\raise -1.8pt\vbox{\hbox{$\scriptscriptstyle#2\backslash$}}#1}}

\newcommand{\cotrm}[2]{\mathchoice
  {\raise -1.8pt\vbox{\hbox{$#2\!\stackrel{\!\rm co}{\backslash}$}}#1}
  {\raise -1.8pt\vbox{\hbox{$#2\!\stackrel{\!\rm co}{\backslash}$}}#1}
  {\raise -1.8pt\vbox{\hbox{$\scriptstyle#2\!\stackrel{\!\rm co}{\backslash}$}}#1}
  {TODO}}

\newcommand\Span[2]{\tr{#1}{#2} \times \trto{#1}{#2}}
\newcommand\Spanoo{\Span{\ooCatOpLaxGray}{Z}}
\newcommand\SpanC{\Span{\ooCatOpLax}{Z}}
\newcommand\Cyl\Gamma

\newcommand\comma{\operatorname{\downarrow}}
\newcommand\commalax{\comma'}
\newcommand\commaCfun{{-} \comma_Z {-}}

\newcommand{\Cda}{\mathcal{C}_{\mathrm{da}}}

\newcommand{\atom}[1]{\langle{#1}\rangle}
\newcommand{\tabld}[2]{\begin{pmatrix}#1^0_0 &\dots &#1^0_{#2-1}
  &#1^0_{#2}\cr\noalign{\vskip 3pt} #1^1_0 &\dots &#1^1_{#2-1}
  &#1^1_{#2}\end{pmatrix}}

\newcommand{\tabll}[2]{\begin{pmatrix}#1^0_0 &#1^0_1 &\dots &#1^0_{#2-1}
  &#1^0_{#2}\cr\noalign{\vskip 3pt} #1^1_0 &#1^1_1 &\dots &#1^1_{#2-1}
  &#1^1_{#2}\end{pmatrix}}

\newcommand\Zdec{\underline{\mathbb{Z'}\kern -2.5pt}\kern 2pt}
\newcommand\cn{\mathsf{c}} 

\newcommand\GrayC{\mathbb{C}}
\newcommand\SesquiC{\mathcal{C}}
\newcommand\SesquiD{\mathcal{D}}

\newcommand\W{\mathcal{W}}

\author{Dimitri Ara}
\address{Aix~Marseille~Univ,~CNRS,~Centrale~Marseille,~I2M,~Marseille,~France}
\email{dimitri.ara@univ-amu.fr}
\urladdr{http://www.i2m.univ-amu.fr/perso/dimitri.ara/}

\author{Georges Maltsiniotis}
\address{%
Institut de Math\'ematiques de Jussieu\\
Universit\'e Paris 7 Denis Diderot\\
Case Postale 7012\\
B\^atiment Sophie Germain\\
75205 Paris Cedex 13\\
France}
\email{georges.maltsiniotis@imj-prg.fr}
\urladdr{http://webusers.imj-prg.fr/%
\raise -3.3pt\vbox{\hbox{$\widetilde{ \ }\,$}}georges.maltsiniotis/}

\title[Un théorème A de Quillen pour les $\infty$-catégories strictes II]{Un
théorème A de Quillen pour\\ les \raise .9pt\hbox{$\infty$-}catégories
strictes II :\\ {la preuve \raise .9pt\hbox{$\infty$-}catégorique}}
\alttitle{A Quillen's Theorem A for strict $\infty$-categories II}

\begin{document}

\frontmatter

\begin{abstract}
Cet article est le second d'une série de deux articles consacrés à une
généralisation du théorème A de Quillen aux \oo-catégories strictes. Dans le
premier, nous avons exposé une preuve de nature simpliciale, rapide mais
quelque peu \forlang{ad hoc}, de ce théorème A. Dans le présent article,
nous en donnons une preuve conceptuelle, de nature \oo-catégorique, basée
sur, d'une part, la théorie du joint et des tranches \oo-catégoriques
développée par les auteurs dans un précédent travail et,
d'autre part, une construction comma pour les \oo-catégories strictes qui
généralise les catégories comma classiques et les $2$\nbd-catégories comma
de Gray. Cette construction comma \oo-catégorique est utilisée par le
premier auteur dans un autre article pour démontrer une généralisation du
théorème B de Quillen aux \oo-catégories strictes. L'importance de
cette construction comma en théorie des \oo-catégories nous semble dépasser
largement le cadre de la théorie de l'homotopie.
\end{abstract}

\begin{altabstract}
This paper is the second in a series of two papers about generalizing
Quillen's Theorem A to strict \oo-categories. In the first one, we presented
a proof of this Theorem A of a simplicial nature, direct but somewhat ad
hoc. In the current paper, we give a conceptual proof of an \oo-categorical
nature of the same theorem. This proof is based on the theory of join and
slices for strict \oo-categories developed by the authors in
a previous paper, and on a comma construction for strict \oo-categories
generalizing classical comma categories and Gray's comma $2$\nbd-categories.
This \oo-categorical comma construction is used by the first author in another
paper to prove a generalization of Quillen's Theorem~B to strict
\oo-categories. We believe that the importance of this comma construction in
the theory of \oo-categories goes far beyond the scope of homotopy theory.
\end{altabstract}

\subjclass{18A25, 18D05, 18G30, 18G35, 18G55, 55P15, 55U10, 55U15, 55U35}

\keywords{\oo-catégories comma, \oo-catégories de Gray, \oo-catégories strictes,
complexes dirigés augmentés, ensembles simpliciaux, joint, nerf de Street,
orientaux, produit tensoriel de Gray, sesquicatégories, théorème A,
tranches, transformations oplax}
\altkeywords{comma \oo-categories, Gray \oo-categories, strict
\oo-categories, augmented directed complexes, simplicial sets, join,
Street's nerve, orientals, Gray tensor product, sesqui\-categories, Theorem A,
slices, oplax transformations}

\maketitle

\vfill
\break

\tableofcontents

\mainmatter

\section*{Introduction}

\mysubsection*{Théorie de l'homotopie des \pdfoo-catégories strictes}

Ce texte fait partie d'un projet consacré à la théorie de l'homotopie des
\oo-catégories strictes, projet constitué actuellement des articles et
prépublications \cite{AraMaltsiNThom, Ara2Cat, AraMaltsiCondE,
AraMaltsiJoint, AraMaltsiThmAI, AraThmB} et du texte en
préparation~\cite{AraMaltsiNerfs}. L'objet de ce projet est l'étude des
relations entre les \oo-catégories strictes et les types d'homotopie
\forlang{via} leur espace classifiant. Rappelons, en effet, qu'à toute
\oo-catégorie stricte $C$, on associe selon Street \cite{StreetOrient} un
ensemble simplicial $N(C)$, appelé son \ndef{nerf de Street}, et donc un
type d'homotopie. On dira qu'un \oo-foncteur strict entre \oo-catégories
strictes $u : C \to D$ est une \ndef{équivalence de Thomason} si le
morphisme simplicial qu'il induit entre les nerfs de Street de $C$ et de $D$
est une équivalence d'homotopie faible.  La théorie de l'homotopie des
\oo-catégories strictes est l'étude de la catégorie $\ooCat$ des
\oo-catégories strictes et des \oo-foncteurs stricts munie de cette notion
d'équivalence faible. Gagna a démontré dans \cite{GagnaIllusieQuillen} une
conjecture que nous avions formulée dans~\cite{AraMaltsiCondE} affirmant que
le nerf de Street induit une équivalence de catégories entre la localisation
de $\ooCat$ par les équivalences de Thomason et la catégorie homotopique des
espaces. Ainsi, étudier la théorie de l'homotopie des \oo-catégories
strictes, c'est étudier les espaces sous un nouvel angle.  On renvoie à
l'introduction de~\cite{AraMaltsiThmAI} pour plus de détails sur notre
projet.

Mais revenons en arrière. Notre projet est inspiré de la théorie de
l'homotopie de~$\Cat$, la catégorie des petites catégories, développée
notamment par Quillen~\cite{QuillenHAKTI}, Thomason~\cite{Thomason} et
Grothendieck~\cite{GrothPS} (voir également \cite{Maltsi, Cisinski}). Le
point de départ de cette théorie est l'idée de Quillen de définir les
groupes de K-théorie algébrique supérieurs comme les groupes d'homotopie
de l'espace classifiant d'une catégorie. Afin d'établir les propriétés
importantes de sa K-théorie algébrique, Quillen démontre ses fameux
théorèmes A et B, établissant ainsi les propriétés fondamentales des
\ndef{équivalences de Thomason}, foncteurs dont le nerf est une équivalence
d'homotopie faible simpliciale. Énonçons une variante relative de son
théorème A.

\begin{theorem*}[Quillen]
  Soit
  \[
    \xymatrix@C=1.5pc{
      A \ar[rr]^u \ar[dr]_v & & B \ar[dl]^w \\
      & C
    }
  \]
  un triangle commutatif de foncteurs entre petites catégories. Si pour tout
  objet $c$ de~$C$, le foncteur $\cotr{A}{c} \to \cotr{B}{c}$ induit par $u$
  est une équivalence de Thomason, alors il en est de même du foncteur~$u$.
\end{theorem*}

Dans cet énoncé, $\cotr{A}{c}$ désigne la \ndef{tranche de $A$ au-dessous de
$c$}, catégorie dont les objets sont les couples $(a, f : c \to v(a))$, où
$a$ est un objet de $A$ et $f$ une flèche de~$C$, et dont les morphismes d'un
objet $(a, f)$ vers un objet $(a', f')$ sont les morphismes $g : a \to a'$
de $A$ tels que $v(g)f = f'$.

Le caractère fondamental du théorème A de Quillen a été mis en évidence par
Grothendieck. En effet, celui-ci a défini une notion de \ndef{localisateur
fondamental}, classe de flèches de $\Cat$ satisfaisant à des axiomes
inspirés des propriétés formelles de la classe des équivalences de Thomason,
le plus important de ces axiomes étant le théorème~A, et il a conjecturé que
la classe des équivalences de Thomason forme le plus petit localisateur
fondamental. Cette conjecture a été démontrée par Cisinski
dans~\cite{CisinskiLFM}. Ce résultat peut s'interpréter de la manière
suivante : le théorème A est le seul moyen non trivial dont on dispose pour
démontrer qu'un foncteur est une équivalence de Thomason.

Le théorème A de Quillen a été généralisé aux $2$-catégories et
$2$-foncteurs stricts par Bullejos et Cegarra \cite{BullCegGeom2Cat}, aux
$2$-catégories et foncteurs lax par del Hoyo~\cite{delHoyoThmA, delHoyoLoop}
et aux triangles de foncteurs lax ne commutant qu'à une transformation près
par Chiche~\cite{ChicheThese, ChicheThmA}. On renvoie à l'introduction
de~\cite{AraMaltsiThmAI} pour plus de détails sur l'historique de ces
théorèmes~A.

\mysubsection*{Théorème A de Quillen pour les \oo-catégories strictes}

Le présent article est le second d'une série de deux articles consacrés à
une généralisation du théorème~A de Quillen aux \oo-catégories strictes.
Dans le premier article~\cite{AraMaltsiThmAI}, nous avons établi, par des
techniques simpliciales, le théorème suivant :

\begin{theorem*}
  Soit
  \[
    \shorthandoff{;}
    \xymatrix@C=1.5pc{
      A \ar[rr]^u \ar[dr]_(0.40){v}_(.60){}="g" & & B \ar[dl]^(0.40){w} \\
      & C
      \ar@{}"g";[ur]_(.15){}="gg"
      \ar@{}"g";[ur]_(.55){}="oo"
      \ar@<-0.0ex>@2"gg";"oo"^\alpha
      &
    }
  \]
  un triangle de \oo-foncteurs stricts commutatif à une transformation
  oplax $\alpha$ près. Si pour tout objet $c$ de $C$, le \oo-foncteur
  $\cotr{A}{c} \to \cotr{B}{c}$ induit par $u$ est une équivalence de
  Thomason, alors il en est de même de $u$.
\end{theorem*}

Dans cet énoncé, $\cotr{A}{c}$ désigne une généralisation \oo-catégorique
adéquate des tranches catégoriques qui peut se définir par des formules
explicites. Par ailleurs, l'adjectif « oplax » dans « transformation oplax »
fait référence à un choix d'orientation des cellules associées à $\alpha$.
Rappelons enfin que nous appelons \ndef{équivalence de Thomason} un
\oo-foncteur strict $u$ dont le nerf de Street $N(u)$ est une équivalence
faible simpliciale.

\emph{Dans la suite de cette introduction, toutes les \oo-catégories et tous
les \oo-foncteurs seront supposés stricts.}

Dans le présent article, on présente une nouvelle preuve de ce théorème, de
nature \oo-catégorique. L'intérêt de ce travail par rapport à
\cite{AraMaltsiThmAI} est multiple :
\begin{itemize}
  \item La preuve présentée du théorème A \oo-catégorique est plus
    conceptuelle : elle s'appuie sur des outils \oo-catégoriques comme le
    joint et les fonctorialités des tranches que nous avons développés dans
    \cite{AraMaltsiJoint} et une nouvelle construction comma pour les
    \oo-catégories.
  \item On établit des propriétés de fonctorialité de cette nouvelle
    construction comma, construction qui, nous semble-t-il, est une
    contribution importante à la théorie des \oo-catégories, indépendamment
    des questions d'homotopie qui nous préoccupent dans ce travail. Cette
    construction est également utilisée par le premier auteur dans sa preuve
    d'un théorème B \oo-catégorique~\cite{AraThmB}.
  \item En utilisant cette construction comma, on montre comment le
    théorème A pour les triangles commutatifs à une transformation oplax près peut
    se ramener formellement au cas des triangles commutatifs, non seulement
    pour les équivalences de Thomason mais également pour des classes de
    \oo-foncteurs plus générales inspirées des localisateurs fondamentaux de
    $\Cat$ de Grothendieck.
  \item Enfin, on montre comment associer à toute transformation
    oplax une homotopie simpliciale.
\end{itemize}

Notre preuve s'articule de la manière suivante. On dégage deux théorèmes~A
abstraits pour les triangles \emph{commutatifs}, le premier cosimplicial et
le second monoïdal, exprimant l'essence de l'argument originel de Quillen.
En appliquant le second de ces théorèmes à $\ooCat$ munie du joint
\oo-catégorique, on ramène le théorème A \oo-catégorique pour les triangles
\emph{commutatifs} à l'énoncé suivant : si $u : A \to B$ est un
\oo-foncteur, alors pour tout \oo-foncteur $b : \On{m} \to B$, où $\On{m}$
désigne le $m$-ième oriental de Street, le \oo-foncteur $\cotr{A}{b} \to
\cotr{A}{b_m}$ induit par fonctorialité des tranches est une équivalence de
Thomason. Dans cet énoncé, $b_m$ désigne l'objet correspondant au
\oo-foncteur $\On{0} \to B$ (l'oriental $\On{0}$ étant la \oo-catégorie
terminale) obtenu en précomposant $b : \On{m} \to B$ par le \oo-foncteur $m
: \On{0} \to \On{m}$, analogue \oo-catégorique du morphisme simplicial $m :
\Deltan{0} \to \Deltan{m}$ correspondant au $m$-ième sommet du
$m$\nbd-simplexe standard.  Par ailleurs, $\cotr{A}{b}$ désigne la tranche
de $A$ au-dessous du \oo-foncteur $b$. Cette tranche, contrairement à celles
qui apparaissent dans l'énoncé du théorème~A, n'est donc pas une tranche
au-dessous d'un objet ; sa définition, plus complexe, utilise la théorie du
joint \oo-catégorique. On démontre cet énoncé par des méthodes
\oo-catégoriques basées sur les fonctorialités des tranches que nous avons
établies dans~\cite{AraMaltsiJoint}, obtenant ainsi le théorème~A
\oo-catégorique pour les triangles commutatifs. Enfin, en utilisant notre
nouvelle construction comma mentionnée plus haut et ses propriétés de
fonctorialité, on déduit le théorème \oo-catégorique pour les triangles
commutatifs à transformation oplax près du théorème A \oo-catégorique pour
les triangles commutatifs.

Voici comment cette preuve se compare à la preuve simpliciale que nous
avions donnée dans \cite{AraMaltsiThmAI}. L'argument originel de Quillen
permet de ramener le théorème A \oo-catégorique pour les triangles
\ndef{commutatifs} aux deux assertions suivantes :
\begin{itemize}
  \item Si $C$ est une \oo-catégorie et $c$ est un objet de~$C$, on a un
  isomorphisme naturel entre $N(\cotr{C}{c})$ et la tranche simpliciale
 $\cotr{N(C)}{c}$, l'objet $c$ correspondant à un $0$-simplexe de $N(C)$.
 \item Si $u : A \to B$ est un \oo-foncteur, alors pour tout $m$-simplexe
   $b$ de $N(B)$, le morphisme simplicial $\cotr{N(A)}{b} \to
   \cotr{N(A)}{b_m}$, où $b_m$ désigne le $m$-ième sommet de~$b$, est une
   équivalence d'homotopie faible. (Notons que les tranches apparaissant ici
   sont des tranches simpliciales au-dessous d'un simplexe, tranches dont la
   définition est classique.)
\end{itemize}
Dans \cite{AraMaltsiThmAI}, on démontre le premier énoncé en construisant
un isomorphisme explicite à l'aide de la théorie des complexes dirigés
augmentés de Steiner \cite{Steiner} et le second, toujours en utilisant la
théorie de Steiner, en montrant que le
morphisme en jeu est la rétraction d'un rétracte par déformation fort en
produisant explicitement une section et une homotopie simpliciale par des
formules \forlang{ad hoc}. Dans le présent article, le premier énoncé
résulte formellement des propriétés de monoïdalité du joint et le second
s'obtient par des fonctorialités des tranches, celles-ci permettant d'obtenir un
\oo-foncteur et une transformation oplax de nerf la rétraction et
l'homotopie simpliciale de \cite{AraMaltsiThmAI} respectivement. Enfin, dans
\cite{AraMaltsiThmAI}, pour obtenir le théorème A \oo-catégorique pour les
triangles commutatifs à une transformation oplax près, on est conduit à modifier
de manière non triviale l'argument originel de Quillen en faisant de nouveau
intervenir la théorie de Steiner, alors que dans le présent texte, on déduit
formellement ce cas de celui des triangles commutatifs grâce à la
construction comma \oo-catégorique.

Détaillons maintenant les éléments de notre preuve.

\mysubsection*{Théorème A pour les triangles commutatifs et théorèmes A abstraits}

La stratégie adoptée dans le présent article pour réduire notre théorème A
\oo-catégorique pour les triangles \emph{commutatifs} au fait que le
\oo-foncteur $\cotr{A}{b} \to \cotr{A}{b_m}$ est une équivalence de Thomason
est axiomatique. Comme on vient de l'expliquer, on commence par dégager deux
théorèmes A abstraits : un théorème A cosimplicial et un théorème A
monoïdal.

Le cadre du théorème A cosimplicial est le suivant. On se donne un objet
cosimplicial $\cO : \cDelta \to \C$ dans une catégorie $\C$ (l'exemple
qui nous intéresse étant celui de l'objet cosimplicial $\cO : \cDelta \to
\ooCat$ des orientaux de Street). Ceci permet de définir un foncteur nerf
$N$ de $\C$ vers la catégorie des ensembles simpliciaux et donc, en
utilisant ce nerf, une notion d'équivalence faible dans $\C$. Le théorème A
cosimplicial affirme alors la chose suivante :
si pour tout morphisme $T \to Z$ de $\C$ et tout $m$-simplexe $z$ de
$N(Z)$, le morphisme simplicial $\cotr{N(T)}{z} \to \cotr{N(T)}{z_m}$ est
une équivalence d'homotopie faible, alors un théorème A est valable dans
$\C$ au sens où, pour tout triangle commutatif
  \[
    \shorthandoff{;}
    \xymatrix@C=1.5pc{
      X \ar[rr]^u \ar[dr]_(0.40){v}_(.60){}="g" & & Y \ar[dl]^(0.40){w} \\
      & Z
      &
    }
  \]
dans $\C$, si pour tout $0$-simplexe $z$ de $N(Z)$, le morphisme
$\cotr{N(X)}{z} \to \cotr{N(Y)}{z}$ induit par $u$ est une équivalence
d'homotopie faible simpliciale, alors $u$ est une équivalence faible de $\C$.

Le théorème A monoïdal, qui s'appuie sur le théorème A cosimplicial, permet
de formuler un théorème A en termes de tranches dans $\C$ et non pas dans
les ensembles simpliciaux. Le cadre est le suivant. On se donne une
catégorie monoïdale $(\C, \joint, \vide)$, où $\vide$ est un objet initial
de $\C$, localement bifermée, au sens où, pour tous objets $X$ et $Y$ de
$\C$, les foncteurs
  \[
    \begin{split}
      \C & \to \cotr{\C}{X} \\
      Z & \mapsto (X \joint Z, X \simeq X \joint \vide \to X \joint Z)
    \end{split}
    \quadet
    \begin{split}
      \C & \to \cotr{\C}{Y} \\
      Z & \mapsto (Z \joint Y, Y \simeq \vide \joint Y \to Z \joint Y)
    \end{split}
  \]
admettent des adjoints à droite. Ces adjoints définissent des foncteurs tranches
\[
  \begin{split}
    \cotr{\C}{X} & \to \C \\
    (Z, u : X \to Z) & \mapsto \cotr{Z}{u}
  \end{split}
  \quadet
  \begin{split}
    \cotr{\C}{Y} & \to \C \\
    (Z, v : Y \to Z) & \mapsto \tr{Z}{v}. 
  \end{split}
\]
Ces considérations sont bien sûr inspirées des propriétés de $\ooCat$
munie du joint \oo-catégorique. Si $\C$ admet un objet final $e$, on définit
un objet cosimplicial $\cO : \cDelta \to \C$ dans $\C$ en posant $\On{m} = e
\joint e \joint \cdots \joint e$, où $e$ apparaît $m + 1$ fois. On en déduit
un foncteur nerf et une notion d'équivalence faible dans $\C$. Le théorème A
cosimplicial appliqué à cet objet cosimplicial $\cO$ permet alors d'obtenir
le théorème A monoïdal : si pour tout morphisme $T \to Z$ de $\C$ et tout
morphisme $z : \On{m} \to Z$, le morphisme $\cotr{T}{z} \to \cotr{T}{z_m}$
induit par fonctorialité des tranches, où $z_m : \On{0} \to \C$ désigne le
morphisme obtenu en précomposant $z$ par le morphisme $m : \On{0} \to
\On{m}$ induit par $m : \Deltan{0} \to \Deltan{m}$, est une équivalence
faible de $\C$, alors un théorème A est valable dans $\C$ au sens où, pour
tout triangle commutatif
  \[
    \shorthandoff{;}
    \xymatrix@C=1.5pc{
      X \ar[rr]^u \ar[dr]_(0.40){v}_(.60){}="g" & & Y \ar[dl]^(0.40){w} \\
      & Z
      &
    }
  \]
dans $\C$, si pour tout morphisme $z : e \to Z$, le morphisme $\cotr{X}{z}
\to \cotr{Y}{z}$ induit par $u$ est une équivalence faible de $\C$, alors il
en est de même de $u$.

Voici comment on utilise le théorème A monoïdal pour obtenir notre
théorème~A \oo-catégorique pour les triangles commutatifs. Considérons la
catégorie $\ooCat$ munie du joint \oo-catégorique. On montre
dans~\cite{AraMaltsiJoint} que l'objet cosimplicial associé
à cette catégorie monoïdale par le procédé qu'on vient de décrire n'est
autre que l'objet cosimplicial des orientaux. Ainsi,
en appliquant le théorème A monoïdal à $\ooCat$ munie du joint, on ramène
le théorème A \oo-catégorique pour les triangles commutatifs à l'énoncé
suivant : pour tout \oo-foncteur $u : A \to B$ et tout \oo-foncteur
$b : \On{m} \to B$, le \oo-foncteur $\cotr{A}{b} \to \cotr{A}{b_m}$
est une équivalence de Thomason. Pour démontrer ceci, on utilise les
propriétés de fonctorialité des tranches établies
dans~\cite{AraMaltsiJoint}. En effet, ce \oo-foncteur $\cotr{A}{b} \to
\cotr{A}{b_m}$ est induit par une fonctorialité des tranches appliquée au
\oo-foncteur $m : \On{0} \to \On{m}$. De la même manière que le morphisme
$m : \Deltan{0} \to \Deltan{m}$ est un rétracte par déformation fort, nous
montrons que le \oo-foncteur $m : \On{0} \to \On{m}$ est ce que nous
appelons un rétracte par transformation oplax fort, notion analogue dans
laquelle la notion d'homotopie simpliciale est remplacée par celle de
transformation oplax. Nos résultats de fonctorialité des tranches entraînent
alors que $\cotr{A}{b} \to \cotr{A}{b_m}$ est la rétraction d'un rétracte
par transformation oplax fort. Or, nous montrons que toute transformation
oplax induit une homotopie simpliciale et nous en déduisons que le nerf de
Street du \oo-foncteur $\cotr{A}{b} \to \cotr{A}{b_m}$, qui n'est autre que
le morphisme simplicial $\cotr{N(A)}{b} \to \cotr{N(A)}{b_m}$, est la
rétraction d'un rétracte par déformation fort.  Nous montrons de plus que la
section et l'homotopie simpliciale qu'on obtient ainsi coïncident avec
celles définies par les formules de~\cite{AraMaltsiThmAI}.

\mysubsection*{Théorème A pour les $2$-triangles et construction comma
\oo-catégorique}

Comme mentionné précédemment, dans le présent article, nous montrons que le
théorème A pour les triangles commutatifs à une transformation oplax près
peut se déduire du théorème A pour les triangles commutatifs
par des outils \oo-catégoriques. Pour cela, nous nous inspirons de
l'article~\cite{ChicheThmA} de Chiche dans lequel est démontré l'assertion
analogue pour la catégorie $\dCat$ des $2$-catégories.  La preuve de Chiche
repose sur la théorie de l'intégration de Grothendieck dans $\dCat$. La
théorie de l'intégration dans~$\ooCat$ n'étant pas encore pleinement
développée, nous l'avons contournée en introduisant une construction comma
pour les \oo-catégories (qui permet d'ailleurs de définir la construction de
Grothendieck).

Soient
  \[
    \xymatrix{
      X \ar[r]^f & Z & Y \ar[l]_g
    }
  \]
deux \oo-foncteurs. On montre qu'il existe une \oo-catégorie $f \comma g$,
qu'on appelle \ndef{\oo-catégorie comma}, satisfaisant à la propriété
universelle suivante. Si $T$ est une \oo-catégorie, la donnée d'un
\oo-foncteur
\[ T \to f \comma g \]
correspond à celle d'un diagramme
  \[
    \shorthandoff{;}
    \xymatrix@C=1.3pc@R=1.3pc{
      & T \ar[dl]_x \ar[dr]^y \\
      X \ar[dr]_f \ar@{}[rr]_(.35){}="x"_(.65){}="y"
      \ar@2"x";"y"^{\lambda} 
      & & Y \ar[dl]^g \\
      & Z & \pbox{,}
    }
  \]
où $x$ et $y$ sont des \oo-foncteurs et $\lambda$ est une transformation
oplax. Lorsque $X$, $Y$ et~$Z$ sont des catégories, on retrouve les
catégories comma classiques telles que définies par exemple dans
\cite[chapitre II, section 6]{MacLane}. Lorsque ce sont des $2$-catégories
ou des $3$\nbd-catégories, on retrouve les constructions comma étudiées par
Gray dans \cite{GrayFCT} (au détail près que nous considérons des comma
oplax alors que Gray privilégie les comma lax).

On vérifie que cette construction est fonctorielle par rapport à des
morphismes de la forme
\[
      \shorthandoff{;}
      \xymatrix@R=1pc@C=3pc{
        X \ar[dd]_u \ar[dr]^f_{}="f" & & Y \ar[dl]_g_{}="g" \ar[dd]^v \\
          & Z \\
        X' \ar[ur]_{f'} & & Y' \ar[ul]^{g'}
        \ar@{}[ll];"f"_(0.35){}="sa"_(0.85){}="ta"
        \ar@2"sa";"ta"^{\alpha}
        \ar@{}[];"g"_(0.35){}="tb"_(0.85){}="sb"
        \ar@2"sb";"tb"^{\beta} \pbox{,}
      }
    \]
où $\alpha$ et $\beta$ sont des transformations oplax. Nous conjecturons que le
foncteur construction comma ainsi obtenu s'étend en un \oo-foncteur de Gray,
c'est-à-dire un foncteur enrichi dans $\ooCat$ munie du produit tensoriel de
Gray, et qu'en particulier, il agit sur une notion adéquate de $n$-cellules
pour tout $n \ge 0$. Dans ce texte, nous parvenons uniquement, par une preuve
quelque peu technique et laborieuse, à étendre la construction comma en un
sesquifoncteur par rapport à des $2$-cellules de la forme
      \[
        \shorthandoff{;:}
        \xymatrix@R=1.5pc@C=3.5pc{
          X \ar@/_2ex/[dd]_(0.62){\phantom{u'}u}_{}="u"
          \ar@/^2ex/[dd]_(0.65){u'\!}_{}="u'"
           \ar[dr]^f_{}="f" & &
          Y \ar@/_2ex/[dd]^(0.65){v'}_{}="v"
          \ar@/^2ex/[dd]^(0.65){v}_{}="v'"
           \ar[dl]_g_{}="g" \\
            & Z \\
          X' \ar[ur]_{f'} & & Y' \ar[ul]^{g'}
          \ar@2"u";"u'"^{\gamma\,\,}
          \ar@2"v'";"v"_{\,\,\delta}
          \ar@{}[ll];"f"_(0.40){}="sa"_(0.85){}="ta"
          \ar@<-1.5ex>@/_1ex/@2"sa";"ta"_(.70){\alpha'}_(0.40){}="a'"
          \ar@<0.0ex>@/^1ex/@{:>}"sa";"ta"^(.70){\alpha}_(0.60){}="a"
          \ar@3"a'";"a"_(.60){\Gamma_{}}
          \ar@{}[];"g"_(0.40){}="tb"_(0.85){}="sb"
          \ar@<-1.5ex>@/_1ex/@2"sb";"tb"_(.30){\beta'\!}_(0.40){}="b'"
          \ar@<0.0ex>@/^1ex/@{:>}"sb";"tb"^(.30){\!\beta}_(0.60){}="b"
          \ar@3"b'";"b"^(.40){\Delta_{}}
          \pbox{,}
        }
      \]
où les lettres grecques minuscules désignent des transformations oplax et
les lettres grecques majuscules des $2$-transformations oplax (parfois également
appelées modifications).

Cette sesquifonctorialité est la quantité minimale de fonctorialité de la
construction comma qui nous permette de déduire le théorème A
\oo-catégorique pour les triangles commutatifs à une transformation oplax
près du cas des triangles commutatifs. En effet, elle entraîne que la
construction comma préserve les rétractes par transformation oplax forts, ce
qui fournit un précieux outil pour montrer que certains \oo-foncteurs sont
des équivalences de Thomason. Cette propriété joue également un rôle
central dans la preuve de la généralisation du théorème B de Quillen aux
\oo-catégories strictes par le premier auteur dans \cite{AraThmB}.

La raison pour laquelle on obtient un sesquifoncteur et non pas un
$2$-foncteur est que les $2$-cellules comme ci-dessus ne se composent pas
horizontalement (horizontalement au sens technique mais verticalement sur le
diagramme). Ceci est lié au fait que les \oo-catégories strictes, \oo-foncteurs
stricts et transformations oplax ne forment pas une $2$-catégorie mais
seulement une sesquicatégorie. En effet, si
  \[
    \shorthandoff{;:}
    \xymatrix@C=3pc{
      X \ar@/^2.3ex/[r]_{}="1"
        \ar@/_2.3ex/[r]_{}="0"
      &
      Y \ar@/^2.3ex/[r]_{}="2"
        \ar@/_2.3ex/[r]_{}="3"
      &
      Z
      \ar@2"1";"0"_{\alpha}
      \ar@2"2";"3"_{\beta}
    }
  \]
sont deux transformations oplax, les deux manières de les composer
en utilisant la composition verticale des transformations oplax et la
composition horizontale d'une transformation oplax et d'un \oo-foncteur ne
coïncident pas. Néanmoins, il existe une $2$-transformation oplax canonique
entre ces deux compositions que nous appelons la \ndef{contrainte de Gray}
associée à $\alpha$ et $\beta$. Cette contrainte de Gray fait partie de la
structure de \oo-catégorie de Gray, c'est-à-dire de catégorie enrichie dans
$\ooCat$ munie du produit tensoriel de Gray oplax, dont sont munis les
\oo-catégories strictes, \oo-foncteurs stricts et $i$-transformations oplax
pour tout $i \ge 1$, structure qu'on notera~$\ooCatOpLaxGray$.

Ainsi, l'étude de la sesquifonctorialité de la construction comma nécessite
une compréhension de la structure de \oo-catégorie de Gray. On
établit quelques propriétés de cette structure dans ce texte. On utilise ces
propriétés pour montrer que si $\GrayC$ est une \oo-catégorie de Gray,
alors pour tout objet $c$ de $\GrayC$, on dispose d'une sesquicatégorie
tranche~$\tr{\GrayC}{c}$, établissant ainsi une conséquence en basse
dimension de notre conjecture~C.24 de \cite{AraMaltsiJoint}
affirmant qu'on dispose d'une \oo-catégorie de Gray
tranche~$\tr{\GrayC}{c}$. C'est à partir de cette sesquicatégorie tranche,
appliquée à~$\GrayC = \ooCatOpLaxGray$, qu'on définit la
sesquicatégorie source de la construction comma.

\mysubsection*{Organisation de l'article}

La première section est consacrée à des préliminaires sur les ensembles
simpliciaux. On y définit une construction comma bisimpliciale et les
tranches simpliciales. On démontre, selon l'argument originel de Quillen, un
théorème A simplicial.

Dans la deuxième section, en utilisant ce théorème A simplicial, on démontre
deux théorèmes A abstraits, abstraits signifiant que ces résultats
s'appliquent dans une catégorie $\C$ (qui n'est pas nécessairement $\Cat$ ou
$\ooCat$) munie de structure supplémentaire. Pour le premier, le théorème A
cosimplicial, cette structure supplémentaire est un objet cosimplicial
(vérifiant certains axiomes). Pour le second, le théorème A monoïdal, cette
structure est celle d'une catégorie monoïdale localement bifermée, notion
introduite dans cette section. On montre comment déduire de ce
formalisme le théorème A de Quillen originel.

La troisième section est dédiée à des rappels sur la théorie des
complexes dirigés augmentés de Steiner \cite{Steiner}. On y présente les
résultats fondamentaux de cette théorie, dus à Steiner, ainsi que quelques
compléments issus de~\cite{AraMaltsiJoint}.

La quatrième section est consacrée à des préliminaires \oo-catégoriques. On
commence par rappeler la théorie du produit tensoriel de Gray oplax,
produit introduit par Al-Agl et Steiner dans~\cite{AlAglSteiner} et
généralisant le produit de Gray $2$-catégorique \cite{GrayFCT}, et la
notion de transformation oplax ou, plus généralement, de $i$-transformation
oplax, ainsi que leurs variantes lax. On définit quelques opérations de
composition sur les transformations oplax et on explicite le lien avec les
transformations strictes. On expose ensuite un résumé de la théorie du joint
et des tranches \oo-catégoriques telle que développée
dans~\cite{AraMaltsiJoint}. On rappelle en particulier que $\ooCat$ munie de
ce joint forme une catégorie localement bifermée au sens de la
section~\ref{sec:thmA_abs} et que les adjoints à droite associés à cette
structure sont les tranches \oo-catégoriques. Enfin, on rappelle les
résultats de sesquifonctorialité des tranches obtenus
dans~\cite{AraMaltsiJoint}.

L'objet de la cinquième section est de démontrer notre théorème A
\oo-catégorique dans le cas d'un triangle commutatif. On commence par
montrer comment le joint \oo-catégorique induit l'objet cosimplicial $\cO :
\cDelta \to \ooCat$ des orientaux de Street et donc le nerf de Street et la
notion d'équivalence de Thomason. On introduit la notion de rétracte par
transformation oplax fort, analogue \oo-catégorique de la notion de rétracte par
déformation fort, et on étudie ses propriétés de stabilité par changement de
base. On observe, en utilisant des résultats de
l'appendice~\ref{app:transformation}, que ces rétractes sont des
équivalences de Thomason. On montre que le $m$-ième oriental $\On{m}$ se
rétracte par transformation oplax sur son $m$-ième objet. On en déduit,
en utilisant les résultats de sesquifonctorialité des tranches rappelés dans
la section précédente, que $\ooCat$ munie du joint vérifie les hypothèses du
théorème A monoïdal. On obtient ainsi le théorème~A annoncé. On étudie les
interactions entre ce théorème A et les dualités de $\ooCat$. On montre par
ailleurs que les tranches \oo-catégoriques de la forme $\cotr{C}{c}$ se
rétractent par transformation oplax sur un objet et on en déduit une version
non relative du théorème~A \oo-catégorique. On termine la section par une
application de notre théorème A donnant une condition suffisante pour que le
nerf d'une \oo-catégorie soit faiblement contractile, résultat qui est
utilisé dans~\cite{AraMaltsiCondE}.

Dans la sixième section, on introduit la notion de \oo-catégorie comma,
généralisation \oo-catégorique des catégories comma classiques et des
$2$-catégories (ou $3$\nbd-catégories) comma de Gray. On montre que cette
construction est fonctorielle. On énonce un résultat de préservation des
rétractes par transformation oplax forts par la construction comma qui découlera
de l'appendice~\ref{app:tr_comma} et jouera un rôle central dans la section
suivante.

La septième section est consacrée au théorème A pour les $2$-triangles,
c'est-à-dire les triangles commutatifs à une transformation oplax près. On
montre comment, en utilisant la fonctorialité de la construction comma et le
fait qu'elle préserve les rétractes par transformation oplax forts, on peut
déduire le théorème A \oo-catégorique pour les $2$-triangles du théorème A
\oo-catégorique pour les triangles commutatifs, non seulement pour les
équivalences de Thomason mais également pour des classes de \oo-foncteurs
plus générales vérifiant des axiomes adéquats. On termine la section par une
étude des interactions entre le théorème A pour les $2$-triangles et les
dualités de $\ooCat$.

Dans l'appendice~\ref{app:transformation}, on associe à toute transformation
oplax une homotopie simpliciale. Ainsi, on obtient que les nerfs de deux
\oo-foncteurs source et but d'une transformation oplax sont homotopes et
donc que les rétractes par transformation oplax forts sont des équivalences
de Thomason.
On vérifie que, lorsque la
transformation est stricte, l'homotopie associée n'est autre que le nerf de
Street de la transformation.

L'appendice~\ref{app:tr_comma} est dédié aux propriétés de
sesquifonctorialité de la construction comma \oo-catégorique introduite dans
la section~\ref{sec:comma_1}. On définit la notion de \oo-catégorie de Gray
et la notion de contrainte de Gray associée à deux cellules composables
horizontalement dans une \oo-catégorie de Gray. On étudie les propriétés de
ces contraintes de Gray. Ceci nous permet de définir, pour $\GrayC$ une
\oo-catégorie de Gray et $c$ un objet de~$\GrayC$, une sesquicatégorie
tranche $\tr{\GrayC}{c}$, ainsi que, par dualité, une variante
$\smash{\trto{\GrayC}{c}}$ de cette sesquicatégorie tranche. En appliquant
ces constructions à $\ooCatOpLaxGray$, la \oo-catégorie de Gray
des \oo-catégories strictes, \oo-foncteurs stricts et $i$-transformations
oplax pour~$i \ge 1$, on obtient, pour $Z$ une \oo-catégorie stricte, des
sesqui\-catégories~$\tr{\ooCatOpLaxGray}{Z}$
et~$\smash{\trto{\ooCatOpLaxGray}{Z}}$. On montre
que la construction comma s'étend en un sesquifoncteur
\[ \commaCfun : \Spanoo \to \ooCatOpLax, \]
où $\ooCatOpLax$ désigne la sesquicatégorie des \oo-catégories strictes,
\oo-foncteurs stricts et transformations oplax. On en déduit le résultat de
préservation des rétractes par transformation oplax forts par la construction
comma annoncé dans la section~\ref{sec:comma_1} et utilisé dans la
section~\ref{sec:thmA_2-tri}.

Enfin, dans l'appendice~\ref{app:thmAI}, on fait le lien entre la preuve du
théorème A \oo-catégorique présentée dans ce texte et celle de notre
précédent article~\cite{AraMaltsiThmAI}. Plus précisément, on montre que
la section simpliciale et l'homotopie simpliciale utilisées pour prouver
ce théorème dans ce précédent article proviennent de la section
\oo-catégorique et de la transformation oplax utilisées pour prouver ce même
théorème dans le présent texte.

\mysubsection*{Remerciements}

Les auteurs remercient vivement le rapporteur anonyme pour sa relecture
attentive et ses nombreuses remarques qui ont grandement amélioré la
qualité de ce texte. En particulier, c'est lui qui a suggéré l'utilisation
de la diagonale d'Alexander-Whitney dans
l'appendice~\ref{app:transformation}.

\section{Préliminaires simpliciaux : le théorème A simplicial}

\begin{paragr}
  On notera $\cDelta$ la \ndef{catégorie des simplexes}. Rappelons que ses
  objets sont les ensembles ordonnés
  \[
    \Deltan{m} = \{0, \dots, m\}, \quad{\text{pour $m \ge 0$}},
  \]
  et ses morphismes les applications croissantes (au sens large) entre tels
  ensembles ordonnés. De même, on notera $\cDeltaAug$ la \ndef{catégorie des
  simplexes augmentée}, c'est-à-dire la catégorie obtenue à partir de
  $\cDelta$ en ajoutant l'ensemble ordonné vide $\Deltan{-1}$.

  La catégorie des \ndef{ensembles simpliciaux}, c'est-à-dire des
  préfaisceaux sur $\cDelta$, sera notée $\pref{\cDelta}$. On considérera le
  foncteur de Yoneda $\cDelta \hookto \pref{\cDelta}$ comme une inclusion.
  Si $X$ est un ensemble simplicial, pour $m \ge 0$, on notera $X_m$
  l'ensemble $X(\Deltan{m})$ de ses \ndef{$m$-simplexes}.

  Soit $x$ un $m$-simplexe d'un ensemble simplicial $X$.
  Pour $I = \{i_0 < \dots < i_p\}$ un sous-ensemble de $\Deltan{m}$,
  on notera $x_{i_0, \dots, i_p}$ le $p$-simplexe
  \[
    x_{i_0, \dots, i_p} = X(f^{}_I)(x),
  \]
  où $f^{}_I : \Deltan{p} \to \Deltan{m}$ est l'application qui envoie $k$ sur
  $i_k$.
\end{paragr}

\begin{paragr}
  On appellera \ndef{équivalences faibles d'ensembles simpliciaux} les
  équivalences d'homotopie faibles d'ensembles simpliciaux, c'est-à-dire les
  morphismes dont la réalisation topologique est une équivalence
  d'homotopie. On dira qu'un ensemble simplicial $X$ est \ndef{faiblement
  contractile} si l'unique morphisme de $X$ vers l'ensemble
  simplicial final~$\Deltan{0}$ est une équivalence faible.
\end{paragr}

\begin{paragr}\label{paragr:def_D}
  On rappelle que la catégorie $\cDelta$ admet un unique automorphisme non
  trivial, automorphisme qui se trouve être une involution et que nous
  noterons $\DDelta : \cDelta \to \cDelta$. Explicitement, le foncteur
  $\DDelta$ est l'identité sur les objets et, si $f : \Deltan{m} \to
  \Deltan{n}$ est un morphisme de $\cDelta$, le morphisme $\DDelta(f)$ est
  donné par
  \[
    \DDelta(f)(i) = n - f(m - i), \qquad\text{pour $0 \le i \le m$}.
  \]
  Cet automorphisme s'étend de manière unique en un automorphisme
  $\DDeltaAug$ de $\cDeltaAug$ et celui-ci vérifie $\DDeltaAug(\Deltan{-1}) =
  \Deltan{-1}$.
  
  L'automorphisme $\DDelta$ induit un automorphisme involutif de
  $\pref{\cDelta}$ qui envoie un ensemble simplicial $X$ sur l'ensemble
  simplicial $X^\op = X \circ \DDelta$.

  On rappelle que $f$ est une équivalence faible simpliciale si et seulement si
  il en est de même de $f^\op$. On vérifie en effet immédiatement que les morphismes
  $f$ et $f^\op$ ont même réalisation topologique.
\end{paragr}

\begin{paragr}\label{paragr:bisimpl}
  On rappelle qu'un \emph{ensemble bisimplicial} est un
  préfaisceau sur $\cDelta \times \cDelta$. Si $X$ est un ensemble
  bisimplicial et $m, n$ sont deux entiers positifs, on notera
  $X_{m, n}$ l'ensemble $X(\Deltan{m}, \Deltan{n})$.

  On notera $p_1, p_2 : \cDelta \times \cDelta \to \cDelta$ les deux
  projections. Ces foncteurs induisent par précomposition des foncteurs
  $p_1^\ast, p_2^\ast : \pref{\cDelta} \to \pref{\cDelta \times \cDelta}$.
  Si $X$ est un ensemble simplicial, $p_1^\ast(X)$ et $p_2^\ast(X)$ sont les
  ensembles bisimpliciaux définis, pour $m$ et $n$ deux entiers positifs,
  par
  \[
    \begin{split}
      p_1^\ast(X)_{m, n} & = X_m, \\
      p_2^\ast(X)_{m, n} & = X_n.
    \end{split}
  \]

  On notera $\delta : \cDelta \to \cDelta \times \cDelta$ le foncteur
  diagonal. Celui-ci induit un foncteur $\delta^* : \pref{\cDelta \times
  \cDelta} \to \pref{\cDelta}$ qui envoie un ensemble bisimplicial $X$ sur
  l'ensemble simplicial~$\delta^*(X)$ défini par $\delta^*(X)_m = X_{m, m}$.
  On appellera \ndef{équivalence faible diagonale} un morphisme d'ensembles
  bisimpliciaux tel que $\delta^\ast(f)$ soit une équivalence
  faible d'ensembles simpliciaux.
\end{paragr}

On rappelle le lemme classique suivant :

\begin{lemme}\label{lemme:bisimpl}
  Soit $f : X \to Y$ un morphisme d'ensembles bisimpliciaux. Si pour tout $m
  \ge 0$, le morphisme $f_{m, \bullet} : X_{m, \bullet} \to Y_{m,
  \bullet}$ est une équivalence faible d'ensembles simpliciaux, alors $f$
  est une équivalence faible diagonale.
\end{lemme}

\begin{proof}
  Voir par exemple \cite[Chapitre XII, paragraphe 4.3]{BousKan} ou
  \cite[proposition 2.1.7]{CisinskiLFM} pour une preuve plus moderne.
\end{proof}

\begin{paragr}
  Soient $g : X \to Z$ et $h : Y \to Z$ deux morphismes d'ensembles
  simpliciaux. On définit un ensemble bisimplicial $g \downarrow h$ en posant
  \[
    \begin{split}
      (g \downarrow h)_{m, n} = & \{ (x \in X_m, y \in Y_n, z \in Z_{m + 1 + n})
        \mid \\
        & \phantom{=1}\qquad  z_{0, \dots, m} = g(x) \text{ et }
      z_{m+1, \dots, m + 1 + n} = h(y)\},
    \end{split}
  \]
  les opérations simpliciales étant définies de la manière évidente. On a des
  morphismes canoniques
  \[
    p^\ast_1(X) \leftarrow g \downarrow h \to p^\ast_2(Y)
  \]
  définis par
  \[ x \mapsfrom (x, y, z) \mapsto y. \]

  On notera respectivement $g \downarrow Z$ et~$Z \downarrow h$ les ensembles
  bisimpliciaux $g \downarrow \id{Z}$ et~$\id{Z} \downarrow h$.
  Notons que dans ces cas, les définitions se simplifient (à isomorphisme
  près) en
  \[
    \begin{split}
      (g \downarrow Z)_{m, n} & =  \{(x \in X_m, z \in Z_{m+1+n}) \mid z_{0, \dots,
      m} = g(x) \}, \\
      (Z \downarrow h)_{m, n} & = \{(y \in Y_n, z \in Z_{m+1+n}) \mid z_{m+1, \dots,
      m + 1 + n} = h(y) \}.
    \end{split}
  \]
\end{paragr}

\begin{paragr}\label{paragr:def_tr_simpl}
  Soient $g : X \to Z$ un morphisme d'ensembles simpliciaux et $m \ge 0$ un
  entier. L'application canonique
  \[
    Z \downarrow g  \to p^\ast_1(Z)
  \]
  induit un morphisme d'ensembles simpliciaux
  \[
    (Z \downarrow g)_{m, \bullet} \to p^\ast_1(Z)_{m, \bullet} = Z_m,
  \]
  où $Z_m$ désigne l'ensemble simplicial constant associé à l'ensemble $Z_m$.

  Si $z$ est un $m$-simplexe de $Z$, on notera $\cotr{X}{z}$ la fibre du
  morphisme ci-dessus en~$z$. On appellera $\cotr{X}{z}$ la \ndef{tranche de
  $X$ au-dessous de $z$}. Explicitement, les $n$-simplexes de~$\cotr{X}{z}$
  sont donnés par
  \[
    \begin{split}
      (\cotr{X}{z})_n & = \{(x \in X_n, z' \in Z_{m+1+n}) \mid \\
      & \phantom{=1}\qquad z'_{0, \dots, m} = z \text{ et } z'_{m+1, \dots, m+1+n} = g(x)\}.
    \end{split}
  \]
  Par définition, l'ensemble simplicial $(Z \downarrow g)_{m,
  \bullet}$ se décompose en
  \[
    (Z \downarrow g)_{m, \bullet} = \coprod_{z \in Z_m} \cotr{X}{z}.
  \]
  Par ailleurs, le morphisme $Z \downarrow g \to p_2^\ast(X)$ induit un
  morphisme canonique
  \[ \cotr{X}{z} \to X \]
  donné explicitement par $(x, z') \mapsto x$.

  Un cas particulièrement important de tranche est celui où $X = Z$ et $g$
  est l'identité de $Z$. Ainsi, si $Z$ est un ensemble simplicial et $z$ est
  un $m$-simplexe de~$Z$, on obtient un ensemble simplicial $\cotr{Z}{z}$.
\end{paragr}

\begin{remark}
  L'ensemble simplicial $\cotr{X}{z}$ peut se définir de manière plus
  conceptuelle à partir du joint simplicial (voir le paragraphe
  \ref{paragr:joint_simpl}).
\end{remark}

\begin{paragr}\label{paragr:funct_tr_simpl}
  Soit
  \[
    \xymatrix@C=1.5pc{
      X \ar[rr]^f \ar[dr]_g & & Y \ar[dl]^h \\
      & Z
    }
  \]
  un triangle commutatif de morphismes d'ensembles simpliciaux. On définit un
  morphisme $Z \downarrow f$ d'ensembles bisimpliciaux
  \[
    Z \downarrow f : Z \downarrow g \to Z \downarrow h
  \]
  en envoyant $(x, z)$ sur $(f(x), z)$. (On prendra garde que la notation $Z
  \downarrow f$ est ambiguë dans le cas $Y = Z$ puisqu'elle désigne à la
  fois un ensemble bisimplicial et un morphisme d'ensembles bisimpliciaux.
  Nous avons pris soin de toujours indiquer clairement quel objet la
  notation désigne dans la suite du texte.)

  On obtient un triangle
  \[
    \xymatrix@C=1.5pc{
      Z \downarrow g \ar[rr]^{Z \downarrow f} \ar[dr] & & Z \downarrow h \ar[dl] \\
      & p_1^\ast(Z)
    }
  \]
  de morphismes bisimpliciaux dont on vérifie immédiatement la
  commutativité. Si $z$ est un $m$-simplexe de $Z$, en prenant la fibre
  au-dessus de $z$ du morphisme $Z \downarrow f$,  on obtient un morphisme
  d'ensembles simpliciaux
  \[
    \cotr{f}{z} : \cotr{X}{z} \to \cotr{Y}{z}
  \]
  qui, explicitement, envoie $(x, z')$ sur~$(f(x),
  z')$. Par définition, on a
  \[
    (Z \downarrow f)_{m, \bullet} = \coprod_{z \in Z_m} \cotr{f}{z} :
    \coprod_{z \in Z_m} \cotr{X}{z}
    \to
    \coprod_{z \in Z_m} \cotr{Y}{z}.
  \]
\end{paragr}

\begin{paragr}\label{paragr:tr_simpl_prod_fib}
  Soient $g : X \to Z$ un morphisme d'ensembles simpliciaux et $z$ un
  $m$-simplexe de $Z$. En vertu du paragraphe précédent, en considérant $g$
  comme un morphisme au-dessus de $Z$, on obtient un morphisme $\cotr{g}{z}
  : \cotr{X}{z} \to \cotr{Z}{z}$. On vérifie immédiatement que le carré
  \[
    \xymatrix{
      \cotr{X}{z} \ar[d]_{\cotr{g}{z}} \ar[r] & X \ar[d]^g \\
      \cotr{Z}{z} \ar[r] & Z
    }
  \]
  est cartésien. Autrement dit, on a
  \[
    \cotr{X}{z} = (\cotr{Z}{z}) \times_Z X.
  \]
\end{paragr}

\begin{paragr}\label{paragr:tr_dual_simpl}
  Soient $g : X \to Z$ un morphisme d'ensembles simpliciaux et $z$ un
  $n$-simplexe de $Z$. On peut définir, de manière similaire à la
  définition des tranches au-dessous, un ensemble simplicial
  $\tr{X}{z}$ \ndef{tranche de $X$ au-dessus de $z$} en considérant les fibres
  du morphisme canonique
  \[
    g \downarrow Z  \to p^\ast_2(Z).
  \]
  On peut également définir $\tr{X}{z}$ à partir de $\cotr{X}{z}$ en utilisant
  la dualité $X \mapsto X^\op$. En effet, on a un morphisme $g^\op : X^\op \to
  Z^\op$ et le $n$-simplexe $z$ de $Z$ peut-être vu comme un $n$-simplexe de
  $Z^\op$. On peut donc considérer l'ensemble simplicial $\cotr{X^\op}{z}$.
  On a alors
  \[ \tr{X}{z} = \big(\cotr{X^\op}{z}\big){}^\op. \]
  Explicitement, pour $m \ge 0$, on a
  \[
    \begin{split}
      (\tr{X}{z})_m & = \{(x \in X_m, z' \in Z_{m+1+n}) \mid \\
        & \phantom{=1}\qquad z'_{0, \dots, m} = g(x) \text{ et } z'_{m+1, \dots,
      m+1+n} = z \}.
    \end{split}
  \]
  Si
  \[
    \xymatrix@C=1.5pc{
      X \ar[rr]^f \ar[dr]_g & & Y \ar[dl]^h \\
      & Z
    }
  \]
  est un triangle commutatif de morphismes d'ensembles simpliciaux et si $z$
  est toujours un $n$-simplexe de $Z$, on définit un morphisme
  \[ \tr{f}{z} : \tr{X}{z} \to \tr{Y}{z} \]
  en posant $\tr{f}{z} = \big(\cotr{f^\op}{z}\big){}^\op$. Explicitement,
  ce morphisme envoie $(x, z')$ sur $(f(x), z')$.
\end{paragr}

\begin{prop}\label{prop:tr_contr}
  Soient $X$ un ensemble simplicial et $x$ un $n$-simplexe de $X$. Alors
  l'ensemble simplicial $\tr{X}{x}$ est contractile.
\end{prop}

\begin{proof}
  Cela résulte de \cite[chapitre 6, proposition 1.4]{Illusie}. Voir
  également \cite[lemme 2.4]{AraMaltsiThmAI} pour une preuve
  élémentaire.
\end{proof}

\begin{paragr}
  Si
  \[
    \xymatrix@C=1.5pc{
      X \ar[rr]^f \ar[dr]_g & & Y \ar[dl]^h \\
      & Z
    }
  \]
  est un triangle commutatif de morphismes d'ensembles simpliciaux, on
  dispose d'un carré
  \[
    \xymatrix@C=2.5pc{
      Z \downarrow g \ar[r]^{Z \downarrow f} \ar[d] & Z \downarrow h \ar[d] \\
      p_2^\ast(X) \ar[r]_{p_2^\ast(f)} & p_2^\ast(Y) \\
    }
  \]
  de morphismes bisimpliciaux dont on vérifie immédiatement la commutativité.

  Le lemme suivant affirme que les flèches verticales de ce carré sont
  des équivalences faibles diagonales :
\end{paragr}

\begin{lemme}\label{lemme:morph_can_eq}
  Soit $g : X \to Z$ un morphisme d'ensembles simpliciaux. Alors le morphisme
  canonique $Z \downarrow g \to p^\ast_2(X)$ est une équivalence faible diagonale.
\end{lemme}

\begin{proof}
  En vertu du lemme \ref{lemme:bisimpl}, il suffit de vérifier que, pour tout
  $n \ge 0$, le morphisme d'ensembles simpliciaux
  \[
    (Z \downarrow g)_{\bullet, n} \to p^\ast_2(X)_{\bullet, n} = X_n,
  \]
  où $X_n$ désigne l'ensemble simplicial constant associé à l'ensemble $X_n$,
  est une équivalence faible. Puisque les équivalences faibles simpliciales
  sont stables par somme, il suffit donc de vérifier que les fibres de ce
  morphisme sont faiblement contractiles. Soit donc $x$ un élément de $X_n$.
  L'ensemble des $m$-simplexes de la fibre en $x$ est
  \[
    \begin{split}
      \MoveEqLeft \{(x' \in X_n, z \in Z_{m+1+n}) \mid x' = x \text{ et } z_{m+1, \dots,
      m+1+n} = g(x')  \} \\
      & \simeq \{z \in Z_{m+1+n} \mid z_{m+1, \dots, m+1+n} = g(x)  \}.
    \end{split}
  \]
  Cette fibre s'identifie ainsi à l'ensemble simplicial $\tr{Z}{g(x)}$. Or
  celui-ci est faiblement contractile en vertu de la
  proposition~\ref{prop:tr_contr}, ce qui achève la démonstration.
\end{proof}

\begin{prop}\label{prop:thmA_simpl_pre}
  Si
  \[
    \xymatrix@C=1.5pc{
      X \ar[rr]^f \ar[dr]_g & & Y \ar[dl]^h \\
      & Z
    }
  \]
  est un triangle commutatif de morphismes d'ensembles simpliciaux, alors le
  morphisme $f$ est une équivalence faible si et seulement si
  \hbox{$Z \downarrow f : Z \downarrow g \to Z \downarrow h$} est une
  équivalence faible diagonale.
\end{prop}

\begin{proof}
  En vertu du lemme précédent, les flèches verticales du carré
  commutatif
  \[
    \xymatrix@C=2.5pc{
      Z \downarrow g \ar[r]^{Z \downarrow f} \ar[d] & Z \downarrow h \ar[d] \\
      p_2^\ast(X) \ar[r]_{p^\ast_2(f)} & p_2^\ast(Y) \\
    }
  \]
  sont des équivalences faibles diagonales. Par deux sur trois, on en déduit
  que le morphisme $p_2^\ast(f) : p_2^\ast(X) \to p_2^\ast(Y)$ est une
  équivalence faible diagonale si et seulement si il en est de même de $Z
  \downarrow f : Z \downarrow g \to Z \downarrow h$. Or
  $\delta^\ast(p_2^\ast(f))$ n'est autre que $f$ et dire que $p_2^\ast(f)$
  est une équivalence faible diagonale signifie précisément que $f$ est une
  équivalence faible, ce qui achève la démonstration.
\end{proof}

\begin{thm}[Théorème A simplicial]\label{thm:thmA_simpl}
  Soit
  \[
    \xymatrix@C=1.5pc{
      X \ar[rr]^f \ar[dr]_g & & Y \ar[dl]^h \\
      & Z
    }
  \]
  un triangle commutatif de morphismes d'ensembles simpliciaux. Si pour tout
  $m \ge 0$ et tout $m$-simplexe $z$ de $Z$, le morphisme $\cotr{f}{z} :
  \cotr{X}{z} \to \cotr{Y}{z}$ est une équivalence faible, alors $f$ est une
  équivalence faible.
\end{thm}

\begin{proof}
  En vertu de la proposition précédente, il s'agit de montrer
  que le morphisme $Z \downarrow f : Z \downarrow g \to Z \downarrow h$ est
  une équivalence faible diagonale. Pour cela, il suffit de montrer, en
  vertu du lemme~\ref{lemme:bisimpl}, que, pour tout $m \ge 0$, le morphisme
  $(Z \downarrow f)_{m, \bullet}$ est une équivalence faible. Or, on a vu
  au paragraphe~\ref{paragr:funct_tr_simpl} que ce morphisme s'identifie au
  morphisme
  \[
    \coprod_{z \in Z_m} \cotr{f}{z} :
    \coprod_{z \in Z_m} \cotr{X}{z}
    \to
    \coprod_{z \in Z_m} \cotr{Y}{z}.
  \]
  Puisque les équivalences faibles simpliciales sont stables par somme,
  l'hypothèse de l'énoncé entraîne que ce morphisme est bien une équivalence
  faible, ce qui achève la démonstration.
\end{proof}

\begin{coro}\label{coro:thmA_simpl_dual}
  Soit
  \[
    \xymatrix@C=1.5pc{
      X \ar[rr]^f \ar[dr]_g & & Y \ar[dl]^h \\
      & Z
    }
  \]
  un triangle commutatif de morphismes d'ensembles simpliciaux. Si pour tout
  $n \ge 0$ et tout $n$-simplexe $z$ de $Z$, le morphisme $\tr{f}{z} :
  \tr{X}{z} \to \tr{Y}{z}$ est une équivalence faible, alors $f$ est une
  équivalence faible.
\end{coro}

\begin{proof}
  Soit $z$ un $n$-simplexe de $Z$. Puisque, par hypothèse, le morphisme
  $\tr{f}{z} : \tr{X}{z} \to \tr{X}{z}$ est une équivalence faible, il en
  est de même de $\big(\tr{f}{z}\big){}^\op$. Or, en vertu du
  paragraphe~\ref{paragr:tr_dual_simpl}, celui-ci s'identifie à
  $\cotr{f^\op}{z}$. Ainsi, on peut appliquer le théorème précédent à
  $f^\op$. On en déduit que $f^\op$ est une équivalence faible et donc qu'il
  en est de même de $f$, ce qu'on voulait démontrer.
\end{proof}

\section{Deux théorèmes A abstraits}
\label{sec:thmA_abs}

\begin{paragr}\label{paragr:def_thmA}
  Soit $\C$ une catégorie munie d'un objet cosimplicial $\cO : \cDelta \to
  \C$.  Pour $n \ge 0$, on pose $\On{n} = \cO(\Deltan{n})$. On notera $N : \C
  \to \EnsSimp$ le \ndef{foncteur nerf} associé défini par
  \[
    X \mapsto \left(\Deltan{n} \mapsto \Hom_\C(\On{n}, X)\right).
  \]
  On appellera \ndef{équivalences faibles} de $\C$ les morphismes de $\C$
  dont le nerf, c'est-à-dire l'image par $N$, est une équivalence faible
  simpliciale.

  On dira que l'objet cosimplicial $\cO : \C \to \EnsSimp$ \ndef{permet un
  théorème A} si, pour tout morphisme $g : X
  \to Z$ de $\C$, tout $m \ge 0$ et tout $m$-simplexe $z$ de $N(Z)$, le
  morphisme~$\cotr{N(X)}{z} \to \cotr{N(X)}{z_m}$, défini sur les
  $n$-simplexes par $(x, z') \mapsto (x, z'_{m, \dots, m+1+n})$, est une
  équivalence faible
  simpliciale.
\end{paragr}

\begin{thm}[Théorème A cosimplicial]\label{thm:thmA_cosimpl}
  Fixons une catégorie $\C$ munie d'un objet cosimplicial $\cO : \cDelta \to \C$
  permettant un théorème A. Soit
  \[
    \xymatrix@C=1.5pc{
      X \ar[rr]^f \ar[dr]_g & & Y \ar[dl]^h \\
      & Z
    }
  \]
  un triangle commutatif dans $\C$. Si pour tout $0$-simplexe $z$ de $N(Z)$,
  le morphisme
  \[ \cotr{N(f)}{z} : \cotr{N(X)}{z} \to \cotr{N(Y)}{z} \]
  est une équivalence faible simpliciale, alors $f$ est une
  équivalence faible de~$\C$.
\end{thm}

\begin{proof}
  Le triangle commutatif
  \[
    \xymatrix@C=1.5pc{
      X \ar[rr]^f \ar[dr]_g & & Y \ar[dl]^h \\
      & Z
    }
  \]
  induit un triangle commutatif d'ensembles simpliciaux
  \[
    \xymatrix@C=1.5pc{
      N(X) \ar[rr]^{N(f)} \ar[dr]_{N(g)} & & N(Y) \ar[dl]^{N(h)} \\
      & N(Z)
    }
  \]
  auquel on va appliquer le théorème~\ref{thm:thmA_simpl}. Pour
  conclure, il suffit donc de montrer que, pour tout $m \ge 0$ et tout
  $m$-simplexe $z$ de $N(Z)$, le morphisme
  \[
    \cotr{N(f)}{z} : \cotr{N(X)}{z} \to \cotr{N(Y)}{z}
  \]
  est une équivalence faible simpliciale. Considérons le carré de
  morphismes d'ensembles simpliciaux
  \[
    \xymatrix{
      \cotr{N(X)}{z} \ar[r] \ar[d] &\cotr{N(Y)}{z} \ar[d] \\
      \cotr{N(X)}{z_m} \ar[r] & \cotr{N(Y)}{z_m}
    }
  \]
  dont on vérifie immédiatement la commutativité. Puisque l'objet
  cosimplicial $\cO$ permet un théorème A, les morphismes
  verticaux de ce carré sont des équivalences faibles. Le morphisme
  $\cotr{N(X)}{z_m} \to \cotr{N(Y)}{z_m}$ étant une équivalence faible par
  hypothèse, on conclut par deux sur trois.
\end{proof}

\emph{Dans la suite de cette section, on fixe une catégorie monoïdale $\C$ de produit
monoïdal noté~$\joint$ et d'unité un objet initial $\vide$ de $\C$.}

\begin{paragraph}\label{paragr:permet_thmA_mon}
  Supposons que $\C$ admette un objet final $e$. L'objet $e$
  admet une et une seule structure de monoïde dans $(\C, \joint, \vide)$. En
  effet, il existe d'uniques morphismes $e \joint e \to e$ et $\vide \to e$
  et ceux-ci vérifient trivialement les axiomes des monoïdes.
  En vertu de la propriété universelle de la catégorie des simplexes
  augmentés $\cDeltaAug$ \cite[chapitre VII, section 5]{MacLane}, ce monoïde
  induit un foncteur monoïdal (défini à unique isomorphisme monoïdal près)
  $\cOAug : \cDeltaAug \to \C$, où $\cDeltaAug$ est munie du produit
  monoïdal défini par la somme ensembliste
  \[ \Deltan{m} \amalg \Deltan{n} = \Deltan{m + 1 + n}, \]
  l'unité étant $\Deltan{-1}$. Ce foncteur monoïdal est caractérisé (à
  unique isomorphisme monoïdal près) par le fait que $\cOAug(\Deltan{0}) =
  e$. En restreignant $\cOAug$ à $\cDelta$, on obtient donc un objet
  cosimplicial $\cO : \cDelta \to \C$. Explicitement, $\On{n} =
  \cO(\Deltan{n})$ est égal à $e \joint \cdots \joint e$, où $e$ apparaît $n
  + 1$ fois (pour un certain choix de parenthésage).

  Ainsi, si $\C$ admet un objet final, on dispose d'un objet cosimplicial
  canonique et on est en position d'utiliser les définitions du début de la
  présente section. En particulier, on dispose d'un foncteur nerf $N : \C
  \to \pref{\cDelta}$ et d'équivalences faibles de $\C$, morphismes
  dont le nerf est une équivalence faible simpliciale.

  On dira que la catégorie monoïdale $\C$ \ndef{permet un théorème $A$} si
  l'objet cosimplicial $\cO : \cDelta \to \C$ permet un théorème $A$ au sens
  du paragraphe~\ref{paragr:def_thmA}.
\end{paragraph}

\begin{paragraph}\label{paragr:def_tr}
  Soient $X$ et $Y$ deux objets de $\C$. On dispose de morphismes
  \[ X \xto{\iota_1} X \joint Y \xot{\iota_2} Y \]
  définis par
  \[ X \simeq X \joint \vide \xto{X \joint \vide_Y} X \joint Y \xot{\,\vide_X
  \joint Y} \vide \joint Y \simeq Y, \]
  où $\vide_Z$, pour $Z$ un objet de $\C$, désigne l'unique morphisme $\vide
  \to Z$.

  On en déduit l'existence d'un foncteur
  \[
    \begin{split}
      \C & \to \cotr{\C}{X} \\
      Y & \mapsto (X \joint Y, \iota_1 : X \to X \joint Y).
    \end{split}
  \]
  On dira que la catégorie monoïdale $\C$ est \ndef{localement fermée à
  gauche} si ce foncteur admet un adjoint à droite. Dans ce cas, on dispose
  d'un foncteur
  \[
    \begin{split}
      \cotr{\C}{X} \quad & \to \quad \C \\
      (Z, g : X \to Z) & \mapsto \cotr{Z}{g}
    \end{split}
  \]
  et de bijections
  \[
  \Hom_{\cotr{\C}{X}}((X \joint Y, \iota_1), (Z, g)) \simeq \Hom_\C(Y, \cotr{Z}{g}),
  \]
  naturelles en $Y$ dans $\C$ et $(Z, g)$ dans $\cotr{\C}{X}$. On appelle
  $\cotr{Z}{g}$ la \ndef{tranche de $Z$ au-dessous de $g$}.

  De même, on dira que $\C$ est \ndef{localement fermée à droite} si le
  foncteur
  \[
    \begin{split}
      \C & \to \cotr{\C}{Y} \\
      X & \mapsto (X \joint Y, \iota_2 : Y \to X \joint Y)
    \end{split}
  \]
  admet un adjoint à droite. Enfin, on dira que $\C$ est \ndef{localement
  bifermée} si elle est localement fermée à gauche et à droite.
\end{paragraph}

\begin{paragraph}\label{paragr:fonct_tr_sur_1}
  Supposons $\C$ localement fermée à gauche et fixons $Z$ un objet de $\C$. On
  dispose d'un foncteur
  \[ (\tr{\C}{Z})^\o \to \C \]
  défini sur les objets par
  \[
      (T, z : T \to Z) \mapsto \cotr{Z}{z}
  \]
  et sur les morphismes de la manière suivante. Soit
  \[
    \xymatrix@C=1.5pc{
      T_1 \ar[dr]_{z_1} \ar[rr]^l & & T_2 \ar[dl]^{z_2} \\
      & Z
    }
  \]
  un morphisme de $(T_1, z_1)$ vers $(T_2, z_2)$ dans $\tr{\C}{Z}$. Pour
  tout objet $S$ de $\C$, le composé
  \[
    \begin{split}
   \Hom_\C(S, \cotr{Z}{z_2}) & \xto{\sim}
   \Hom_{\cotr{\C}{T_2}}((T_2 \joint S, \iota_1), (Z, z_2)) \\
   & \qquad\qquad\qquad\qquad \big\downarrow \\
   & \phantom{\xto{\sim}''} \Hom_{\cotr{\C}{T_1}}((T_1 \joint S, \iota_1), (Z, z_1)) 
   \xto{\sim} \Hom_\C(S, \cotr{Z}{z_1}),
   \end{split}
  \]
  où la flèche verticale est induite par le morphisme $l \joint S : T_1
  \joint S \to T_2 \joint S$, est naturel en~$S$. On obtient ainsi, par le
  lemme de Yoneda, le morphisme $l^\ast : \cotr{Z}{z_2} \to \cotr{Z}{z_1}$
  associé. La fonctorialité de ce morphisme est immédiate.

  En particulier, en appliquant cette construction au triangle
  \[
    \xymatrix@C=1pc{
      \vide \ar[dr]_{\vide_Z} \ar[rr]^{\vide_X} & & X \ar[dl]^{g} \\
      & Z &
    }
  \]
  et en tenant compte de l'isomorphisme canonique $\cotr{Z}{\vide_Z} \simeq
  Z$, on obtient un morphisme $U : \cotr{Z}{g} \to Z$ qu'on appellera
  \ndef{morphisme d'oubli}.

  Par fonctorialité, le foncteur
  \[
    \begin{split}
      (\tr{\C}{Z})^\o & \to \C \\
      (T, z : T \to Z) & \mapsto \cotr{Z}{z}
    \end{split}
  \]
  se relève, le long du foncteur d'oubli $\tr{\C}{Z} \to \C$, en un
  foncteur
  \[
    \begin{split}
      (\tr{\C}{Z})^\o & \to \tr{\C}{Z} \\
      (T, z : T \to Z) & \mapsto (\cotr{Z}{z}, U),
    \end{split}
  \]
  où $U : \cotr{Z}{z} \to Z$ désigne le morphisme d'oubli que l'on vient de
  définir.
\end{paragraph}

\begin{paragraph}\label{paragr:fonct_tr_sur_2}
  On suppose toujours $\C$ localement fermée à gauche et on suppose de plus
  que $\C$ admet des produits fibrés. Pour $g : X \to Z$ et $z : T \to Z$ des
  morphismes de $\C$, on pose
  \[ \cotr{X}{z} = \cotr{Z}{z} \times_Z X, \]
  où le morphisme $\cotr{Z}{z} \to Z$ est le morphisme d'oubli. On appellera
  la seconde projection $\cotr{X}{z} \to X$ le \ndef{morphisme d'oubli}.

  À $z : T \to Z$ fixé, on obtient un foncteur
  \[
    \begin{split}
      \tr{\C}{Z} \quad & \to \quad \C \\
      (X, g : X \to Z) & \mapsto \cotr{X}{z}
    \end{split}
  \]
  en composant les foncteurs
  \[ \tr{\C}{Z} \to \tr{\C}{(\cotr{Z}{z})} \to \C, \]
  le foncteur de gauche étant le foncteur de changement de base le long du
  morphisme d'oubli $\cotr{Z}{z} \to Z$ et celui de droite le foncteur
  d'oubli.

  De même, à $g : X \to Z$ fixé, on obtient un foncteur
  \[
    \begin{split}
      (\tr{\C}{Z})^\o \quad & \to \quad \C \\
      (T, z : T \to Z) & \mapsto \cotr{X}{z}
    \end{split}
  \]
  en composant les foncteurs
  \[ (\tr{\C}{Z})^\o \to \tr{\C}{Z} \to \tr{\C}{X} \to \C, \]
  le foncteur de gauche étant le foncteur du paragraphe précédent, celui du
  milieu le foncteur de changement de base le long de $g : X \to Z$ et celui
  de droite le foncteur d'oubli.
  Si
  \[
    \xymatrix@C=1.5pc{
      T_1 \ar[dr]_{z_1} \ar[rr]^l & & T_2 \ar[dl]^{z_2} \\
      & Z
    }
  \]
  est un morphisme de $(T_1, z_1)$ vers $(T_2, z_2)$ dans $\tr{\C}{Z}$, on
  notera $l^\ast : \cotr{X}{z_2} \to \cotr{X}{z_1}$ le morphisme associé par
  ce foncteur.
\end{paragraph}

\begin{paragraph}\label{paragr:joint_simpl}
  Dans ce paragraphe, nous allons nous placer dans le cas où $\C$ est la
  catégorie $\pref{\cDelta}$ des ensembles simpliciaux. Rappelons la
  définition du joint simplicial (voir par exemple~\cite[section
  3]{JoyalQuasiKan}). La catégorie des simplexes augmentée $\cDeltaAug$ est
  munie d'une structure de catégorie monoïdale de produit monoïdal $\joint$
  induit par la somme ensembliste
  \[
    \Deltan{m} \joint \Deltan{n} =
    \Deltan{m} \amalg \Deltan{n} = \Deltan{m + 1 + n},
  \]
  et d'objet unité $\vide = \Deltan{-1}$. La catégorie $\cDeltaAug$ est
  naturellement une sous-catégorie pleine de $\pref{\cDelta}$ et on définit
  le \ndef{joint} $X \joint Y$ de deux ensembles simpliciaux $X$ et $Y$ par
  la formule
  \[
    X \joint Y = \limind_{\substack{\Deltan{m} \to X\\\,\Deltan{n} \to Y}}
    \Deltan{m + 1 + n},
  \]
  où $\Deltan{m}$ et $\Deltan{n}$ varient dans $\cDeltaAug$. On montre qu'on
  obtient ainsi une structure de catégorie monoïdale sur les ensembles
  simpliciaux d'objet unité l'ensemble simplicial vide. Par ailleurs, on
  vérifie que cette structure est localement bifermée. En particulier, si
  $X$ est un ensemble simplicial, on en déduit l'existence d'un foncteur
  \[
    \begin{split}
      \cotr{\pref{\cDelta}}{X} \quad & \to \quad \pref{\cDelta} \\
      (Z, g : X \to Z) & \mapsto \cotr{Z}{g}
    \end{split}
  \]
  et de bijections naturelles
  \[
  \Hom_{\cotr{\pref{\cDelta}}{X}}((X \joint Y, \iota_1), (Z, g)) \simeq
  \Hom_{\pref{\cDelta}}(Y, \cotr{Z}{g}).
  \]
  Soient $Z$ un ensemble simplicial et $z : \Deltan{m} \to Z$ un
  $m$-simplexe. En spécialisant la bijection ci-dessus au cas $X =
  \Deltan{m}$, $Y = \Deltan{n}$ et $g = z$, on obtient que les $n$-simplexes
  de~$\cotr{Z}{z}$ correspondent aux morphismes $x : \Deltan{m + 1 + n} \to
  Z$ rendant commutatif le triangle
  \[
    \xymatrix{
      \Deltan{m + 1 + n} \ar[r]^-x & Z \\
      \Deltan{m} \ar[u]^{\iota_1} \ar[ur]_-z & \pbox{,}
    }
  \]
  où $\iota_1$ désigne l'inclusion comme section initiale. C'est exactement
  la description de la tranche $\cotr{Z}{z}$ définie par des formules
  explicites dans le paragraphe~\ref{paragr:def_tr_simpl}. De plus, si $g :
  X \to Z$ est quelconque, on a $\cotr{X}{z} \simeq \cotr{Z}{z} \times_Z X$
  pour les tranches au sens du présent paragraphe comme pour celles
  au sens du paragraphe~\ref{paragr:def_tr_simpl} (voir le
  paragraphe~\ref{paragr:tr_simpl_prod_fib}). On en déduit que les tranches
  au sens de ces deux paragraphes coïncident.
\end{paragraph}

\begin{remark}
  L'objet cosimplicial associé à la catégorie des ensembles simpliciaux
  munie du joint s'identifie au foncteur de Yoneda $\cDelta \to
  \pref{\cDelta}$. Ainsi, le foncteur nerf associé est l'identité.
\end{remark}

\emph{Dans la suite de cette section, on suppose que $\C$ est localement fermée à
gauche et admet des limites projectives finies. On note $e$ un objet final
de $\C$.}

\begin{proposition}\label{prop:comp_nerf_tr}
  Soient $X$ un objet de $\C$ et $x : \On{m} \to X$ un $m$-simplexe
  de~$N(X)$. On a un isomorphisme canonique
  \[ N(\cotr{X}{x}) \simeq \cotr{N(X)}{x}, \]
  naturel en $x$ dans la catégorie des éléments $\tr{\cDelta}{N(X)}$ de $N(X)$.
\end{proposition}

\begin{proof}
  Pour tout $n \ge 0$, on a des isomorphismes canoniques
  \[
    \begin{split}
      N(\cotr{X}{x})_n &  \simeq \Hom_\C(\On{n}, \cotr{X}{x}) \\
    & \simeq \Hom_{\cotr{\C}{\On{m}}}((\On{m} \joint \On{n}, \iota_1), (X,
    x)) \\
    & \simeq \Hom_{\cotr{\C}{\On{m}}}((\On{m + 1 + n}, \iota_1), (X,
    x)) \\
    & \simeq \Hom_{\cotr{\pref{\cDelta}}{\Deltan{m}}}((\Deltan{m+1+n}, \iota_1),
    (N(X), x)) \\
    & \simeq \Hom_{\cotr{\pref{\cDelta}}{\Deltan{m}}}((\Deltan{m} \joint
    \Deltan{n}, \iota_1), (N(X), x)) \simeq (\cotr{N(X)}{x})_n,
    \end{split}
  \]
  d'où le résultat.
\end{proof}

\begin{remark}
  Si on suppose de plus que la catégorie $\C$ est cocomplète, le foncteur
  nerf $N : \C \to \pref{\cDelta}$ admet un adjoint à gauche $c :
  \pref{\cDelta} \to \C$. On peut alors vérifier que ce foncteur $c$ est
  monoïdal pour le joint simplicial et le produit monoïdal~$\joint$ de $\C$. On en
  déduit que la proposition précédente reste vraie si on remplace le
  $m$-simplexe $x : \On{m} \to X$ par un morphisme $c(Y) \to X$ quelconque,
  où $Y$ est un ensemble simplicial, le cas de la proposition précédente
  étant celui où $Y = \Deltan{m}$.
\end{remark}

\begin{corollary}\label{coro:tr_simpl_mon}
  Soient $X \to Z$ un morphisme de $\C$ et $z : \On{m} \to Z$ un
  $m$\nbd-simplexe de $N(Z)$. On a un isomorphisme canonique
  \[ N(\cotr{X}{z}) \simeq \cotr{N(X)}{z}, \]
  naturel en $z$ dans la catégorie des éléments $\tr{\cDelta}{N(Z)}$ de
  $N(Z)$ et en $X$ dans $\tr{\C}{Z}$.
\end{corollary}

\begin{proof}
  On a, en utilisant la proposition précédente et la commutation du nerf
  aux produits fibrés, des isomorphismes canoniques
  \[ N(\cotr{X}{z}) \simeq N(\cotr{Z}{z} \times_Z X) \simeq N(\cotr{Z}{z})
  \times_{N(Z)} N(X) \simeq \cotr{N(Z)}{z} \times_{N(Z)} N(X) \simeq
  \cotr{N(X)}{z},
  \]
  ce qui prouve l'assertion.
\end{proof}

\begin{proposition}\label{prop:def_equiv_permet_A}
  La catégorie monoïdale $\C$ permet un théorème A si et seulement si, pour
  tout morphisme $g : X \to Z$ de $\C$, tout $m \ge 0$ et tout morphisme $z
  : \On{m} \to Z$, le morphisme $m^\ast : \cotr{X}{z} \to
  \cotr{X}{z_m}$, associé en vertu du paragraphe~\ref{paragr:fonct_tr_sur_2} au
  triangle commutatif
  \[
    \xymatrix@C=1.5pc{
      \On{0} \ar[dr]_{z_m} \ar[rr]^m & & \On{m} \ar[dl]^{z} \\
      & Z & \pbox{,}
    }
  \]
  où $m$ désigne l'image par $\cO : \cDelta \to \C$ du morphisme $m :
  \Deltan{0} \to \Deltan{m}$, est une équivalence faible de $\C$.
\end{proposition}

\begin{proof}
  Par définition, la catégorie monoïdale $\C$ permet un théorème~A si, pour tout
  morphisme $g : X \to Z$ de $\C$ et tout $m$-simplexe $z : \On{m} \to Z$ de
  $N(Z)$, le morphisme $\cotr{N(X)}{z} \to \cotr{N(X)}{z_m}$, défini sur les
  $n$-simplexes par $(x, z') \mapsto (x, z'_{m, \dots, m+1+n})$, est une
  équivalence faible simpliciale. Or, ce morphisme n'est autre que le
  morphisme $m^\ast$, associé en vertu du
  paragraphe~\ref{paragr:fonct_tr_sur_2} au triangle
  \[
    \xymatrix@C=1.5pc{
      \Deltan{0} \ar[dr]_{z_m} \ar[rr]^m & & \Deltan{m} \ar[dl]^{z} \\
      & N(X) & \pbox{,}
    }
  \]
  pour $\C$ la catégorie des ensembles simpliciaux munie du joint. La
  naturalité de l'isomorphisme du corollaire~\ref{coro:tr_simpl_mon} donne
  donc un carré commutatif
  \[
    \xymatrix{
      N(\cotr{X}{z}) \ar[d]_*[@]{\sim} \ar[r] & N(\cotr{X}{z_m})
      \ar[d]^*[@]{\sim_{}} \\
      \cotr{N(X)}{z} \ar[r] & \cotr{N(X)}{z_m} \pbox{,}
    }
  \]
  où la flèche horizontale du haut est le nerf du morphisme de l'énoncé et
  celle du bas est le morphisme simplicial mentionné ci-dessus, ce qui
  entraîne le résultat.
\end{proof}

\begin{theorem}[Théorème A monoïdal]\label{thm:thmA_monoid}
  On suppose que la catégorie monoïdale $\C$ permet un théorème A.
  Soit
  \[
    \xymatrix@C=1.5pc{
      X \ar[rr]^f \ar[dr]_g & & Y \ar[dl]^h \\
      & Z
    }
  \]
  un triangle commutatif dans $\C$. Si pour tout morphisme $z : e \to Z$ de
  $\C$, le morphisme $\cotr{X}{z} \to \cotr{Y}{z}$
  induit par $f$ est une équivalence faible de $\C$, alors il en est de même
  du morphisme~$f$.
\end{theorem}

\begin{proof}
  Puisque la catégorie monoïdale $\C$ permet un théorème A, on peut
  appliquer le théorème A cosimplicial (théorème \ref{thm:thmA_cosimpl}) à
  l'objet cosimplicial associé. Pour montrer que $f$ est une équivalence
  faible, il suffit donc de montrer que l'hypothèse de ce théorème est
  satisfaite, à savoir que, pour tout $0$-simplexe $z$ de $N(Z)$, le
  morphisme $\cotr{N(X)}{z} \to \cotr{N(Y)}{z}$ est une équivalence faible
  simpliciale. Or, en vertu de la proposition~\ref{coro:tr_simpl_mon}, ce
  morphisme s'identifie au nerf du morphisme $\cotr{X}{z} \to \cotr{Y}{z}$
  de $\C$, qui est une équivalence faible par hypothèse, d'où le résultat.
\end{proof}

Dans la suite de cette section, on va déduire le théorème A originel de
Quillen \cite{QuillenHAKTI} du résultat précédent.

\begin{paragraph}\label{paragr:def_joint_1}
  Rappelons la définition du joint de deux catégories (voir par
  exemple~\cite[section 3.1]{JoyalQCatAppl}). Celui-ci peut se
  définir essentiellement comme on l'a fait pour les ensembles simpliciaux au
  paragraphe~\ref{paragr:joint_simpl}. La catégorie des simplexes augmentée
  $\cDeltaAug$ peut être considérée de manière évidente comme une
  sous-catégorie pleine de la catégorie $\Cat$ des petites catégories. Si
  $A$ et $B$ sont deux petites catégories, on peut alors définir leur
  \ndef{joint} $A \joint B$ par la formule
  \[
    A \joint B = \limind_{\substack{\Deltan{m} \to A\\\,\Deltan{n} \to B}}
    \Deltan{m + 1 + n},
  \]
  où $\Deltan{m}$ et $\Deltan{n}$ varient dans $\cDeltaAug$. La catégorie $A
  \joint B$ peut se décrire explicitement de la manière suivante : le graphe
  sous-jacent à $A \joint B$ est obtenu à partir du graphe sous-jacent à $A
  \amalg B$ en ajoutant, pour tout objet $a$ de $A$ et tout objet $b$ de
  $B$, une flèche de $a$ vers $b$ ; les identités et composés sont
  définis de la manière évidente. En particulier, la catégorie $\Deltan{0}
  \joint A$ est la catégorie obtenue à partir de $A$ en ajoutant librement
  un objet initial.  On montre qu'on obtient bien ainsi une structure de
  catégorie monoïdale sur $\Cat$ d'objet unité la catégorie
  vide. Par ailleurs, on vérifie que cette structure est localement
  bifermée. En particulier, si $A$ est une petite catégorie, on en déduit
  l'existence d'un foncteur 
  \[
    \begin{split}
      \cotr{\Cat}{A} \quad & \to \quad \Cat \\
      (C, v : A \to C) & \mapsto \cotr{C}{v}
    \end{split}
  \]
  et de bijections naturelles
  \[
    \Hom_{\cotr{\Cat}{A}}((A \joint B, \iota_1), (C, v)) \simeq
      \Hom_{\Cat}(B, \cotr{C}{v}).
  \]
  Soient $C$ une petite catégorie et $c : \Deltan{m} \to C$ un foncteur pour
  un entier $m \ge 0$, c'est-à-dire une suite de flèches
  \[ \xymatrix{ c_0 \ar[r]^{f_1} & c_1 \ar[r]^{f_2} & \cdots \ar[r]^{f_m} & c_m } \]
  de $C$. En spécialisant la bijection ci-dessus au cas $A = \Deltan{m}$, $B
  = \Deltan{0}$ et $v = c$, on obtient que les objets de $\cotr{C}{c}$
  correspondent aux foncteurs $x : \Deltan{m+1} \to C$ rendant commutatif le
  triangle
  \[
    \xymatrix{
      \Deltan{m + 1} \ar[r]^-x & C \\
      \Deltan{m} \ar[u]^{\iota_1} \ar[ur]_-c & \pbox{,}
    }
  \]
  où $\iota_1$ désigne l'inclusion comme section initiale, c'est-à-dire aux
  suites de flèches
  \[
    \xymatrix{
      c_0 \ar[r]^{f_1} & c_1 \ar[r]^{f_2} & \cdots \ar[r]^{f_m} 
        & c_m \ar[r]^{g} & d \pbox{.}
    }
  \]
  De même, on obtient que les flèches de $\cotr{C}{c}$ correspondent aux
  suites de flèches
  \[ 
    \xymatrix{
      c_0 \ar[r]^{f_1} & c_1 \ar[r]^{f_2} & \cdots
          \ar[r]^{f_m} & c_m \ar[r]^{g} & d \ar[r]^{h} & d' \pbox{,}
    }
  \]
  et qu'une telle flèche a pour source
  \[
    \xymatrix{
      c_0 \ar[r]^{f_1} & c_1 \ar[r]^{f_2} &
      \cdots \ar[r]^{f_m} & c_m \ar[r]^{g} & d
    }
  \]
  et pour but
  \[ 
    \xymatrix{ c_0 \ar[r]^{f_1} & c_1 \ar[r]^{f_2} &
      \cdots \ar[r]^{f_m} & c_m \ar[r]^{hg} & d' \pbox{.}
    }
  \]
  On vérifie enfin que les identités et les composés sont obtenus de la
  manière évidente. Il résulte de cette description que, d'une part, dans le
  cas $m = 0$ où $c$ correspond à un objet de $C$, la catégorie
  $\cotr{C}{c}$ est la tranche usuelle et, d'autre part, le foncteur
  $\cotr{C}{c} \to \cotr{C}{c_m}$ induit par le triangle commutatif 
  \[
    \xymatrix@C=1.5pc{
      \Deltan{0} \ar[dr]_{c_m} \ar[rr]^m & & \Deltan{m} \ar[dl]^{c} \\
                                         & C &
    }
  \]
  est un isomorphisme. On en déduit que, plus généralement, si $v : A \to C$
  est un foncteur, le foncteur $\cotr{A}{c} \to \cotr{A}{c_m}$ induit par ce
  même triangle est un isomorphisme.
\end{paragraph}

\begin{paragraph}
  Considérons toujours la catégorie $\Cat$ munie du joint catégorique.
  L'objet cosimplicial $\cDelta \to \Cat$ induit par cette structure de
  catégorie monoïdale est, essentiellement par définition, l'inclusion
  canonique. En particulier, le foncteur nerf associé $N : \Cat \to
  \pref{\cDelta}$ est le foncteur nerf usuel.

  On appellera \ndef{équivalences de Thomason} les équivalences faibles
  associées, c'est-à-dire les foncteurs qui sont envoyés par le nerf usuel
  sur une équivalence faible simpliciale.
\end{paragraph}

\begin{proposition}
  La catégorie $\Cat$ des petites catégories munie du joint catégorique
  permet un théorème A.
\end{proposition}

\begin{proof}
  Le paragraphe~\ref{paragr:def_joint_1} montre que le morphisme de la
  proposition~\ref{prop:def_equiv_permet_A} est un isomorphisme. C'est donc
  une équivalence faible et on conclut par cette même proposition.
\end{proof}

\begin{thm}[Théorème A de Quillen originel]
  \noindent
  Soit
  \[
    \xymatrix@C=1.5pc{
      A \ar[rr]^u \ar[dr]_v & & B \ar[dl]^w \\
      & C
    }
  \]
  un triangle commutatif de foncteurs entre petites catégories. Si pour tout
  objet $c$ de~$C$, le foncteur $\cotr{A}{c} \to \cotr{B}{c}$ est une
  équivalence de Thomason, alors il en est de même du foncteur~$u$.
\end{thm}

\begin{proof}
  Cela résulte du théorème A monoïdal (théorème~\ref{thm:thmA_monoid}),
  qu'on peut appliquer en vertu de la proposition précédente.
\end{proof}

\section{Rappels sur la théorie de Steiner}

Dans cette section, on rappelle brièvement quelques éléments de la théorie
des complexes dirigés augmentés de Steiner \cite{Steiner}, ainsi que
quelques compléments issus de~\cite{AraMaltsiJoint}. On renvoie à
\cite[chapitre 2]{AraMaltsiJoint} pour des rappels plus complets.

\begin{paragraph}
  On notera $\ooCat$ la catégorie des \oo-catégories \emph{strictes} et des
  \oo-foncteurs \emph{stricts} entre celles-ci. Sauf mention expresse du
  contraire, les \oo-catégories et les \oo-foncteurs considérés dans ce
  texte seront supposés stricts. Rappelons qu'une \oo-catégorie $C$ consiste
  en la donnée, pour tout $i \ge 0$, d'un ensemble $C_i$ de
  \ndef{$i$-cellules}. Si $x$ est une $i$-cellule pour un $i \ge 1$, on
  dispose de sa \ndef{source} $s(x)$ et de son \ndef{but} $t(x)$ qui sont
  deux $(i-1)$-cellules. Si $x$ est une $i$-cellule pour un $i \ge 0$, on
  dispose de son \ndef{identité}~$\id{x}$ qui est une $(i+1)$-cellule. Par
  ailleurs, pour $i > j \ge 0$ et $x, y$ deux $i$-cellules telles que
  la $j$-cellule source itérée de $x$ soit égale à la $j$-cellule but
  itéré de $y$, on dispose d'une $i$-cellule composée $x \comp_j y$. Ces
  données sont soumises à des axiomes qui doivent être vérifiés à égalité
  près (et non à des contraintes supérieures près).
\end{paragraph}

\begin{paragr}
  Un \ndef{complexe dirigé augmenté} est un complexe de chaînes de groupes
  abéliens en degrés positifs augmenté
  \[
    \xymatrix{\cdots \ar[r]^-{d_{i+1}} & K_i \ar[r]^-{d_i} & K_{i-1}
    \ar[r]^-{d_{i-1}} & \cdots \ar[r]^{d_2} & K_1 \ar[r]^{d_1} & K_0
    \ar[r]^e & \Z
    }
  \]
  muni, pour tout $i \ge 0$, d'un sous-monoïde $K^\ast_i$ de $K^{}_i$ qu'on
  appellera \ndef{sous-monoïde de positivité}. Un \ndef{morphisme de
  complexes dirigés augmentés} est un morphisme de complexes de chaînes
  augmentés qui envoie les sous-monoïdes de positivité de sa source dans les
  sous-monoïdes de positivité de son but. On obtient ainsi une catégorie
  qu'on notera $\Cda$.
\end{paragr}

\begin{paragr}
  À toute \oo-catégorie $C$, on associe un complexe dirigé augmenté
  $\lambda(C)$ de la manière suivante. Pour tout $i \ge 0$, le groupe
  abélien $\lambda(C)_i$ est engendré par des générateurs $[x]$, où
  $x$ varie parmi les $i$-cellules de $C$, soumis aux relations
  $[x \ast_j y] = [x] + [y]$, où $x$ et $y$ varient parmi les $i$-cellules
  de $C$ pour lesquelles le composé $x \comp_j y$ est défini. Le sous-monoïde de
  positivité $\lambda(C)_i^\ast$ est le sous-monoïde engendré par les $[x]$,
  où $x$ varie parmi les $i$-cellules de $C$. Pour $i > 0$, la
  différentielle $d_i : \lambda(C)_i \to \lambda(C)_{i-1}$ est définie par
  \[
    d_i([x]) = [t(x)] - [s(x)].
  \]
  Enfin, l'augmentation $e : \lambda(C)_0 \to \Z$ est définie par $e([x]) = 1$.

  On vérifie qu'on obtient ainsi un foncteur $\lambda : \ooCat \to \Cda$,
  l'action sur les morphismes étant définie de la manière évidente.
\end{paragr}

\begin{proposition}[Steiner]
  Le foncteur $\lambda : \ooCat \to \Cda$ admet un adjoint à droite $\nu :
  \Cda \to \ooCat$.
\end{proposition}

\begin{proof}
  Voir \cite[théorème 2.11]{Steiner}.
\end{proof}

\begin{paragr}
  Soit $K$ un complexe dirigé augmenté. Pour $i \ge 0$, les $i$-cellules de
  $\nu(K)$, où $\nu$ désigne le foncteur de la proposition précédente, sont
  les tableaux
  \[
    \tabld{x}{i}
  \]
  où
  \begin{enumerate}
    \item $x^\epsilon_k$ appartient à $K^\ast_i$ pour $\epsilon = 0, 1$ et $0
      \le k \le i$ ;
    \item $d(x^\epsilon_k) = x^1_{k-1} - x^0_{k-1}$ pour $\epsilon = 0, 1$ et $0
      < k \le i$ ;
    \item $e(x^\epsilon_0) = 1$ pour $\epsilon = 0, 1$ ;
    \item $x_i^0 = x_i^1$.
  \end{enumerate}
  Les opérations de la \oo-catégorie $\nu(K)$ peuvent se décrire
  aisément en termes de ces tableaux mais nous n'aurons pas besoin de
  cette description dans ce texte.
\end{paragr}

\begin{paragr}
  Une \ndef{base} d'un complexe dirigé augmenté $K$ est un ensemble gradué
  \hbox{$B = (B_i)_{i \ge 0}$} tel que, pour tout $i \ge 0$,
  \begin{enumerate}
    \item $B_i$ est une base du $\Z$-module $K_i$ ;
    \item $B_i$ engendre le sous-monoïde $K^\ast_i$ de $K^{}_i$.
  \end{enumerate}
  On vérifie que si un complexe dirigé augmenté admet une base, cette base
  est unique.
\end{paragr}

\begin{paragr}
  Soit $K$ un complexe admettant une base $B = (B_i)$. Si $x$ est un élément
  de~$K_i$, son \ndef{support} est l'ensemble des éléments de $B_i$ qui
  apparaissent avec un coefficient non nul dans la décomposition de $x$
  selon la base $B_i$. Il est immédiat que tout élément~$x$ de~$K_i$ se décompose
  de manière unique en $x = x_+ - x_-$, où $x_-$ et $x_+$ sont deux éléments
  de~$K^\ast_i$ à supports disjoints.
\end{paragr}

\begin{paragr}
  Soit $K$ un complexe dirigé augmenté admettant une base. Si $x$ est un
  élément de degré $i$ de la base de $K$, on lui associe un tableau
  \[
    \atom{x}=\tabll{\atom{x}}{i},
  \]
  où les $\atom{x}^\epsilon_k$ sont définis par récurrence descendante sur $k$
  de $i$ à $0$ :
  \begin{itemize}
    \item $\atom{x}^0_i = x = \atom{x}^1_i$ ;
    \item $\atom{x}^0_{k - 1} = d(\atom{x}^0_k)_-$ et $\atom{x}^1_{k - 1} =
      d(\atom{x}^1_k)_+$ pour $0 < k \le i$.
  \end{itemize}
  Ce tableau est une cellule de $\nu(K)$ si et seulement si $e(\atom{x}^0_0)
  = 1 = e(\atom{x}^1_0)$.

  On dit que la base $B$ de $K$ est \ndef{unitaire} si, pour tout $i \ge 0$ et
  tout $x$ dans $B_i$, le tableau $\atom{x}$ est une $i$-cellule de $\nu(K)$.
\end{paragr}

\begin{paragr}\label{paragr:def_le_N}
  Soit $K$ un complexe dirigé augmenté admettant une base $B$. On notera
  $\leN$ la plus petite relation de préordre sur $B$ pour laquelle, pour
  tout $i \ge 1$ et tout $x$ dans~$B_i$, si $y$ est dans le support de
  $d(x)_-$ et $z$ est dans le support de $d(x)_+$, on a~$y \leN x \leN
  z$.

  On dira que la base $B$ est \ndef{fortement sans boucle} si cette relation
  $\leN$ est une relation d'ordre.
\end{paragr}

\begin{paragr}
  On appellera \ndef{complexe de Steiner fort} un complexe dirigé augmenté
  admettant une base unitaire et fortement sans boucle.
\end{paragr}

\begin{thm}[Steiner]\label{thm:Steiner}
  La restriction du foncteur $\nu : \Cda \to \ooCat$ à la sous-catégorie
  pleine formée des complexes de Steiner forts est un foncteur pleinement
  fidèle.
\end{thm}

\begin{proof}
  Voir \cite[théorème 5.6]{Steiner}.
\end{proof}

La notion suivante, introduite dans \cite{AraMaltsiJoint}, joue un rôle
technique dans les fonctorialités du joint \oo-catégorique qu'on rappellera
dans la section suivante.

\begin{paragraph}\label{paragr:def_rig_ord}
  Soit $i : K \to L$ un monomorphisme entre des complexes dirigés augmentés
  admettant une base. On dira que $i$ est une \ndef{inclusion rigide
  ordonnée} si, d'une part, $i$~envoie tout élément de la base de $K$ sur un
  élément de la base de $L$ et, d'autre part, si $x$ et $y$ sont des
  éléments de la base de $K$, on a $x \leN y$ si et seulement si on a~$f(x) \leN
  f(y)$.
\end{paragraph}

Terminons cette section par quelques rappels sur les antihomotopies de
complexes dirigés augmentés.

\begin{paragr}\label{paragr:def_antihomot}
  Soient $f, g : K \to L$ deux morphismes de complexes dirigés augmentés.
  Une \ndef{antihomotopie} $h$ de $f$ vers $g$ consiste en la donnée,
  pour tout $i \ge 0$, de morphismes de groupes abéliens $h_i : K_i \to
  L_{i+1}$ envoyant $K^\ast_i$ dans $L^\ast_{i+1}$ et vérifiant, pour tout
  $i \ge 0$,
  \[ d_{i+1}h_i - h_{i-1}d_i = (-1)^i(g_i - f_i), \]
  en convenant que $h_{-1} = 0$ et $d_0 = 0$, de sorte qu'on a $d_1h_0 = g_0
  - f_0$ pour $i = 0$.

  Si maintenant $h$ et $k$ sont deux antihomotopies de $f$ vers $g$, une
  $2$-antihomotopie~$H$ de $h$ vers $k$ consiste en la donnée, pour tout $i
  \ge 0$, de morphismes de groupes abéliens $H_i : K_i \to L_{i+2}$ envoyant
  $K^\ast_i$ dans $L^\ast_{i+2}$ et vérifiant, pour tout $i \ge 0$,
  \[ d_{i + 2}H_i - H_{i-1}d_i = (-1)^i(k_i - h_i), \]
  en convenant que $H_{-1} = 0$ et $d_0 = 0$.
\end{paragr}

\begin{paragr}
  Soit $f : K \to L$ un morphisme de complexes dirigés augmentés. On notera
  $\id{f}$ l'antihomotopie de $f$ vers $f$ définie par $(\id{f})_i = 0$ pour
  tout $i \ge 0$. De même, si $h$ est une antihomotopie, on notera $\id{h}$
  la $2$-antihomotopie de $h$ vers $h$ définie par $(\id{h})_i = 0$ pour
  tout $i \ge 0$.

  On renvoie à la fin de la chapitre 2 de \cite{AraMaltsiJoint} pour les
  définitions de diverses opérations de composition pour les antihomotopies
  et les $2$-antihomotopies. Dans ce texte, nous manipulerons les
  antihomotopies et les $2$-antihomotopies comme des morphismes de groupes
  abéliens gradués (de degré $1$ et $2$) en les additionnant et les
  composant degré par degré. Nous renverrons à \cite{AraMaltsiJoint} pour
  le fait que les formules que nous utiliserons définissent bien des
  antihomotopies ou des $2$-antihomotopies.
\end{paragr}

\section{Préliminaires \pdfoo-catégoriques : produit tensoriel, joint et tranches}

Cette section est essentiellement un résumé de \cite{AraMaltsiJoint}, avec
quelques emprunts à~\cite{Steiner}. Commençons par des rappels sur le
produit tensoriel \oo-catégorique introduit par Al-Agl et
Steiner~\cite{AlAglSteiner} et généralisant le produit de Gray $2$-catégorique
\cite{GrayFCT}.

\begin{paragraph}\label{paragr:def_tens_comp}
  Soient $K$ et $L$ deux complexes dirigés augmentés. On définit leur
  \ndef{produit tensoriel} $K \otimes L$ de la manière suivante. Le
  complexe sous-jacent est le produit tensoriel usuel des complexes :
  pour tout $r \ge 0$, on pose
  \[
   (K \otimes L)_r = \bigoplus_{\substack{p + q = r\\p \ge 0,\,q \ge 0}}
   K_p \otimes L_q
  \]
  et, pour $x \otimes y$ dans $K_p \otimes L_q$ avec $p + q > 0$,
  \[
  d(x \otimes y) = d(x) \otimes y + (-1)^{p} x \otimes d(y),
  \]
  où on convient que $d(z) = 0$ lorsque $z$ est de degré $0$.
  L'augmentation, pour $x \otimes y$ dans $K_0 \otimes L_0$, est
  définie par
  \[ e(x \otimes y) = e(x)e(y). \]
  Enfin, pour $r \ge 0$, le sous-monoïde de positivité $(K \otimes L)^\ast_r$ est
  le sous-monoïde de~$(K \otimes L)_r$ engendré par les éléments de la forme
  $x \otimes y$, avec $x$ dans $K^\ast_p$,  $y$ dans $L^\ast_q$ et~$p + q
  = r$.
\end{paragraph}

\begin{proposition}[Steiner]\label{prop:tens_Steiner}
  Si $K$ et $L$ sont des complexes de Steiner forts admettant respectivement
  $X$ et $Y$ pour base, alors $K \otimes L$ est un complexe de Steiner fort
  admettant $X \otimes Y = \{x \otimes y \mid x \in X, \, y \in Y \}$ pour
  base.
\end{proposition}

\begin{proof}
  Voir \cite[exemple 3.10]{Steiner}.
\end{proof}

\begin{paragraph}
  Le produit tensoriel définit une structure de catégorie monoïdale sur la
  catégorie $\Cda$ des complexes dirigés augmentés. L'unité est le complexe
  dirigé augmenté~$\lambda(\Dn{0})$, où $\Dn{0}$ désigne la \oo-catégorie
  terminale. En vertu de la proposition précédente (et du fait évident que
  $\lambda(\Dn{0})$ est un complexe de Steiner fort), cette structure se
  restreint à la sous-catégorie pleine des complexes de Steiner forts.
\end{paragraph}

\begin{theorem}\label{thm:prod_tens}
  Il existe une et une seule structure de catégorie monoïdale bifermée (à
  unique isomorphisme monoïdal près) sur la catégorie $\ooCat$ des \oo-catégories
  telle que la restriction du foncteur $\nu : \Cda \to \ooCat$ à la
  sous-catégorie pleine de~$\Cda$ formée des complexes de Steiner forts
  munie du produit tensoriel soit un foncteur monoïdal.
\end{theorem}

\begin{proof}
  Voir \cite[théorème A.15]{AraMaltsiJoint}.
\end{proof}

\begin{paragraph}
  On appellera \ndef{produit tensoriel} le produit monoïdal
  \begin{alignat*}{2}
    \otimes & : \ooCat \times \ooCat & \to \ooCat\\
    & \phantom{=1}\quad\,\,(A, B) & \mapsto A \otimes B
  \end{alignat*}
  donné par le théorème précédent. L'unité de ce produit tensoriel est la
  \oo-catégorie terminale $\Dn{0}$.

  Le fait que le produit tensoriel est bifermé signifie qu'il existe des
  foncteurs
  \[
    \HomOpLax : \ooCat^\o \times \ooCat \to \ooCat
    \quadet
    \HomLax : \ooCat^\o \times \ooCat \to \ooCat
  \]
  tels qu'on ait des bijections
  \[
      \Hom_{\ooCat}(A \otimes B, C)
      \simeq \Hom_{\ooCat}(A, \HomOpLax(B, C))
  \]
  et
  \[
      \Hom_{\ooCat}(A \otimes B, C)
      \simeq \Hom_{\ooCat}(B, \HomLax(A, C)),
  \]
  naturelles en $A$, $B$ et $C$ dans $\ooCat$.
\end{paragraph}

Le produit tensoriel est compatible aux principales dualités de $\ooCat$
dont on rappelle maintenant les définitions.

\begin{paragraph}\label{paragr:dual_ooCat}
  Soit $J$ un ensemble d'entiers strictement positifs. On dispose d'un
  \oo-foncteur involutif $D_J : \ooCat \to \ooCat$ envoyant une
  \oo-catégorie $C$ sur la \oo-catégorie $D_J(C)$ obtenue en inversant le
  sens des $i$-cellules pour tout $i$ dans $J$.

  Outre la dualité triviale (le cas $J = \vide$), trois dualités jouent un
  rôle particulièrement important en théorie des \oo-catégories. Si $J$
  est l'ensemble de tous les entiers strictement positifs, on note $C^\o$
  la \oo-catégorie $D_J(C)$ et on parle du \ndef{dual total} de $C$ ; si $J$
  est l'ensemble des entiers impairs, on note $C^\op$ la \oo-catégorie
  $D_J(C)$ et on parle du \ndef{dual impair} de~$C$ ; enfin, si $J$ est
  l'ensemble des entiers pairs strictement positifs, on note $C^\co$ la
  \oo-catégorie~$D_J(C)$ et on parle du \ndef{dual pair} de $C$.
\end{paragraph}

\begin{proposition}
  Soient $A$ et $B$ deux \oo-catégories. On a des isomorphismes canoniques
  \[
    (A \otimes B)^\op \simeq B^\op \otimes A^\op,
    \quad
    (A \otimes B)^\co \simeq B^\co \otimes A^\co
    \quadet
    (A \otimes B)^\o \simeq A^\o \otimes B^\o,
  \]
  naturels en $A$ et $B$.
\end{proposition}

\begin{proof}
  Voir par exemple \cite[proposition A.22]{AraMaltsiJoint}.
\end{proof}

\begin{proposition}\label{prop:tens_dual}
  Soient $A$ et $B$ deux \oo-catégories. On a des isomorphismes canoniques
  \[
    \HomOpLax(A, B)^\op \simeq \HomLax(A^\op, B^\op),
    \quad
    \HomOpLax(A, B)^\co \simeq \HomLax(A^\co, B^\co)
  \]
  et
  \[
    \HomOpLax(A, B)^\o \simeq \HomOpLax(A^\o, B^\o),
  \]
  naturels en $A$ et $B$.
\end{proposition}

\begin{proof}
  Voir par exemple \cite[proposition A.23]{AraMaltsiJoint}.
\end{proof}

Passons maintenant à des rappels sur les transformations oplax. Commençons
par introduire quelques notations.

\begin{paragraph}\label{paragr:def_Dn}
  Pour $i \ge 0$, on notera $\Dn{i}$ la \oo-catégorie coreprésentant le
  foncteur associant à une \oo-catégorie $C$ l'ensemble $C_i$ de ses
  $i$-cellules. On a donc une bijection naturelle
  \[ \Hom_{\ooCat}(\Dn{i}, C) \simeq C_i. \]
  La \oo-catégorie $\Dn{i}$ est en fait une $i$-catégorie. Elle possède une
  unique $i$-cellule n'étant pas une identité qu'on appellera sa
  \ndef{cellule principale}. Pour $k$ tel que $0 \le k < i$, elle admet
  exactement deux $i$-cellules qui ne sont pas des identités ; ces cellules
  sont la source et le but itérés en dimension $k$ de la cellule principale.

  Pour $i > 0$, on notera $\sigma$ et
  $\tau$ les \oo-foncteurs de $\Dn{i-1}$ vers $\Dn{i}$ qui correspondent
  respectivement à la source et au but de la cellule principale de $\Dn{i}$.
  De même, si $i \ge 0$, on notera $\kappa$ le \oo-foncteur de $\Dn{i+1}$
  vers $\Dn{i}$ correspondant à l'identité de la cellule principale de
  $\Dn{i}$.
\end{paragraph}

\begin{paragraph}\label{paragr:trans_oplax}
  Soient $A$ et $B$ deux \oo-catégories. Pour $i \ge 0$, les $i$-cellules de
  $\HomOpLax(A, B)$ seront appelées des \ndef{$i$-transformation oplax} de
  $0$-source $A$ et de $0$-but $B$. En vertu de la bijection canonique
  \[
    \Hom_{\ooCat}(\Dn{i}, \HomOpLax(A, B)) \simeq \Hom_{\ooCat}(\Dn{i}
    \otimes A, B),
  \]
  une $i$-transformation oplax s'identifie à un \oo-foncteur $\Dn{i}
  \otimes A \to B$. Ainsi, pour $i = 0$, en vertu de l'isomorphisme $\Dn{0}
  \otimes A \simeq A$, une $0$-transformation oplax n'est~autre qu'un
  \oo-foncteur \emph{strict} $A \to B$. On appellera \ndef{transformations
  oplax} les $1$\nbd-transformations oplax. Pour $i > 0$, une
  $i$-transformation oplax $\Lambda : \Dn{i} \otimes A \to B$ a une
  $(i-1)$\nbd-transformation oplax source $s(\Lambda)$ et une
  $(i-1)$\nbd-transformation oplax but $t(\Lambda)$ obtenues en précomposant
  $\Lambda$ par $\sigma \otimes A, \tau \otimes A : \Dn{i-1} \otimes A \to
  \Dn{i} \otimes A$ respectivement, où $\sigma$ et $\tau$ désignent les
  \oo-foncteurs du paragraphe précédent.

  En particulier, si $u$ et $v$ sont deux \oo-foncteurs, une
  transformation oplax $\alpha$ de $u$ vers $v$ est un \oo-foncteur $\alpha
  : \Dn{1} \otimes A \to B$ rendant commutatif le diagramme
  \[
    \xymatrix@C=3pc{
    A \ar[dr]^u \ar[d]_{\sigma \otimes A} \\
    \Dn{1} \otimes A \ar[r]^\alpha & B \\
    A \ar[ur]_v \ar[u]^{\tau \otimes A} & \pbox{,} \\
    }
  \]
  où on a identifié $A$ et $\Dn{0} \otimes A$.

  De même, on définit une notion de \ndef{$i$-transformation lax} en remplaçant la
  \oo-catégorie $\HomOpLax(A, B)$ par la \oo-catégorie $\HomLax(A, B)$. En
  vertu de la proposition~\ref{prop:tens_dual}, les $i$\nbd-transformations lax
  peuvent se définir par dualité à partir des $i$\nbd-transformations oplax (et
  réciproquement). En particulier, si $u, v : A \to B$ sont deux
  \oo-foncteurs, une transformation lax de $u$ vers $v$ n'est rien d'autre
  qu'une transformation oplax de $v^\op$ vers $u^\op$ (ou de $u^\co$ vers
  $v^\co$).
\end{paragraph}

\begin{remark}
  La donnée d'une transformation oplax $\alpha$ entre deux \oo-foncteurs de
  $A$ vers $B$ revient à la donnée, pour toute $i$-cellule $x$ de $A$, d'une
  $(i+1)$-cellule $\alpha_x$ de $B$ avec des sources et buts prescrits et
  vérifiant des axiomes de compatibilités aux identités et compositions.
  C'est cette définition concrète qui est utilisée dans
  \cite{AraMaltsiJoint} (voir le paragraphe 1.9 pour la définition et le
  corollaire B.2.6 pour la comparaison).
\end{remark}

\begin{paragraph}\label{paragr:def_sesqui_oplax}
  Si $v : C \to D$ est un \oo-foncteur, on notera $\id{v}$ la transformation
  oplax identité de $v$ dans $\HomOpLax(C, D)$. Explicitement, elle est
  donnée par le composé
  \[ \Dn{1} \otimes C \xto{\kappa \otimes C} C \xto{\,\,v\,\,} D, \]
  où $\kappa$ désigne le \oo-foncteur du paragraphe~\ref{paragr:def_Dn} et
  où on a identifié $C$ et $\Dn{0} \otimes C$.

  Soient maintenant $v_0, v_1 : C \to D$ deux \oo-foncteurs et $\alpha$ une
  transformation oplax de~$v_0$ vers $v_1$. Si $u : B \to C$ est un
  \oo-foncteur, on notera $\alpha \comp u$ la transformation oplax de $v_0u$
  vers $v_1u$ donnée par le composé
  \[
    \Dn{1} \otimes B \xto{\Dn{1} \otimes u} \Dn{1} \otimes C
      \xto{\,\,\alpha\,\,} D.
  \]
  De même, si $w : D \to E$ est un \oo-foncteur, on notera $w \comp \alpha$
  la transformation oplax de $wv_0$ vers $wv_1$ donnée par le composé
  \[ \Dn{1} \otimes C \xto{\,\,\alpha\,\,} D \xto{\,\,w\,\,} E. \]
\end{paragraph}

\begin{paragraph}\label{paragr:img_inv_trans}
  Considérons un diagramme
  \[
    \shorthandoff{;}
    \xymatrix@R=2.5pc@C=2.5pc{
      A \ar[r]^f
      \ar@/_2.5ex/[d]_(.50){\phantom{u'}u}_{}="u"
      \ar@/^2.5ex/[d]^(.50){u'}_{}="up"
      \ar@2"u";"up"^{\alpha}
      &
      C
      \ar@/_2.5ex/[d]_(.50){\phantom{w'}w}_{}="w"
      \ar@/^2.5ex/[d]^(.50){w'}_{}="wp"
      \ar@2"w";"wp"^{\gamma}
      &
      B \ar[l]_g
      \ar@/_2.5ex/[d]_(.50){\phantom{v'}v}_{}="v"
      \ar@/^2.5ex/[d]^(.50){v'}_{}="vp"
      \ar@2"v";"vp"^{\beta}
      \\
      A' \ar[r]_{f'} & C' & B' \ar[l]^{g'}
    }
  \]
  de \oo-catégories, où $\alpha$, $\beta$ et $\gamma$ sont des
  transformations oplax de $u$ vers $u'$, de $v$ vers $v'$ et de $w$ vers
  $w'$ respectivement. On suppose le
  diagramme commutatif au sens où
  \[ \gamma \comp f = f' \comp \alpha \quadet \gamma \comp g = g' \comp
  \beta. \]
  On définit alors une transformation oplax $\alpha \times_\gamma \beta$ de
  $u \times_w v$ vers $u' \times_{w'} v'$ (qui sont deux \oo-foncteurs de $A
  \times_C B$ vers $A' \times_{C'} B'$) par le composé
  \[
    \Dn{1} \otimes (A \times_C B) \to (\Dn{1} \otimes A) \times_{\Dn{1}
    \otimes C} (\Dn{1} \otimes B) \xto{\alpha \times_\gamma \beta} A'
    \times_{C'} B',
  \]
  où la flèche de gauche est le morphisme canonique et celle de droite est
  bien définie par l'hypothèse de commutativité du diagramme.

  Un cas particulièrement important est celui où 
  \[ 
     B = B',\quad v = \id{B} = v',\quad \beta = \id{\id{B}}, \quad
     C = C',\quad w = \id{C} = w',\quad \gamma = \id{\id{C}},
  \]
  c'est-à-dire celui d'un diagramme
  \[
    \shorthandoff{;:}
    \xymatrix@C=1.5pc@R=3pc{
      & A \ar@/^2ex/[rr]^(.50){u}_{}="1" \ar@/_2ex/[rr]_(.50){u'}_{}="0"
      \ar[dr]_{}="f"_{\phantom{f'}f}
      \ar@2"1";"0"_{\alpha\,}
      & & A' \ar[dl]^{f'} \\
      B \ar[rr]^g && C & \pbox{,}
      }
  \]
  commutatif au sens où $f' \comp \alpha = \id{f}$. Dans ce cas, on obtient
  une transformation oplax qu'on notera plus simplement $\alpha \times_C B$
  de $u \times_C B$ vers $u' \times_C B$.
\end{paragraph}

\begin{paragraph}\label{paragr:def_comp_trans}
  Soient $u_0, u_1, u_2 : C \to D$ trois \oo-foncteurs et soient $\alpha :
  u_0 \tod u_1$ et $\beta : u_1 \tod u_2$ deux transformations oplax. On
  définit une transformation oplax $\beta \alpha : u_0 \tod u_2$ de la
  manière suivante. Par définition, les transformations oplax~$\alpha$ et
  $\beta$ sont des $1$-cellules de $\HomOpLax(C, D)$, la première de source
  $u_0$ et de but $u_1$, et la seconde de source $u_1$ et de but $u_2$. La
  composition des $1$-cellules de~$\HomOpLax(C, D)$ fournit une $1$-cellule
  de $u_0$ vers $u_2$ qui par définition est la transformation oplax~$\beta
  \alpha : u_0 \tod u_2$.
\end{paragraph}

\begin{paragraph}\label{paragr:def_ooCatOpLax}
  On se gardera de croire que les \oo-catégories, les \oo-foncteurs et les
  transformations oplax munis des opérations de composition définies dans
  les paragraphes~\ref{paragr:def_sesqui_oplax}
  et~\ref{paragr:def_comp_trans} forment une $2$-catégorie. Ils forment
  néanmoins une sesquicatégorie (voir la définition~\ref{paragr:def_sesqui})
  que l'on notera $\ooCatOpLax$. On renvoie à l'appendice~\ref{app:tr_comma}
  pour plus de détails et notamment à l'exemple~\ref{ex:OpLaxGray} et à la
  remarque~\ref{rem:sesqui_sous-jacente} pour une justification du fait que
  $\ooCatOpLax$ est bien une sesquicatégorie.
\end{paragraph}

Le but des paragraphes qui suivent est d'expliciter le lien entre les
notions de transformation oplax et de transformation stricte.

\begin{paragraph}\label{paragr:def_q}
  Soient $A$ et $B$ deux \oo-catégories. On dispose d'un \oo-foncteur
  canonique $q_1 : A \otimes B \to A$ défini par le composé
  \[ A \otimes B \xto{A \otimes p} A \otimes \Dn{0} \xto{\,\sim\,} A, \]
  où $p$ désigne l'unique \oo-foncteur de $B$ vers $\Dn{0}$. De même, on
  dispose d'un \oo-foncteur canonique $q_2 : A \otimes B \to B$. On obtient
  donc un \oo-foncteur
  \[ q = (q_1, q_2) : A \otimes B \to A \times B. \]
  Les \oo-foncteurs $q_1$ et $q_2$ étant naturels en $A$ et $B$, il en est
  de même de $q$.
\end{paragraph}

\begin{proposition}
  Soient $A$ et $B$ deux \oo-catégories. Le \oo-foncteur
  \[ q : A \otimes B \to A \times B \]
  est surjectif sur les cellules et est en particulier un épimorphisme.
\end{proposition}

\begin{proof}
  Fixons $i \ge 0$. Soient $a$ une $0$-cellule de $A$ et $y$ une $i$-cellule
  de $B$.  Considérons le carré de naturalité
  \[
    \xymatrix{
      \Dn{0} \otimes \Dn{i} \ar[r]^q \ar[d]_{a \otimes y} &
      \Dn{0} \times \Dn{i} \ar[d]^{a \times y} \\
      A \otimes B \ar[r]_q & A \times B \pbox{.}
    }
  \]
  Il résulte du fait que $\Dn{0}$ est à la fois l'unité du produit cartésien
  et du produit tensoriel que la flèche horizontale du haut de ce carré est un
  isomorphisme. Ainsi, la $i$-cellule~$(\id{a}, y)$ de $A \times B$,
  où $\id{a}$ désigne l'identité itérée de $a$ en dimension $i$, est
  atteinte par le \oo-foncteur $q$. Plus précisément, la $i$-cellule de $A
  \otimes B$ correspondant au \oo-foncteur $a \otimes y$, cellule que l'on
  notera également $a \otimes y$, est envoyée sur~$(\id{a}, y)$ par~$q$. En
  particulier, le \oo-foncteur $q : A \otimes B \to A \times B$ est
  surjectif sur les $0$-cellules. On montre de même que si $x$ est une
  $i$-cellule de $A$ et $b$ est une $0$-cellule de $B$, alors $q$ envoie $x
  \otimes b$ sur $(x, \id{b})$. Or, si $(x, y)$ est une $i$-cellule
  de $A \times B$ avec $i > 0$, on a
  \[ (x, y) = (x, \id{t_0(y)}) \comp_0 (\id{s_0(x)}, y), \]
  où $s_0$ et $t_0$ désignent les sources et buts en dimension $0$.
  Puisque
  \[ s_0(x \otimes t_0(y)) = s_0(x) \otimes t_0(y) = t_0(s_0(x) \otimes y), \]
  on dispose d'une $i$-cellule $(x \otimes t_0(y)) \comp_0 (s_0(x) \otimes y)$
  dans $A \otimes B$. En vertu des considérations précédentes, cette cellule
  est envoyée par $q$ sur $(x, \id{t_0(y)}) \comp_0 (\id{s_0(x)}, y)$,
  c'est-à-dire sur~$(x, y)$, ce qu'il fallait démontrer.
\end{proof}

\begin{paragraph}
  Rappelons que la catégorie $\ooCat$ des \oo-catégories est cartésienne
  fermée. Si $B$ et $C$ sont deux \oo-catégories, on notera $\HomStr(B, C)$
  le $\Hom$ interne associé. Par définition, si $A$ est une troisième
  \oo-catégorie, on a une bijection
  \[
    \Hom_{\ooCat}(A \times B, C) \simeq \Hom_{\ooCat}(A, \HomStr(B, C)),
  \]
  naturelle en $A$, $B$ et $C$.

  Soient $A$ et $B$ deux \oo-catégories. On déduit du \oo-foncteur
  naturel $q$ du paragraphe~\ref{paragr:def_q} un \oo-foncteur
  $i : \HomStr(A, B) \to \HomOpLax(A, B)$ naturel en $A$ et $B$. En effet,
  si $T$ est une troisième \oo-catégorie, en vertu du lemme de Yoneda, il
  suffit de définir une application naturelle
  \[
    \Hom_{\ooCat}(T, \HomStr(A, B)) \to \Hom_{\ooCat}(T, \HomOpLax(A, B)),
  \]
  c'est-à-dire une application naturelle
  \[
    \Hom_{\ooCat}(T \times A, B) \to \Hom_{\ooCat}(T \otimes A, B).
  \]
  Or, le \oo-foncteur $q : T \otimes A \to T \times A$ induit bien une telle
  application. Par ailleurs, puisque $q$ est un épimorphisme en vertu de
  la proposition précédente, le \oo-foncteur~$i$ est un monomorphisme.

  De même, on a un monomorphisme canonique $j : \HomStr(A, B) \to \HomLax(A,
  B),$ naturel en $A$ et $B$ dans $\ooCat$.
\end{paragraph}

\begin{paragraph}\label{paragr:trans_str}
  Soient $A$ et $B$ deux \oo-catégories. Pour tout $i \ge 0$, on définit comme
  dans le paragraphe~\ref{paragr:trans_oplax} une notion de
  \ndef{$i$-transformation stricte} de $0$-source~$A$ et de $0$-but $B$ en
  remplaçant $\HomOpLax(A, B)$ par $\HomStr(A, B)$. Le monomorphisme
  $\HomStr(A, B) \to \HomOpLax(A, B)$ défini au paragraphe précédent permet
  de considérer toute $i$-transformation stricte comme une
  $i$-transformation oplax. En particulier, si $u, v : A \to B$ sont deux
  \oo-foncteurs, une transformation stricte de $u$ vers $v$ est un
  \oo-foncteur $h : \Dn{1} \times A \to B$ rendant commutatif le diagramme
  \[
    \xymatrix@C=3pc{
    A \ar[dr]^u \ar[d]_{\sigma \times A} \\
    \Dn{1} \times A \ar[r]^h & B \\
    A \ar[ur]_v \ar[u]^{\tau \times A} & \pbox{,} \\
    }
  \]
  où on a identifié $A$ et $\Dn{0} \times A$, et la transformation oplax
  associée est donnée par le composé
  \[ \Dn{1} \otimes A \xto{\,q\,} \Dn{1} \times A \xto{\,h\,} B, \]
  où $q$ est le \oo-foncteur du paragraphe~\ref{paragr:def_q}.
\end{paragraph}

Nous terminons ces rappels liés au produit tensoriel par une proposition, de
nature technique, qui n'interviendra que dans l'appendice~\ref{app:tr_comma}.

\begin{prop}\label{prop:tens_morph}
  Soient $f : K \to K'$ et $g : L \to L'$ des morphismes entre complexes
  de Steiner forts et soient $x \otimes y$ un élément de la
  base de~$K \otimes L$. Supposons qu'il existe un élément $x' \otimes y'$
  de la base de $K' \otimes L'$ tel que
  \[
    \nu(f)(\atom{x}) = \id{\atom{x'}}
    \quadet
    \nu(f)(\atom{y}) = \id{\atom{y'}},
  \]
  où $\id{}$ désigne une identité itérée (éventuellement $0$ fois). Alors on
  a
  \[ \nu(f \otimes g)(\atom{x \otimes y}) = \id{\atom{x' \otimes y'}}. \]
  En particulier, lorsque $x' = f(x)$ et $y' = g(y)$ vérifient les
  hypothèses ci-dessus, on a
  \[ \nu(f \otimes g)(\atom{x \otimes y}) = \atom{f(x) \otimes g(y)}. \]
\end{prop}

\begin{proof}
  Voir \cite[proposition A.7]{AraMaltsiJoint}.
\end{proof}

Passons maintenant à des rappels sur le joint \oo-catégorique introduit dans
\cite{AraMaltsiJoint}.

\begin{paragraph}
  Soient $K$ et $L$ deux complexes dirigés augmentés. On définit leur
  \ndef{joint} $K \joint L$ de la manière suivante.
  Pour tout $r \ge 0$, on pose
  \[
   (K \joint L)_r = \bigoplus_{\substack{p + 1 + q = r\\p \ge -1,\,q \ge -1}}
   K_p \otimes L_q,
  \]
  où on convient que $K_{-1} = \Z$ et $L_{-1} = \Z$. On notera $\vide$ le
  générateur positif de $K_{-1}$ et $L_{-1}$. Par ailleurs, si $x$ est dans
  $K_p$ et $y$ dans $L_q$ avec $q + 1 + p \ge 0$, l'élément correspondant de
  $K \joint L$ sera noté $x \joint y$. Avec ces notations, la différentielle
  de $K \joint L$ est définie par, pour $x \joint y$ dans $K_p \otimes L_q$
  avec $p + 1 + q > 0$,
  \[
  d(x \joint y) = d(x) \joint y + (-1)^{p+1} x \joint d(y),
  \]
  où on convient, d'une part, que $d(z) = e(z)\vide$ lorsque $z$ est de
  degré $0$ et, d'autre part, que $d(\vide) = 0$.  L'augmentation, pour $x$
  dans $K_0$ et $y$ dans $L_0$, est définie par
  \[ e(x \joint \vide) = e(x) \quadet e(\vide \joint y) = e(y). \]
  Enfin, pour $r \ge 0$, le sous-monoïde de positivité $(K \joint L)^\ast_r$ est
  le sous-monoïde de~$(K \joint L)_r$ engendré par les éléments de la forme
  $x \joint y$, avec $x$ dans $K^\ast_p$, $y$ dans~$L^\ast_q$ et $p + 1
  + q = r$, en convenant que $K^\ast_{-1} = \N$ et $L^\ast_{-1} = \N$.
\end{paragraph}

\begin{proposition}\label{prop:joint_Steiner}
  Si $K$ et $L$ sont des complexes de Steiner forts admettant $X$ et $Y$
  pour bases respectives, alors $K \joint L$ est un complexe de Steiner
  fort admettant
  \[
    X \joint Y = 
      \{x \joint \vide \mid x \in X\} \cup
      \{\vide \joint y \mid y \in Y\} \cup
      \{ x \joint y \mid x \in X,\, y \in Y\}
  \]
  pour base.
\end{proposition}

\begin{proof}
  Voir \cite[paragraphe 6.13 et corollaire 6.21]{AraMaltsiJoint}.
\end{proof}

\begin{paragraph}
  Le joint définit une structure de catégorie monoïdale sur la catégorie
  $\Cda$ des complexes dirigés augmentés. L'unité est le complexe
  dirigé augmenté~$\lambda(\vide)$, où $\vide$ désigne la \oo-catégorie
  initiale. On notera également $\vide$ ce complexe dirigé augmenté. En
  vertu de la proposition précédente (et du fait évident que $\vide$ est un
  complexe de Steiner fort), cette structure se restreint à la
  sous-catégorie pleine des complexes de Steiner forts.
\end{paragraph}

\begin{theorem}
  Il existe une et une seule structure de catégorie monoïdale
  (à unique isomorphisme monoïdal près) localement bifermée (au sens du
  paragraphe~\ref{paragr:def_tr}) sur la catégorie $\ooCat$ des
  \oo-catégories telle que la restriction du foncteur $\nu : \Cda \to
  \ooCat$ à la sous-catégorie pleine de~$\Cda$ formée des complexes de
  Steiner forts munie du joint soit un foncteur monoïdal.
\end{theorem}

\begin{proof}
  Voir \cite[théorème 6.29]{AraMaltsiJoint}.
\end{proof}

\begin{paragraph}
  On appellera \ndef{joint} le produit monoïdal
  \begin{alignat*}{2}
    \joint & : \ooCat \times \ooCat & \to \ooCat\\
    & \phantom{=1}\quad\,\,(A, B) & \mapsto A \joint B
  \end{alignat*}
  donné par le théorème précédent. L'unité de ce produit tensoriel est la
  \oo-catégorie initiale $\vide$.

  On notera, comme dans le paragraphe~\ref{paragr:def_tr}, pour $A$ et
  $B$ deux \oo-catégories,
  \[ A \xto{\iota_1} A \joint B \xot{\iota_2} B \]
  les deux \oo-foncteurs canoniques. Le fait que le joint est localement
  bifermé signifie que les foncteurs
  \[
    \begin{split}
      \ooCat & \to \cotr{\ooCat}{A} \\
      B & \mapsto (A \joint B, \iota_1 : A \to A \joint B)
    \end{split}
  \]
  et
  \[
    \begin{split}
      \ooCat & \to \cotr{\ooCat}{B} \\
      A & \mapsto (A \joint B, \iota_2 : B \to A \joint B) \pbox{,}
    \end{split}
  \]
  pour $A$ fixé pour le premier foncteur et $B$ fixé pour le second,
  admettent des adjoints à droite. On obtient donc des foncteurs
  \[
    \begin{split}
      \cotr{\ooCat}{A} & \to \ooCat \\
      (C, A \xto{v} C) & \mapsto \cotr{C}{v} \\
    \end{split}
  \]
  et
  \[
    \begin{split}
      \cotr{\ooCat}{B} & \to \ooCat \\
      (C, B \xto{w} C) & \mapsto \trm{C}{w} \\
    \end{split}
  \]
  et des bijections naturelles
  \[
    \begin{split}
      \Hom_{\cotr{\ooCat}{A}}((A \joint B, \iota_1), (C, v))
      & \simeq \Hom_{\ooCat}(B, \cotr{C}{v}), \\
      \Hom_{\cotr{\ooCat}{B}}((A \joint B, \iota_2), (C, w))
      & \simeq \Hom_{\ooCat}(A, \trm{C}{w}). \\
    \end{split}
  \]
  Si $v : A \to C$ est un \oo-foncteur, la \oo-catégorie $\cotr{C}{v}$ sera
  appelé la \ndef{tranche de $C$ au-dessous de $v$}.
\end{paragraph}

\begin{remark}
  Si $w : B \to C$ est un \oo-foncteur, ce n'est pas la
  \oo-catégorie $\trm{C}{w}$ qu'on appelle la tranche de $C$ au-dessus
  de $w$ dans \cite{AraMaltsiJoint} mais la \oo-catégorie~$\tr{C}{w} =
  \big(\cotr{C^\o}{w^\o}\big)^\o$.  Néanmoins, cette tranche n'interviendra
  pas dans ce texte. On renvoie à \cite[remarque 6.37]{AraMaltsiJoint} pour
  plus de détails.
\end{remark}

\begin{proposition}\label{prop:dual_joint}
  Soient $A$ et $B$ deux \oo-catégories. On a un isomorphisme canonique
  \[
    (A \joint B)^\op \simeq B^\op \joint A^\op,
  \]
  naturel en $A$ et $B$.
\end{proposition}

\begin{proof}
  Voir \cite[proposition 6.35]{AraMaltsiJoint}.
\end{proof}

\begin{proposition}\label{prop:tr_op}
  Soient $C$ une \oo-catégorie et $v : A \to C$ un \oo-foncteur. On a un
  isomorphisme canonique
  \[
    (\cotr{C}{v})^\op \simeq \trm{C^\op}{v^\op},
  \]
  naturel en $A$ et $C$.
\end{proposition}

\begin{proof}
  Voir \cite[proposition 6.36]{AraMaltsiJoint}.
\end{proof}

\begin{paragraph}\label{paragr:obj_tr}
  Soient $C$ une \oo-catégorie et $c$ un objet de $C$. En considérant $c$
  comme un \oo-foncteur $\Dn{0} \to C$, on obtient une \oo-catégorie
  $\cotr{C}{c}$. Par adjonction, les $i$-cellules de $\cotr{C}{c}$
  correspondent aux \oo-foncteurs $\Dn{0} \joint \Dn{i} \to C$ rendant
  le triangle
  \[
    \xymatrix{
      \Dn{0} \joint \Dn{i} \ar[r] & C \\
      \Dn{0} \ar[u]^{\iota_1} \ar[ur]_{c}
    }
  \]
  commutatif. On peut décrire explicitement ces \oo-foncteurs et la
  structure de \oo-catégorie résultante (voir \cite[chapitre
  9]{AraMaltsiJoint}). En particulier, lorsque $C$ est une $1$\nbd-catégorie,
  $\cotr{C}{c}$ est la $1$-catégorie tranche usuelle. Dans ce texte, nous
  aurons seulement besoin de la description explicite des objets de
  $\cotr{C}{c}$. On a un isomorphisme canonique $\Dn{0} \joint \Dn{0} \simeq
  \Dn{1}$ (cela résulte par exemple de \cite[corollaire
  7.10]{AraMaltsiJoint}) et les objets de $\cotr{C}{c}$ correspondent donc à
  des $1$-cellules $f$ de $C$ de source $c$. On notera $(c, f)$ un tel
  objet.

  Par ailleurs, d'après~\cite[proposition B.5.2]{AraMaltsiJoint}, la \oo-catégorie
  $\cotr{C}{c}$ peut se décrire par le carré cartésien
  \[
    \xymatrix{
      \cotr{C}{c} \ar[d] \ar[r] & \HomLax(\Dn{1}, C) \ar[d]^{\pi^{}_0} \\
      \Dn{0} \ar[r]_-c & C \pbox{,}
    }
  \]
  où $\pi_0$ désigne le \oo-foncteur
  $\HomLax(\sigma, \Dn{1}) : \HomLax(\Dn{1}, C) \to \HomLax(\Dn{0}, C)
  \simeq C$.
\end{paragraph}

\begin{paragraph}\label{paragr:contr}
  L'opération $C \mapsto \Dn{0} \joint C$ a une interprétation
  particulièrement simple en termes de transformations oplax. En effet, si
  $A$ et $B$ sont deux \oo-catégories, en vertu de \cite[paragraphe B.5.5 et
  corollaire B.5.6]{AraMaltsiJoint}, on dispose d'un \oo-foncteur $\Dn{1}
  \otimes A \to \Dn{0} \joint A$ rendant commutatif le triangle
  \[
    \xymatrix@C=1.5pc{
      \Dn{1} \otimes A \ar[rr] && \Dn{0} \joint A \\
      & A \ar[ul]^{\tau \otimes A} \ar[ur]_{\iota_2} & \pbox{,}
    }
  \]
  et un \oo-foncteur $\Dn{1} \otimes A \to B$ se factorise par $\Dn{0}
  \joint A$ si et seulement si sa source en tant que transformation oplax
  est un \oo-foncteur constant, cette valeur constante correspondant alors
  au composé $\Dn{0} \xto{\iota_1} \Dn{0} \joint A \to B$.

  En particulier, si $C$ est une \oo-catégorie, se donner une transformation
  oplax d'un \oo-endofoncteur constant de $C$ vers l'identité de $C$
  revient à se donner un \oo-foncteur $\Dn{0} \joint C \to C$ rendant
  le triangle
  \[
    \xymatrix{
      \Dn{0} \joint C \ar[r] & C \\
      C \ar[u]^{\iota_2} \ar[ur]_{\id{C}}
    }
  \]
  commutatif. Par adjonction, il revient au même de se donner un objet $c$
  de $C$ et un \oo-foncteur $C \to \cotr{C}{c}$ rendant commutatif le triangle
  \[
    \xymatrix{
      C \ar[r] \ar[dr]_{\id{C}} & \cotr{C}{c} \ar[d]^{U} \\
      & C \pbox{,}
    }
  \]
  où $U$ désigne le \oo-foncteur d'oubli.
\end{paragraph}

Nous terminons cette section par des rappels sur les résultats de
fonctorialités des tranches établis dans \cite[section 11]{AraMaltsiJoint},
qui sont centraux à la démonstration du théorème~A \oo-catégorique présentée
dans ce texte.

\medskip

\emph{Dans la suite de cette section, on fixe une \oo-catégorie $C$, un
complexe de Steiner fort~$L$ et un \oo-foncteur $b : \nu(L) \to C$.}

\begin{paragraph}\label{paragr:fonct_tri}
  Considérons un diagramme
  \[
    \shorthandoff{;}
    \xymatrix@C=1.5pc{
      K \ar[rr]^f \ar[dr]_{g}_{}="f" & & K' \ar[dl]^(0.42){g'} \\
      & L
      \ar@{}"f";[ur]_(.15){}="ff"
      \ar@{}"f";[ur]_(.55){}="oo"
      \ar@<-0.5ex>@2"ff";"oo"^{h}
      & \pbox{,}
    }
  \]
  où $K$ et $K'$ sont des complexes de Steiner forts, $f$ et $g$ sont des
  morphismes de complexes dirigés augmentés quelconques, $g'$ est une
  inclusion rigide ordonnée (voir le paragraphe~\ref{paragr:def_rig_ord}) et
  $h$ est une antihomotopie de $g$ vers $g'f$. Le théorème 11.2.2
  de~\cite{AraMaltsiJoint} associe à un tel diagramme un \oo-foncteur
  \[ (f, h, b)^\ast : \cotr{C}{c'} \to \cotr{C}{c}, \]
  où on a posé
  \[ c = b\nu(g) \quadet c' = b\nu(g'). \]
  La définition précise de ce \oo-foncteur ne jouera aucun rôle dans ce
  texte et nous utiliserons seulement quelques propriétés que nous allons
  maintenant rappeler.
\end{paragraph}

\begin{proposition}
  Le \oo-foncteur $(f, h, b)^\ast$ du paragraphe précédent est au-dessus de
  $C$ au sens où le triangle
  \[
  \xymatrix@C=1.5pc{
      \cotr{C}{c'} \ar[dr] \ar[rr]^{(f, h, b)^\ast} & & \cotr{C}{c} \ar[dl]
     \\
     & C & \pbox{,} 
  }
  \] 
  où les flèches obliques désignent les \oo-foncteurs d'oubli, est
  commutatif.
\end{proposition}

\begin{proof}
  Voir \cite[proposition 11.3.6]{AraMaltsiJoint}.
\end{proof}

\begin{proposition}\label{prop:fonct_tri_comm}
  Soit
  \[
    \shorthandoff{;}
    \xymatrix@C=1.5pc{
      K \ar[rr]^f \ar[dr]_{g}_{}="f" & & K' \ar[dl]^(0.42){g'} \\
      & L
      \ar@{}"f";[ur]_(.15){}="ff"
      \ar@{}"f";[ur]_(.55){}="oo"
      \ar@<-0.5ex>@2"ff";"oo"^{\id{g}}
    }
  \]
  un diagramme commutatif de complexes de Steiner forts, avec $g'$ une
  inclusion rigide ordonnée. Alors on a
  \[ (f, \id{g}, b)^\ast = \nu(f)^\ast, \]
  où $\nu(f)^\ast$ est le \oo-foncteur associé au
  triangle commutatif
   \[
     \xymatrix@C=1.5pc{
       \nu(K) \ar[rr]^{\nu(f)} \ar[dr]_{c} & & \nu(K') \ar[dl]^(0.42){c'} \\
       & C
     }
   \]
   en vertu du paragraphe~\ref{paragr:fonct_tr_sur_1}.
\end{proposition}

\begin{proof}
  Voir \cite[proposition 11.2.5]{AraMaltsiJoint}.
\end{proof}

\begin{proposition}\label{prop:fonct_tri_id}
  Soit $g : K \to L$ une inclusion rigide ordonnée entre complexes de Steiner
  forts. Considérons le diagramme commutatif
  \[
    \shorthandoff{;}
    \xymatrix@C=1.5pc{
      K \ar[rr]^{\id{K}} \ar[dr]_{g}_{}="f" & & K \ar[dl]^{g} \\
      & L
      \ar@{}"f";[ur]_(.15){}="ff"
      \ar@{}"f";[ur]_(.55){}="oo"
      \ar@<-0.5ex>@2"ff";"oo"^{\id{g}} & \pbox{.}
    }
  \]
  Alors on a
  \[ (\id{K}, \id{g}, b)^\ast = \id{\cotr{C}{c}}. \]
\end{proposition}

\begin{proof}
  Voir \cite[proposition 11.3.2]{AraMaltsiJoint}. (Le résultat est en fait
  une conséquence directe de la proposition précédente.)
\end{proof}

\begin{proposition}\label{prop:fonct_tri}
  Soit
  \[
    \shorthandoff{;}
    \xymatrix{
      K \ar[r]^f \ar[dr]_{}="g"_(.40){g}
      & K' \ar[r]^{f'}_(.75){}="fp" \ar[d]_(.70){}="gp"_(.50){g'} & K''
      \ar[dl]_{}="gpp"^(.33){g''} \\
      & L
      \ar@{}"g";[u]_(0.10){}="x"
      \ar@{}"g";[u]_(.85){}="y"
      \ar@<-0.1ex>@2"x";"y"^(.30)h
      \ar@{}"gp";"fp"_(.25){}="x2"
      \ar@{}"gp";"fp"_(.75){}="y2"
      \ar@<0.4ex>@2"x2";"y2"^(0.40){h'\!}
    }
  \]
  un diagramme de complexes de Steiner forts, où $g'$ et $g''$ sont des inclusions
  rigides ordonnées, et $h$ et $h'$ sont des antihomotopies de
  $g$ vers $g'f$ et de $g'$ vers $g''f'$ respectivement. Considérons le
  diagramme composé
  \[
    \shorthandoff{;}
    \xymatrix@C=1.5pc{
      K \ar[rr]^{f'f} \ar[dr]_{g}_{}="f" & & K' \ar[dl]^(0.42){g''} \\
      & L
      \ar@{}"f";[ur]_(.15){}="ff"
      \ar@{}"f";[ur]_(.55){}="oo"
      \ar@<-0.5ex>@2"ff";"oo"^{h''}
      & \pbox{,}
    }
  \]
  où $h'' = h'f + h$. Alors on a
  \[ (f, h, b)^\ast (f', h', b)^\ast = (f'f, h'f + h, b)^\ast. \]
\end{proposition}

\begin{proof}
  Voir \cite[proposition 11.3.4]{AraMaltsiJoint}.
\end{proof}

\begin{paragraph}\label{paragr:fonct_cone}
   Considérons un diagramme
   \[
      \shorthandoff{;:}
      \xymatrix@C=1.5pc@R=3pc{
        K \ar@/^2ex/[rr]^(.33){f'}_{}="1" \ar@/_2ex/[rr]^(.30)f_{}="0"
        \ar[dr]_{}="f"_{\phantom{g'}g}
        \ar@2"0";"1"_k
        & & K' \ar[dl]^{g'} \\
        & L
        \ar@{}"f";[ur]_(.15){}="ff"
        \ar@{}"f";[ur]_(.55){}="oo"
        \ar@<-0.5ex>@/^1ex/@{:>}"ff";"oo"^(.18){h'\!\!}_(.30){}="h'"
        \ar@<-2.0ex>@/^-1ex/@2"ff";"oo"_(.36){h}_(.80){}="h"
        \ar@3"h";"h'"_(.20){H_{}}
        & \pbox{,}
        }
  \]
  où $K$ et $K'$ sont des complexes de Steiner forts, $f$, $f'$ et
  $g$ sont des morphismes de complexes dirigés augmentés quelconques, $g'$
  est une inclusion rigide ordonnée, $h$,  $h'$ et~$k$ sont des antihomotopies
  de $g$ vers $g'f$, de $g$ vers $g'f'$ et de $f$ vers $f'$ respectivement
  et $H$ est une $2$-antihomotopie de $g'k + h$ vers $h'$ (qui sont deux
  antihomotopies de $g$ vers $g'f'$). Le théorème 11.4.2 de
  \cite{AraMaltsiJoint} associe à un tel diagramme une transformation oplax
  \[
   \shorthandoff{;}
   (k, H, b)^\ast :
    \xymatrix@C=4pc{
      \cotr{C}{c'}
      \ar@/^2.5ex/[r]^{(f', h', b)^\ast}_{}="0"
      \ar@/_2.5ex/[r]_{(f, h, b)^\ast}^{}="1"
      \ar@2"0";"1"
      & \cotr{C}{c} \pbox{,}
    }
  \]
  où on pose toujours
  \[ c = b\nu(g) \quadet c' = b\nu(g'). \]
  Ici encore, la définition précise de cette transformation oplax ne jouera
  aucun rôle dans ce texte et nous utiliserons seulement quelques propriétés
  que nous allons maintenant rappeler.
\end{paragraph}

\begin{proposition}\label{prop:fonct_cone_sur}
  La transformation oplax $(k, H, b)^\ast$ du paragraphe précédent est
  au-dessus de $C$ au sens où on a
  \[ U \comp (k, H, b)^\ast = \id{U'}, \]
  où $U : \cotr{C}{c} \to C$ et $U' : \cotr{C}{c'} \to C$ désignent
  les \oo-foncteurs d'oubli.
\end{proposition}

\begin{proof}
  Voir \cite[proposition 11.5.6]{AraMaltsiJoint}.
\end{proof}

\begin{proposition}\label{prop:fonct_cone_id}
  Soit
  \[
    \shorthandoff{;}
    \xymatrix@C=1.5pc{
      K \ar[rr]^f \ar[dr]_{g}_{}="f" & & K' \ar[dl]^(0.42){g'} \\
      & L
      \ar@{}"f";[ur]_(.15){}="ff"
      \ar@{}"f";[ur]_(.55){}="oo"
      \ar@<-0.5ex>@2"ff";"oo"^{h}
      &
    }
  \]
  un diagramme de complexes de Steiner forts, avec $g'$ une inclusion rigide
  ordonnée et $h$ une antihomotopie de $g$ vers $g'f$. Considérons le
  diagramme
  \[
    \shorthandoff{;:}
    \xymatrix@C=1.5pc@R=3pc{
    K \ar@/^2ex/[rr]^(.33){f}_{}="1" \ar@/_2ex/[rr]^(.30)f_{}="0"
    \ar[dr]_{}="f"_{\phantom{g'}g}
    \ar@2"0";"1"_{\id{f}}
    & & K' \ar[dl]^{g'} \\
    & L
    \ar@{}"f";[ur]_(.15){}="ff"
    \ar@{}"f";[ur]_(.55){}="oo"
    \ar@<-0.5ex>@/^1ex/@{:>}"ff";"oo"^(.18){h\!\!}_(.30){}="h'"
    \ar@<-2.0ex>@/^-1ex/@2"ff";"oo"_(.36){h}_(.80){}="h"
    \ar@3"h";"h'"_(.20){\,\id{h}}
    & \pbox{.}
    }
  \]
  Alors on a
  \[ (\id{f}, \id{h}, b)^\ast = \id{(f, h, b)^\ast}. \]
\end{proposition}

\begin{proof}
  Voir \cite[proposition 11.5.2]{AraMaltsiJoint}.
\end{proof}

\begin{proposition}\label{prop:sesqui_cone_1}
  Soit
  \[
      \shorthandoff{;:}
      \xymatrix@C=3.5pc@R=3.5pc{
      K  \ar[r]^f \ar[dr]_{}="g"_(.61){\phantom{g''}g} &
      K' \ar@/^2ex/[r]^(.33){f''}_{}="1"
      \ar@/_2ex/[r]^(.30){f'}_{}="0"_(.70){}="fp"
      \ar[d]_(.50){}="gp2"_(.20){}="gp"^(0.73){g'} &
      K'' \ar[dl]^(.61){g''} &
      \ar@2"0";"1"_{k} \\
        & L
      \ar@{}"g";[u]_(0.10){}="x"
      \ar@{}"g";[u]_(.75){}="y"
      \ar@<-0.1ex>@2"x";"y"^(.30)h
      \ar@{}"gp2";"fp"_(.10){}="ff2"
      \ar@{}"gp2";"fp"_(.55){}="oo2"
      \ar@<+0.5ex>@/^1ex/@{:>}"ff2";"oo2"^{\!\!h''}_(.30){}="h'''"
      \ar@<-0.5ex>@/^-1.5ex/@2"ff2";"oo2"_(.50){\!\!\!h'}_(.80){}="h''"
      \ar@3"h''";"h'''"_(.20){H_{}}
      }
  \]
  un diagramme de complexes de Steiner forts, où $g'$ et $g''$ sont des
  inclusions rigides ordonnés, $h$, $h'$, $h''$ et $k$ sont des
  antihomotopies de $g$ vers $g'f$, de $g'$ vers $g''f'$, de $g'$ vers
  $g''f''$ et de $f'$ vers $f''$ respectivement et $H$ est une
  $2$-antihomotopie de $g''k + h'$ vers $h''$ (qui sont deux antihomotopies
  de $g'$ vers $g''f''$). Considérons le diagramme composé
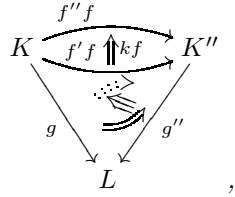
  \[
    \shorthandoff{;:}
    \xymatrix@C=1.5pc@R=3pc{
      K \ar@/^2ex/[rr]^(.33){f''f}_{}="1" \ar@/_2ex/[rr]^(.30){f'f}_{}="0"
      \ar[dr]_{}="f"_{\phantom{g'}g}
      \ar@2"0";"1"_{kf}
      & & K'' \ar[dl]^{g''} \\
      & L
      \ar@{}"f";[ur]_(.15){}="ff"
      \ar@{}"f";[ur]_(.55){}="oo"
      \ar@<-0.5ex>@/^1ex/@{:>}"ff";"oo"^(.18){\!\!}_(.30){}="h'"
      \ar@<-2.0ex>@/^-1ex/@2"ff";"oo"_(.36){}_(.80){}="h"
      \ar@3"h";"h'"_(.20){} & \pbox{,}
      }
  \]
  la $2$-cellule courbée de devant étant l'antihomotopie $h'f + h$, celle de
  derrière l'antihomotopie $h''f + h$ et la $3$-cellule la $2$-antihomotopie
  $Hf$. Alors le composé
  \[
   \shorthandoff{;}
    \xymatrix@C=4pc{
      \cotr{C}{c''}
      \ar@/^2.5ex/[r]^{(f'', h'', b)^\ast}_{}="0"
      \ar@/_2.5ex/[r]_{(f', h', b)^\ast}^{}="1"
      \ar@2"0";"1"
      &
      \cotr{C}{c'}
      \ar[r]^{(f, h, b)^\ast}
      &
      \cotr{C}{c}
      \pbox{,}
    }
  \]
  où la $2$-cellule du diagramme est $(k, H, b)^\ast$ et où on a posé \[ c =
  b\nu(g), \quad c' = b\nu(g') \quadet c'' = b\nu(g''), \]
  est égal à
  \[
   \shorthandoff{;}
   (kf, Hf, b)^\ast :
    \xymatrix@C=4pc{
      \cotr{C}{c''}
      \ar@/^2.5ex/[r]^{(f''f, h''f + h, b)^\ast}_{}="0"
      \ar@/_2.5ex/[r]_{(f'f, h'f + h , b)^\ast}^{}="1"
      \ar@2"0";"1"
      & \cotr{C}{c} \pbox{.}
    }
  \]
  Autrement dit, on a
   \[ (f, h, b)^\ast \comp (k, H, b)^\ast = (kf, Hf, b)^\ast. \]
\end{proposition}

\begin{proof}
  Voir \cite[proposition 11.5.4]{AraMaltsiJoint}.
\end{proof}

\begin{proposition}\label{prop:fonct_cone_2}
  Soit
  \[
    \shorthandoff{;:}
    \xymatrix@C=3.5pc@R=3.5pc{
    K
    \ar@/^2ex/[r]^(.33){f'}_{}="1"
    \ar@/_2ex/[r]^(.30){f}_{}="0"_(.70){}="f"
    \ar[dr]_{}="g"_(.61){\phantom{g''}g}
    \ar@2"0";"1"_{k}
    &
    K' \ar[r]^{f''}_(.75){}="fp"
       \ar[d]_(.70){}="gp2"_(.20){}="gp"^(0.73){g'}
    &
    K'' \ar[dl]^(.61){g''}
    \\
    & L
    \ar@{}"g";"gp"_(.15){}="ff1"
    \ar@{}"g";"gp"_(.80){}="oo1"
    \ar@<-0.0ex>@/^1ex/@{:>}"ff1";"oo1"^(.35){h'\!\!}_(.30){}="h'"
    \ar@<-1.0ex>@/^-1.5ex/@2"ff1";"oo1"_(.36){\!\!\!h}_(.80){}="h"
    \ar@3"h";"h'"_(.20)H
    \ar@{}"gp2";"fp"_(.25){}="x2"
    \ar@{}"gp2";"fp"_(.75){}="y2"
    \ar@<0.4ex>@2"x2";"y2"^{h''}
    }
  \]
  un diagramme de complexes de Steiner forts, où $g'$ et $g''$ sont des
  inclusions rigides ordonnées, $h$, $h'$, $h''$ et $k$ sont des
  antihomotopies de $g$ vers $g'f$, de $g$ vers $g'f'$, de~$g'$ vers
  $g''f''$ et de $f$ vers $f'$ respectivement et $H$ est une
  $2$-antihomotopie de $g'k + h$ vers~$h'$ (qui sont deux antihomotopies de
  $g$ vers $g'f'$). Considérons le diagramme composé
  \[
    \shorthandoff{;:}
    \xymatrix@C=1.5pc@R=3pc{
      K \ar@/^2ex/[rr]^(.33){f''\!f'}_{}="1" \ar@/_2ex/[rr]^(.30){f''\!f}_{}="0"
      \ar[dr]_{}="f"_{\phantom{g'}g}
      \ar@2"0";"1"_{f''\!k}
      & & K'' \ar[dl]^{g''} \\
      & L
      \ar@{}"f";[ur]_(.15){}="ff"
      \ar@{}"f";[ur]_(.55){}="oo"
      \ar@<-0.5ex>@/^1ex/@{:>}"ff";"oo"^(.18){\!\!}_(.30){}="h'"
      \ar@<-2.0ex>@/^-1ex/@2"ff";"oo"_(.36){}_(.80){}="h"
      \ar@3"h";"h'"_(.20){\,\phantom{H'}} & \pbox{,}
      }
  \]
  la $2$-cellule courbée de devant étant l'antihomotopie $h''f + h$, celle de
  derrière l'antihomotopie $h''f + h'$ et la $3$-cellule la $2$-antihomotopie
  $h''k + H$. Alors le composé
  \[
   \shorthandoff{;}
    \xymatrix@C=4pc{
      \cotr{C}{c''}
      \ar[r]^{(f'', h'', b)^\ast}
      &
      \cotr{C}{c'}
      \ar@/^2.5ex/[r]^{(f', h', b)^\ast}_{}="0"
      \ar@/_2.5ex/[r]_{(f, h, b)^\ast}^{}="1"
      \ar@2"0";"1"
      &
      \cotr{C}{c'}
      \pbox{,}
    }
  \]
  où la $2$-cellule du diagramme est $(k, H, b)^\ast$ et où on pose toujours
  \[ c = b\nu(g), \quad c' = b\nu(g') \quadet c'' = b\nu(g''), \]
  est égal à
  \[
   \shorthandoff{;}
   (f''k, h''k + H, b)^\ast :
    \xymatrix@C=4pc{
      \cotr{C}{c''}
      \ar@/^2.5ex/[r]^{(f''f', h''f' + h', b)^\ast}_{}="0"
      \ar@/_2.5ex/[r]_{(f''f, h''f + h , b)^\ast}^{}="1"
      \ar@2"0";"1"
      & \cotr{C}{c} \pbox{.}
    }
  \]
  Autrement dit, on a
  \[ (k, H, b)^\ast \comp (f'', h'', b)^\ast = (f''k, h''k + H, b)^\ast. \]
\end{proposition}

\begin{proof}
  Voir \cite[proposition 11.5.8]{AraMaltsiJoint}.
\end{proof}

\begin{remark}
  Les transformations oplax $(k, H, b)^\ast$ sont également compatibles à la
  composition verticale des cônes (voir \cite[proposition
  11.5.10]{AraMaltsiJoint}) mais nous n'aurons pas besoin de cette
  fonctorialité dans ce texte. L'ensemble de ces résultats de fonctorialités
  des tranches peut s'exprimer par l'existence d'un sesquifoncteur. La
  source de ce sesquifoncteur est liée à la construction de la sesquicatégorie
  tranche d'une \oo-catégorie de Gray qui sera étudiée dans
  l'appendice~\ref{app:tr_comma}. D'ailleurs, la $2$\nbd-antihomotopie~$h''k$
  apparaissant dans la proposition précédente est la « contrainte de Gray »
  pour la composition horizontale des antihomotopies $k$ et $h''$.
\end{remark}

\section{Un théorème A \pdfoo-catégorique pour les triangles commutatifs}

\begin{paragraph}\label{paragr:def_orient}
  Considérons la catégorie $\ooCat$ des \oo-catégories munie du joint
  \oo-catégorique. À cette catégorie monoïdale, le
  paragraphe~\ref{paragr:permet_thmA_mon} associe un foncteur
  \[
      \cOAug : \cDeltaAug \to \ooCat
  \]
  défini par
  \[ \Deltan{n} \mapsto \On{n} = \On{0} \joint \cdots \joint \On{0}, \]
  où $\On{0} = \Dn{0}$ apparaît $n + 1$ fois, ainsi que, par restriction,
  un objet cosimplicial
  \[
      \cO : \cDelta \to \ooCat
  \]
  et donc un foncteur nerf
  \[
      N : \ooCat \to \pref{\cDelta} \\
  \]
  défini par
  \[
      C \mapsto (\Deltan{n} \mapsto \Hom_{\ooCat}(\On{n}, C)).
  \]
  On appellera $\On{n}$ le \ndef{$n$-ième oriental} et $N$ le \ndef{nerf de
  Street}. Ces objets coïncident avec ceux définis par Street dans
  \cite{StreetOrient} en vertu de \cite[chapitre 7]{AraMaltsiJoint}. En
  particulier, on a
  \[
    \shorthandoff{;}
    \On{0} = \Dn{0} = \xymatrix{\bullet}, \qquad
    \On{1} = \Dn{1} = \xymatrix{\bullet \ar[r] & \bullet}
    \quadet
    \On{2} =
    \raisebox{1.5pc}{
    $\xymatrix@C=1.5pc{
      \bullet \ar[rr] \ar[dr]_{}="f" & & \bullet \ar[dl] \\
      & \bullet
      \ar@{}"f";[ur]_(.15){}="ff"
      \ar@{}"f";[ur]_(.55){}="oo"
      \ar@<-0.5ex>@2"ff";"oo"
    }$
    }
    \text{.}
  \]

  On dira qu'un \oo-foncteur $u : A \to B$ est une
  \ndef{équivalence de Thomason} si $u$ est une équivalence faible au sens
  du paragraphe~\ref{paragr:permet_thmA_mon}, c'est-à-dire si son nerf
  $N(u)$ est une équivalence faible simpliciale. On dira qu'une
  \oo-catégorie $C$ est \ndef{asphérique} si l'unique \oo-foncteur de $C$
  vers la \oo-catégorie terminale est une équivalence de Thomason.
\end{paragraph}

\begin{proposition}\label{prop:nerf_op}
  Soit $C$ une \oo-catégorie. On a un isomorphisme canonique d'ensembles
  simpliciaux
  \[ N(C^\op) \simeq N(C)^\op, \]
  naturel en $C$.
\end{proposition}

\begin{proof}
  Si $\C$ est une catégorie monoïdale de produit tensoriel $\otimes$, on
  appellera dans cette preuve \ndef{transposée} de $\C$ la catégorie
  monoïdale $\trans{\C}$ de même catégorie sous-jacente et de produit
  tensoriel $(X, Y) \mapsto Y \otimes X$. Notons que la transposée est
  fonctorielle sur les foncteurs monoïdaux de manière évidente.
  Avec cette terminologie, la
  proposition~\ref{prop:dual_joint} affirme que la dualité \oo-catégorique
  $C \mapsto C^\op$ est un foncteur monoïdal de la catégorie $\ooCat$ munie
  du joint vers sa transposée. De même, l'automorphisme~$\DDeltaAug$ de la
  catégorie $\cDeltaAug$ (voir le paragraphe~\ref{paragr:def_D}) est un
  foncteur monoïdal de $\cDeltaAug$ munie de la somme ensembliste vers sa
  transposée. Ainsi, on dispose d'un carré
  \[
    \xymatrix{
      \cDeltaAug \ar[r]^-{\cOAug} \ar[d]_{\DDeltaAug} & \ooCat \ar[d]^{\op} \\
      \trans{\!\cDeltaAug} \ar[r]_-{\trans{\cOAug}} & \trans{\ooCat} \\
    }
  \]
  de foncteurs monoïdaux, les structures de catégories monoïdales étant
  celles mentionnées ci-dessus. Nous allons montrer que ce carré est
  commutatif à isomorphisme canonique près. Notons que les deux foncteurs que
  l'on veut comparer envoient $\Deltan{0}$ sur $\Dn{0}$. Ainsi, en vertu de
  la propriété universelle de $\cDeltaAug$ \cite[chapitre VII,
  section~5]{MacLane}, chacun de ces foncteurs monoïdaux correspond à une
  structure de monoïde sur $\Dn{0}$ dans~$\trans{\ooCat}$. Or, $\Dn{0}$
  étant un objet final, il existe une unique telle structure. La propriété
  universelle de $\cDeltaAug$ fournit donc l'isomorphisme recherché.

  Par restriction, on obtient un isomorphisme entre les objets cosimpliciaux
  en \oo-catégories $\cO \circ \DDelta$ et $\op \circ \cO$. Ainsi, les deux
  foncteurs nerfs associés sont canoniquement isomorphes. Or, le foncteur
  nerf associé au premier objet cosimplicial envoie une \oo-catégorie $C$
  sur $N(C) \circ \DDelta = N(C)^\op$, alors que le foncteur nerf associé au
  second envoie~$C$ sur
  $\Hom_{\ooCat}(\On{\var}^\op, C) \simeq \Hom_{\ooCat}(\On{\var}, C^\op) \simeq
  N(C^\op)$, d'où le résultat.
\end{proof}

\begin{corollary}\label{coro:Thom_op}
  Un \oo-foncteur $u  : A \to B$ est une équivalence de Thomason si et
  seulement si $u^\op : A^\op \to B^\op$ en est une.
\end{corollary}

\begin{proof}
  Cela résulte de la proposition précédente et du fait que la
  classe des équivalences faibles simpliciales est stable par la dualité~$X
  \mapsto X^\op$.
\end{proof}

Notre but est maintenant de montrer que la catégorie $\ooCat$ munie du joint
permet un théorème A au sens du paragraphe~\ref{paragr:permet_thmA_mon}.

\begin{paragraph}\label{paragr:def_retr}
  Soit $i : A \to B$ un \oo-foncteur. Une \ndef{structure de rétracte par
  transformation oplax à gauche} (resp. \ndef{à droite}) sur $i$ consiste en
  la donnée de :
  \begin{enumerate}
    \item une rétraction $r : B \to A$ de $i$ (de sorte qu'on a $ri =
    \id{A}$) ;
    \item une transformation oplax $\alpha$ de $ir$ vers $\id{B}$ (resp. de
    $\id{B}$ vers $ir$).
  \end{enumerate}
  On omettra parfois les indications « à gauche » ou « à droite » dans les
  énoncés abstraits qui sont valables pour les deux variantes (à condition
  de rester cohérent dans un même énoncé).

  Si on dispose d'un \oo-foncteur $q : B \to C$, on dira que la structure
  est \ndef{au-dessus de $C$} si on a
  \[ qir = q \quadet q \comp \alpha = \id{q}. \]
  En particulier, on pourra utiliser cette notion pour $C = A$ et $q = r$
  (auquel cas, l'égalité $qir = q$ est automatique).
  De même, si on dispose d'un \oo-foncteur $j : C \to B$, on dira que la
  structure est \ndef{au-dessous de $C$} si on a
  \[ irj = j \quadet \alpha \comp j = \id{j}. \]
  Le cas où $C = A$ et $j = i$ (ce qui entraîne l'égalité $irj = j$) est
  particulièrement important : dans ce cas, on dira que la structure est
  \ndef{forte}.

  On dit que $i$ est un \ndef{rétracte par transformation oplax à gauche} (resp.
  \emph{à droite}) si $i$ admet une structure de rétracte par transformation
  oplax à gauche (resp. à droite). On qualifiera un tel rétracte d'au-dessus
  de $C$, d'au-dessous de $C$ ou de fort en fonction des propriétés des
  structures que $i$ peut admettre.

  On appellera \ndef{rétraction d'un rétracte par transformation oplax} non
  pas n'importe quelle rétraction d'un tel rétracte mais une rétraction $r$
  faisant partie d'une structure~$(r, \alpha)$.

  Toutes les notions introduites dans ce paragraphe admettent également des
  variantes lax obtenues en remplaçant la transformation oplax $\alpha$ par
  une transformation lax.
\end{paragraph}

\begin{paragraph}\label{paragr:contr_fort}
  Soit $C$ une \oo-catégorie. En vertu du paragraphe~\ref{paragr:contr},
  la donnée d'une structure de rétracte par transformation oplax à gauche
  sur un \oo-foncteur $c : \Dn{0} \to C$ correspondant à un objet $c$ de $C$
  est équivalente à celle d'un \oo-foncteur $\Dn{0} \joint C \to C$
  au-dessous de $C$ tel que $\Dn{0} \xto{\iota_1} \Dn{0} \joint C \to C$
  soit $c$, ou encore à celle d'un \oo-foncteur $C \to \cotr{C}{c}$
  au-dessus de~$C$.

  On vérifie qu'une telle structure est forte si et seulement si, dans la
  première description, le composé
  \[
     \Dn{0} \joint \Dn{0} \xto{\,\Dn{0} \joint c\,} \Dn{0} \joint C
     \xto{\phantom{\Dn{0} \joint c}} C
  \]
  correspond à la $1$-cellule $\id{c}$ de $C$ (rappelons qu'on a $\Dn{0}
  \joint \Dn{0} \simeq \Dn{1}$) ou, dans la deuxième description, le composé
  \[
    \Dn{0} \xto{c} C \to \cotr{C}{c}
  \]
  correspond à l'objet $(c, \id{c})$ de $\cotr{C}{c}$.
\end{paragraph}

\begin{prop}\label{prop:retr_univ}
  Si $i : A \to B$ est un rétracte par transformation oplax au-dessus d'une
  \oo-catégorie $C$, alors tout changement de base de $i$ au-dessus de $C$
  est un rétracte par transformation oplax. Autrement dit, pour tout
  \oo-foncteur \hbox{$D \to C$}, le \oo-foncteur
  \[
    i \times_C D : A \times_C D \to B \times_C D
  \]
  est un rétracte par transformation oplax.

  Plus précisément, si $(r, \alpha)$ est une structure de rétracte par
  transformation oplax sur $i$ au-dessus de $C$, alors $(r \times_C D, \alpha
  \times_C D)$ est une structure de rétracte par transformation oplax sur $i
  \times_C D$ au-dessus de $D$. De plus, si la structure $(r, \alpha)$ est
  au-dessous d'une \oo-catégorie $E$, alors la nouvelle structure est
  au-dessous de $E \times_C D$.
\end{prop}

\begin{proof}
  Puisque $r$ est au-dessus de $C$, le \oo-foncteur $r \times_C D$ a un sens
  et est bien une rétraction de $i \times_C D$ par fonctorialité du
  changement de base. De plus, $\alpha$ étant au-dessus de $C$, le
  transformation oplax $\alpha \times_C D$ a également un sens en vertu du
  paragraphe~\ref{paragr:img_inv_trans} et on conclut de nouveau par
  fonctorialité du changement de base.
\end{proof}

\begin{prop}\label{prop:retr_Thomason}
  Un rétracte par transformation oplax est une équivalence de Thomason.
\end{prop}

\begin{proof}
  C'est un cas particulier du corollaire~\ref{coro:retr_Thom}.
\end{proof}

\begin{coro}\label{coro:retr_equiv_Thom}
  Un rétracte par transformation oplax au-dessus d'une \oo-catégorie
  $C$, de même que la rétraction d'un tel rétracte, est une équivalence de
  Thomason et le reste après tout changement de base au-dessus de $C$.
\end{coro}

\begin{proof}
  Cela résulte des deux propositions précédentes.
\end{proof}

\begin{paragr}\label{paragr:def_retr_Steiner}
  Soit $i : K \to L$ un morphisme de complexes dirigés augmentés.
  Une \ndef{structure de rétracte par antihomotopie fort} sur $i$ est
  la donnée de :
  \begin{enumerate}
    \item une rétraction $r : L \to K$ de $i$ ;
    \item une antihomotopie $h$ de $\id{L}$ vers $ir$ vérifiant $hi = 0$.
  \end{enumerate}
  Si $rh = 0$, on dira que la structure est \ndef{au-dessus de $K$}. Si $hh
  = 0$, on parlera de \ndef{structure de rétracte par antihomotopie de carré
  nul fort}.

  On dira que $i$ est un \ndef{rétracte par antihomotopie fort} s'il admet
  une structure de rétracte par antihomotopie fort. On qualifiera $i$ de
  rétracte par antihomotopie \ndef{de carré nul} fort ou \ndef{au-dessus de
  $K$} en fonction des propriétés des structures que $i$ peut admettre.
\end{paragr}

\begin{remark}
  Tout comme la notion de rétracte par transformation \oo-catégorique
  du paragraphe~\ref{paragr:def_retr}, la définition du paragraphe précédent
  admet de nombreuses variantes. On a ici privilégié les structures fortes,
  à droite et les antihomotopies (qui correspondent aux transformations lax,
  voir \cite[remarque B.4.11]{AraMaltsiJoint}). La condition « de carré nul »,
  qu'on n'a pas introduite dans le contexte \oo-catégorique même si elle
  a également un sens, correspond à la trivialité d'une « contrainte de
  Gray ».
\end{remark}

La proposition suivante est le cœur de la démonstration présentée dans ce
texte du théorème A \oo-catégorique. Elle repose sur les résultats de
fonctorialité des tranches rappelés dans la section précédente.

\begin{prop}\label{prop:retr}
  Soit $i : K \to L$ une inclusion rigide ordonnée entre complexes de
  Steiner forts qui est un rétracte par antihomotopie de carré nul fort
  au-dessus de~$K$ et soit $C$ une \oo-catégorie munie d'un \oo-foncteur $b
  : \nu(L) \to C$. Posons $c = b\nu(i)$ et $c' = b$. Alors le foncteur
  $\nu(i)^\ast : \cotr{C}{c'} \to \cotr{C}{c}$ (du
  paragraphe~\ref{paragr:fonct_tr_sur_1}) est la rétraction d'un rétracte
  par transformation oplax à gauche fort au-dessus de $C$.
\end{prop}

\begin{proof}
  Notons tout d'abord que, d'après la proposition~\ref{prop:fonct_tri_comm},
  le \oo-foncteur $\nu(i)^\ast$ coïncide avec le \oo-foncteur $(i, \id{i},
  b)^\ast$ associé au diagramme
  \[
    \shorthandoff{;}
    \xymatrix@C=1.5pc{
      K \ar[rr]^i \ar[dr]_{i}_{}="f" & & L \ar@{=}[dl]_{}="g" \\
      & L
      \ar@{}"f";[ur]_(.15){}="ff"
      \ar@{}"f";[ur]_(.55){}="oo"
      \ar@{}"f";"g"^{\textstyle =}
      &
    }
  \]
  en vertu du paragraphe~\ref{paragr:fonct_tri}. Soit $(r, h)$ une structure
  de rétracte par antihomotopie fort sur~$i$ et soit
  $(r, h, b)^\ast$ le \oo-foncteur associé au diagramme
  \[
    \shorthandoff{;}
    \xymatrix@C=1.5pc{
      L \ar[rr]^r \ar@{=}[dr]_{}="f" & & K \ar[dl]^(0.42){i} \\
      & L
      \ar@{}"f";[ur]_(.15){}="ff"
      \ar@{}"f";[ur]_(.55){}="oo"
      \ar@<-0.5ex>@2"ff";"oo"^{h}
      &
    }
  \]
  en vertu du paragraphe~\ref{paragr:fonct_tri}. Considérons le diagramme
  \[
    \shorthandoff{;}
    \xymatrix{
      K \ar[r]^i \ar[dr]_{}="g"_(.40)i
      & L \ar[r]^r_(.75){}="fp" \ar@{=}[d]_(.70){}="gp"_(.50){}="gp1" & K
      \ar[dl]_{}="gpp"^(.40)i \\
      & L
      \ar@{}"g";[u]_(0.10){}="x"
      \ar@{}"g";[u]_(.85){}="y"
      \ar@{}^{\textstyle =}"g";"gp1"
      \ar@{}"gp";"fp"_(.25){}="x2"
      \ar@{}"gp";"fp"_(.75){}="y2"
      \ar@<0.4ex>@2"x2";"y2"^h
      & \pbox{.}
    }
  \]
  D'après la proposition~\ref{prop:fonct_tri}, on a
  \[
    (i, \id{i}, b)^\ast (r, h, b)^\ast = (ri, hi, b)^\ast = (\id{K},
    \id{\id{K}}, b)^\ast = \id{\cotr{C}{c}},
  \]
  la dernière égalité résultant de la proposition~\ref{prop:fonct_tri_id},
  et $(i, \id{i}, b)^\ast$ est donc une rétraction de $(r, h, b)^\ast$. On
  obtient de même, en considérant le diagramme
  \[
    \shorthandoff{;}
    \xymatrix{
      L \ar[r]^r \ar@{=}[dr]_{}="g"
      & K \ar[r]^i_(.75){}="fp" \ar[d]_{}="gp"_(.50)i & L
      \ar@{=}[dl]_{}="gpp" \\
      & L
      \ar@{}"g";[u]_(0.10){}="x"
      \ar@{}"g";[u]_(.85){}="y"
      \ar@<-0.1ex>@2"x";"y"^(.30)h
      \ar@{}^{\textstyle =}"gp";"gpp"
      & \pbox{,}
    }
  \]
  l'égalité
  \[
    (r, h, b)^\ast (i, \id{i}, b)^\ast = (ir, h, b)^\ast.
  \]
  Par ailleurs, la construction du paragraphe~\ref{paragr:fonct_cone}
  associe au diagramme
  \[
      \shorthandoff{;:}
      \xymatrix@C=1.5pc@R=3pc{
        L \ar@/^2ex/[rr]^(.33){ir}_{}="1" \ar@{=}@/_2ex/[rr]_{}="0"
        \ar@{=}[dr]_{}="f"
        \ar@2"0";"1"_h
        & & L \ar@{=}[dl] \\
        & L
        \ar@{}"f";[ur]_(.10){}="ff"
        \ar@{}"f";[ur]_(.50){}="oo"
        \ar@<-0.5ex>@/^1ex/@{:>}"ff";"oo"^(.25){h\!}_(.30){}="h'"
        \ar@<-2.0ex>@/^-1ex/@{=}"ff";"oo"_(.80){}="h"
        \ar@3{-}"h";"h'" & \pbox{,}
        }
  \]
  une transformation oplax $(h, \id{h}, b)^\ast$ de $(ir, h, b)^\ast$
  vers~$(\id{L}, \id{\id{L}}, b)^\ast = \id{\cotr{C}{c'}}$. Cette
  transformation est au-dessus de $C$ en vertu de la
  proposition~\ref{prop:fonct_cone_sur}.

  Supposons enfin que $rh = 0$ et $hh = 0$, et considérons le
  diagramme
  \[
      \shorthandoff{;:}
      \xymatrix@C=3.5pc@R=3.5pc{
      L
      \ar@/^2ex/[r]^(.33){ir}_{}="1"
      \ar@/_2ex/@{=}[r]_{}="0"_(.70){}="f"
      \ar@{=}[dr]_{}="g"
      \ar@2"0";"1"_{h}
      &
      L \ar[r]^{r}_(.75){}="fp"
         \ar@{=}[d]_(.70){}="gp2"_(.20){}="gp"
      &
      K \ar[dl]^(.62){i}
      \\
      & L
      \ar@{}"g";"gp"_(.15){}="ff1"
      \ar@{}"g";"gp"_(.80){}="oo1"
      \ar@<-0.0ex>@/^1ex/@{:>}"ff1";"oo1"^(.25){h\!}_(.30){}="h'"
      \ar@<-1.0ex>@/^-1.5ex/@{=}"ff1";"oo1"_(.80){}="h"
      \ar@3{-}"h";"h'"
      \ar@{}"gp2";"fp"_(.25){}="x2"
      \ar@{}"gp2";"fp"_(.75){}="y2"
      \ar@<0.4ex>@2"x2";"y2"^{h} & \pbox{.}
      }
    \]
    La proposition~\ref{prop:fonct_cone_2} entraîne qu'on a
    \[ (h, \id{h}, b)^\ast \comp (r, h, b)^\ast = (rh, hh, b)^\ast
     = (\id{r}, \id{h}, b)^\ast = \id{(r, h, b)^\ast}, \]
    la dernière égalité résultant de la
    proposition~\ref{prop:fonct_cone_id}, d'où le résultat.
\end{proof}

\begin{remark}
  On a montré plus précisément que si $(r, h)$ est une structure de rétracte
  par antihomotopie fort sur~$i$, alors $((i, \id{i}, b)^\ast, (h, \id{h},
  b)^\ast)$ est une structure de rétracte par transformation oplax à gauche
  au-dessus de $C$ sur~$(r, h, b)^\ast$ et que si, de plus, la structure
  $(r, h)$ est au-dessus de $K$ et l'antihomotopie $h$ est de carré nul,
  alors la structure $((i, \id{i}, b)^\ast, (h, \id{h}, b)^\ast)$ est forte.
\end{remark}

\begin{coro}\label{coro:anticontr_Thomason}
  Soit $i : K \to L$ une inclusion rigide ordonnée entre complexes de
  Steiner forts qui est un rétracte par antihomotopie de carré nul fort
  au-dessus de~$K$ et soient $u : A \to C$ et $b : \nu(L) \to C$ deux
  \oo-foncteurs. Posons $c = b\nu(i)$ et $c' = b$. Alors le \oo-foncteur
  $\nu(i)^\ast : \cotr{A}{c'} \to \cotr{A}{c}$ (du
  paragraphe~\ref{paragr:fonct_tr_sur_2}) est la rétraction d'un rétracte
  par transformation oplax à gauche fort au-dessus de $A$ et, en
  particulier, une équivalence de Thomason.
\end{coro}

\begin{proof}
  Le \oo-foncteur $\nu(i)^\ast : \cotr{A}{c'} \to \cotr{A}{c'}$ s'obtient
  par définition comme changement de base le long de $A \to C$ du \oo-foncteur
  $\nu(i)^\ast : \cotr{C}{c'} \to \cotr{C}{c}$ (du
  paragraphe~\ref{paragr:fonct_tr_sur_1}). Or ce dernier \oo-foncteur
  est la rétraction d'un rétracte par transformation oplax à gauche fort
  au-dessus de $C$ d'après la proposition précédente. Le résultat découle
  donc des propositions~\ref{prop:retr_univ} et~\ref{prop:retr_Thomason}.
\end{proof}

\begin{paragraph}\label{paragr:def_cn}
  Considérons la catégorie $\Cda$ des complexes dirigés augmentés munie du
  joint des complexes. En vertu du paragraphe~\ref{paragr:permet_thmA_mon},
  on associe à cette catégorie monoïdale un objet cosimplicial
  \[ \cn : \cDelta \to \Cda \]
  défini par
  \[
    \Deltan{n} \mapsto \cn(\Deltan{n}) =
     \cn(\Deltan{0}) \joint \cdots \joint \cn(\Deltan{0}),
  \]
  où $\cn(\Deltan{0}) = \lambda(\Dn{0})$ apparaît $n + 1$ fois. Notons que,
  d'après la proposition~\ref{prop:joint_Steiner}, le complexe
  $\cn(\Deltan{n})$ est un complexe de Steiner fort. Ainsi, il résulte du
  fait que la restriction du foncteur $\nu : \Cda \to \ooCat$ aux complexes
  de Steiner forts est monoïdale pour le joint que l'objet cosimplicial
  $\cO : \cDelta \to \ooCat$ du paragraphe~\ref{paragr:def_orient} se
  factorise (à isomorphisme canonique près) en
  \[ \cDelta \xto{\,\cn\,} \Cda \xto{\,\nu\,} \ooCat. \]
  En particulier, on a un isomorphisme canonique $\On{n} \simeq
  \nu(\cn(\Deltan{n}))$ qu'on considérera comme une égalité.
\end{paragraph}

\begin{paragraph}\label{paragr:base_cn}
  Fixons $n \ge 0$. Le complexe dirigé augmenté $\cn(\Deltan{n})$ du
  paragraphe précédent  peut se décrire explicitement de la manière
  suivante (voir \cite[paragraphe 7.3 et remarque 7.7]{AraMaltsiJoint}).
  Pour $p \ge 0$, le groupe abélien $\cn(\Deltan{n})_p$ est le groupe
  abélien libre sur l'ensemble
  \[
    B_p = \{(i_0, \dots, i_p) \mid 0 \le i_0 < \cdots < i_p \le n\}.
  \]
  Si $p \ge 1$ et si $(i_0, \dots, i_p)$ est dans $B_p$, on a
  \[
    d(i_0, \dots, i_p) = \sum_{k = 0}^p (-1)^k (i_0, \dots, \hat i_k,
    \dots, i_p),
  \]
  où on a posé $(i_0, \dots, \hat i_k, \dots, i_p) =
  (i_0, \dots, i_{k-1}, i_{k+1}, \dots, i_p)$.
  Si $(i_0)$ est dans $B_0$, on a
  \[ e(i_0) = 1. \]
  Enfin, les sous-monoïdes de positivité $\cn(\Deltan{n})^\ast_p$ sont les
  sous-monoïdes engendrés par les ensembles $B_p$. Notons que les $B_p$
  forment une base du complexe dirigé augmenté~$\cn(\Deltan{n})$.

  De plus, si $f : \Deltan{n} \to \Deltan{n'}$ est un morphisme de
  $\cDelta$, le morphisme associé $\cn(f) : \cn(\Deltan{n}) \to
  \cn(\Deltan{n'})$ est donné sur la base de $\cn(\Deltan{n})$ par
  \[ f(i_0, \dots, i_p) = (f(i_0), \dots, f(i_p)), \]
  où on convient que $(j_0, \dots, j_p) = 0$ si la suite des $j_k$ n'est pas
  strictement croissante. Notons qu'avec cette convention, les formules
  définissant la différentielle $d$ et l'application $f$ restent valables
  pour $(i_0, \dots, i_p)$ avec $0 \le i_0 \le \cdots \le i_p \le n$.
\end{paragraph}

\begin{prop}\label{prop:Deltan_contr}
  Pour tout $m \ge 0$, le morphisme $m : \cn(\Deltan{0}) \to
  \cn(\Deltan{m})$, image par $\cn$ du morphisme $m : \Deltan{0} \to
  \Deltan{m}$, est un rétracte par antihomotopie de carré nul fort au-dessus
  de $\cn(\Deltan{0})$.
\end{prop}

\begin{proof}
  L'unique morphisme $\Deltan{m} \to \Deltan{0}$ dans $\cDelta$ induit
  une rétraction $r : \cn(\Deltan{m}) \to \cn(\Deltan{0})$ de $m$. Nous
  allons produire une antihomotopie $h$ de $\id{\cn(\Deltan{m})}$ vers~$mr$.
  Notons que $mr$ vérifie
  \[
    mr(i_0, \dots, i_p) =
    \begin{cases}
      (m) & \text{si $p = 0$,} \\
      0 & \text{sinon,}
    \end{cases}
  \]
  sur la base de $\cn(\Deltan{m})$ ou, plus généralement, pour $0 \le i_0
  \le \cdots \le i_p \le m$, avec la convention du paragraphe précédent. On
  définit $h$ sur cette même base par
  \[
    h(i_0, \dots, i_p) = (i_0, \dots, i_p, m),
  \]
  où, suivant la convention du paragraphe précédent, cette expression est
  nulle si $i_p = m$. Encore une fois, cette expression reste valable pour
  $0 \le i_0 \le \cdots \le i_p \le m$. Vérifions que cette formule définit
  bien une antihomotopie de $\id{\cn(\Deltan{m})}$ vers $mr$. Fixons un
  élément $(i_0, \dots, i_p)$ de la base de $\cn(\Deltan{m})$. On distingue
  deux cas :
  \begin{itemize}
    \item Si $p = 0$, on a
      \[ dh(i_0) = d(i_0, m) = (m) - (i_0) = mr(i_0) - (i_0). \]
    \item Si $p \ge 1$, on a
      \[
        \begin{split}
          \MoveEqLeft
          dh(i_0, \dots, i_p) - hd(i_0, \dots, i_p)  \\
          & = d(i_0, \dots, i_p, m) - \sum_{k = 0}^p (-1)^k h(i_0, \dots, \hat i_k, \dots, i_p) \\
          & = \sum_{k = 0}^p (-1)^k (i_0, \dots, \hat i_k, \dots, i_p, m) +
          (-1)^{p+1}(i_0, \dots, i_p) \\*
          & \phantom{=1} \qquad - \sum_{k = 0}^p (-1)^k (i_0, \dots, \hat i_k, \dots, i_p, m) \\
          & = (-1)^p (0 - (i_0, \dots, i_p)) \\
          & = (-1)^p (mr(i_0, \dots, i_p) - (i_0, \dots, i_p)).
        \end{split}
      \]
  \end{itemize}
  Enfin, il est immédiat qu'on a bien $hm = 0$, $rh =
  0$ et $hh = 0$.
\end{proof}

\begin{prop}\label{prop:permet_thmA}
  Soient $v : A \to C$ un \oo-foncteur et $c$ un $m$-simplexe
  de~$N(C)$. Alors le \oo-foncteur $m^\ast :
  \cotr{A}{c} \to \cotr{A}{c_m}$ est la rétraction d'un rétracte par
  transformation oplax à gauche fort au-dessus de $A$ et, en particulier,
  une équivalence de Thomason.
\end{prop}

\begin{proof}
  C'est exactement le contenu du corollaire~\ref{coro:anticontr_Thomason}
  appliqué au morphisme $m : \cn(\Deltan{0}) \to \cn(\Deltan{m})$, ses
  hypothèses étant vérifiées en vertu de la proposition précédente (et du
  fait immédiat que $m$ est bien une inclusion rigide ordonnée).
\end{proof}

\begin{coro}
  La catégorie $\ooCat$ des \oo-catégories munie du joint permet un théorème
  A au sens du paragraphe~\ref{paragr:permet_thmA_mon}.
\end{coro}

\begin{proof}
  En vertu de la proposition~\ref{prop:def_equiv_permet_A}, l'assertion est
  équivalente à la proposition précédente, d'où le résultat.
\end{proof}

\begin{thm}\label{thm:thmA}
  Soit
  \[
    \xymatrix@C=1.5pc{
      A \ar[rr]^u \ar[dr]_v & & B \ar[dl]^w \\
      & C
    }
  \]
  un triangle commutatif de \oo-foncteurs. Si pour tout objet $c$ de $C$, le
  \oo-foncteur \hbox{$\cotr{A}{c} \to \cotr{B}{c}$} est une équivalence de
  Thomason, alors il en est de même de $u$.
\end{thm}

\begin{proof}
  En vertu du corollaire précédent, on peut appliquer le théorème~A
  monoïdal (théorème~\ref{thm:thmA_monoid}) à $\ooCat$ munie du joint. On
  obtient alors exactement l'assertion qu'on voulait démontrer.
\end{proof}



\begin{remark}\label{rem:thmA_dual}
  La stabilité des équivalences de Thomason par la dualité $C \mapsto C^\op$
  (voir le corollaire~\ref{coro:Thom_op}) permet de déduire du théorème précédent
  un théorème analogue pour les tranches de type \smash{$\trm{C}{c}$} (voir
  le théorème~\ref{thm:thmA_lax} pour un énoncé plus général).
  Les résultats analogues pour les tranches de type $\tr{C}{c}$ ou
  \smash{$\cotrm{C}{c}$} (voir \cite[remarque~6.37]{AraMaltsiJoint}) sont également
  vrais. Néanmoins, pour les établir, on a besoin de savoir que la classe des
  équivalences de Thomason est stable par la dualité~$C \mapsto C^\co$, ce
  qui sera démontré dans~\cite{AraMaltsiNerfs}.
\end{remark}

Les théorèmes A que l'on vient d'établir sont relatifs au sens où ils
traitent de \oo-foncteurs au-dessus d'une \oo-catégorie. Afin de déduire de
ces résultats des théorèmes A absolus, nous avons besoin de montrer que les
tranches $\cotr{C}{c}$, où $c$ est un objet de $C$, sont asphériques. Cela
résulte facilement de l'énoncé analogue pour les ensembles simpliciaux (voir
la remarque~\ref{rem:tr_asp}). Nous exposons maintenant une
démonstration alternative dont les conséquences joueront un rôle important
dans la section~\ref{sec:thmA_2-tri}.

\begin{proposition}
  Soient $C$ une \oo-catégorie et $c$ un objet de $C$. On a un isomorphisme
  canonique
  \[
    \cotr{\big(\cotr{C}{c}\big)}{(c, \id{c})}
    \simeq
    \cotr{C}{\id{c}},
  \]
  où, à gauche, le couple $(c, \id{c})$ est considéré comme un objet de
  $\cotr{C}{c}$ (voir le paragraphe \ref{paragr:obj_tr}) et, à droite,
  $\id{c}$ est considéré comme une $1$-cellule de $C$, c'est-à-dire un
  \oo-foncteur $\Dn{1}
  \to C$.
\end{proposition}

\begin{proof}
  Soit $T$ une \oo-catégorie. Par adjonction, se donner un \oo-foncteur $T
  \to \cotr{\big(\cotr{C}{c}\big)}{(c, \id{c})}$ revient à se donner un
  \oo-foncteur $\Dn{0} \joint T \to \cotr{C}{c}$ rendant le triangle
  \[
    \xymatrix{
      \Dn{0} \joint T \ar[r] & \cotr{C}{c} \\
      \Dn{0} \ar[u]^{\iota_1} \ar[ur]_{(c, \id{c})}
    }
  \]
  commutatif. De nouveau par adjonction, en vertu de la définition de
  $(c, \id{c})$ (voir le paragraphe~\ref{paragr:obj_tr}), cela revient à se
  donner un \oo-foncteur $\Dn{0} \joint \Dn{0} \joint T \to C$ rendant
  commutatif le triangle
  \[
    \xymatrix{
      (\Dn{0} \joint \Dn{0}) \joint T \ar[r] & C \\
      \Dn{0} \joint \Dn{0} \ar[u]^{\iota_1} \ar[ur]_{\id{c}} & \pbox{,}
    }
  \]
  où on a identifié $\Dn{0} \joint \Dn{0}$ et $\Dn{1}$. Par adjonction, cela
  revient à se donner un \oo-foncteur~$T \to \cotr{C}{\id{c}}$, d'où le
  résultat en vertu du lemme de Yoneda.
\end{proof}

\begin{proposition}\label{prop:tr_fort_asp}
  Soient $C$ une \oo-catégorie et $c$ un objet de $C$. Le \oo-foncteur
  $\Dn{0} \to \cotr{C}{c}$ correspondant à l'objet $(c, \id{c})$ est un
  rétracte par transformation oplax à gauche fort.
\end{proposition}

\begin{proof}
  En vertu du paragraphe \ref{paragr:contr_fort}, il suffit de construire
  un \oo-foncteur $\cotr{C}{c} \to \cotr{\big(\cotr{C}{c}\big)}{(c,
  \id{c})}$ au-dessus de $\cotr{C}{c}$ (satisfaisant à une propriété
  additionnelle pour le caractère fort) ou encore, en utilisant la
  proposition précédente, un \oo-foncteur $\cotr{C}{c} \to \cotr{C}{\id{c}}$
  au-dessus de $\cotr{C}{c}$. On obtient ce \oo-foncteur comme le
  \oo-foncteur~$\kappa^\ast$ associé, en vertu du
  paragraphe~\ref{paragr:fonct_tr_sur_1}, au triangle commutatif
  \[
    \xymatrix@C=1.5pc{
      \Dn{1} \ar[dr]_{\id{c}} \ar[rr]^{\kappa} & & \Dn{0} \ar[dl]^{c} \\
                                               & C & \pbox{,}
    }
  \]
  où $\kappa$ désigne le \oo-foncteur du paragraphe~\ref{paragr:def_Dn}.
  Le \oo-foncteur
  $\cotr{C}{\id{c}} \to \cotr{C}{c}$ correspondant au \oo-foncteur d'oubli
  $\cotr{\big(\cotr{C}{c}\big)}{(c, \id{c})} \to \cotr{C}{c}$ à
  travers l'isomorphisme de la proposition précédente s'identifie lui au
  \oo-foncteur~$\tau^\ast$ associé au triangle commutatif
  \[
    \xymatrix@C=1.5pc{
      \Dn{0} \ar[dr]_{c} \ar[rr]^{\tau} & & \Dn{1} \ar[dl]^{\id{c}} \\
      & C & \pbox{,}
    }
  \]
  où $\tau$ désigne le \oo-foncteur du paragraphe~\ref{paragr:def_Dn}.
  On en déduit que $\kappa^\ast : \cotr{C}{c} \to \cotr{C}{\id{C}}$ est bien
  au-dessus de $\cotr{C}{c}$ par fonctorialité de la construction du
  paragraphe~\ref{paragr:fonct_tr_sur_1}. Il nous reste à vérifier le
  caractère fort du rétracte par transformation. Par définition
  de~$\kappa^\ast$, l'objet
  \[ \Dn{0} \xto{(c, \id{c})} \cotr{C}{c} \xto{\,\kappa^\ast\,} \cotr{C}{\id{c}} \]
  de $\cotr{C}{\id{c}}$ correspond au \oo-foncteur
  \[ \Dn{1} \joint \Dn{0} \xto{\kappa \joint \Dn{0}} \Dn{0} \joint \Dn{0}
  \xto{\,\,\id{c}\,\,} C, \]
  c'est-à-dire à un \oo-foncteur constant de valeur $c$, ce qui montre bien
  en vertu du paragraphe~\ref{paragr:contr_fort} que le rétracte par
  transformation de l'énoncé est fort.
\end{proof}

\begin{corollary}\label{coro:tr_asp}
  Soient $C$ une \oo-catégorie et soit $c$ un objet de $C$. Les
  \oo-catégories $\cotr{C}{c}$ et \smash{$\trm{C}{c}$} sont asphériques.
\end{corollary}

\begin{proof}
  Le cas de $\cotr{C}{c}$ résulte de la proposition précédente en vertu du
  corollaire~\ref{coro:retr_equiv_Thom} (et du fait que $N(\Dn{0}) =
  \Deltan{0}$). Celui de \smash{$\trm{C}{c}$} en résulte
  par dualité, grâce à la proposition~\ref{prop:tr_op} et au
  corollaire~\ref{coro:Thom_op} :
  \[ N(\trm{C}{c}) \simeq N((\cotr{C^\op}{c})^\op) \simeq
  N(\cotr{C^\op}{c})^\op. \qedhere \]
\end{proof}

\begin{remark}\label{rem:tr_asp}
  On peut montrer le corollaire précédent de manière plus directe. En effet,
  en vertu de la proposition~\ref{prop:comp_nerf_tr}, on a $N(\cotr{C}{c})
  \simeq \cotr{N(C)}{c}$ et cet ensemble simplicial est faiblement
  contractile d'après la proposition~\ref{prop:tr_contr}. L'asphéricité de
  $\smash{\trm{C}{c}}$ s'en déduit par dualité comme dans la preuve
  précédente.
\end{remark}

\begin{corollary}
  Soit $u : A \to B$ un \oo-foncteur. Si pour tout objet $b$ de $B$, la
  \oo-catégorie $\cotr{A}{b}$ est asphérique, alors $u$ est une équivalence
  de Thomason.
\end{corollary}

\begin{proof}
  Considérons le triangle commutatif
  \[
    \xymatrix@C=1.5pc{
      A \ar[rr]^u \ar[dr]_u & & B \ar[dl]^{\id{B}} \\
      & B & \pbox{.}
    }
  \]
  L'hypothèse permettant d'appliquer le théorème~\ref{thm:thmA} (et donc de
  conclure que $u$ est une équivalence de Thomason) est que, pour tout objet
  $b$ de $B$, le \oo-foncteur $\cotr{A}{b} \to \cotr{B}{b}$ est une
  équivalence de Thomason. Puisque $\cotr{B}{b}$ est asphérique en vertu du
  corollaire~\ref{coro:tr_asp}, cette hypothèse est équivalente au fait que
  $\cotr{A}{b}$ soit asphérique, ce qui entraîne le résultat.
\end{proof}

\begin{corollary}
  Soit $u : A \to B$ un \oo-foncteur. Si pour tout objet $b$ de $B$, la
  \oo-catégorie \smash{$\trm{A}{b}$} est asphérique, alors $u$ est une équivalence
  de Thomason.
\end{corollary}

\begin{proof}
  L'assertion se déduit du corollaire précédent par dualité.
\end{proof}

Terminons cette section par une application de ces théorèmes $A$.

\begin{theorem}
  Soit $C$ une \oo-catégorie admettant un objet $c_0$ ayant la
  propriété suivante : pour tout objet $c$ de $C$, la \oo-catégorie
  $\Homi_C(c_0, c)$ (provenant de l'enrichissement de $\ooCat$ sur
  elle-même) est asphérique. Alors $C$ est asphérique.
\end{theorem}

\begin{proof}
  Nous allons appliquer le corollaire précédent au \oo-foncteur $c_0 :
  \Dn{0} \to C$. Soit $c$ un objet de $C$. Il s'agit donc de vérifier que
  \smash{$\trm{\Dn{0}}{c}$} est asphérique. Commençons par calculer
  $\cotr{\Dn{0}}{c}$.  D'après le paragraphe~\ref{paragr:obj_tr}, on a
  \[
    \cotr{\Dn{0}}{c} = \cotr{C}{c} \times_C \{c_0\} \simeq \{c\}
    \times^{\pi_0}_C
    \HomLax(\Dn{1}, C) \,{}^{\pi_1}\!\!\times_C \{c_0\},
  \]
  où $\pi_0$ et $\pi_1$ désignent respectivement les \oo-foncteurs
  $\HomLax(\sigma, C)$ et $\HomLax(\tau, C)$ de $\HomLax(\Dn{1}, C)$ vers
  $\HomLax(\Dn{0}, C) \simeq C$. Or, en vertu de \cite[proposition
  B.6.2]{AraMaltsiJoint}, ce dernier produit fibré est isomorphe à
  $\Homi_C(c, c_0)^\o$. Ainsi, en utilisant la proposition~\ref{prop:tr_op},
  on obtient des isomorphismes
  \[ \trm{\Dn{0}}{c} \simeq (\cotr{\Dn{0}^\op}{c})^\op
  \simeq \big(\Homi_{C^\op}(c, c_0)^\o\big)^\op
  \simeq \Homi_C(c_0, c)^{\co\,\o\,\op}
  \simeq \Homi_C(c_0, c),
  \]
  d'où le résultat.
\end{proof}

\begin{remark}
  Il est également vrai que si $C$ admet un objet $c_0$ tel que, pour tout
  objet $c$ de $C$, la \oo-catégorie $\Homi_C(c, c_0)$ est asphérique, alors
  $C$ est asphérique. Ceci peut se déduire du résultat précédent pourvu
  qu'on sache que la classe des \oo-catégories asphériques est stable par la
  dualité qui inverse le sens des $1$-cellules, ce qui sera établi dans
  \cite{AraMaltsiNerfs}.
\end{remark}

\section{\pdfoo-catégories comma}\label{sec:comma_1}

Dans cette section, on introduit une généralisation \oo-catégorique des
catégories comma. Ces \oo-catégories comma nous permettront, dans la section
suivante, de déduire du théorème A pour les triangles commutatifs un
théorème A pour les triangles commutatifs à transformation près.

\medskip

\emph{On fixe une \oo-catégorie $Z$.}

\begin{paragraph}\label{paragr:def_comma}
  Soient
  \[
    \xymatrix{
      X \ar[r]^f & Z & Y \ar[l]_g
    }
  \]
  deux \oo-foncteurs. On définit la \ndef{\oo-catégorie comma} $f \comma_Z
  g$, qu'on notera également plus simplement $f \comma g$, par le
  produit fibré itéré
  \[
    f \comma g = X \times_Z \HomLax(\Dn{1}, Z) \times_Z Y,
  \]
  limite projective du diagramme
  \[
    \xymatrix{
      X \ar[r]^f & Z & \HomLax(\Dn{1}, Z) \ar[l]_-{\pi_0} \ar[r]^-{\pi_1}
                 & Z & Y \ar[l]_g,
    }
  \]
  où $\pi_\e$, pour $\e = 0, 1$, désigne le \oo-foncteur $\HomLax(\Dn{1}, Z)
  \to \HomLax(\{\e\}, Z) \simeq Z$ induit par l'inclusion~$\{\e\} \hookto
  \Dn{1}$.

  Notons que les projections canoniques fournissent des \oo-foncteurs
  \[
    \xymatrix{X & f \comma g \ar[l]_{p_1} \ar[r]^{p_2} & Y}
  \]
  et donc un \oo-foncteur
  \[ p : f \comma g \to X \times Y. \]
\end{paragraph}

\begin{paragraph}\label{paragr:comma_pu}
  Soient toujours
  \[
    \xymatrix{
      X \ar[r]^f & Z & Y \ar[l]_g
    }
  \]
  deux \oo-foncteurs. La \oo-catégorie comma $f \comma g$ a la propriété
  universelle suivante. Soit $T$ une \oo-catégorie. Par adjonction, la
  donnée d'un \oo-foncteur $\lambda$ de $T$ vers~$\HomLax(\Dn{1}, Z)$
  correspond à celle d'une transformation oplax de~$\pi_0\lambda$
  vers~$\pi_1\lambda$.
  Ainsi, la donnée d'un \oo-foncteur $T \to f \comma g$ correspond à celle
  d'un diagramme
  \[
    \shorthandoff{;}
    \xymatrix@C=1.5pc@R=1.5pc{
      & T \ar[dl]_x \ar[dr]^y \\
      X \ar[dr]_f \ar@{}[rr]_(.35){}="x"_(.65){}="y"
      \ar@2"x";"y"^{\lambda} 
      & & Y \ar[dl]^g \\
      & Z & \pbox{,}
    }
  \]
  où $x$ et $y$ sont des \oo-foncteurs et $\lambda$ est une transformation
  oplax. On notera $(x, \lambda, y)$ le \oo-foncteur $T \to f \comma g$
  correspondant à un tel diagramme.
\end{paragraph}

\begin{remark}\label{rem:comma_lax}
  Le paragraphe précédent exprime une propriété universelle de la
  \oo-catégorie $f \comma g$ en termes de transformations oplax. En
  remplaçant dans cette propriété universelle les transformations oplax par
  des transformations lax, on obtient une \oo-catégorie $f \commalax g$
  définie par
  \[
    f \comma' g = X \times_Z \HomOpLax(\Dn{1}, Z) \times_Z Y.
  \]
  Pour différencier ces deux \oo-catégories, on pourra parler de
  \ndef{\oo-catégorie comma oplax} pour $f \comma g$ et de
  \ndef{\oo-catégorie comma lax} pour $f \comma' g$. Il résulte de la
  dualité entre transformations oplax et transformations lax (voir la fin du
  paragraphe~\ref{paragr:trans_oplax}) qu'on a
  \[ f \comma' g = \big(g^\op \comma f^\op\big)^\op. \]
  Dans ce texte, on travaillera uniquement avec des \oo-catégories comma
  oplax.
\end{remark}

Dans la suite de cette section, on va montrer que la construction comma
définit un foncteur
  \[ \commaCfun : \SpanC \to \ooCat, \]
  où $\tr{\ooCatOpLax}{Z}$ et $\trto{\ooCatOpLax}{Z}$ sont des catégories
  que l'on va maintenant décrire. On montrera dans
  l'appendice~\ref{app:tr_comma} que ce foncteur provient en fait d'un
sesquifoncteur.

\begin{paragraph}
  On définit une catégorie $\tr{\ooCatOpLax}{Z}$, où $\ooCatOpLax$ désigne la
  sesquicatégorie des \oo-catégories, \oo-foncteurs et transformations oplax
  (voir le paragraphe~\ref{paragr:def_ooCatOpLax}), de la manière suivante.
  \begin{itemize}[wide]
    \item Les objets de $\tr{\ooCatOpLax}{Z}$ sont les couples $(X, f)$, où $X$
      est une \oo-catégorie et $f : X \to Z$ un \oo-foncteur,
      c'est-à-dire les diagrammes
      \[
        \xymatrix{
          X \ar[d]_f \\ Z
        }
      \]
      dans $\ooCat$.
    \item Les morphismes sont les diagrammes
  \[
    \shorthandoff{;}
    \xymatrix@C=1.5pc{
      X \ar[rr]^u \ar[dr]_(0.40){\phantom{f'}f}_(.60){}="f" & & X' \ar[dl]^(0.40){f'} \\
      & Z
      \ar@{}"f";[ur]_(.15){}="ff"
      \ar@{}"f";[ur]_(.55){}="oo"
      \ar@<-0.0ex>@2"oo";"ff"_\alpha
      &
    }
  \]
  dans $\ooCatOpLax$, où
  \[ \alpha : f' u \tod f \]
  est donc une transformation oplax. La source d'un tel morphisme est $(X,
  f)$ et son but est $(X', f')$.
    \item L'identité d'un objet
      \[
        \xymatrix{
          X \ar[d]_f \\ Z
        }
      \]
      est le morphisme
  \[
    \shorthandoff{;}
    \xymatrix@C=1.5pc{
      X \ar[rr]^{\id{X}} \ar[dr]_(0.40){f}_(.60){}="f" & & X \ar[dl]^(0.40){f} \\
      & Z
      \ar@{}"f";[ur]_(.15){}="ff"
      \ar@{}"f";[ur]_(.55){}="oo"
      \ar@<-0.0ex>@2"oo";"ff"_{\id{f}}
      & \pbox{.}
    }
  \]
    \item Le composé de deux morphismes composables
  \[
    \shorthandoff{;}
    \xymatrix{
      X \ar[r]^u \ar[dr]_{}="g"_(.40){f}
      & X' \ar[r]^{u'}_(.75){}="fp" \ar[d]_(.70){}="gp"_(.56){f'} & X''
      \ar[dl]_{}="gpp"^(.38){f''} \\
      & Z
      \ar@{}"g";[u]_(0.10){}="x"
      \ar@{}"g";[u]_(.85){}="y"
      \ar@<-0.1ex>@{<=}"x";"y"^(.30){\alpha}
      \ar@{}"gp";"fp"_(.25){}="x2"
      \ar@{}"gp";"fp"_(.75){}="y2"
      \ar@<0.4ex>@{<=}"x2";"y2"^(0.40){\alpha'\!}
    }
  \]
  est le morphisme
  \[
    \shorthandoff{;}
    \xymatrix@C=1.5pc{
      X \ar[rr]^{u''} \ar[dr]_(0.40){\phantom{f'}f}_(.60){}="f"
        & & X'' \ar[dl]^(0.40){f''} \\
      & Z
      \ar@{}"f";[ur]_(.15){}="ff"
      \ar@{}"f";[ur]_(.55){}="oo"
      \ar@<-0.0ex>@2"oo";"ff"_{\alpha''}
      & \pbox{,}
    }
  \]
  où
  \[ u'' = u' u \quadet \alpha'' = \alpha (\alpha' \comp{} u). \]
  \end{itemize}
  Il résulte facilement du fait que $\ooCatOpLax$ est une sesquicatégorie
  qu'on obtient bien ainsi une catégorie. On verra dans
  l'appendice~\ref{app:tr_comma} (voir le
  paragraphe~\ref{paragr:def_tr_Gray} et l'exemple~\ref{ex:OpLaxGray}) que
  cette catégorie est la catégorie
  sous-jacente à une sesquicatégorie.
\end{paragraph}

\begin{paragraph}
  De même, on définit une catégorie $\trto{\ooCatOpLax}{Z}$ de la manière
  suivante.
  \begin{itemize}[wide]
    \item Les objets de $\smash{\trto{\ooCatOpLax}{Z}}$ sont les mêmes que
      ceux de $\tr{\ooCatOpLax}{Z}$.
    \item Les morphismes sont les diagrammes
  \[
    \shorthandoff{;}
    \xymatrix@C=1.5pc{
      Y \ar[rr]^v \ar[dr]_(0.40){\phantom{g'}g}_(.60){}="f" & & Y' \ar[dl]^(0.40){g'} \\
      & Z
      \ar@{}"f";[ur]_(.15){}="ff"
      \ar@{}"f";[ur]_(.55){}="oo"
      \ar@<-0.0ex>@2"ff";"oo"^\beta
      &
    }
  \]
  dans $\ooCatOpLax$, où $\beta : g \tod g'v$ est donc une transformation
  oplax. La source d'un tel morphisme est $(Y,
  g)$ et son but est $(Y', g')$.
    \item Les identités et la composition des morphismes sont définies de
      manière analogue à celles de~$\tr{\ooCatOpLax}{Z}$. (Une description
      précise peut être extraite du paragraphe suivant.)
  \end{itemize}
  On verra dans l'appendice~\ref{app:tr_comma} (voir le
  paragraphe~\ref{paragr:def_trto}) que cette catégorie est la catégorie
  sous-jacente à une sesquicatégorie qui se déduit par dualité de la
  sesquicatégorie de catégorie sous-jacente~$\tr{\ooCatOpLax}{Z}$ mentionnée
  dans le paragraphe précédent.
\end{paragraph}

\begin{paragraph}\label{paragr:def_cat_span}
  Décrivons maintenant la catégorie produit $\SpanC$.
  \begin{itemize}[wide]
    \item Les objets sont les diagrammes
    \[
      \xymatrix{
        X \ar[r]^f & Z & Y \ar[l]_g
      }
    \]
    dans $\ooCat$. On notera $(X, f, g, Y)$ un tel objet.
    \item Les morphismes sont les diagrammes
      \[
        \shorthandoff{;}
        \xymatrix@R=1pc@C=3pc{
          X \ar[dd]_u \ar[dr]^f_{}="f" & & Y \ar[dl]_g_{}="g" \ar[dd]^v \\
            & Z \\
          X' \ar[ur]_{f'} & & Y' \ar[ul]^{g'}
          \ar@{}[ll];"f"_(0.35){}="sa"_(0.85){}="ta"
          \ar@2"sa";"ta"^{\alpha}
          \ar@{}[];"g"_(0.35){}="tb"_(0.85){}="sb"
          \ar@2"sb";"tb"^{\beta}
        }
      \]
      dans $\ooCatOpLax$, où
      \[ \alpha : f' u \tod f \quadet \beta : g \tod g' v \]
      sont donc des transformations oplax. On notera $(u, \alpha, \beta, v)$
      un tel morphisme, sous-entendant ainsi les morphismes $f$, $f'$, $g$
      et $g'$. La source de $(u, \alpha, \beta, v)$ est $(X, f, g,
      Y)$ et son but est~$(X', f', g', Y')$.
    \item L'identité d'un objet
    \[
      \xymatrix{
        X \ar[r]^f & Z & Y \ar[l]_g
      }
    \]
    est le morphisme
      \[
        \shorthandoff{;}
        \xymatrix@R=1pc@C=3pc{
          X \ar[dd]_{\id{X}} \ar[dr]^f_{}="f" & & Y \ar[dl]_g_{}="g"
          \ar[dd]^{\id{Y}} \\
            & Z \\
          X \ar[ur]_{f} & & Y \ar[ul]^{g}
          \ar@{}[ll];"f"_(0.35){}="sa"_(0.85){}="ta"
          \ar@2"sa";"ta"^{\id{f}}
          \ar@{}[];"g"_(0.35){}="tb"_(0.85){}="sb"
          \ar@2"sb";"tb"^{\id{g}} \pbox{.}
        }
      \]
      \item Le composé de deux morphismes composables
     \[
        \shorthandoff{;}
        \xymatrix@R=2pc@C=3pc{
          X \ar[d]_u \ar[dr]^(0.50)f_{}="f" & & Y \ar[dl]_(0.60)g_{}="g"
          \ar[d]^v \\
          X' \ar[r]^(0.60){f'}_(.70){}="f'" \ar[d]_{u'\!}_(.70){}="u'"
          \ar@2[];"f"^(.68){\alpha\phantom{'}}
           & Z & Y' \ar[l]_(0.65){g'}_(.70){}="g'" \ar[d]^{v'}_(.70){}="v'"
          \ar@2"g";[]^(0.30){\beta} 
           \\
          X'' \ar[ur]_(0.51){f''}
          \ar@{}"u'";"f'"_(.20){}="sa'"^(.80){}="ta'"
          \ar@2"sa'";"ta'"^{\alpha'}
          & & Y'' \ar[ul]^(0.52){g''}_{}="g''"
          \ar@{}"g'";"v'"_(.20){}="sb'"^(.80){}="tb'"
          \ar@2"sb'";"tb'"^{\beta'}
        }
      \]
      est le morphisme
    \[
    \shorthandoff{;}
    \xymatrix@R=1pc@C=3pc{
      X \ar[dd]_{u''} \ar[dr]^f_{}="f" & & Y \ar[dl]_g_{}="g" \ar[dd]^{v''} \\
        & Z \\
      X'' \ar[ur]_{f''} & & Y'' \ar[ul]^{g''}
      \ar@{}[ll];"f"_(0.35){}="sa"_(0.85){}="ta"
      \ar@2"sa";"ta"^{\alpha''}
      \ar@{}[];"g"_(0.35){}="tb"_(0.85){}="sb"
      \ar@2"sb";"tb"^{\beta''}
    }
    \]
    où
    \[
      u'' = u' u,
      \quad
      \alpha'' = \alpha (\alpha' \comp u),
      \quad
      \beta'' = (\beta' \comp v) \beta
      \quadet
      v'' = v' v.
    \]
  \end{itemize}
\end{paragraph}

\begin{paragraph}\label{paragr:comma_act_fonct}
  Considérons
  \[
    \shorthandoff{;}
    (u, \alpha, \beta, v) =
    \raisebox{2pc}{
    $\xymatrix@R=1pc@C=3pc{
      X \ar[dd]_u \ar[dr]^f_{}="f" & & Y \ar[dl]_g_{}="g" \ar[dd]^v \\
        & Z \\
      X' \ar[ur]_{f'} & & Y' \ar[ul]^{g'}
      \ar@{}[ll];"f"_(0.35){}="sa"_(0.85){}="ta"
      \ar@2"sa";"ta"^{\alpha}
      \ar@{}[];"g"_(0.35){}="tb"_(0.85){}="sb"
      \ar@2"sb";"tb"^{\beta}
    }$}
  \]
  un morphisme de $\SpanC$. On lui associe un \oo-foncteur
  \[ (u, \alpha) \comma_Z (\beta, v) : f \comma_Z g \to f' \comma_Z g', \]
  qu'on notera plus simplement $(u, \alpha) \comma (\beta, v)$ et parfois
  également $(u, \alpha, \beta, v)_\ast$, de la manière suivante. Soit $T$
  une \oo-catégorie et soit $(x, \lambda, y) : T \to f \comma g$ un
  \oo-foncteur (voir le paragraphe~\ref{paragr:comma_pu}). En composant le
  diagramme
  \[
    \shorthandoff{;}
    \xymatrix@R=1pc@C=3pc{
      & T \ar[dl]_x \ar[dr]^y \\
      X \ar[dd]_u \ar[dr]^f_{}="f" & & Y \ar[dl]_g_{}="g" \ar[dd]^v
      \ar@{}[ll];[]_(0.40){}="x"_(0.60){}="y"
      \ar@2"x";"y"^{\lambda}
       \\
        & Z \\
      X' \ar[ur]_{f'} & & Y' \ar[ul]^{g'}
      \ar@{}[ll];"f"_(0.35){}="sa"_(0.85){}="ta"
      \ar@2"sa";"ta"^{\alpha}
      \ar@{}[];"g"_(0.35){}="tb"_(0.85){}="sb"
      \ar@2"sb";"tb"^{\beta}
      \pbox{,}
    }
  \]
  on obtient un carré correspondant au \oo-foncteur
  \[
  (u x, (\beta \comp y) \lambda (\alpha \comp x)
  ,v y)
  : T \to f' \comma g'.
  \]
  Il résulte du fait que $\ooCatOpLax$ est une sesquicatégorie que cette
  correspondance est naturelle en $T$. En vertu du lemme de Yoneda, on
  a donc bien défini un \oo-foncteur~$f \comma g \to f' \comma g'$.
\end{paragraph}

\begin{prop}\label{prop:comma_fonct}
  Soit $Z$ une \oo-catégorie. Les applications
  \[
    \begin{split}
      (f, g) & \mapsto f \comma_Z g \\
      (u, \alpha, \beta, v) & \mapsto (u, \alpha) \comma_Z (\beta, v)
    \end{split}
  \]
  définissent un foncteur
    \[ \commaCfun : \SpanC \to \ooCat. \]
\end{prop}

\begin{proof}
  Fixons
  \[
    (X, f, g, Y) =
    \xymatrix{
       X \ar[r]^f & Z & Y \ar[l]_g
    }
  \]
  un objet de $\SpanC$, $T$ une \oo-catégorie et $(x, \lambda, y) : T \to f
  \comma g$ un \oo-foncteur, c'est-à-dire un diagramme
  \[
    \shorthandoff{;}
     \xymatrix@C=1.5pc@R=1.5pc{
      & T \ar[dl]_x \ar[dr]^y \\
      X \ar[dr]_f
      \ar@{}[rr]_(.35){}="x"_(.65){}="y"
      \ar@2"x";"y"^{\lambda}
      & & Y \ar[dl]^g \\
        & Z
  }
  \]
  dans $\ooCatOpLax$. On va vérifier la fonctorialité de $\commaCfun$ en
  utilisant le lemme de Yoneda, c'est-à-dire en précomposant les égalités
  qu'on veut démontrer par $(x, \lambda, y)$.

  Commençons par la compatibilité à l'identité de l'objet $(f, X, Y, g)$. On a
  \[
    \begin{split}
      (\id{(X, f, g, Y)})_\ast (x, \lambda, y) & = (\id{X}, \id{f}, \id{g},
      \id{Y})_\ast (x, \lambda, y) \\
      & = (\id{X}\, x, (\id{g} \comp y)
      \lambda (\id{f} \comp x), \id{Y}\, y) \\
      & = (x, \lambda, y) \\
      & = (\id{f \comma g}) (x, \lambda, y),
    \end{split}
  \]
  d'où la compatibilité recherchée.

  Soit maintenant
     \[
        \shorthandoff{;}
        \xymatrix@R=2pc@C=3pc{
          X \ar[d]_u \ar[dr]^(0.50)f_{}="f" & & Y \ar[dl]_(0.60)g_{}="g"
          \ar[d]^v \\
          X' \ar[r]^(0.60){f'}_(.70){}="f'" \ar[d]_{u'\!}_(.70){}="u'"
          \ar@2[];"f"^(.68){\alpha\phantom{'}}
           & Z & Y' \ar[l]_(0.65){g'}_(.70){}="g'" \ar[d]^{v'}_(.70){}="v'"
          \ar@2"g";[]^(0.30){\beta} 
           \\
          X'' \ar[ur]_(0.51){f''}
          \ar@{}"u'";"f'"_(.20){}="sa'"^(.80){}="ta'"
          \ar@2"sa'";"ta'"^{\alpha'}
          & & Y'' \ar[ul]^(0.52){g''}_{}="g''"
          \ar@{}"g'";"v'"_(.20){}="sb'"^(.80){}="tb'"
          \ar@2"sb'";"tb'"^{\beta'}
        }
      \]
   deux morphismes composables de $\SpanC$. Vérifions la compatibilité de
   $\commaCfun$ à leur composition. On a
   \[
    \begin{split}
      \MoveEqLeft
      (u', \alpha', \beta', v')_\ast (u, \alpha, \beta, v)_\ast(x, \lambda,
      y) \\
      & =
      (u', \alpha', \beta', v')_\ast \big(u x, (\beta \comp y)
      \lambda (\alpha \comp x), v y\big) \\
      & =
      (u' u x, \lambda', v' v y),
    \end{split}
  \]
  où
  \[
    \begin{split}
    \lambda' & =
    (\beta' \comp (v y))
      (\beta \comp y)
      \lambda (\alpha \comp x) (\alpha' \comp (u x)) \\
      & =
      \big(((\beta' \comp v) \beta) \comp y\big)
      \lambda
      \big((\alpha (\alpha' \comp u)) \comp x\big).
    \end{split}
  \]
  D'où
  \[
    \begin{split}
      \MoveEqLeft
      (u', \alpha', \beta', v')_\ast (u, \alpha, \beta, v)_\ast(x, \lambda,
      y) \\
      & =
      \big(u' u, \alpha (\alpha' \comp u),
        (\beta' \comp v) \beta, v' v\big)_\ast
        (x, \lambda, y) \\
      & =
        \big((u', \alpha', \beta', v')(u, \alpha, \beta, v)\big)_\ast(x, \lambda,
      y),
    \end{split}
  \]
  ce qui achève de montrer la fonctorialité de $\commaCfun$.
\end{proof}

\begin{proposition}\label{prop:comma_au-dessus}
  Si
  \[
    \shorthandoff{;}
    (u, \alpha, \beta, v) =
    \raisebox{2pc}{
    $\xymatrix@R=1pc@C=3pc{
      X \ar[dd]_u \ar[dr]^f_{}="f" & & Y \ar[dl]_g_{}="g" \ar[dd]^v \\
        & Z \\
      X' \ar[ur]_{f'} & & Y' \ar[ul]^{g'}
      \ar@{}[ll];"f"_(0.35){}="sa"_(0.85){}="ta"
      \ar@2"sa";"ta"^{\alpha}
      \ar@{}[];"g"_(0.35){}="tb"_(0.85){}="sb"
      \ar@2"sb";"tb"^{\beta}
    }$}
  \]
  est un morphisme de $\SpanC$, alors le carré
  \[
    \xymatrix@C=3pc{
      f \comma g \ar[d]_p \ar[r]^-{(u, \alpha) \comma (\beta, v)} &
      f' \comma g' \ar[d]^p \\
      X \times Y \ar[r]_-{u \times v} & X' \times Y'
    }
  \]
  est commutatif.
\end{proposition}

\begin{proof}
  On va procéder comme dans la preuve précédente. Soit $T$ une \oo-catégorie
  et soit $(x, \lambda, y) : T \to f \comma g$ un \oo-foncteur. On a
  \[
    \begin{split}
      p(u,\alpha,\beta,v)_\ast (x, \lambda, y)
      & =
      p \big(u x, (\beta \comp y) \lambda (\alpha \comp x), v y\big) \\
      & = (ux, vy) = (u \times v)(x, y) \\
      & = (u \times v)p(x, \lambda, y),
    \end{split}
  \]
  d'où l'égalité recherchée en vertu du lemme de Yoneda.
\end{proof}

\begin{paragraph}\label{paragr:comma_carre}
  Soit $f : X \to Z$ un \oo-foncteur. Si
  \[
    \shorthandoff{;}
    \xymatrix@C=1.5pc{
      Y \ar[rr]^v \ar[dr]_(0.40){\phantom{g'}g}_(.60){}="g" & & Y' \ar[dl]^(0.40){g'} \\
      & Z
      \ar@{}"g";[ur]_(.15){}="gg"
      \ar@{}"g";[ur]_(.55){}="oo"
      \ar@<-0.0ex>@2"gg";"oo"^\beta
      &
    }
  \]
  est un diagramme dans $\ooCatOpLax$, on notera
  \[ f \comma (\beta, v) : f \comma g \to f \comma g' \]
  le \oo-foncteur $(\id{X}, \id{f}) \comma (\beta, v)$. On obtient ainsi un
  \oo-foncteur
  \[ f \comma {-} : \trto{\ooCatOpLax}{Z} \to \ooCat \]
  (voir le paragraphe~\ref{paragr:comma_var_fixe} pour un énoncé plus
  précis). Dans
  le cas particulier où~$X = Z$ et $f = \id{Z}$, on notera simplement $Z
  \comma {-}$ ce \oo-foncteur.

  De même, si $g : Y \to Z$ est un \oo-foncteur et si
  \[
    \shorthandoff{;}
    \xymatrix@C=1.5pc{
      X \ar[rr]^u \ar[dr]_(0.40){\phantom{f'}f}_(.60){}="f" & & X' \ar[dl]^(0.40){f'} \\
      & Z
      \ar@{}"f";[ur]_(.15){}="ff"
      \ar@{}"f";[ur]_(.55){}="oo"
      \ar@<-0.0ex>@2"oo";"ff"_\alpha
      &
    }
  \]
  est un diagramme dans $\ooCatOpLax$, on notera
  \[ (u, \alpha) \comma g : f \comma g \to f' \comma g \]
  le \oo-foncteur $(u, \alpha) \comma (\id{g}, \id{Y})$. On obtient ainsi un
  \oo-foncteur
  \[ {-} \comma g : \tr{\ooCatOpLax}{Z} \to \ooCat \]
  (voir également le paragraphe~\ref{paragr:comma_var_fixe}) qu'on notera
  simplement ${-} \comma Z$ dans le cas particulier où $Y = Z$ et $g =
  \id{Z}$.

  De plus, si
  \[
    \shorthandoff{;}
    (u, \alpha, \beta, v) =
    \raisebox{2pc}{
    $\xymatrix@R=1pc@C=3pc{
      X \ar[dd]_u \ar[dr]^f_{}="f" & & Y \ar[dl]_g_{}="g" \ar[dd]^v \\
        & Z \\
      X' \ar[ur]_{f'} & & Y' \ar[ul]^{g'}
      \ar@{}[ll];"f"_(0.35){}="sa"_(0.85){}="ta"
      \ar@2"sa";"ta"^{\alpha}
      \ar@{}[];"g"_(0.35){}="tb"_(0.85){}="sb"
      \ar@2"sb";"tb"^{\beta}
    }$}
  \]
  est un morphisme de $\SpanC$, l'égalité
  \[
    (u, \alpha, \id{g'}, \id{Y'})(\id{X}, \id{f}, \beta, v)
    =
    (\id{X'}, \id{f'}, \beta, v)(u, \alpha, \id{g}, \id{Y})
  \]
  des deux décompositions de $(u, \alpha, \beta, v)$ dans $\SpanC$ entraîne
  que le carré
  \[
    \xymatrix@C=3pc{
      f \comma g \ar[d]_-{(u, \alpha) \comma g} \ar[r]^-{f \comma (\beta, v)} &
      f \comma g' \ar[d]^-{(u, \alpha) \comma g'} \\
      f' \comma g \ar[r]_-{f' \comma (\beta, v)} & f' \comma g'
    }
  \]
  est commutatif.
\end{paragraph}

\begin{proposition}\label{prop:assoc_comma}
  Soit
  \[
    \xymatrix{
      X \ar[r]^f & Z & Y \ar[l]_g
    }
  \]
  un diagramme dans $\ooCat$. On a un isomorphisme canonique
  \[
    (p^{}_{2, f \comma Z}) \comma g \simeq f \comma (p^{}_{1, Z \comma g}),
  \]
  où
  \[
    p^{}_{2, f \comma Z} = p_2 : f \comma Z \to Z
    \quadet
    p^{}_{1, Z \comma g} = p_1 : Z \comma g \to Z,
  \]
  naturel en $(X, f, g, Y)$ dans $\SpanC$.
\end{proposition}

\begin{proof}
  Les deux \oo-catégories en jeu sont limites projectives du diagramme
  \[
    \xymatrix{
      X \ar[r]^f & Z & \HomLax(\Dn{1}, Z) \ar[l]_-{\pi_0} \ar[r]^-{\pi_1} & Z &
     \HomLax(\Dn{1}, Z) \ar[l]_-{\pi_0} \ar[r]^-{\pi_1} & Z & Y \ar[l]_g,
    }
  \]
  d'où le résultat.
\end{proof}

\begin{remark}
  Si $T$ est une \oo-catégorie, un \oo-foncteur de $T$ vers les deux
  \oo-catégories isomorphes de la proposition précédente
  correspond à un diagramme
  \[
    \shorthandoff{;}
    \xymatrix{
      & T \ar[dl]_x \ar[dd]_(0.30)z_(0.5){}="z"
      \ar[dr]^y \\
      X \ar[dr]_f
      & & Y \ar[dl]^g \\
      & Z
      \ar@{}[ul];"z"_(.30){}="sl"_(0.85){}="tl"
      \ar@2"sl";"tl"^\lambda
      \ar@{}"z";[ur]_(.15){}="sr"_(0.70){}="tr"
      \ar@2"sr";"tr"^{\lambda'}
    }
  \]
  dans $\ooCatOpLax$.
\end{remark}

On étudiera dans l'appendice~\ref{app:tr_comma} les propriétés de
$2$-fonctorialité de la construction comma. Le but de cette étude est
essentiellement de démontrer le résultat suivant qui jouera un rôle central
dans la section suivante :

\begin{proposition}\label{prop:comma_retr}
  Soit $i : X' \to X$ un rétracte par transformation oplax à gauche fort.
  Alors, pour tout diagramme
    \[
      \xymatrix{
        X \ar[r]^f & Z & Y \ar[l]_g
      }
    \]
  dans $\ooCat$, le \oo-foncteur
  \[ (i, \id{fi}) \comma g :  (fi) \comma g \to f \comma g \]
  est également un rétracte par transformation oplax à gauche fort.
\end{proposition}

\begin{proof}
  C'est l'une des assertions du corollaire~\ref{coro:comma_retr_app}.
\end{proof}

\section{Un théorème A \pdfoo-catégorique pour les $2$-triangles}
\label{sec:thmA_2-tri}

Le but de cette section est de déduire du théorème A \oo-catégorique pour
les triangles commutatifs un théorème A pour les $2$-triangles, c'est-à-dire
les triangles commutatifs à transformation près. Nous montrerons que cette
réduction du cas des $2$-triangles au cas des triangles commutatifs est
valable pour toute classe de \oo-foncteurs vérifiant des axiomes adéquats.
Notre principal outil pour y parvenir sera la construction comma
\oo-catégorique développée dans la section précédente.

\medbreak

Commençons par faire le lien entre \oo-catégories comma et tranches.

\begin{proposition}\label{prop:ident_tr_comma}
  Soient $C$ une \oo-catégorie et $c$ un objet de $C$. Pour toute
  \oo-catégorie $A$ et tout \oo-foncteur $v : A \to C$, on a un isomorphisme
  canonique
  \[ \cotr{A}{c} \simeq c \comma v, \]
  où on considère $c$ comme un \oo-foncteur $\Dn{0} \to C$,
  naturel en $A$ et $v$. De plus, cet isomorphisme est au-dessus de $A$ au
  sens où le triangle
  \[
    \shorthandoff{;}
    \xymatrix@C=1.5pc{
      \cotr{A}{c} \ar[rr]^{\sim} \ar[dr]_(0.40){U}_(.60){}="f"
      & & c \comma v \ar[dl]^(0.40){p_2} \\
      & A & \pbox{,}
    }
  \]
  où $U : \cotr{A}{c} \to A$ désigne le \oo-foncteur d'oubli,
  est commutatif.
\end{proposition}

\begin{proof}
  En vertu du paragraphe~\ref{paragr:obj_tr}, on a
  \[
    \begin{split}
     c \comma v
     & =
     \Dn{0} \times^{\pi_0}_C \HomLax(\Dn{1}, C)\, {}^{\pi_1}\!\!\times_C A
     \\
     & \simeq
     \cotr{C}{c} \times_C A
     \\
     & =
     \cotr{A}{c},
   \end{split}
  \]
  d'où l'isomorphisme recherché. Le fait que cet isomorphisme est au-dessus
  de $A$ est immédiat. De plus, si
  \[
    \shorthandoff{;}
    \xymatrix@C=1.5pc{
      A \ar[rr]^u \ar[dr]_(0.40){\phantom{v'}v}_(.60){}="f" & & A' \ar[dl]^(0.40){v'} \\
      & C
    }
  \]
  est un triangle commutatif, on vérifie que le \oo-foncteur $c \comma
  (\id{v}, u) : c \comma v \to c \comma v'$ est égal à $\Dn{0} \times_C
  \HomLax(\Dn{1}, C) \times_C u$, ce qui établit la naturalité de
  l'isomorphisme.
\end{proof}

Les trois lemmes suivants isolent les aspects techniques relatifs aux
tranches et aux \oo-catégories comma qui apparaîtront dans notre
démonstration du théorème A pour les $2$-triangles.

\begin{lemme}\label{lem:tr_comma}
  Soient $v : A \to C$ un \oo-foncteur et $c$ un objet de $C$. On a un
  isomorphisme canonique naturel
  \[ \cotr{(C \comma v)}{c} \simeq {U \comma v}, \]
  où la tranche $\cotr{(C \comma v)}{c}$ est relative au \oo-foncteur $p_1 :
  C \comma v \to C$ et où $U : \cotr{C}{c} \to C$ désigne le \oo-foncteur
  d'oubli.
\end{lemme}

\begin{proof}
  En vertu de la proposition précédente et de la
  proposition~\ref{prop:assoc_comma}, avec les notations de cette dernière,
  on a des isomorphismes naturels
  \[
    \cotr{(C \comma v)}{c} \simeq c \comma (p^{}_{1, C \comma v})
    \simeq (p^{}_{2, c \comma C}) \comma v
    \simeq U \comma v,
  \]
  d'où le résultat.
\end{proof}

\begin{lemme}\label{lem:comma_tr_retr}
  Soient $v : A \to C$ un \oo-foncteur et $c$ un objet de $C$.
  Alors le \oo-foncteur
  \[
    c \comma v \to U \comma v
  \]
  associé au triangle commutatif
  \[
    \shorthandoff{;}
    \xymatrix@C=1.5pc{
      \Dn{0} \ar[rr]^{(c, \id{c})} \ar[dr]_(0.40){c}
      & & \cotr{C}{c} \ar[dl]^(0.40){U} \\
      & C & \pbox{,}
    }
  \]
  où $U : \cotr{C}{c} \to C$ désigne le \oo-foncteur d'oubli, est un
  rétracte par transformation oplax à gauche fort.
\end{lemme}

\begin{proof}
  En vertu de la proposition~\ref{prop:tr_fort_asp}, le foncteur $(c,
  \id{c}) : \Dn{0} \to \cotr{C}{c}$ est un rétracte par transformation oplax
  à gauche fort et l'assertion résulte donc de la
  proposition~\ref{prop:comma_retr}.
\end{proof}

\begin{lemme}\label{lem:comma_pi_2_retr}
  Si $v : A \to C$ est un \oo-foncteur, alors le \oo-foncteur
  \[ p_2 : C \comma v \to A \]
  est la rétraction d'un rétracte par transformation oplax à droite fort.
\end{lemme}

\begin{proof}
  Par définition, le carré
  \[
    \xymatrix{
      C \comma v \ar[r] \ar[d]_{p_2} & \HomLax(\Dn{1}, C) \ar[d]^{\pi_1} \\
      A \ar[r]_v & C
    }
  \]
  est cartésien. Nous allons montrer que le \oo-foncteur $\pi_1$ est la
  rétraction d'un rétracte par transformation oplax à droite fort au-dessus
  de $C$. L'assertion résultera alors des propriétés de stabilité de ces
  rétractes par changement de base (proposition~\ref{prop:retr_univ}). Par
  définition, le \oo-foncteur $\pi_1$ est l'image du \oo-foncteur $1 : \Dn{0} \to
  \Dn{1}$ par le foncteur~$\HomLax({-}, C)$. Il est immédiat que $1
  : \Dn{0} \to \Dn{1}$ est un rétracte par transformation lax à droite fort
  et au-dessus de $\Dn{0}$. Or, pour des raisons formelles
  (voir~\cite[exemple C.23.(f)]{AraMaltsiJoint}), le foncteur $\HomLax({-}, C)$
  s'étend en un sesquifoncteur $(\ooCatLax)^\op \to \ooCatOpLax$, où
  $\ooCatLax$ (resp.~$\ooCatOpLax$) désigne la sesquicatégorie des
  \oo-catégories, \oo-foncteurs et transformations lax (resp.
  transformations oplax) et, si $\C$ est une sesquicatégorie, $\C^\op$
  désigne la sesquicatégorie obtenue en inversant le sens de ses
  $1$-cellules. Le foncteur $\HomLax({-}, C)$ transforme donc rétractes
  par transformation lax à droite fort $i : A \to B$ au-dessus de $A$ en
  rétractions de rétracte par transformation oplax à droite fort au-dessus
  de $\HomLax(A, C)$, ce qui achève la démonstration.
\end{proof}

Nous pouvons maintenant formuler et prouver notre théorème A pour les
triangles commutatifs à transformation oplax près.

\begin{paragraph}
  Soit
  \[
    \shorthandoff{;}
    \xymatrix@C=1.5pc{
      A \ar[rr]^u \ar[dr]_(0.40){v}_(.60){}="g" & & B \ar[dl]^(0.40){w} \\
      & C
      \ar@{}"g";[ur]_(.15){}="gg"
      \ar@{}"g";[ur]_(.55){}="oo"
      \ar@<-0.0ex>@2"gg";"oo"^\alpha
      &
    }
  \]
  un diagramme dans $\ooCatOpLax$. Pour tout objet $c$ de $C$, on dispose
  d'un \oo-foncteur
  \[ \cotr{(u, \alpha)}{c} : \cotr{A}{c} \to \cotr{B}{c}. \]
  En effet, en vertu de la proposition~\ref{prop:ident_tr_comma}, il revient
  au même de définir un \oo-foncteur~$c \comma v \to c \comma w$. Or, $c
  \comma (\alpha, u)$ est un tel \oo-foncteur.
\end{paragraph}

\begin{remark}
  Dans \cite{AraMaltsiThmAI}, le \oo-foncteur $\cotr{(u, \alpha)}{c}$ (qui y
  est noté $\cotr{\mathcal{T}}{c}$, où $\mathcal{T}$ désigne le $2$-triangle
  en jeu) est défini sans référence aux \oo-catégories comma. Néanmoins,
  ce \oo-foncteur est défini dans \cite[paragraphe 5.1]{AraMaltsiThmAI} en
  termes d'une propriété universelle des tranches (exprimée par
  \cite[proposition 4.3]{AraMaltsiThmAI}) qui est exactement celle des
  \oo-catégories comma de la forme $c \comma v$. En particulier, les deux
  définitions coïncident.
\end{remark}

\begin{theorem}\label{thm:thmA_strict_oplax}
  Soit $\W$ une classe de \oo-foncteurs satisfaisant aux propriétés
  suivantes :
  \begin{enumerate}
    \item\label{item:23} $\W$ contient les identités et satisfait à la
      propriété du deux sur trois ;
    \item\label{item:retr} tout rétracte par transformation oplax (à gauche
      comme à droite) fort est dans~$\W$ ;
    \item\label{item:thmA} $\W$ vérifie un théorème A pour les triangles
      commutatifs au sens où, pour tout triangle commutatif de \oo-foncteurs
    \[
      \xymatrix@C=1.5pc{
        A \ar[rr]^u \ar[dr]_v & & B \ar[dl]^w \\
                              & C & \pbox{,}
      }
    \]
    si pour tout objet $c$ de~$C$, le foncteur $\cotr{A}{c} \to \cotr{B}{c}$
    induit par $u$ est dans $\W$, alors il en est de même du foncteur~$u$.
  \end{enumerate}
  Alors $\W$ vérifie un théorème A pour les $2$-triangles au sens où, pour
  tout triangle de \oo-foncteurs commutatif à une transformation oplax
  $\alpha$ près
  \[
    \shorthandoff{;}
    \xymatrix@C=1.5pc{
      A \ar[rr]^u \ar[dr]_(0.40){v}_(.60){}="g" & & B \ar[dl]^(0.40){w} \\
      & C
      \ar@{}"g";[ur]_(.15){}="gg"
      \ar@{}"g";[ur]_(.55){}="oo"
      \ar@<-0.0ex>@2"gg";"oo"^\alpha
      & \pbox{,}
    }
  \]
  si pour tout objet $c$ de $C$, le \oo-foncteur $\cotr{(u, \alpha)}{c} :
  \cotr{A}{c} \to \cotr{B}{c}$ est dans $\W$, alors il en est de même de
  $u$.
\end{theorem}

\begin{proof}
  En vertu de la proposition~\ref{prop:comma_au-dessus}, on dispose de
  diagrammes commutatifs
  \[
    \xymatrix@C=1.5pc{
      C \comma v \ar[rr]^-{C \comma (\alpha, u)} \ar[dr]_{p_1} &&
      C \comma w \ar[dl]^{p_1} \\
      & C & \pbox{,}
    }
    \qquad
    \qquad
    \xymatrix@C=3pc{
      C \comma v \ar[r]^-{C \comma (\alpha, u)} \ar[d]_{p_2} &
      C \comma w \ar[d]^{p_2} \\
      A \ar[r]_u & B \pbox{.}
    }
  \]
  D'après le lemme~\ref{lem:comma_pi_2_retr}, les \oo-foncteurs
  verticaux du carré sont des rétractions de rétractes par transformation
  oplax à droite forts, et sont donc dans $\W$ en vertu des
  conditions~\ref{item:23} et \ref{item:retr}. Ainsi, en vertu de
  la condition~\ref{item:23}, pour montrer que $u$ est dans~$\W$, il suffit
  donc de montrer que $C \comma (\alpha, u)$ est dans $\W$.
  Pour ce faire, nous allons appliquer le
  théorème A pour les triangles commutatifs (condition~\ref{item:thmA}) au
  triangle ci-dessus. Il s'agit donc de montrer que, pour tout objet $c$ de
  $C$, le \oo-foncteur
  \[
    \cotr{(C \comma (\alpha, u))}{c} :
    \cotr{(C \comma v)}{c} \to \cotr{(C \comma w)}{c}
  \]
  est dans $\W$. En vertu du
  lemme~\ref{lem:tr_comma}, ce \oo-foncteur s'identifie au
  \oo-foncteur
  \[
    U \comma (\alpha, u) : U \comma v \to U \comma w,
  \]
  où $U : \cotr{C}{c} \to C$ désigne le \oo-foncteur d'oubli.
  Or, en appliquant la fonctorialité de la construction comma (voir la fin du
  paragraphe~\ref{paragr:comma_carre}) au diagramme
  \[
    \shorthandoff{;}
    \xymatrix@R=1pc@C=3pc{
      \Dn{0} \ar[dd]_{(c, \id{c})}_{}="m" \ar[dr]^c_{}="f" & &
        A \ar[dl]_v_{}="g" \ar[dd]^u \\
        & C \\
      \cotr{C}{c} \ar[ur]_{U} & & B \ar[ul]^{w} \pbox{,}
      \ar@{}[ll];"f"_(0.35){}="sa"_(0.85){}="ta"
      \ar@{}[];"g"_(0.35){}="tb"_(0.85){}="sb"
      \ar@2"sb";"tb"^{\alpha}
      \ar@{}"m";[lu]|(.40){\textstyle =}
    }
  \]
  on obtient un carré commutatif
  \[
    \xymatrix@C=3pc@R=3pc{
      c \comma v \ar[r]^-{c \comma (\alpha, u)}
      \ar[d]_{((c, \id{c}), \id{c}) \comma v}
      &
      c \comma w \ar[d]^{((c, \id{c}), \id{c}) \comma w}
      \\
      U \comma v \ar[r]_-{U \comma (\alpha, u)} & U \comma w
      \pbox{.}
    }
  \]
  En vertu du lemme~\ref{lem:comma_tr_retr}, les flèches
  verticales de ce carré sont des rétractes par transformation oplax à
  gauche forts, et sont donc dans $\W$ en vertu de la condition~\ref{item:retr}.
  Par ailleurs, par définition, la flèche horizontale du haut du carré
  s'identifie au \oo-foncteur~$\cotr{(u, \alpha)}{c}$ qui est dans $\W$ par
  hypothèse. On en déduit que $U \comma (\alpha, u)$ est dans $\W$, ce qui
  achève de prouver que le triangle commutatif introduit au début de cette
  preuve vérifie bien les hypothèses du théorème A et termine la
  démonstration.
\end{proof}

\begin{theorem}\label{thm:thmA_oplax}
  Soit
  \[
    \shorthandoff{;}
    \xymatrix@C=1.5pc{
      A \ar[rr]^u \ar[dr]_(0.40){v}_(.60){}="g" & & B \ar[dl]^(0.40){w} \\
      & C
      \ar@{}"g";[ur]_(.15){}="gg"
      \ar@{}"g";[ur]_(.55){}="oo"
      \ar@<-0.0ex>@2"gg";"oo"^\alpha
      &
    }
  \]
  un triangle de \oo-foncteurs commutatif à une transformation oplax $\alpha$
  près. Si pour tout objet $c$ de $C$, le \oo-foncteur $\cotr{(u,
  \alpha)}{c} : \cotr{A}{c} \to \cotr{B}{c}$ est une équivalence de
  Thomason, alors il en est de même de $u$.
\end{theorem}

\begin{proof}
  Cela résulte du théorème précédent appliqué à $\W$ la classe des
  équivalences de Thomason, les hypothèses de l'énoncé étant satisfaites en
  vertu de la proposition~\ref{prop:retr_Thomason} et du
  théorème~\ref{thm:thmA}.
\end{proof}

\begin{remark}
  \newcommand\Ws{\mathsf{W}}
  La classe $\W_\infty$ des équivalences de Thomason a été définie à partir
  de la classe $\Ws_\infty$ des équivalences d'homotopie
  simpliciales faibles par la formule
  \[ \W_\infty = N^{-1}(\Ws_\infty). \]
  Une inspection attentive des preuves précédentes révèle que les seules
  propriétés de la classe $\Ws_\infty$ que l'on a utilisées pour démontrer
  le théorème précédent
  (ainsi que les théorèmes~\ref{thm:thmA_simpl}, \ref{thm:thmA_cosimpl} et
  \ref{thm:thmA_monoid}) sont les suivantes :
  \begin{enumerate}
    \item $\Ws_\infty$ est faiblement saturée, c'est-à-dire contient les
      identités, satisfait à la propriété du deux sur trois et, tout
      morphisme $i$ admettant une rétraction $r$ telle que $ir$ soit dans
      $\Ws_\infty$ est dans $\Ws_\infty$ ;
    \item $\Ws_\infty$ est stable par petites sommes ;
    \item $\Ws_\infty$ vérifie le lemme bisimplicial
      (lemme~\ref{lemme:bisimpl}) ;
    \item pour tout ensemble simplicial $X$, la projection $\Deltan{1}
      \times X \to X$ est dans $\Ws_\infty$.
  \end{enumerate}
  Ainsi, le théorème précédent, ainsi que les théorèmes mentionnés
  ci-dessus, se généralisent à toute classe $\W = N^{-1}(\Ws)$ de
  \oo-foncteurs, où $\Ws$ est une classe de morphismes simpliciaux vérifiant
  les quatre propriétés ci-dessus. On peut montrer qu'une telle classe
  $\Ws$ correspond exactement à ce qui est appelé un $\Delta$-localisateur
  test \hbox{dans \cite[section 4.2]{Cisinski}}.
\end{remark}

\begin{paragraph}
  Soit
  \[
    \shorthandoff{;}
    \xymatrix@C=1.5pc{
      A \ar[rr]^u \ar[dr]_(0.40){v}_(.60){}="f" & & B \ar[dl]^(0.40){w} \\
      & C
      \ar@{}"f";[ur]_(.15){}="ff"
      \ar@{}"f";[ur]_(.55){}="oo"
      \ar@<-0.0ex>@2"oo";"ff"_\alpha
      &
    }
  \]
  un diagramme dans $\ooCatLax$, la $2$-flèche $\alpha$ étant donc une
  transformation lax. Si $c$ est un objet de $C$, on définit un \oo-foncteur
  \[
    \trm{(u, \alpha)}{c} : \trm{A}{c} \to \trm{B}{c}
  \]
  de la manière suivante. En appliquant la dualité $X \mapsto X^\op$ à ce
  diagramme, on obtient un diagramme
  \[
    \shorthandoff{;}
    \xymatrix@C=1.5pc{
      A^\op \ar[rr]^{u^\op} \ar[dr]_(0.40){v^\op}_(.60){}="g" & & B^\op
      \ar[dl]^(0.40){w^\op} \\
      & C^\op
      \ar@{}"g";[ur]_(.20){}="gg"
      \ar@{}"g";[ur]_(.50){}="oo"
      \ar@<-0.0ex>@2"gg";"oo"^{\alpha^\op}
      & \pbox{,}
    }
  \]
  où $\alpha^\op$ est une transformation oplax (voir la fin du
  paragraphe~\ref{paragr:trans_oplax}). On dispose ainsi d'un
  \oo-foncteur
  \[ \cotr{(u^\op, \alpha^\op)}{c} : \cotr{A^\op}{c} \to \cotr{B^\op}{c}. \]
  En appliquant de nouveau la dualité $X \mapsto X^\op$, on obtient, en
  vertu de la proposition~\ref{prop:tr_op}, le \oo-foncteur \smash{$\trm{A}{c} \to
  \trm{B}{c}$} recherché.
\end{paragraph}

\begin{remark}
  L'apparente asymétrie entre la définition directe de la \oo-catégorie
  $\cotr{(u, \alpha)}{c}$ et celle, par dualité, de $\smash{\trm{(u,
  \alpha)}{c}}$ vient du fait qu'on
  a privilégié la construction comma oplax par rapport à la construction
  comma lax (voir la remarque~\ref{rem:comma_lax}). En effet, si $v : A \to
  C$ est un \oo-foncteur et $c$ est un objet de $C$, alors la tranche
  $\smash{\trm{A}{c}}$ est canoniquement isomorphe à la \oo-catégorie comma
  lax $v \commalax c$ et on peut définir le \oo-foncteur $\smash{\trm{(u,
  \alpha)}{c}}$ par fonctorialité de la construction comma lax.
\end{remark}

\begin{theorem}\label{thm:thmA_lax}
  Soit
  \[
    \shorthandoff{;}
    \xymatrix@C=1.5pc{
      A \ar[rr]^u \ar[dr]_(0.40){v}_(.60){}="f" & & B \ar[dl]^(0.40){w} \\
      & C
      \ar@{}"f";[ur]_(.15){}="ff"
      \ar@{}"f";[ur]_(.55){}="oo"
      \ar@<-0.0ex>@2"oo";"ff"_\alpha
      &
    }
  \]
  un triangle de \oo-foncteurs commutatif à une transformation lax $\alpha$
  près. Si pour tout objet $c$ de $C$, le \oo-foncteur
  \smash{$\trm{(u, \alpha)}{c} : \trm{A}{c} \to \trm{B}{c}$} est une
  équivalence de Thomason, alors il en est de même de $u$.
\end{theorem}

\begin{proof}
  Soit $c$ un objet de $C$. En vertu du corollaire~\ref{coro:Thom_op},
  l'hypothèse entraîne que le \oo-foncteur
  \[ (\trm{(u, \alpha)}{c})^\op : \big(\trm{A}{c}\big)_{}^\op \to
  \big(\trm{B}{c}\big)_{}^\op \]
  est une équivalence de Thomason. Or, par définition, ce \oo-foncteur
  n'est autre que le \oo-foncteur 
  \[ \cotr{(u^\op, \alpha^\op)}{c} : \cotr{(A^\op)}{c} \to \cotr{(B^\op)}{c}. \]
  On est donc en mesure d'appliquer le théorème~\ref{thm:thmA_oplax} au
  triangle
  \[
    \shorthandoff{;}
    \xymatrix@C=1.5pc{
      A^\op \ar[rr]^{u^\op} \ar[dr]_(0.40){v^\op}_(.60){}="g" & & B^\op
      \ar[dl]^(0.40){w^\op} \\
      & C^\op
      \ar@{}"g";[ur]_(.20){}="gg"
      \ar@{}"g";[ur]_(.50){}="oo"
      \ar@<-0.0ex>@2"gg";"oo"^{\alpha^\op}
      & \pbox{.}
    }
  \]
  On en déduit que $u^\op$ est une équivalence de Thomason et donc, en
  appliquant de nouveau le corollaire~\ref{coro:Thom_op}, que $u$ est une
  équivalence de Thomason, ce qu'il fallait démontrer.
\end{proof}

\begin{remark}
  Les deux théorèmes précédents admettent des variantes pour les tranches de type
  $\tr{C}{c}$ ou \smash{$\cotrm{C}{c}$} (voir \cite[remarque 6.37]{AraMaltsiJoint}).
  Néanmoins, pour les établir, comme dans le cas du théorème A pour les
  triangles commutatifs (voir la remarque~\ref{rem:thmA_dual}), on a besoin
  de savoir que la classe des équivalences de Thomason est stable par la
  dualité~$C \mapsto C^\co$, ce qu'on démontrera dans~\cite{AraMaltsiNerfs}.
\end{remark}

\appendix

\section{Transformations oplax et homotopies simpliciales}\label{app:transformation}

Le but de cet appendice est d'associer à toute transformation oplax $\alpha$
d'un \oo-foncteur $u$ vers un \oo-foncteur $v$ une homotopie simpliciale
$N(\alpha)$ de~$N(u)$ vers~$N(v)$.

\begin{paragraph}
  Si $A$ et $B$ sont deux \oo-catégories, on dispose d'un morphisme
  d'ensembles simpliciaux
  \[ N(q) : N(A \otimes B) \to  N(A) \times N(B), \]
  où $q$ désigne le \oo-foncteur du paragraphe~\ref{paragr:def_q}.
  On va construire une section
  \[ s : N(A) \times N(B) \to N(A \otimes B) \]
  de ce morphisme, naturelle en $A$ et $B$.
\end{paragraph}

\begin{paragraph}
  On rappelle que, pour $n \ge 0$, on dispose d'un morphisme de complexes
  \[
    \begin{split}
      \nabla : \cn(\Deltan{n}) & \to \cn(\Deltan{n}) \otimes \cn(\Deltan{n}) \\
      (i_0, \ldots, i_p) & \mapsto \sum^p_{l = 0} (i_0, \ldots, i_l) \otimes
      (i_l, \ldots, i_p),
    \end{split}
  \]
  appelé diagonale d'Alexander-Whitney,
  naturel en $\Deltan{n}$ dans $\cDelta$, faisant de $\cn(\Deltan{n})$ une
  cogèbre différentielle graduée coassociative et coünitaire de coünité
  $c(p)$, où $p$ désigne l'unique morphisme de $\cDelta$ de $\Deltan{n}$
  vers $\Deltan{0}$.

  En appliquant le foncteur $\nu : \Cda \to \ooCat$, on obtient donc un
  \oo-foncteur
  \[
    \nu(\nabla) : \On{n} \to \On{n} \otimes \On{n},
  \]
  naturel en $\Deltan{n}$ dans $\cDelta$, faisant de $\On{n}$ une cogèbre
  coassociative et coünitaire de coünité l'unique \oo-foncteur de $\On{n}$
  vers $\Dn{0}$. En effet, cela résulte du théorème~\ref{thm:prod_tens}
  affirmant que le foncteur $\nu$ restreint aux complexes de Steiner forts
  est monoïdal pour le produit tensoriel, ainsi que du
  paragraphe~\ref{paragr:def_cn} et en particulier de l'isomorphisme
  canonique~$\nu\cn(\Deltan{n}) \simeq \On{n}$.
\end{paragraph}

\begin{paragraph}
  Soient $A$ et $B$ deux \oo-catégories. On définit un morphisme d'ensembles
  simpliciaux
  \[ s : N(A) \times N(B) \to N(A \otimes B) \]
  de la manière suivante. Considérons $(x, y) : \Deltan{n} \to N(A) \times N(B)$ un
  $n$-simplexe de~$N(A) \times N(B)$. Les morphismes $x$ et $y$
  correspondent à des \oo-foncteurs $\On{n} \to A$ et $\On{n} \to B$
  respectivement qu'on notera également $x$ et $y$. Le
  morphisme simplicial $s$ associe à $(x, y)$ le $n$-simplexe de
  $N(A \otimes B)$ défini par le composé
  \[
     \On{n} \xto{\nu(\nabla)} \On{n} \otimes \On{n}
     \xto{x \otimes\, y} A \otimes B.
  \]
\end{paragraph}

\begin{proposition}
  L'application $s$ est bien un morphisme d'ensembles simpliciaux.
\end{proposition}

\begin{proof}
  Soit $(x, y)$ un $n$-simplexe de~$N(A) \times N(B)$ et soit $\psi
  : \Deltan{n'} \to \Deltan{n}$ un morphisme de $\cDelta$. Notons $(x',
  y')$ le $n'$\nbd-simplexe $(N(A) \times N(B))(\psi)(x, y)$. Par
  définition, les triangles
  \[
    \xymatrix@C=1pc{
      \On{n'} \ar[rr]^{\nu\cn(\psi)} \ar[dr]_{x'} & & \On{n}
      \ar[ld]^{x{\phantom{x'}}} \\
      & A
    }
    \qquad
    \xymatrix@C=1pc{
      \On{n'} \ar[rr]^{\nu\cn(\psi)} \ar[dr]_{y'} & & \On{n}
      \ar[ld]^{y{\phantom{y'}}} \\
      & B
    }
  \]
  sont commutatifs et il s'agit de montrer que les deux \oo-foncteurs de
  $\On{n'}$ vers $A \otimes B$ donnés par le bord du diagramme
  \[
    \xymatrix@R=1pc@C=3pc{
      \On{n'} \ar[r]^-{\nu(\nabla)} \ar[dd]_{\nu\cn(\psi)} & \On{n'}
      \otimes \On{n'} \ar[dd]_{\nu\cn(\psi) \otimes \,\nu\cn(\psi)}
      \ar[dr]^{x' \otimes \, y'} \\
      & & A \otimes B \\
      \On{n} \ar[r]_-{\nu(\nabla)} & \On{n} \otimes \On{n}
      \ar[ur]_{\,x \otimes \, y} \\
    }
  \]
  sont égaux. Or, le triangle de ce diagramme est commutatif car il est le
  produit tensoriel des deux triangles commutatifs mentionnés précédemment
  dans la preuve et le carré est commutatif par naturalité de $\nu(\nabla)$,
  d'où le résultat.
\end{proof}

\begin{proposition}
  Le morphisme
  \[ s : N(A) \times N(B) \to N(A \otimes B) \]
  est une section de
  \[ N(q) : N(A \otimes B) \to N(A) \times N(B) \]
  naturelle en $A$ et $B$.
\end{proposition}

\begin{proof}
  Commençons par montrer que $s$ est une section de $N(q)$. Soit $(x, y)$
  un $n$-simplexe de $N(A) \times N(B)$. Il s'agit de montrer que le
  composé
  \[
    \On{n} \xto{\nu(\nabla)} \On{n} \otimes \On{n} \xto{x \otimes  y}
    A \otimes B \xto{\,\,q\,\,} A \times B
  \]
  est égal à
  \[ (x, y) : \On{n} \to A \times B. \]
  Pour cela, il suffit de montrer que le diagramme
  \[
    \xymatrix@R=1pc@C=2.5pc{
      & \On{n} \otimes \On{n} \ar[r]^{x \otimes y} \ar[dd]_q &
      A \otimes B \ar[dd]^q  \\
      \On{n} \ar[ru]^{\nu(\nabla)} \ar[rd]_-{\Delta} \\
      & \On{n} \times \On{n} \ar[r]_{x \times y} &
      A \times B \pbox{,} \\
    }
  \]
  où $\Delta$ désigne le \oo-foncteur diagonal, est commutatif. Le carré de
  ce diagramme étant commutatif par naturalité de $q$, il suffit de montrer
  que les deux triangles
  \[
    \xymatrix@R=1pc@C=2.5pc{
      & \On{n} \otimes \On{n} \ar[dd]^{q_1} \\
      \On{n} \ar[ru]^-{\nu(\nabla)} \ar[rd]_{\id{\On{n}}} \\
      & \On{n} \\
    }
    \qquad
    \xymatrix@R=1pc@C=2.5pc{
      & \On{n} \otimes \On{n} \ar[dd]^{q_2} \\
      \On{n} \ar[ru]^-{\nu(\nabla)} \ar[rd]_{\id{\On{n}}} \\
      & \On{n} \\
    }
  \]
  sont commutatifs, ce qui résulte du caractère coünitaire de la cogèbre
  $\On{n}$.

  Montrons maintenant la naturalité de $s$ en $A$ et $B$. Soient $u : A \to
  A'$ et $v : B \to B'$ deux \oo-foncteurs et soit $(x, y)$ un $n$-simplexe
  de $N(A) \times N(B)$. En considérant $s(x, y)$ comme un
  \oo-foncteur de~$\On{n}$ vers $A \otimes B$, on a
  \[
    (u \otimes v)s(x, y) = (u \otimes v)(x \otimes y)\nu(\nabla) =
    (ux \otimes vy)\nu(\nabla) = s(ux, vx),
  \]
  ce qu'on voulait démontrer.
\end{proof}

\begin{remark}
  Il résulte immédiatement de la coassociativité et du caractère coünitaire
  du coproduit $\nu(\nabla)$ que le morphisme
  \[ s : N(A) \times N(B) \to N(A \otimes B) \]
  fait du nerf de Street $N : \ooCat \to \pref{\cDelta}$ un foncteur
  monoïdal lax de $\ooCat$ munie du produit tensoriel vers $\pref{\cDelta}$
  munie du produit cartésien.
\end{remark}

\begin{remark}
  On peut montrer que le morphisme $s : N(A) \times N(B) \to N(A \otimes B)$
  n'est \emph{pas} le nerf d'un \oo-foncteur $A \times B \to A \otimes B$.
  Moralement, le morphisme $s$ correspond à un \oo-foncteur \emph{oplax}.
\end{remark}

\begin{paragraph}\label{paragr:g_phi}
  En particulier, lorsque $C$ est une \oo-catégorie,
  on obtient une section
  \[ s : \Deltan{1} \times N(C) \to N(\Dn{1} \otimes C), \]
  naturelle en $C$, du morphisme
  \[ N(q) : N(\Dn{1} \otimes C) \to \Deltan{1} \times N(C), \]
  compatible aux extrémités au sens où  le diagramme
  \[
    \xymatrix{
      \Deltan{1} \times N(C) \ar[r] & N(\Dn{1} \otimes C) \\
      \{\e\} \times N(C) \ar@{^(->}@<0.7ex>[u] \ar[r]^{\sim} 
      & N(\{\e\} \otimes C) \ar@<-1.7ex>@{^(->}[u] 
    }
  \]
  est commutatif pour $\e = 0, 1$.

  Explicitons le morphisme $s$ dans ce cas. Soit $(\phi, x) : \Deltan{n} \to
  \Deltan{1} \times N(C)$ un $n$\nbd-simplexe de $\Deltan{1} \times N(C)$. Par
  définition, le morphisme $s$ associe à $(\phi, x)$ le $n$-simplexe de
  $N(\Dn{1} \otimes C)$ défini par le composé
  \[
    \On{n} \xto{\nu(\nabla)} \On{n} \otimes \On{n} \xto{\nu\cn(\phi) \otimes x}
    \Dn{1} \otimes C.
  \]
  Autrement dit, en définissant un morphisme
  \[ g_\phi : \cn(\Deltan{n}) \to \cn(\Deltan{1}) \otimes \cn(\Deltan{n}) \]
  par le composé
  \[
    \cn(\Deltan{n}) \xto{\,\,\,\,\nabla\,\,\,\,} \cn(\Deltan{n}) \otimes \cn(\Deltan{n})
    \xto{\cn(\phi) \otimes \cn(\Deltan{n})}
    \cn(\Deltan{1}) \otimes \cn(\Deltan{n}),
  \]
  on a
  \[ s(\phi, x) = (\Dn{1} \otimes x)\nu(g_\phi). \]
  On vérifie que si $(i_0, \dots, i_p)$ est dans la base de $\cn(\Deltan{n})$
  (voir le paragraphe~\ref{paragr:base_cn}) et si $r$ désigne le nombre de
  $0$ parmi $\phi(i_0), \dots, \phi(i_p)$, on a
  \[
     g_\phi(i_0, \dots, i_p) =
     \begin{cases}
       (1) \otimes (i_0, \dots, i_p) & \text{si $r = 0$,} \\
       (0) \otimes (i_0, \dots, i_p) + (01) \otimes (i_1, \dots, i_p) &
       \text{si $r = 1$,} \\
       (0) \otimes (i_0, \dots, i_p) & \text{si $r \ge 2$,} \\
     \end{cases}
  \]
  en convenant que $(i_1, \dots, i_p) = 0$ lorsque $p = 0$.
\end{paragraph}

\begin{proposition}
  Le morphisme
  \[ s :  \Deltan{1} \times N(C) \to N(\Dn{1} \otimes C) \]
  est l'unique section du morphisme
  \[ N(q) : N(\Dn{1} \otimes C) \to \Deltan{1} \times N(C) \]
  qui soit à la fois naturelle en $C$ et compatible aux extrémités au sens
  du paragraphe précédent.
\end{proposition}

\begin{proof}
  Soit $s'$ une seconde section de $N(q)$ vérifiant les conditions de l'énoncé.
  Soit $\phi : \Deltan{n} \to \Deltan{1}$ un $n$-simplexe de $\Deltan{1}$.
  Considérons le $n$\nbd-simplexe $(\phi, \id{\On{n}})$ de $\Deltan{1}
  \times N(\On{n})$. On en déduit un \oo-foncteur
  \[ s'(\phi, \id{\On{n}}) : \On{n} \to \Dn{1} \otimes \On{n}. \]
  Si $x$ est un $n$-simplexe de $N(C)$, alors, par naturalité de $s'$, on a
  \[ s'(\phi, x) = (\Dn{1} \otimes x)s'(\phi, \id{\On{n}}) \]
  et les $s'(\phi, \id{\On{n}})$ déterminent donc $s'$. Par ailleurs,
  puisqu'on a des isomorphismes
  \[
   \nu\cn(\Deltan{n}) \simeq \On{n}
   \quadet
    \nu(\cn(\Deltan{1}) \otimes \cn(\Deltan{n})) 
    \simeq \Dn{1} \otimes \On{n}
  \]
  et que les complexes $\cn(\Deltan{n})$ et $\cn(\Deltan{1}) \otimes
  \cn(\Deltan{n})$ sont de Steiner forts (voir le
  paragraphe~\ref{paragr:def_cn} et la proposition~\ref{prop:tens_Steiner}),
  par pleine fidélité du foncteur $\nu$ restreint aux complexes de Steiner
  forts (théorème \ref{thm:Steiner}), on obtient l'existence d'un unique
  morphisme de complexes dirigés augmentés
  \[
    g'_\phi : \cn(\Deltan{n}) \to \cn(\Deltan{1}) \otimes \cn(\Deltan{n})
  \]
  tel que $s(\phi, \id{\On{n}}) = \nu(g'_\phi)$. Quand on applique ces
  considérations à la section $s$ de l'énoncé, on obtient le morphisme
  $g_\phi$ du paragraphe~\ref{paragr:g_phi}. Pour conclure, il suffit donc
  de montrer l'égalité $g'_\phi = g_\phi$.

  Les propriétés de la section $s'$ se traduisent de la manière suivante sur
  les $g'_\phi$ :
  \begin{enumerate}[label=(\alph*)]
    \item\label{item:sec} les triangles
  \[
    \xymatrix@R=1pc@C=2.5pc{
      & \cn(\Deltan{1}) \otimes \cn(\Deltan{n}) \ar[dd]^{q_1} \\
      \cn(\Deltan{n}) \ar[ru]^{g'_\phi} \ar[rd]_{\cn(\phi)} \\
      & \cn(\Deltan{1}) \\
    }
    \qquad
    \xymatrix@R=1pc@C=2.5pc{
      & \cn(\Deltan{1}) \otimes \cn(\Deltan{n}) \ar[dd]^{q_2} \\
      \cn(\Deltan{n}) \ar[ru]^{g'_\phi} \ar[rd]_{\id{\cn(\Deltan{n})}} \\
      & \cn(\Deltan{n}) \pbox{,} \\
    }
  \]
      où $q_1$ et $q_2$ désignent les deux projections, sont commutatifs ;
    \item\label{item:nat} le carré
      \[
        \xymatrix{
          \cn(\Deltan{n'}) \ar[r]^-{g'_{\phi'}} \ar[d]_{\cn(\psi)} &
          \cn(\Deltan{1}) \otimes \cn(\Deltan{n'})
          \ar[d]^{\cn(\Deltan{1}) \otimes \cn(\psi)} \\
          \cn(\Deltan{n}) \ar[r]_-{g'_{\phi}} & \cn(\Deltan{1}) \otimes
          \cn(\Deltan{n}) \pbox{,}
        }
      \]
      où $\phi : \Deltan{n} \to \Deltan{1}$ est un $n$-simplexe de
      $\Deltan{1}$, $\psi : \Deltan{n'} \to \Deltan{n}$ est un morphisme
      de~$\cDelta$ et $\phi ' = \phi\psi$, est commutatif ;
    \item\label{item:extr} si $\phi = \{\e\}$ est l'application constante de
      valeur $\e = 0,1$, alors $g'_\phi$ s'identifie à l'inclusion $\{\e\}
      \otimes \cn(\Deltan{n}) \hookto \cn(\Deltan{1}) \otimes
      \cn(\Deltan{n})$.
  \end{enumerate}

  On va montrer par récurrence sur $n \ge 0$ que, pour tout $\phi :
  \Deltan{n} \to \Deltan{1}$, on a l'égalité~$g'_\phi = g_\phi$. Pour $n =
  0$, le morphisme $\phi$ est nécessairement constant et la valeur de
  $g'_\phi$ est imposée par la propriété~\ref{item:extr}. Soit $n > 0$. En
  vertu de la propriété~\ref{item:nat} et de l'hypothèse de récurrence, on a
  $g'_\phi(i_0, \dots, i_p) = g_\phi(i_0, \dots, i_p)$ dès que $p < n$. Pour
  conclure, il suffit donc de montrer l'égalité $g'_\phi(0, \dots, n) =
  g_\phi(0, \dots, n)$. Notons qu'on a
  \[
    dg'_\phi(0, \dots, n) = g'_\phi d(0, \dots, n) = g_\phi d(0, \dots, n) =
    dg_\phi(0, \dots, n).
  \]
  Par ailleurs, la commutativité du deuxième triangle de la
  propriété~\ref{item:sec} montre que
  \[
    g'_\phi(0, \dots, n) = (\epsilon) \otimes (0, \dots, n) + (01) \otimes x
  \]
  pour $\epsilon = 0, 1$ et $x$ un élément de $\cn(\Deltan{n})^\ast_{n-1}$.
  Ainsi, on a
  \[
    dg'_\phi(0, \dots, n) = (\epsilon) \otimes d(0, \dots, n) + (1) \otimes x
      - (0) \otimes x - (01) \otimes dx,
  \]
  en convenant que $dx = 0$ dans le cas $n = 1$, et cette expression est
  égale à~$dg_\phi(0, \dots, n)$.

  Commençons par traiter le cas $n = 1$. Si l'application $\phi$ est
  constante, la valeur de~$g'_\phi$ est de nouveau imposée par la
  propriété~\ref{item:extr}. Il reste donc à considérer le cas où $\phi$ est
  l'identité de $\Deltan{1}$. Dans ce cas, on a
  \[ g_\phi(01) = (0) \otimes (01) + (01) \otimes (1) \]
  et donc
  \[
    \begin{split}
      dg_\phi(01)
      & = (0) \otimes (1) - (0) \otimes (0) + (1) \otimes (1) - (0) \otimes
        (1) \\
      & = (1) \otimes (1) - (0) \otimes (0).
    \end{split}
  \]
  Ainsi, pour déterminer $g'_\phi(01)$, on est conduit à résoudre l'équation
  \[
    (\e) \otimes (1) - (\e) \otimes (0) + (1) \otimes x - (0) \otimes x
    = (1) \otimes (1) - (0) \otimes (0).
  \]
  Si $\e = 0$, alors, par identification, on a $x = (1)$ et on trouve bien
  $g'_\phi(01) = g_\phi(01)$. Si $\e = 1$, alors, toujours par
  identification, on a $x = (0)$ et
  \[
    g'_\phi(01) = (1) \otimes (01) + (01) \otimes (0)
  \]
  est une seconde solution de l'équation. Néanmoins, on va montrer que cette
  formule pour~$g'_\phi(01)$ est en contradiction avec le cas $n = 2$.
  Considérons le morphisme $\psi : \Deltan{2} \to \Deltan{1}$ envoyant $0$
  et $1$ sur $0$, et $2$ sur $1$. En utilisant cette valeur de
  $g'_\phi(01)$, on aurait, en vertu des conditions~\ref{item:nat}
  et~\ref{item:extr},
  \[
    \begin{split}
      dg'_\psi(012) & = g'_\psi(12) -
      g'_\psi(02) + g'_\psi(01) \\
      & = \big((1) \otimes (12) + (01) \otimes (1)\big) - \big((1) \otimes
        (02) + (01) \otimes (0)\big) + (0) \otimes (01) \\
      & = (0) \otimes (01) + (1) \otimes \big((12) - (02)\big) + (01) \otimes
      \big((1) - (0)\big)
    \end{split}
  \]
  et cette expression vaudrait
  \[
    (\e) \otimes d(012) + (1) \otimes x - (0) \otimes x - (01) \otimes dx.
  \]
  Si $\e = 0$, par identification, on aurait $x = (12) - (02)$, ce qui est
  impossible puisque cet élément n'est pas positif au sens où il
  n'appartient pas au sous-monoïde de positivité~$\cn(\Deltan{2})_1^\ast$.
  De même, si $\e = 1$, on aurait $x = -(01)$, ce qui est également
  impossible. Ainsi, la seconde formule considérée pour $g'_\phi(01)$ n'est
  pas correcte et on a bien établi le cas $n = 1$ de notre récurrence.

  Passons maintenant au cas $n > 1$. Notons $r$ le nombre de $0$ parmi les
  entiers~$\phi(0), \dots, \phi(n)$. On distingue trois cas suivant la
  définition de $g_\phi$ :
  \begin{itemize}
    \item Si $r = 0$, alors 
      \[
        g_\phi(0, \dots, n) = (1) \otimes (0, \dots, n)
      \]
      et donc
      \[
        dg_\phi(0, \dots, n) = (1) \otimes d(0, \dots, n).
      \]
      Il suffit de montrer que l'équation
      \[
        (\epsilon) \otimes d(0, \dots, n) + (1) \otimes x
        - (0) \otimes x - (01) \otimes dx = (1) \otimes d(0, \dots, n)
      \]
      admet pour unique solution $\e = 1$ et $x = 0$. Si $\e = 0$, alors,
      par identification,~$x = d(0, \dots, n)$, ce qui est impossible
      puisque cet élément n'est pas positif. Si $\e = 1$, alors, toujours
      par identification, $x = 0$, ce qu'on voulait démontrer.
    \item Si $r = 1$, alors
      \[
        g_\phi(0, \dots, n) = (0) \otimes (0, \dots, n) + (01) \otimes (1,
        \dots, n)
      \]
      et donc $dg_\phi(0, \dots, n)$ est égal à
      \[
        (0) \otimes d(0, \dots, n) + (1) \otimes (1, \dots, n) - (0) \otimes
        (1, \dots, n) - (01) \otimes d(1, \dots, n).
      \]
      On cherche donc à résoudre l'équation affirmant l'égalité de cette
      dernière expression et de
      \[
        (\epsilon) \otimes d(0, \dots, n) + (1) \otimes x - (0) \otimes x -
        (01) \otimes dx.
      \]
      Si $\e = 1$, alors $x = (1, \dots, n) - d(0, \dots, n)$, ce qui est
      impossible puisque cet élément n'est pas positif (car $n > 1$). Si $\e
      = 0$, alors $x = (1, \dots, n)$, ce qu'on voulait démontrer.
    \item Si $r \ge 2$, alors
      \[
        g_\phi(0, \dots, n) = (0) \otimes (0, \dots, n)
      \]
      et donc
      \[
        dg_\phi(0, \dots, n) = (0) \otimes d(0, \dots, n).
      \]
      Ainsi, il s'agit de résoudre l'équation
      \[
        (\epsilon) \otimes d(0, \dots, n) + (1) \otimes x
        - (0) \otimes x - (01) \otimes dx = (0) \otimes d(0, \dots, n).
      \]
      Si $\e = 1$, alors $x = -d(0, \dots, n)$ qui n'est pas positif. Si $\e
      = 0$, alors $x = 0$, ce qui achève la démonstration de la proposition.
      \qedhere
  \end{itemize}
\end{proof}

\begin{paragraph}
  Soient $u, v : C \to D$ deux \oo-foncteurs et $\alpha$ une transformation
  oplax de~$u$ vers $v$. On définit une homotopie simpliciale $N(\alpha) :
  \Deltan{1} \times N(C) \to N(D)$ de $N(u)$ vers
  $N(v)$ en composant
  \[
    \Deltan{1} \times N(C) \xto{\,\,s\,\,} N(\Dn{1} \otimes C)
    \xto{N(\alpha)} N(D).
  \]
\end{paragraph}

\begin{theorem}\label{thm:nerf_trans_oplax}
  Le morphisme $N(\alpha) : \Deltan{1} \times N(C) \to N(D)$ est bien
  une homotopie simpliciale de~$N(u)$ vers $N(v)$.
\end{theorem}

\begin{proof}
  Il nous suffit vérifier que l'homotopie simpliciale $N(\alpha)$ a bien
  pour source $N(u)$ et pour but $N(v)$. Cela résulte de la compatibilité de
  la section $s$ aux extrémités et plus précisément de la commutativité du
  diagramme
  \[
    \xymatrix{
      \Deltan{1} \times N(C) \ar[r]^s & N(\Dn{1} \otimes C)
        \ar[r]^-{N(\alpha)} & D \\
      \{\e\} \times N(C) \ar@{^(->}@<0.7ex>[u] \ar[r]^{\sim} 
      & N(\{\e\} \otimes C) \ar@<-1.7ex>@{^(->}[u] 
    }
  \]
  pour $\e = 0, 1$.
\end{proof}

\begin{remark}
  On prendra garde au fait que si $\alpha : \Dn{1} \otimes C \to D$ est une
  transformation oplax, alors $N(\alpha)$ désigne \forlang{a priori} deux objets
  distincts : d'une part, l'homotopie simpliciale du théorème précédent et,
  d'autre part, le nerf du \oo-foncteur~$\alpha$ qui est un morphisme
  simplicial de $N(\Dn{1} \otimes C)$ vers $N(D)$. Dans la suite, sauf
  mention expresse du contraire, c'est toujours le premier objet, à savoir
  l'homotopie simpliciale, qui sera désigné par la notation $N(\alpha)$.
\end{remark}

\begin{corollary}\label{coro:retr_Thom}
  Soit $u : C \to D$ un \oo-foncteur. On suppose qu'il existe un
  \oo-foncteur $v : D \to C$ et des transformations oplax entre, d'une part,
  $vu$ et $\id{C}$ et, d'autre part, $uv$ et $\id{D}$. Alors $N(u)$ est une
  équivalence d'homotopie et, en particulier, $u$ est une équivalence de
  Thomason.
\end{corollary}

\begin{proof}
  Cela résulte immédiatement de la proposition précédente.
\end{proof}

\begin{proposition}\label{prop:nerf_trans_str}
  Soient $u, v$ deux \oo-foncteurs, $h : \Dn{1} \times C \to D$ une
  transformation stricte de $u$ vers $v$ et $\alpha_h$ la transformation
  oplax associée à $h$ (voir le paragraphe~\ref{paragr:trans_str}). Alors
  les homotopies simpliciales $N(\alpha_h)$ et $N(h)$ coïncident.
\end{proposition}

\begin{proof}
  Par définition, l'homotopie simpliciale $N(\alpha_h)$ est donnée par le composé
  \[ \Deltan{1} \times N(C) \xto{\,\,s\,\,}
    N(\Dn{1} \otimes C) \xto{N(q)} N(\Dn{1} \times C)
    \xto{\,N(h)\,} N(D).
  \]
  Puisque $s$ est une section de $N(q)$, ce composé n'est autre que $N(h)$,
  ce qu'on voulait démontrer.
\end{proof}

\begin{proposition}
  Soient $v_0, v_1 : C \to D$ deux \oo-foncteurs et $\alpha$ une
  transformation oplax de $v_0$ vers $v_1$.
  \begin{enumerate}
    \item Si $u : B \to C$ est un \oo-foncteur, alors on a
     \[ N(\alpha \comp u) = N(\alpha)(\Delta_1 \times N(u)). \]
    \item Si $w : D \to E$ est un \oo-foncteur, alors on a
     \[ N(w \comp \alpha) = N(w)N(\alpha). \]
  \end{enumerate} 
\end{proposition}

\begin{proof}
  Ces égalités résultent des formules définissant $\alpha \comp u$ et $w
  \comp \alpha$ (voir le paragraphe~\ref{paragr:def_sesqui_oplax}) et de la
  naturalité de $s$ pour la première.
\end{proof}

\begin{proposition}
  Considérons un diagramme
  \[
    \shorthandoff{;}
    \xymatrix@R=2.5pc@C=2.5pc{
      A \ar[r]^f
      \ar@/_2.5ex/[d]_(.50){\phantom{u'}u}_{}="u"
      \ar@/^2.5ex/[d]^(.50){u'}_{}="up"
      \ar@2"u";"up"^{\alpha}
      &
      C
      \ar@/_2.5ex/[d]_(.50){\phantom{w'}w}_{}="w"
      \ar@/^2.5ex/[d]^(.50){w'}_{}="wp"
      \ar@2"w";"wp"^{\gamma}
      &
      B \ar[l]_g
      \ar@/_2.5ex/[d]_(.50){\phantom{v'}v}_{}="v"
      \ar@/^2.5ex/[d]^(.50){v'}_{}="vp"
      \ar@2"v";"vp"^{\beta}
      \\
      A' \ar[r]_{f'} & C' & B' \ar[l]^{g'}
    }
  \]
  de \oo-catégories, où $\alpha$, $\beta$ et $\gamma$ sont de
  transformations oplax de $u$ vers $u'$, de $v$ vers~$v'$ et de $w$ vers
  $w'$ respectivement, commutatif au
  sens où
  \[ \gamma \comp f = f' \comp \alpha \quadet \gamma \comp g = g' \comp
  \beta. \]
  Alors on a
  \[
    N(\alpha \times_\gamma \beta) = N(\alpha) \times_{N(\gamma)} N(\beta),
  \]
  où $\alpha \times_\gamma \beta$ est la transformation oplax du
  paragraphe~\ref{paragr:img_inv_trans}.
\end{proposition}

\begin{proof}
  Notons tout d'abord qu'en vertu de la proposition précédente, le diagramme
  \[
    \xymatrix@C=4pc{
      \Deltan{1} \times N(A) \ar[r]^{\Deltan{1} \times N(f)}
      \ar[d]_{N(\alpha)}
      &
      \Deltan{1} \times N(C) 
      \ar[d]_{N(\gamma)}
      &
      \Deltan{1} \times N(B) \ar[l]_{\Deltan{1} \times N(g)}
      \ar[d]^{N(\beta)}
      \\
      N(A') \ar[r]_{N(f')} & N(C') & N(B') \ar[l]^{N(g')}
    }
  \]
  est commutatif et le produit fibré $N(\alpha) \times_{N(\gamma)}
  N(\beta)$ est donc bien défini. Soit $(\phi, z)$ un $n$-simplexe de
  $\Deltan{1} \times N(A \times_C B)$. Notons $x : \On{n} \to A$ et $y :
  \On{n} \to B$ les composantes de~$z$ de sorte qu'on a $z = (x, y) : \On{n}
  \to A \times_C B$. L'assertion résulte de la commutativité du diagramme
  \[
    \xymatrix{
      & \Dn{1} \otimes A \ar@{=}[r]
      & \Dn{1} \otimes A \ar[r]^\alpha
      & A'
      \\
      \Dn{1} \otimes \On{n} \ar[r]^-{\Dn{1} \otimes z}
      \ar[ur]^{\Dn{1} \otimes x} \ar[dr]_{\Dn{1} \otimes y}
      & \Dn{1} \otimes (A \times_C B) \ar[r] \ar[u] \ar[d]
      &
      (\Dn{1} \otimes A) \times_{\Dn{1} \otimes C} (\Dn{1} \otimes B)
      \ar[r]^-{\alpha \times_\gamma \beta} \ar[u] \ar[d]
      &
      A' \times_{C'} B' \ar[u] \ar[d]
      \\
      & \Dn{1} \otimes B \ar@{=}[r]
      & \Dn{1} \otimes B \ar[r]_\beta
      & B' \pbox{,}
    }
  \]
  où les flèches non nommées sont les flèches canoniques.
\end{proof}

\begin{corollary}\label{coro:nerf_trans_prod_fib}
  Considérons un diagramme
  \[
    \shorthandoff{;:}
    \xymatrix@C=1.5pc@R=3pc{
      & A \ar@/^2ex/[rr]^(.50){u}_{}="1" \ar@/_2ex/[rr]_(.50){u'}_{}="0"
      \ar[dr]_{}="f"_{\phantom{f'}f}
      \ar@2"1";"0"_{\alpha\,}
      & & A' \ar[dl]^{f'} \\
      B \ar[rr]^g && C &
      }
  \]
  de \oo-catégories, commutatif au sens où $f' \comp \alpha = \id{f}$. Alors
  on a 
  \[ N(\alpha \times_C B) = N(\alpha) \times_C B. \]
\end{corollary}

\begin{proof}
  En vertu de la proposition précédente, on a
  \[ N(\alpha \times_C B) = N(\alpha) \times_{N(\id{\id{C}})} N(\id{\id{B}}). \]
  Or, il résulte de la proposition~\ref{prop:nerf_trans_str} que, si $D$ est
  une \oo-catégorie, alors $N(\id{\id{D}})$ est la projection canonique
  $\Deltan{1} \times D \to D$, d'où le résultat.
\end{proof}

\begin{paragraph}
  Soient $u, v : C \to D$ deux \oo-foncteurs et $\alpha$ une transformation
  \emph{lax} de $u$ vers~$v$. On va associer à $\alpha$ une homotopie
  simpliciale $N(\alpha)$ de $N(u)$ vers~$N(v)$. Considérons la
  transformation \emph{oplax} $\alpha^\op$ de $v^\op$ vers $u^\op$
  et $N(\alpha^\op) : \Deltan{1} \times N(C^\op) \to
  N(D^\op)$ l'homotopie simpliciale de $N(v^\op)$ vers $N(u^\op)$ associée.
  En appliquant la dualité simpliciale $X \mapsto X^\op$,  on obtient
  un morphisme $N(\alpha^\op)^\op : \Deltan{1}^\op \times N(C^\op)^\op \to
  N(D^\op)^\op$ qui définit une homotopie simpliciale de~$N(u^\op)^\op$ vers
  $N(v^\op)^\op$. Or, en vertu de la proposition~\ref{prop:nerf_op}, ces
  deux morphismes s'identifient respectivement à $N(u)$ et $N(v)$. Ainsi,
  $N(\alpha^\op)^\op$ est bien une homotopie simpliciale de~$N(u)$ vers
  $N(v)$ et c'est cette homotopie simpliciale qu'on notera~$N(\alpha)$.

  Tous les résultats obtenus dans le cas oplax s'adaptent immédiatement au
  cas lax.
\end{paragraph}

\begin{remark}
  On peut également définir directement $N(\alpha)$, pour $\alpha$ une transformation
  lax, comme on l'a fait dans le cas oplax. Pour ce faire, il suffit de
  remplacer le morphisme $g_\phi :
  \cn(\Deltan{n}) \to \cn(\Deltan{1}) \otimes \cn(\Deltan{n})$ du
  paragraphe~\ref{paragr:g_phi} par le morphisme $g'_\phi :
  \cn(\Deltan{n}) \to \cn(\Deltan{n}) \otimes \cn(\Deltan{1})$ défini par
  le composé
  \[
    \cn(\Deltan{n}) \xto{\,\,\,\,\nabla\,\,\,\,} \cn(\Deltan{n}) \otimes
    \cn(\Deltan{n}) \xto{\cn(\Deltan{n}) \otimes \cn(\phi)}
    \cn(\Deltan{n}) \otimes \cn(\Deltan{1}).
  \]
  Explicitement, on a
  \[
     g'_\phi(i_0, \dots, i_p) =
     \begin{cases}
       (i_0, \dots, i_p) \otimes (0) &
       \mkern-10mu
       \text{si $r' = 0$,} \\
       (i_0, \dots, i_p) \otimes (1) + (i_0, \dots, i_{p-1}) \otimes (01) &
       \mkern-10mu
       \text{si $r' = 1$,} \\
       (i_0, \dots, i_p) \otimes (1) &
       \mkern-10mu
       \text{si $r' \ge 2$,} \\
     \end{cases}
  \]
  où $r'$ désigne le nombre de $1$ parmi $\phi(i_0), \dots, \phi(i_p)$ et où
  on a convenu que $(i_0, \dots, i_{p-1}) = 0$ si $p = 0$.
\end{remark}

\section{Tranches sesquicatégoriques et \pdfoo-catégories comma}
\label{app:tr_comma}

Le but de cet appendice est d'étudier les propriétés de $2$-fonctorialité de
la construction comma introduite dans la section~\ref{sec:comma_1} afin de
démontrer la proposition~\ref{prop:comma_retr} qu'on a admise dans le corps
de ce texte. Dans la première sous-section, on construit des
sesquicatégories $\tr{\GrayC}{c}$ et $\smash{\trto{\GrayC}{c}}$ associées à
une \oo-catégorie de Gray $\GrayC$ et un objet $c$ de $\GrayC$ et, dans la
seconde, on montre que, pour toute \oo-catégorie
$Z$, la construction comma ${-} \comma_Z {-}$ définit un sesquifoncteur
\[ \Spanoo \to \ooCatOpLax, \]
où $\ooCatOpLaxGray$ désigne la \oo-catégorie de Gray des \oo-catégories,
\oo-foncteurs et $i$\nbd-transformations oplax pour $i \ge 1$ et
$\ooCatOpLax$ la sesquicatégorie des \oo-catégories, \oo-foncteurs et
transformations oplax.

\subsection{Tranche sesquicatégorique d'une \pdfoo-catégorie de Gray}

\begin{paragraph}\label{paragr:def_oo-cat_Gray}
  Une \ndef{\oo-catégorie de Gray} est une catégorie enrichie dans la
  catégorie monoïdale des \oo-catégories munie du produit tensoriel de Gray.
  Ainsi, si $\GrayC$ est une \oo-catégorie de Gray, on dispose d'un ensemble
  $\Ob(\GrayC)$ appelé \ndef{ensemble des objets} ou des \ndef{$0$-cellules}
  de $\GrayC$ et, pour tous objets $x$ et~$y$, d'une \oo-catégorie
  $\Homi_{\GrayC}(x, y)$, qu'on notera parfois aussi $\GrayC_{x, y}$.
  On dispose également, pour toute $0$-cellule $x$, d'un objet
  \ndef{identité}~$\id{x}$ de~$\Homi_{\GrayC}(x, x)$ et, pour
  tous objets $x$, $y$ et $z$ de $\GrayC$, d'un \oo-foncteur de
  \emph{composition}
  \[
     \circ_{z,y,x} :
       \Homi_{\GrayC}(y, z) \otimes \Homi_{\GrayC}(x, y)
       \to \Homi_{\GrayC}(x, z).
  \]
  Ces données sont soumises à des axiomes affirmant que les identités sont
  des neutres pour la composition et que la composition est associative.

  Soit $\GrayC$ une \oo-catégorie de Gray. Pour $i \ge 1$, on appellera
  \ndef{$i$-cellule de $\GrayC$} une $(i-1)$\nbd-cellule de $\Homi_{\GrayC}(x,
  y)$ pour $x$ et $y$ deux objets de $\GrayC$. On dira que $x$ est la
  \ndef{$0$-source} et que $y$ est le \ndef{$0$-but} d'une telle cellule.
  Les cellules de
  $\GrayC$ forment de manière évidente un ensemble globulaire. Soient
  $\alpha$ et $\beta$ deux $i$-cellules de $\GrayC$, pour~$i > 1$,
  ayant même $0$-source $x$ et même $0$-but $y$. On notera $\beta \comp_j
  \alpha$ le composé $\beta \comp_{j-1} \alpha$ de $\Homi_{\GrayC}(x, y)$,
  pour $1 \le j < i$, si celui-ci est bien défini.
\end{paragraph}

\begin{example}\label{ex:ooCatStrGray}
  Le foncteur identité $\ooCat \to \ooCat$ est un foncteur monoïdal lax de
  source $\ooCat$ munie du produit cartésien et de but $\ooCat$
  munie du produit tensoriel de Gray, la contrainte tensorielle étant donnée par
  la transformation naturelle $q : A \otimes B \to A \times B$ du
  paragraphe~\ref{paragr:def_q}. Une \oo-catégorie stricte pouvant être
  considérée comme une catégorie enrichie dans la catégorie monoïdale des
  \oo-catégories strictes munie du produit cartésien, on en déduit un
  foncteur des \oo-catégories strictes vers les \oo-catégories de Gray.
  Si $C$ est une \oo-catégorie stricte, pour tout $i \ge 0$, les
  $i$-cellules de la \oo-catégorie de Gray associée à $C$ coïncident avec
  les $i$-cellules de $C$, ce qui justifie notre terminologie pour les
  cellules des \oo-catégories de Gray.
\end{example}

\begin{example}\label{ex:OpLaxGray}
  Les \oo-catégories, \oo-foncteurs, transformations oplax,
  $2$-trans\-formations oplax, etc.~s'organisent naturellement en une
  \oo-catégorie de Gray qu'on notera $\ooCatOpLaxGray$. Plus précisément,
  les objets de $\ooCatOpLaxGray$ sont les \oo-catégories et, si $A$ et $B$
  sont deux \oo-catégories, on pose
  \[
    \Homi_{\ooCatOpLaxGray}(A, B) = \HomOpLax(A, B).
  \]
  Il résulte formellement de la relation d'adjonction entre le produit
  tensoriel $\otimes$ et~$\HomOpLax$ qu'on dispose de compositions et
  d'unités et que $\ooCatOpLaxGray$ est bien une \oo-catégorie de Gray (voir
  \cite[exemples C.10 et C.18]{AraMaltsiJoint}).
\end{example}

Nous allons maintenant expliciter une partie de la structure des
\oo-catégories de Gray. Nous commençons par des préliminaires sur les
disques et les produits tensoriels de disques.

\begin{proposition}[Steiner]\label{prop:Steiner_disques}
  Soit $i \ge 0$. Le complexe dirigé augmenté $\lambda(\Dn{i})$ est de
  Steiner fort. Sa base est constituée des $[x]$, où $x$ varie parmi les
  cellules de $\Dn{i}$ qui ne sont pas des identités. De plus, le morphisme
  d'adjonction $\Dn{i} \to \nu\lambda(\Dn{i})$ est un isomorphisme.
\end{proposition}

\begin{proof}
  Voir \cite[section 3]{SteinerTheta} ou \cite[chapitre 4]{AraMaltsiJoint}.
\end{proof}

\begin{paragraph}\label{paragr:def_can_Dij}
  Soient $i \ge 0$ et $j \ge 0$ deux entiers. On définit un \oo-foncteur
  \[ c : \Dn{i+j} \to \Dn{i} \otimes \Dn{j} \]
  de la manière suivante. En vertu de la
  proposition précédente et de la compatibilité aux produits tensoriels du
  foncteur $\nu$ restreint aux complexes de Steiner forts
  (voir le théorème~\ref{thm:prod_tens}), on a un
  isomorphisme canonique
  \[
    \Dn{i} \otimes \Dn{j} \simeq \nu(\lambda(\Dn{i}) \otimes
    \lambda(\Dn{j})).
  \]
  Ainsi, en notant $c_k$, pour $k \ge 0$, la cellule principale de $\Dn{k}$
  (voir le paragraphe~\ref{paragr:def_Dn}), on dispose d'une $(i+j)$-cellule
  \[ \atom{[c_i] \otimes [c_j]} \]
  de $\Dn{i} \otimes \Dn{j}$. On appellera cette cellule la \ndef{cellule
  principale} de $\Dn{i} \otimes \Dn{j}$ et on la notera~$c_{i,j}$. Elle
  définit un \oo-foncteur $c : \Dn{i+j} \to \Dn{i} \otimes
  \Dn{j}$ comme annoncé.

  Dans la suite, on identifiera $c_k$, pour $k \ge 0$, avec $[c_k]$ et on
  notera donc simplement $\atom{c_i \otimes c_j}$ la cellule principale de
  $\Dn{i} \otimes \Dn{j}$.
\end{paragraph}

\begin{remark}
  On peut montrer que $c_{i,j} = \atom{c_i \otimes c_j}$ est l'unique
  $(i+j)$-cellule de $\Dn{i} \otimes \Dn{j}$ qui ne soit pas une identité.
\end{remark}

\begin{paragraph}
  Soit $\GrayC$ une \oo-catégorie de Gray et soient $i \ge 1$ et $j \ge 1$ deux
  entiers. Si $\alpha$ est une $i$-cellule de $\GrayC$ de $0$\nbd-source~$x$
  et $0$\nbd-but $y$ et $\beta$ est une $j$-cellule de $\GrayC$ de
  $0$-source~$y$ et de $0$-but~$z$, on définit une $(i+j-1)$\nbd-cellule
  $\beta \circ \alpha$ de~$\GrayC$ de $0$-source $x$ et de $0$-but $z$ de la
  manière suivante. Les cellules $\alpha$ et $\beta$ correspondent à des
  \oo-foncteurs
  \[
    \alpha : \Dn{i-1} \to \Homi_{\GrayC}(x, y)
    \quadet
    \beta : \Dn{j-1} \to \Homi_{\GrayC}(y, z)
  \]
  et la cellule $\beta \circ \alpha$ est définie par le \oo-foncteur
  \[
    \xymatrix@C=2.5pc{
      \Dn{i+j-2}
      \ar[r]^-{c}
      & \Dn{j-1} \otimes \Dn{i-1}
      \ar[r]^-{\beta \otimes \alpha}
      & \Homi_{\GrayC}(y, z) \otimes \Homi_{\GrayC}(x, y)
      \ar[r]^-{\circ_{z,y,x}}
      & \Homi_{\GrayC}(x, z),
    }
  \]
  où $c$ désigne le \oo-foncteur du
  paragraphe~\ref{paragr:def_can_Dij}.

  Si $\alpha$ est une $1$-cellule $f$, on notera $\beta \comp_0 f$ la
  $j$-cellule $\beta \circ f$. De même, si $\beta$ est une
  $1$-cellule $f$, on notera $f \comp_0 \alpha$ la $i$-cellule $f \circ
  \alpha$.
\end{paragraph}

\begin{proposition}\label{prop:Gray_fonct_1-cell}
  Soit $f$ une $1$-cellule d'une \oo-catégorie de Gray $\GrayC$ de
  source~$x$ et de but $y$.
  \begin{enumerate}
    \item Soit $i \ge 1$ et soit $\alpha$ une $i$-cellule de $0$-source $y$
      et de $0$-but $z$. Alors la $i$-cellule $\alpha \comp_0 f$ est l'image
      de $\alpha$ par le \oo-foncteur
      \[
        \xymatrix@C=2.5pc{
        \Homi_{\GrayC}(y, z)
        \ar[r]^-{\id{} \otimes f}
        &
        \Homi_{\GrayC}(y, z) \otimes \Homi_{\GrayC}(x, y)
        \ar[r]^-{\circ_{z,y,x}}
        &
        \Homi_{\GrayC}(x, z),
      }
      \]
      où on a identifié $\Homi_{\GrayC}(y, z)$ à $\Homi_{\GrayC}(y, z)
      \otimes \Dn{0}$. En particulier, ${-} \comp_0 f$ est un \oo-foncteur.
    \item Soit $i \ge 1$ et soit $\alpha$ une $i$-cellule de $0$-source $t$
      et de $0$-but $x$. Alors la $i$-cellule $f \comp_0 \alpha$ est l'image
      de $\alpha$ par le \oo-foncteur
      \[
        \xymatrix@C=2.5pc{
        \Homi_{\GrayC}(t, x)
        \ar[r]^-{f \otimes \id{}}
        &
        \Homi_{\GrayC}(x, y) \otimes \Homi_{\GrayC}(t, x)
        \ar[r]^-{\circ_{y,x,t}}
        &
        \Homi_{\GrayC}(t, y).
      }
      \]
      En particulier, $f \comp_0 {-}$ est un \oo-foncteur.
  \end{enumerate}
\end{proposition}

\begin{proof}
  Démontrons la première assertion, la seconde se démontrant de manière
  analogue. Il s'agit de démontrer l'égalité des deux composés de $\Dn{i-1}$
  vers~$\GrayC_{x, z}$ du bord du diagramme
  \[
    \xymatrix@R=1pc@C=2.5pc{
      \Dn{i-1} \ar[r]^-{c} \ar[dd]_{\alpha} &
      \Dn{i-1} \otimes \Dn{0} \ar[dr]^{\alpha \otimes f} 
      \ar[dd]_{\alpha \otimes \Dn{0}}
      \\
      & & \GrayC_{y, z} \otimes \GrayC_{x, y} \ar[r]^-{\circ_{z,y,x}}
      & \GrayC_{x, z}
      \\
      \GrayC_{y, z} \ar[r] &
      \GrayC_{y, z} \otimes \Dn{0} \ar[ur]_{\,\GrayC_{y, z} \otimes f}
      & & \pbox{,}
    }
  \]
  où la flèche horizontale du bas est la contrainte d'unité du produit
  tensoriel. Le triangle de ce diagramme étant commutatif par définition,
  il suffit de vérifier la commutativité du carré. Or, le morphisme $c :
  \Dn{i-1} \to \Dn{i-1} \otimes \Dn{0}$ n'est autre que la contrainte
  d'unité du produit tensoriel et le carré est donc commutatif par
  naturalité de celle-ci.
\end{proof}

\begin{remark}
  La proposition précédente entraîne que si $f$ est une $1$-cellule d'une
  \oo-catégorie de Gray, on a
  \begin{align*}
    f \comp_0 \id{\alpha} & = \id{f \comp_0 \alpha}
    & \id{\alpha} \comp_0 f & = \id{\alpha \comp_0 f} \\
    f \comp_0 (\beta \comp_j \alpha) & = (f \comp_0 \beta) \comp_j (f
      \comp_0 \alpha)
    & (\beta \comp_j \alpha) \comp f & = (\beta \comp_0 f) \comp_j (\alpha
      \comp_0 f),
  \end{align*}
  où, dans les égalités du haut, $\alpha$ est une $i$-cellule avec $i \ge 1$
  et, dans celles du bas, $\alpha$~et $\beta$ sont des $i$-cellules et $i >
  j \ge 1$, dès que ces compositions ont un sens (voir le
  paragraphe~\ref{paragr:def_oo-cat_Gray} pour la définition de $\comp_j$).
\end{remark}

\begin{proposition}\label{prop:circ_assoc}
  L'opération $\circ$ d'une \oo-catégorie de Gray $\GrayC$ est associative.
  Autrement dit, pour $i \ge 1$, $j \ge 1$ et $k \ge 1$, si $\alpha$ est une
  $i$-cellule de $\GrayC$ de $0$-source~$x$ et de $0$-but $y$, $\beta$ est
  une $j$-cellule de $\GrayC$ de $0$-source $y$ et de $0$-but $z$, et
  $\gamma$ est une $k$-cellule de $\GrayC$ de $0$-source $z$ et de $0$-but
  $t$, alors on a l'égalité
  \[
    (\gamma \circ \beta) \circ \alpha = \gamma \circ (\beta \circ \alpha)
  \]
  de $(i + j + k - 2)$\nbd-cellules de $0$-source $x$ et de $0$-but $t$.
\end{proposition}

\begin{proof}
  Posons $i' = i - 1$, $j' = j - 1$ et $k' = k - 1$. Il s'agit de montrer la
  commutativité du bord du diagramme
  \[
    \xymatrix@C=1.4pc@R=1pc{
      &
      \Dn{j'+k'} \otimes \Dn{i'}
      \ar[r]^-{\substack{c \otimes \Dn{i'}\\\phantom{x}}}
      &
      (\Dn{k'} \otimes \Dn{j'}) \otimes \Dn{i'}
      \ar[r]^-{\substack{(\gamma \otimes \beta) \otimes \alpha\\\ }}
      \ar[dd]_*[@]{\sim}
      &
      (\GrayC_{z,t} \otimes \GrayC_{y,z}) \otimes \GrayC_{x,y}
      \ar[dr] \ar[dd]_*[@]{\sim}
      \\
      \Dn{i'+j'+k'} \ar[ur]^c \ar[dr]_c
      & & & &
      \GrayC_{x,t}
      \\
      &
      \Dn{k'} \otimes \Dn{i'+j'}
      \ar[r]_-{\substack{\phantom{x}\\\Dn{k'} \otimes c}}
      &
      \Dn{k'} \otimes (\Dn{j'} \otimes \Dn{i'})
      \ar[r]_-{\substack{\phantom{x}\\\gamma \otimes (\beta \otimes \alpha)}}
      &
      \GrayC_{z,t} \otimes (\GrayC_{y,z} \otimes \GrayC_{x,y})
      \ar[ur]
      &
      ,
    }
  \]
  où les isomorphismes verticaux sont la contrainte d'associativité du
  produit tensoriel et les flèches obliques de droite sont induites par la
  composition de $\GrayC$. Le carré du diagramme est commutatif par
  naturalité de la contrainte d'associativité et le triangle de droite par
  associativité de la composition de $\GrayC$. Il suffit donc de montrer la
  commutativité du pentagone. Par
  définition, et avec les notations du paragraphe~\ref{paragr:def_can_Dij},
  la flèche oblique d'en haut à gauche correspond à la cellule
  $\atom{c_{j'+k'} \otimes c_{i'}}$. Puisque le \oo-foncteur 
  $c : \Dn{j'+k'} \to \Dn{k'} \otimes \Dn{j'}$ envoie $\atom{c_{j'+k'}}$ sur
  $\atom{c_{k'} \otimes c_{j'}}$, en vertu de la
  proposition~\ref{prop:tens_morph}, la flèche horizontale d'en haut à gauche envoie
  $\atom{c_{j'+k'} \otimes c_{i'}}$ sur $(\atom{c_{k'} \otimes c_{j'})
  \otimes c_{i'}}$. Ainsi, le \oo-foncteur $\Dn{i'+j'+k'} \to (\Dn{k'}
  \otimes \Dn{j'}) \otimes \Dn{i'}$ du diagramme correspond à la
  cellule~$\atom{(c_{k'} \otimes c_{j'}) \otimes c_{i'}}$.
  De même, le \oo-foncteur $\Dn{i'+j'+k'} \to \Dn{k'} \otimes (\Dn{j'}
  \otimes \Dn{i'})$ du diagramme correspond à la
  cellule~$\atom{c_{k'} \otimes (c_{j'} \otimes c_{i'})}$,
  ce qui prouve la commutativité du pentagone et donc du diagramme, d'où le
  résultat.
\end{proof}

\begin{proposition}\label{prop:circ_ident}
  L'opération $\circ$ d'une \oo-catégorie de Gray $\GrayC$ vérifie les
  compatibilités aux unités suivantes :
  \begin{enumerate}
    \item Soit $i \ge 1$ et soit $\alpha$ une $i$-cellule de $\GrayC$ de
      $0$-source $x$ et de $0$-but $y$. On a
  \[ \id{y} \comp_0 \alpha = \alpha \quadet \alpha \comp_0 \id{x} = \alpha.  \]
    \item Soient $i \ge 1$ et $j \ge 1$ deux entiers et soient $\alpha$ une
      $i$-cellule de $\GrayC$ de $0$-source $x$ et de $0$-but $y$, et
      $\beta$ une $j$-cellule de $\GrayC$ de $0$-source $y$ et de $0$-but
      $z$. Alors on a
      \[
        \id{\beta} \circ \alpha = \id{\beta \circ \alpha}
        \quadet
        \beta \circ \id{\alpha} = \id{\beta \circ \alpha}.
      \]
  \end{enumerate}
\end{proposition}

\begin{proof}
  \begin{enumerate}[wide]
    \item Démontrons la première égalité, la seconde se démontrant de manière
      analogue. Posons $i' = i - 1$. Il s'agit de montrer que le
      \oo-foncteur composé
      \[
        \xymatrix{
          \Dn{i'}
          \ar[r]^-{c}
          &
          \Dn{0} \otimes \Dn{i'}
          \ar[r]^-{\id{y} \otimes \alpha}
          &
          \GrayC_{y,y} \otimes \GrayC_{x, y}
          \ar[r]^-{\circ_{y,y,x}}
          &
          \GrayC_{x, y}
        }
      \]
      correspond à la cellule~$\alpha$. Considérons le diagramme
      \[
        \xymatrix{
          \Dn{i'}
          \ar[r]^-{c}
          &
          \Dn{0} \otimes \Dn{i'}
          \ar[r]^-{\id{y} \otimes \alpha}
          \ar[dr]_{\Dn{0} \otimes \alpha}
          &
          \GrayC_{y,y} \otimes \GrayC_{x, y}
          \ar[r]^-{\circ_{y,y,x}}
          &
          \GrayC_{x, y}
          \\
          & &
          \Dn{0} \otimes \GrayC_{x, y}
          \ar[u]^(0.60){\id{y} \otimes \GrayC_{x, y}\!}
          \ar[ur]_{\sim}
          & \pbox{,}
        }
      \]
      où la flèche oblique de droite est la contrainte d'unité du produit
      tensoriel. Le triangle de gauche est commutatif par définition et
      celui de droite par l'axiome d'unité des \oo-catégories de Gray. Or,
      le morphisme $c : \Dn{i'} \to \Dn{0} \otimes \Dn{i'}$ n'est autre que
      la contrainte d'unité du produit tensoriel et on conclut par
      naturalité de celle-ci.

    \item Démontrons la première égalité, la seconde se démontrant de manière
      analogue. Posons $i' = i - 1$ et $j' = j - 1$. Il s'agit de montrer
      l'égalité des deux composés de~$\Dn{i'+j'+1}$ vers $\GrayC_{x, z}$
      du bord du diagramme
    \[
    \xymatrix@R=1pc{
      \Dn{i'+j'+1}
      \ar[r]^-c
      \ar[dd]_{\kappa}
      &
      \Dn{j'+1} \otimes \Dn{i'}
      \ar[dr]^{\id{\beta} \otimes \alpha}
      \ar[dd]_{\kappa \otimes \Dn{i'}}
      \\
      & &
      \GrayC_{y,z} \otimes \GrayC_{x,y}
      \ar[r]^-{\circ_{z,y,x}}
      &
      \GrayC_{x,z}
      \\
      \Dn{i'+j'}
      \ar[r]_-c
      &
      \Dn{j'} \otimes \Dn{i'}
      \ar[ur]_{\beta \otimes \alpha}
       & &
      \pbox{,}
    }
  \]
  où $\kappa$ désigne le \oo-foncteur du paragraphe~\ref{paragr:def_Dn}. Le
  triangle du diagramme étant commutatif par définition, il suffit de
  montrer la commutativité du carré.  Il s'agit donc de montrer que la
  flèche verticale de droite envoie $\atom{c_{j'+1} \otimes c_{i'}}$ sur
  $\id{\atom{c_{j'} \otimes c_{i'}}}$, ce qui résulte de la
  proposition~\ref{prop:tens_morph}. \qedhere
  \end{enumerate}
\end{proof}

\begin{proposition}\label{prop:contr_Gray_str}
  Soit $\GrayC$ une \oo-catégorie de Gray provenant d'une \oo-catégorie stricte
  (voir l'exemple~\ref{ex:ooCatStrGray}), soient $i \ge 1$ et $j \ge 1$ deux
  entiers et soient $\alpha$ une $i$-cellule de $0$-source~$x$ et de
  $0$-but~$y$, et $\beta$ une $j$-cellule de $0$-source $y$ et de $0$-but
  $z$. Alors la $(i+j-1)$\nbd-cellule $\beta \circ \alpha$ est l'identité
  itérée de la cellule $\beta \comp_0 \alpha$.  En particulier, si $i = 1$
  ou $j = 1$, les deux définitions de $\beta \comp_0 \alpha$ coïncident.
\end{proposition}

\begin{proof}
  Posons $i' = i - 1$ et $j' = j -1$. Il s'agit de montrer la commutativité
  du bord du diagramme
  \[
    \xymatrix@R=1pc{
      &
      \Dn{j'} \otimes \Dn{i'}
      \ar[dd]_q
      \ar[r]^-{\beta \otimes \alpha}
      &
      \GrayC_{y,z} \otimes \GrayC_{x,y}
      \ar[dd]_q
      \ar[dr]^-{\circ_{z,y,x}}
      \\
      \Dn{i'+j'}
      \ar[ur]^c \ar[dr]
      & & &
      \GrayC_{x,z}
      \\
      &
      \Dn{j'} \times \Dn{i'}
      \ar[r]^-{\beta \times \alpha}
      &
      \GrayC_{y,z} \times \GrayC_{x,y}
      \ar[ur]
      & \pbox{,}
    }
  \]
  où la flèche oblique d'en bas à gauche correspond à la $(i' + j')$-cellule
  $(\id{c_{j'}}, \id{c_{i'}})$ et celle d'en bas à droite est la composition
  de la \oo-catégorie stricte de laquelle provient $\GrayC$. Or, le triangle
  de droite est commutatif par définition et le carré central est commutatif
  par naturalité de $q$. Il s'agit donc de vérifier la commutativité du
  triangle de gauche, c'est-à-dire le fait que les « projections » $\Dn{j'}
  \otimes \Dn{i'} \to \Dn{j'}$ et $\Dn{j'} \otimes \Dn{i'} \to \Dn{i'}$
  envoient $\atom{c_{j'} \otimes c_{i'}}$ sur $\id{\atom{c_{j'}}}$ et
  $\id{\atom{c_{i'}}}$ respectivement. Ceci résulte de la
  proposition~\ref{prop:tens_morph} (qu'on applique en identifiant
  $\Dn{j'}$ à $\Dn{j'} \otimes \Dn{0}$ et $\Dn{i'}$ à $\Dn{0} \otimes
  \Dn{i'}$), d'où le résultat.
\end{proof}

On va maintenant décrire l'opération~$\circ$ sur les $2$-cellules d'une
\oo-catégorie de Gray.

\begin{paragraph}\label{paragr:def_contr_Gray}
  Soit $\GrayC$ une \oo-catégorie de Gray et soient $\alpha$ une $2$-cellule
  de $0$-source $x$ et de $0$-but $y$, et $\beta$ une $2$-cellule de
  $0$-source $y$ et de $0$-but $z$
  \[
    \shorthandoff{;:}
    \xymatrix@C=3pc{
      x \ar@/^2.3ex/[r]_{}="1"
        \ar@/_2.3ex/[r]_{}="0"
      &
      y \ar@/^2.3ex/[r]_{}="2"
        \ar@/_2.3ex/[r]_{}="3"
      &
      z
      \pbox{.}
      \ar@2"1";"0"_{\alpha}
      \ar@2"2";"3"_{\beta}
    }
  \]
  Si $\GrayC$ était une \oo-catégorie stricte, on disposerait d'une
  $2$-cellule $\beta \comp_0 \alpha$ composée horizontale de $\beta$ et
  $\alpha$. Selon la règle de Godement, cette composée s'exprimerait en termes
  de la composition verticale des $2$-cellules des deux manières suivantes :
  \[
    (\beta \comp_0 t(\alpha)) \comp_1 (s(\beta) \comp_0 \alpha)
    \quadet
    (t(\beta) \comp_0 \alpha)) \comp_1 (\beta \comp_0 s(\alpha)).
  \]
  En général, ces deux composés diffèrent dans une \oo-catégorie de Gray.
  Néanmoins, la proposition suivante affirme que l'opération $\circ$ produit
  une $3$-cellule de comparaison qu'on appellera \ndef{contrainte de Gray}.
\end{paragraph}

\begin{proposition}\label{prop:s_t_Gray}
  Soit
  \[
    \shorthandoff{;:}
    \xymatrix@C=3pc{
      x \ar@/^2.3ex/[r]_{}="1"
        \ar@/_2.3ex/[r]_{}="0"
      &
      y \ar@/^2.3ex/[r]_{}="2"
        \ar@/_2.3ex/[r]_{}="3"
      &
      z
      \ar@2"1";"0"_{\alpha}
      \ar@2"2";"3"_{\beta}
    }
  \]
  un diagramme dans une \oo-catégorie de Gray. Le \oo-foncteur composé
  \[
    \xymatrix@C=3pc{
      \Dn{1} \otimes \Dn{1} \ar[r]^-{\beta \otimes \alpha} & \Homi_{\GrayC}(y,
      z) \otimes \Homi_{\GrayC}(x, y) \ar[r]^-{\circ_{z,y,x}} & \Homi_{\GrayC}(x, z)
    }
  \]
  correspond à un diagramme
  \[
    \shorthandoff{;:}
    \xymatrix@C=3pc@R=3pc{
      \bullet
      \ar@2[d]_{\beta \comp_0 s(\alpha)}
      \ar@2[r]^{s(\beta) \comp_0 \alpha}
      &
      \bullet \ar@2[d]^{\beta \comp_0 t(\alpha)}
      \\
      \bullet \ar@2[r]_{t(\beta) \comp_0 \alpha}
      &
      \bullet
      \ar@{}[u];[l]_(.30){}="x"
      \ar@{}[u];[l]_(.70){}="y"
      \ar@3"x";"y"_{\beta \circ \alpha\!}
    }
  \]
  dans $\GrayC$. En particulier, on a
  \[
    \xymatrix@C=1.7pc{
    (\beta \comp_0 t(\alpha)) \comp_1 (s(\beta) \comp_0 \alpha)
    \ar@3[r]^-{\beta \circ \alpha}
    &
    (t(\beta) \comp_0 \alpha)) \comp_1 (\beta \comp_0 s(\alpha)).
  }
  \]
\end{proposition}

\begin{proof}
  En vertu de \cite[proposition B.1.6]{AraMaltsiJoint}, en notant $a$ et $b$
  les cellules principales des copies de $\Dn{1}$ apparaissant de droite à
  gauche dans $\Dn{1} \otimes \Dn{1}$, on a
  \[
    \shorthandoff{;:}
    \Dn{1} \otimes \Dn{1} \simeq
    \raisebox{2pc}{
    $\xymatrix@C=3pc@R=3pc{
      \bullet
      \ar[d]_{\atom{b \otimes s(a)}}
      \ar[r]^{\atom{s(b) \otimes a}}
      &
      \bullet \ar[d]^{\atom{b \otimes t(a)}}
      \\
      \bullet \ar[r]_{\atom{t(b) \otimes a}}
      &
      \bullet
      \ar@{}[u];[l]_(.30){}="x"
      \ar@{}[u];[l]_(.70){}="y"
      \ar@2"x";"y"_{\atom{b \otimes a}\!\!\!}
    }$
    }.
  \]
  Par définition, le \oo-foncteur $\Dn{2} \to \Dn{1} \otimes \Dn{1}$ du
  paragraphe~\ref{paragr:def_can_Dij} correspond à la $2$\nbd-cellule~$\atom{b
  \otimes a}$. La composition $\comp_0$ étant un cas particulier de la
  composition $\circ$, on en déduit un carré
  \[
    \shorthandoff{;:}
    \xymatrix@C=3pc@R=3pc{
      \bullet
      \ar[d]_{\beta \comp_0 s(\alpha)}
      \ar[r]^{s(\beta) \comp_0 \alpha}
      &
      \bullet \ar[d]^{\beta \comp_0 t(\alpha)}
      \\
      \bullet \ar[r]_{t(\beta) \comp_0 \alpha}
      &
      \bullet
      \ar@{}[u];[l]_(.30){}="x"
      \ar@{}[u];[l]_(.70){}="y"
      \ar@2"x";"y"_{\beta \circ \alpha\!}
    }
  \]
  dans $\Homi_{\GrayC}(x, z)$, ce qu'on voulait démontrer.
\end{proof}

\begin{proposition}\label{prop:circ_horiz}
  Soit $\GrayC$ une \oo-catégorie de Gray.
  \begin{enumerate}
    \item
  Si
  \[
    \shorthandoff{;:}
    \xymatrix@C=3pc{
      \bullet
        \ar@/^5ex/[r]_{}="0"
        \ar[r]_{}="1"
        \ar@/_5ex/[r]_{}="2"
      &
      \bullet
        \ar@/^3ex/[r]_{}="3"
        \ar@/_3ex/[r]_{}="4"
      &
      \bullet
      \ar@2"0";"1"_{\alpha}
      \ar@2"1";"2"_{\beta}
      \ar@2"3";"4"_{\gamma}
    }
  \]
  est un diagramme dans $\GrayC$, alors on a
  \[
    \gamma \circ (\beta \comp_1 \alpha) =
      \big((t(\gamma) \comp_0 \beta) \comp_1 (\gamma \circ \alpha)\big)
      \comp_2
      \big((\gamma \circ \beta) \comp_1 (s(\gamma) \comp_0 \alpha)\big).
  \]
  \item
  De même, si
  \[
    \shorthandoff{;:}
    \xymatrix@C=3pc{
      \bullet
        \ar@/^3ex/[r]_{}="1"
        \ar@/_3ex/[r]_{}="0"
      &
      \bullet
        \ar@/^5ex/[r]_{}="2"
        \ar[r]_{}="3"
        \ar@/_5ex/[r]_{}="4"
      &
      \bullet
      \ar@2"1";"0"_{\alpha}
      \ar@2"2";"3"_{\beta}
      \ar@2"3";"4"_{\gamma}
    }
  \]
  est un diagramme $\GrayC$, alors on 
  \[
    (\gamma \comp_1 \beta) \circ \alpha
    = \big((\gamma \circ \alpha) \comp_1 (\beta \comp_0 s(\alpha))\big)
        \comp_2 \big((\gamma \comp_0 t(\alpha)) \comp_1 (\beta \circ
        \alpha)\big).
  \]
  \end{enumerate}
\end{proposition}

\begin{proof}
  Démontrons la première assertion, la seconde se démontrant de manière
  analogue. Considérons donc un diagramme
  \[
    \shorthandoff{;:}
    
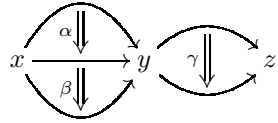
\xymatrix@C=3pc{
      x
        \ar@/^5ex/[r]_{}="0"
        \ar[r]_{}="1"
        \ar@/_5ex/[r]_{}="2"
      &
      y
        \ar@/^3ex/[r]_{}="3"
        \ar@/_3ex/[r]_{}="4"
      &
      z
      \ar@2"0";"1"_{\alpha}
      \ar@2"1";"2"_{\beta}
      \ar@2"3";"4"_{\gamma}
    }
  \]
  dans $\GrayC$. Notons $\nabla : \Dn{1} \to \Dn{1}
  {}^\sigma\!\!\amalg^{\tau}_{\Dn{0}} \Dn{1}$, où $\sigma$ et $\tau$ désignent
  les \oo-foncteurs du paragraphe~\ref{paragr:def_Dn}, le \oo-foncteur
  correspondant au composé des deux $1$-cellules de
  \[ \Dn{1} \amalg_{\Dn{0}}
    \Dn{1} = \xymatrix{\bullet \ar[r] & \bullet
    \ar[r] & \bullet \pbox{.}}
  \]
  Par définition, la $2$-cellule $\gamma \circ (\beta \comp_1 \alpha)$
  correspond au \oo-foncteur
  \[
    \xymatrix@C=2.5pc{
      \Dn{2}
      \ar[r]^c
      &
      \Dn{1} \otimes \Dn{1}
      \ar[r]^-{\Dn{1} \otimes \nabla}
      &
      \Dn{1} \otimes (\Dn{1} \amalg_{\Dn{0}} \Dn{1})
      \ar[r]^-{\gamma \otimes (\beta, \alpha)}
      &
      \GrayC_{y,z} \otimes \GrayC_{x, y}
      \ar[r]^-{\circ_{z,y,x}}
      &
      \GrayC_{x,z}
      \pbox{.}
    }
  \]
  Notons que, le produit tensoriel commutant aux limites inductives en
  chaque variable, on a un isomorphisme canonique
  \[
    \Dn{1} \otimes (\Dn{1} \amalg_{\Dn{0}} \Dn{1})
    \simeq
    (\Dn{1} \otimes \Dn{1}) \amalg_{\Dn{1}} (\Dn{1} \otimes \Dn{1})
  \]
  et le \oo-foncteur $\gamma \otimes (\beta, \alpha)$ s'identifie, à travers
  cet isomorphisme, au \oo-foncteur~$(\gamma \otimes \beta, \gamma \otimes
  \alpha)$. En vertu de \cite[proposition B.1.6]{AraMaltsiJoint}, en notant
  $a$, $b$ et $c$ les cellules principales des copies de $\Dn{1}$
  apparaissant de droite à gauche dans $\Dn{1} \otimes (\Dn{1}
  \amalg_{\Dn{0}} \Dn{1})$, on a
  \[
    \shorthandoff{;:}
    \Dn{1} \otimes (\Dn{1} \amalg_{\Dn{0}} \Dn{1}) \simeq
    \raisebox{2pc}{
    $\xymatrix@C=3pc@R=3pc{
      \bullet
      \ar[d]_{\atom{c \otimes s(a)}}
      \ar[r]^{\atom{s(c) \otimes a}}
      &
      \bullet
      \ar[d]
      \ar[r]^{\atom{s(c) \otimes b}}
      &
      \bullet
      \ar[d]^{\atom{c \otimes t(b)}}
      \\
      \bullet
      \ar[r]_{\atom{t(c) \otimes a}}
      &
      \bullet
      \ar[r]_{\atom{t(c) \otimes b}}
      \ar@{}[u];[l]_(.30){}="x"
      \ar@{}[u];[l]_(.70){}="y"
      \ar@2"x";"y"_{\atom{c \otimes a}\!\!}
      &
      \bullet
      \ar@{}[u];[l]_(.30){}="x"
      \ar@{}[u];[l]_(.70){}="y"
      \ar@2"x";"y"_{\atom{c \otimes b}\!\!}
    }$
    }\,,
  \]
  la flèche verticale du milieu étant $\atom{c \otimes t(a)} = \atom{c
  \otimes s(b)}$. L'image de ce diagramme par le \oo-foncteur de $\Dn{1}
  \otimes (\Dn{1} \amalg_{\Dn{0}} \Dn{1})$ vers $\GrayC_{x,z}$ est le
  diagramme
  \[
    \shorthandoff{;:}
    \xymatrix@C=3pc@R=3pc{
      \bullet
      \ar[d]_{\gamma \comp_0 s(\alpha)}
      \ar[r]^{s(\gamma) \comp_0 \alpha}
      &
      \bullet
      \ar[d]
      \ar[r]^{s(\gamma) \comp_0 \beta}
      &
      \bullet
      \ar[d]^{\gamma \comp_0 t(\beta)}
      \\
      \bullet
      \ar[r]_{t(\gamma) \comp_0 \alpha}
      &
      \bullet
      \ar[r]_{t(\gamma) \comp_0 \beta}
      \ar@{}[u];[l]_(.30){}="x"
      \ar@{}[u];[l]_(.70){}="y"
      \ar@2"x";"y"_{\gamma \circ \alpha\!\!}
      &
      \bullet
      \ar@{}[u];[l]_(.30){}="x"
      \ar@{}[u];[l]_(.70){}="y"
      \ar@2"x";"y"_{\gamma \circ \beta\!\!}
      \pbox{,}
    }
  \]
  la flèche verticale du milieu étant $\gamma \comp_0 t(\alpha) = \gamma
  \comp_0 s(\beta)$. Puisque le \oo-foncteur $\Dn{1} \otimes \nabla$
  correspond à la composition des carrés, la cellule $\gamma \circ (\beta,
  \alpha)$ est la $2$-cellule du composé des deux carrés ci-dessus, ce qui
  donne la formule qu'on cherchait à établir.
\end{proof}

Dans \cite{AraMaltsiJoint}, nous conjecturons que si $\GrayC$ est une
\oo-catégorie de Gray et $c$ est un objet de $\GrayC$, alors il existe une
\oo-catégorie de Gray tranche $\tr{\GrayC}{c}$ (voir la
conjecture~C.24). À défaut de prouver cette conjecture, nous allons
maintenant prouver qu'il existe une sesquicatégorie tranche
$\tr{\GrayC}{c}$.

\begin{paragraph}\label{paragr:def_sesqui}
  Rappelons qu'une \ndef{sesquicatégorie} est une catégorie $\SesquiC$ munie d'un foncteur
  \[ \Homi_{\SesquiC} : \SesquiC^\o \times \SesquiC \to \Cat \]
  rendant le triangle
  \[
    \shorthandoff{;}
    \xymatrix@C=1.5pc{
      \SesquiC^\o \times \SesquiC 
      \ar[rr]^-{\Homi_{\SesquiC}}
      \ar[dr]_-{\Hom_\SesquiC}
      & &
      \Cat
      \ar[dl]^-{\Ob}
      \\
      &
      \Ens
      &
    }
  \]
  commutatif.
  Ainsi, si $\C$ est une sesquicatégorie, on dispose d'objets, aussi appelés
  $0$-cellules, de $1$-cellules et de $2$-cellules. Les objets et les
  $1$-cellules ont des identités. On peut par ailleurs composer les
  $1$\nbd-cellules et composer verticalement les $2$-cellules. On notera ces
  deux compositions par la concaténation. On ne peut néanmoins pas composer
  horizontalement les $2$\nbd-cellules mais on dispose d'une composition
  $\comp$ d'une $1$-cellule suivie d'une $2$-cellule ou d'une $2$-cellule
  suivie d'une $1$-cellule. En notant $x, y$ des objets, $f, g$ des
  $1$-cellules et $\alpha, \beta$ des $2$-cellules de $\SesquiC$, les
  axiomes vérifiés par cette composition $\comp$ sont les suivants :
  \begin{align*}
    \qquad & \mkern50mu \zbox{$(g \comp \alpha) \comp f = g \comp (\alpha
    \comp f)$} \\
    \id{y} \comp \alpha & = \alpha & \alpha \comp \id{x} & = \alpha \\
    g \comp (f \comp \alpha) & = (gf) \comp \alpha &
      (\alpha \comp g) \comp f & = \alpha \comp (gf) \\
    g \comp \id{f} & = \id{gf} & \id{g} \comp f & = \id{gf} \\
    g \comp (\beta\alpha) & = (g \comp \beta)(g \comp \alpha) &
      (\beta\alpha) \comp f & = (\beta \comp f)(\alpha \comp f),
  \end{align*}
  lorsque ces compositions ont un sens.

  Si $\SesquiC$ et $\SesquiD$ sont deux sesquicatégories, un
  \ndef{sesquifoncteur} $F : \SesquiC \to \SesquiD$ est la donnée d'un
  foncteur $F$ de la catégorie sous-jacente à $\SesquiC$ vers la catégorie
  sous-jacente à $\SesquiD$ et d'une transformation naturelle
  \[
    \shorthandoff{;}
    \xymatrix@C=1.5pc{
      \SesquiC^\o \times \SesquiC \ar[rr]^-{F^\o \times F}
      \ar[dr]_-{\Homi_{\SesquiC}}_(.60){}="g"
      & &
      \SesquiD^\o \times \SesquiD \ar[dl]^-{\Homi_{\SesquiD}}
      \\
      & \Cat
      \ar@{}"g";[ur]_(.10){}="sa"_(.50){}="ta"
      \ar@2"sa";"ta"^{\phi}
    }
  \]
  au-dessus de $\Ens$ au sens où, pour tous objets $x$ et $y$ de $\SesquiC$,
  l'application
  \[ \Ob(\phi_{x, y}) : \Hom_{\SesquiC}(x, y) \to \Hom_{\SesquiD}(F(x),
  F(y)) \]
  est celle induite par le foncteur $F$. Explicitement, un sesquifoncteur $F
  : \SesquiC \to \SesquiD$ associe à toute $i$\nbd-cellule~$x$ de
  $\SesquiC$, pour $i = 0, 1, 2$, une $i$-cellule $F(x)$ de $\SesquiD$, ceci
  de manière compatible aux sources et aux buts, de sorte qu'on ait
  \begin{align*}
    F(\id{x}) & = \id{F(x)} &
    F(gf) & = F(g)F(f) &
    F(\beta \comp f) & = F(\beta) \comp F(f) \\
    F(\id{f}) & = \id{F(f)} &
    F(\beta\alpha) & = F(\beta)F(\alpha) &
    F(g \comp \alpha) & = F(g) \comp F(\alpha),
  \end{align*}
  où $x$ désigne un objet, $f, g$ des $1$-cellules et $\alpha, \beta$ des
  $2$-cellules de~$\SesquiC$, dès que ces compositions ont un sens.
\end{paragraph}

\begin{remark}\label{rem:sesqui_sous-jacente}
  Il découle des résultats de cette section qu'une \oo-catégorie de Gray
  a une sesquicatégorie sous-jacente et même une \oo-sesquicatégorie
  sous-jacente (la notion de \oo-sesquicatégorie est obtenue en remplaçant
  $\Cat$ par $\ooCat$ dans la définition de sesquicatégorie,
  voir~\cite[paragraphe C.3]{AraMaltsiJoint}). Cela résulte en fait de
  considérations purement formelles (voir \cite[paragraphe
  C.9]{AraMaltsiJoint}).
\end{remark}

\begin{paragraph}\label{paragr:def_tr_Gray}
  Soit $\GrayC$ une \oo-catégorie de Gray et soit $c$ un objet de $\GrayC$.
  On va définir une sesquicatégorie $\tr{\GrayC}{c}$. Commençons par définir
  ses cellules.
  \begin{itemize}[wide]
    \item Les objets de $\tr{\GrayC}{c}$ sont les couples $(x, f)$, où $x$
      est un objet de $\GrayC$ et $f : x \to c$ une $1$-cellule, 
      c'est-à-dire les diagrammes
      \[
        \xymatrix{
          x \ar[d]_f \\ c
        }
      \]
      dans $\GrayC$.
    \item Les $1$-cellules sont les diagrammes
  \[
    \shorthandoff{;}
    \xymatrix@C=1.5pc{
      x \ar[rr]^u \ar[dr]_(0.40){\phantom{f'}f}_(.60){}="f" & & x' \ar[dl]^(0.40){f'} \\
      & c
      \ar@{}"f";[ur]_(.15){}="ff"
      \ar@{}"f";[ur]_(.55){}="oo"
      \ar@<-0.0ex>@2"oo";"ff"_\alpha
      &
    }
  \]
  dans $\GrayC$, où
  \[ u : x \to x' \quadet \alpha : f'\comp_0 u \tod f. \]
  On notera $(u, \alpha)$ une telle $1$-cellule, sous-entendant
  ainsi $f$ et $f'$. La source de $(u, \alpha)$ est $(x, f)$ et son but est
  $(x', f')$.
    \item Les $2$-cellules sont les diagrammes
   \[
      \shorthandoff{;:}
      \xymatrix@C=1.5pc@R=3pc{
        x \ar@/^2ex/[rr]^(.33){u}_{}="1" \ar@/_2ex/[rr]^(.30){u'}_{}="0"
        \ar[dr]_{}="f"_{\phantom{f'}f}
        \ar@2"1";"0"^{\,\gamma}
        & & x' \ar[dl]^{f'} \\
        & c
        \ar@{}"f";[ur]_(.15){}="ff"
        \ar@{}"f";[ur]_(.55){}="oo"
        \ar@<2.0ex>@/^1ex/@{=>}"oo";"ff"^(.70){\alpha'}_(.30){}="h'"
        \ar@<0.5ex>@/^-1ex/@{:>}"oo";"ff"_(.65){\alpha}_(.80){}="h"
        \ar@3"h'";"h"_(.20){\Gamma_{}}
        &
        }
  \]
  dans $\GrayC$, où
  \[ \gamma : u \tod u' \quadet \Gamma : \alpha' \comp_1 (f' \comp_0 \gamma)
  \tot \alpha. \]
  On notera $(\gamma, \Gamma)$ une telle $2$-cellule, sous-entendant ainsi,
  comme dans le cas des $1$\nbd-cellules, les autres cellules du diagramme.
  La source de $(\gamma, \Gamma)$ est $(u, \alpha)$ et son but est~$(u',
  \alpha')$.
  \end{itemize}

  Définissons maintenant les identités.
  \begin{itemize}[wide]
    \item L'identité d'un objet
      \[
        \xymatrix{
          x \ar[d]_f \\ c
        }
      \]
      est la $1$-cellule
  \[
    \shorthandoff{;}
    
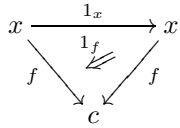
\xymatrix@C=1.5pc{
      x \ar[rr]^{\id{x}} \ar[dr]_(0.40){f}_(.60){}="f" & & x \ar[dl]^(0.40){f} \\
      & c
      \ar@{}"f";[ur]_(.15){}="ff"
      \ar@{}"f";[ur]_(.55){}="oo"
      \ar@<-0.0ex>@2"oo";"ff"_{\id{f}}
      & \pbox{.}
    }
  \]
    \item L'identité d'une $1$-cellule
  \[
    \shorthandoff{;}
    
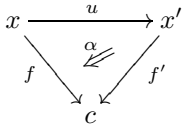
\xymatrix@C=1.5pc{
      x \ar[rr]^u \ar[dr]_(0.40){\phantom{f'}f}_(.60){}="f" & & x' \ar[dl]^(0.40){f'} \\
      & c
      \ar@{}"f";[ur]_(.15){}="ff"
      \ar@{}"f";[ur]_(.55){}="oo"
      \ar@<-0.0ex>@2"oo";"ff"_\alpha
      &
    }
  \]
  est la $2$-cellule
   \[
      \shorthandoff{;:}
      
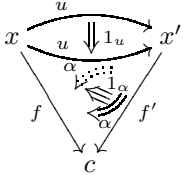
\xymatrix@C=1.5pc@R=3pc{
        x \ar@/^2ex/[rr]^(.33){u}_{}="1" \ar@/_2ex/[rr]^(.30){u}_{}="0"
        \ar[dr]_{}="f"_{\phantom{f'}f}
        \ar@2"1";"0"^{\,\id{u}}
        & & x' \ar[dl]^{f'} \\
        & c
        \ar@{}"f";[ur]_(.15){}="ff"
        \ar@{}"f";[ur]_(.55){}="oo"
        \ar@<2.0ex>@/^1ex/@{=>}"oo";"ff"^(.65){\alpha}_(.30){}="h'"
        \ar@<0.5ex>@/^-1ex/@{:>}"oo";"ff"_(.65){\alpha}_(.80){}="h"
        \ar@3"h'";"h"_(.20){\id{\alpha}}
        & \pbox{.}
        }
  \]
  \end{itemize}

  Définissons enfin les compositions.
  \begin{itemize}[wide]
    \item Le composé de deux $1$-cellules composables
  \[
    \shorthandoff{;}
    
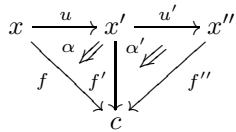
\xymatrix{
      x \ar[r]^u \ar[dr]_{}="g"_(.40){f}
      & x' \ar[r]^{u'}_(.75){}="fp" \ar[d]_(.70){}="gp"_(.56){f'} & x''
      \ar[dl]_{}="gpp"^(.38){f''} \\
      & c
      \ar@{}"g";[u]_(0.10){}="x"
      \ar@{}"g";[u]_(.85){}="y"
      \ar@<-0.1ex>@{<=}"x";"y"^(.30){\alpha}
      \ar@{}"gp";"fp"_(.25){}="x2"
      \ar@{}"gp";"fp"_(.75){}="y2"
      \ar@<0.4ex>@{<=}"x2";"y2"^(0.40){\alpha'\!}
    }
  \]
  est la $1$-cellule
  \[
    \shorthandoff{;}
    
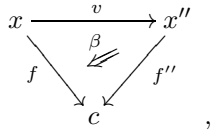
\xymatrix@C=1.5pc{
      x \ar[rr]^{v} \ar[dr]_(0.40){\phantom{f'}f}_(.60){}="f" 
        & & x'' \ar[dl]^(0.40){f''} \\
      & c
      \ar@{}"f";[ur]_(.15){}="ff"
      \ar@{}"f";[ur]_(.55){}="oo"
      \ar@<-0.0ex>@2"oo";"ff"_\beta
      & \pbox{,}
    }
  \]
  où
  \[ v = u' \comp_0 u \quadet \beta = \alpha \comp_1 (\alpha'
  \comp_0 u). \]
  \item Le composé verticale de deux $2$-cellules
  \[
      \shorthandoff{;:!}
      
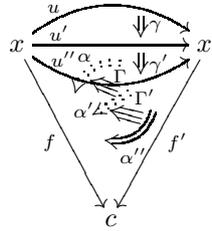
\xymatrix@C=2pc@R=4.5pc{
        x \ar@/^3.5ex/[rr]^(.22)*+<-.2em>{\labelstyle u}_(.65){}="2"
        \ar[rr]^(.22)*+<-.3em>{\labelstyle u'}_(.65){}="1"
        \ar@/_3.5ex/[rr]^(.20)*+<-.3em>{\labelstyle u''}_(.65){}="0"
        \ar[dr]_{}="f"_{\phantom{f'}f}
        \ar@{<=}"0";"1"_{\gamma'}
        \ar@{<=}"1";"2"_{\gamma}
        & & x' \ar[dl]^{f'} \\
        & c
        \ar@{}"f";[ur]_(.15){}="ff"
        \ar@{}"f";[ur]_(.55){}="oo"
        \ar@<2.0ex>@/^1.5ex/@{<:}"ff";"oo"^(.50)*+<-.1em>{\labelstyle \alpha}_(.30){}="h''"
        \ar@<0ex>@/^0ex/@{<:}"ff";"oo"^(.0)*+<-.5em>{\labelstyle{\alpha'\!\!}}_(.30){}="h'"_(.70){}="h'2"
        \ar@<-2.0ex>@/^-1.5ex/@2{<-}"ff";"oo"_(.36){\alpha''}_(.80){}="h"
        \ar@3"h'2";"h''"_(.28){\Gamma}
        \ar@3"h";"h'"_(.20){\Gamma'_{}}
        }
  \]
  est la $2$-cellule
  \[
      \shorthandoff{;:}
      
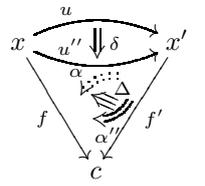
\xymatrix@C=1.5pc@R=3pc{
        x \ar@/^2ex/[rr]^(.32){u}_{}="1" \ar@/_2ex/[rr]^(.30){u''}_{}="0"
        \ar[dr]_{}="f"_{\phantom{f'}f}
        \ar@2"1";"0"^{\,\delta}
        & & x' \ar[dl]^{f'} \\
        & c
        \ar@{}"f";[ur]_(.15){}="ff"
        \ar@{}"f";[ur]_(.55){}="oo"
        \ar@<2.0ex>@/^1ex/@{=>}"oo";"ff"^(.70){\alpha''}_(.30){}="h'"
        \ar@<0.5ex>@/^-1ex/@{:>}"oo";"ff"_(.65){\alpha}_(.80){}="h"
        \ar@3"h'";"h"_(.20){\Delta_{}}
        & \pbox{,}
        }
  \]
  où
  \[
    \delta = \gamma' \comp_1 \gamma
    \quadet
    \Delta = \Gamma \comp_2 \big(\Gamma' \comp_1 (f' \comp_0 \gamma)\big).
  \]
  \item Le composé horizontal d'une $1$-cellule suivie d'une $2$-cellule
  \[
      \shorthandoff{;:}
      
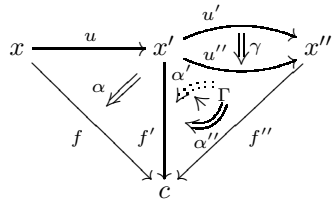
\xymatrix@C=3.5pc@R=3.5pc{
      x  \ar[r]^u \ar[dr]_{}="g"_{\phantom{f''}f} &
      x' \ar@/^2ex/[r]^(.33){u'}_{}="1"
      \ar@/_2ex/[r]^(.30){u''}_{}="0"_(.70){}="fp"
      \ar[d]_(.50){}="gp2"_(.20){}="gp"_(.62){f'} &
      x'' \ar[dl]^{f''} &
      \ar@2{<-}"0";"1"_{\gamma} \\
        & c
      \ar@{}"g";[u]_(0.10){}="x"
      \ar@{}"g";[u]_(.75){}="y"
      \ar@<-0.1ex>@2{<-}"x";"y"^(.30){\alpha}
      \ar@{}"gp2";"fp"_(.10){}="ff2"
      \ar@{}"gp2";"fp"_(.55){}="oo2"
      \ar@<+0.5ex>@/^1ex/@{<:}"ff2";"oo2"^{\!\alpha'}_(.40){}="h'''"
      \ar@<-0.5ex>@/^-1.5ex/@2{<-}"ff2";"oo2"_(.45){\!\!\!\alpha''}_(.70){}="h''"
      \ar@3"h''";"h'''"_(.20){\Gamma_{}}
      }
  \]
  est la $2$-cellule
  \[
      \shorthandoff{;:}
      
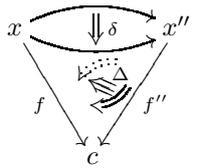
\xymatrix@C=1.5pc@R=3pc{
        x \ar@/^2ex/[rr]^(.32){}_{}="1" \ar@/_2ex/[rr]^(.30){}_{}="0"
        \ar[dr]_{}="f"_{\phantom{f''}f}
        \ar@2"1";"0"^{\,\delta}
        & & x'' \ar[dl]^{f''} \\
        & c
        \ar@{}"f";[ur]_(.15){}="ff"
        \ar@{}"f";[ur]_(.55){}="oo"
        \ar@<2.0ex>@/^1ex/@{=>}"oo";"ff"^(.70){}_(.30){}="h'"
        \ar@<0.5ex>@/^-1ex/@{:>}"oo";"ff"_(.65){}_(.80){}="h"
        \ar@3"h'";"h"_(.20){\Delta_{}}
        & \pbox{,}
        }
  \]
  où
  \[
    \delta = \gamma \comp_0 u
    \quadet
    \Delta = \alpha \comp_1 (\Gamma \comp_0 u).
  \]
  \item Enfin, le composé horizontal d'une $2$-cellule suivie d'une
    $1$-cellule
  \[
    \shorthandoff{;:}
    
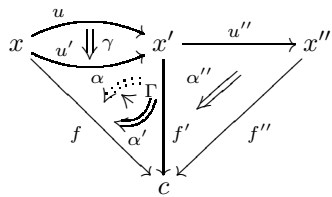
\xymatrix@C=3.5pc@R=3.5pc{
    x
    \ar@/^2ex/[r]^(.31){u}_{}="1"
    \ar@/_2ex/[r]^(.30){u'}_{}="0"_(.70){}="f"
    \ar[dr]_{}="g"_{\phantom{f''}f}
    \ar@2{<-}"0";"1"_{\,\gamma}
    &
    x' \ar[r]^{u''}_(.75){}="fp"
       \ar[d]_(.70){}="gp2"_(.20){}="gp"^(0.62){f'}
    &
    x'' \ar[dl]^{f''}
    \\
    & c
    \ar@{}"g";"gp"_(.15){}="ff1"
    \ar@{}"g";"gp"_(.80){}="oo1"
    \ar@<-0.0ex>@/^1ex/@{<:}"ff1";"oo1"^(.35){\alpha}_(.40){}="h'"
    \ar@<-1.0ex>@/^-1.5ex/@2{<-}"ff1";"oo1"_(.46){\!\!\!\alpha'}_(.70){}="h"
    \ar@3"h";"h'"_(.20){\Gamma_{}}
    \ar@{}"gp2";"fp"_(.25){}="x2"
    \ar@{}"gp2";"fp"_(.75){}="y2"
    \ar@<0.4ex>@2{<-}"x2";"y2"^{\alpha''}
    }
  \]
  est la $2$-cellule
  \[
      \shorthandoff{;:}
      
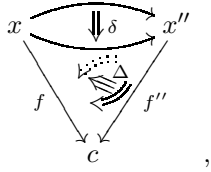
\xymatrix@C=1.5pc@R=3pc{
        x \ar@/^2ex/[rr]^(.32){}_{}="1" \ar@/_2ex/[rr]^(.30){}_{}="0"
        \ar[dr]_{}="f"_{\phantom{f''}f}
        \ar@2"1";"0"^{\,\delta}
        & & x'' \ar[dl]^{f''} \\
        & c
        \ar@{}"f";[ur]_(.15){}="ff"
        \ar@{}"f";[ur]_(.55){}="oo"
        \ar@<2.0ex>@/^1ex/@{=>}"oo";"ff"^(.70){}_(.30){}="h'"
        \ar@<0.5ex>@/^-1ex/@{:>}"oo";"ff"_(.65){}_(.80){}="h"
        \ar@3"h'";"h"_(.20){\Delta_{}}
        & \pbox{,}
        }
  \]
  où
  \[
    \delta = u'' \comp_0 \gamma
    \quadet
    \Delta = (\Gamma \comp_1 (\alpha'' \comp_0 u)) \comp_2 (\alpha' \comp_1
    (\alpha'' \circ \gamma)).
  \]
  Notons que $\Delta$ a les source et but attendus. En effet, en vertu de la
  proposition~\ref{prop:s_t_Gray}, on a
  \[
    \alpha' \comp_1 (\alpha'' \circ \gamma) :
    \alpha' \comp_1 (\alpha'' \comp_0 u') \comp_1 (f'' \comp_0 u'' \comp_0
    \gamma)
    \tot
    \alpha' \comp_1 (f' \comp_0 \gamma) \comp_1 (\alpha'' \comp_0 u)
  \]
  et
  \[ \Gamma \comp_1 (\alpha'' \comp_0 u) :
    \alpha' \comp_1 (f' \comp_0 \gamma) \comp_1 (\alpha'' \comp_0 u)
    \tot
    \alpha \comp_1 (\alpha'' \comp_0 u).
  \]
  \end{itemize}
\end{paragraph}

\begin{theorem}
  Soit $\GrayC$ une \oo-catégorie de Gray et soit $c$ un objet de $\GrayC$.
  Alors $\tr{\GrayC}{c}$ est bien une sesquicatégorie.
\end{theorem}

\begin{proof}
  Les formules définissant les cellules de $\tr{\GrayC}{c}$, ainsi que leurs
  identités et compositions, mis à part la composition horizontale d'une
  $2$-cellule suivie d'une $1$-cellule, sont les mêmes que pour les tranches
  $\tr{C}{c}$ pour $C$ une \oo-catégorie stricte (voir \cite[propositions
  9.6 et 9.15]{AraMaltsiJoint} pour la tranche $\cotr{C}{c}$). On vérifie
  par les mêmes calculs que tous les axiomes des sesquicatégories ne
  faisant par intervenir la composition horizontale mentionnée ci-dessus
  sont vérifiés par $\tr{\GrayC}{c}$. Il nous reste donc à vérifier le
  premier axiome et les axiomes de la colonne de gauche de la définition de
  sesquicatégorie donnée au paragraphe~\ref{paragr:def_sesqui}.

  Commençons par le premier axiome. Soit donc
  \[
      \shorthandoff{;:}
      
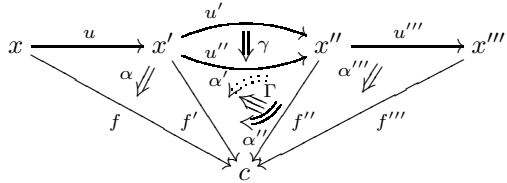
\xymatrix@C=1.5pc@R=3pc{
        x \ar[rr]^u
        \ar[drrr]_(0.47)f_{}="nf"
        & &
        x'
        \ar@/^2ex/[rr]^(.33){u'}_{}="1"
        \ar@/_2ex/[rr]^(.30){u''}_{}="0"
        \ar[dr]_{}="f"_{f'\!}
        \ar@2"1";"0"^{\,\gamma}
        & &
        x'' \ar[rr]^{u'''}_(0.40){}="u4"
        \ar[dl]^{\!f''}_(0.80){}="f2"
        & &
        x'''
        \ar[dlll]^(0.45){f'''}_(0.60){}="f3"
        \\
        & & & c
        \ar@{}"f";[ur]_(.15){}="ff"
        \ar@{}"f";[ur]_(.55){}="oo"
        \ar@<2.0ex>@/^1ex/@{=>}"oo";"ff"^(.70){\alpha''}_(.30){}="h'"
        \ar@<0.5ex>@/^-1ex/@{:>}"oo";"ff"_(.75){\alpha'\!}_(.80){}="h"
        \ar@3"h'";"h"_(.20){\Gamma_{}}
        \ar@{}[ul];"nf"_(0.20){}="sa"
        \ar@{}[ul];"nf"_(0.90){}="ta"
        \ar@2"sa";"ta"_{\alpha\,}
        \ar@{}"u4";"f3"_(0.15){}="sa4"
        \ar@{}"u4";"f3"_(0.70){}="ta4"
        \ar@2"sa4";"ta4"_{\alpha'''\!\!}
        &
        }
  \]
  un diagramme dans $\GrayC$. Il s'agit de vérifier qu'on a
  \[
    \big((u''', \alpha''') \comp (\gamma, \Gamma)\big) \comp (u, \alpha)
    =
    (u''', \alpha''') \comp \big((\gamma, \Gamma) \comp (u, \alpha)\big).
  \]
  Ces deux $2$-cellules sont des « $3$-cônes ». On vérifie facilement, en
  utilisant les propriétés de sesquicatégorie de $\GrayC$, que ces cônes ont
  mêmes objets, $1$-cellules et $2$-cellules. Il reste à vérifier que leurs
  $3$-cellules coïncident. Or, on a
  \begin{align*}
    \MoveEqLeft
    \alpha \comp_1 \Big[\Big[\big(\Gamma \comp_1 (\alpha''' \comp_0 u')\big)
    \comp_2 \big(\alpha'' \comp_1 (\alpha''' \circ \gamma)\big)\Big]
    \comp_0 u\Big]
    \\
    & =
    \alpha \comp_1 \Big[\big((\Gamma \comp_0 u) \comp_1 (\alpha''' \comp_0
      u' \comp_0 u)\big)
    \comp_2 \big((\alpha'' \comp_0 u) \comp_1 ((\alpha''' \circ \gamma)
    \comp_0 u)\big)\Big]
    \\
    & =
    \big(\alpha \comp_1 (\Gamma \comp_0 u) \comp_1 (\alpha''' \comp_0
      u' \comp_0 u)\big)
    \comp_2 \big(\alpha \comp_1 (\alpha'' \comp_0 u) \comp_1 ((\alpha''' \circ \gamma)
    \comp_0 u)\big)
    \\
    & =
    \big(\alpha \comp_1 (\Gamma \comp_0 u) \comp_1 (\alpha''' \comp_0
      u' \comp_0 u)\big)
    \comp_2 \big(\alpha \comp_1 (\alpha'' \comp_0 u) \comp_1 (\alpha'''
    \circ (\gamma \comp_0 u))\big),
  \end{align*}
  la dernière égalité résultant de l'associativité de l'opération $\circ$
  (proposition~\ref{prop:circ_assoc}), ce qu'il fallait démontrer.

  Nous allons maintenant vérifier les quatre axiomes de la colonne de
  gauche, dans l'ordre. Considérons donc
  \[
    \shorthandoff{;:}
    
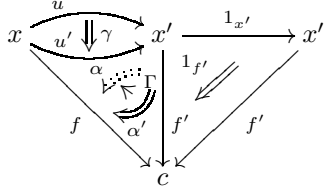
\xymatrix@C=3.5pc@R=3.5pc{
    x
    \ar@/^2ex/[r]^(.31){u}_{}="1"
    \ar@/_2ex/[r]^(.30){u'}_{}="0"_(.70){}="f"
    \ar[dr]_{}="g"_{\phantom{f'}f}
    \ar@2{<-}"0";"1"_{\,\gamma}
    &
    x' \ar[r]^{\id{x'}}_(.75){}="fp"
       \ar[d]_(.70){}="gp2"_(.20){}="gp"^(0.62){f'}
    &
    x' \ar[dl]^{f'}
    \\
    & c
    \ar@{}"g";"gp"_(.15){}="ff1"
    \ar@{}"g";"gp"_(.80){}="oo1"
    \ar@<-0.0ex>@/^1ex/@{<:}"ff1";"oo1"^(.35){\alpha}_(.40){}="h'"
    \ar@<-1.0ex>@/^-1.5ex/@2{<-}"ff1";"oo1"_(.46){\!\!\!\alpha'}_(.70){}="h"
    \ar@3"h";"h'"_(.20){\Gamma_{}}
    \ar@{}"gp2";"fp"_(.25){}="x2"
    \ar@{}"gp2";"fp"_(.75){}="y2"
    \ar@<0.4ex>@2{<-}"x2";"y2"^{\id{f'}}
    }
  \]
  un diagramme dans $\GrayC$ et montrons qu'on a
  \[ \id{(x', f')} \comp (\gamma, \Gamma) = (\gamma, \Gamma). \]
  La seule vérification non triviale est celle de l'égalité des $3$-cellules
  de ces deux $3$-cônes mais, en utilisant les compatibilités des
  contraintes de Gray aux identités (proposition~\ref{prop:circ_ident}), on
  a
  \begin{align*}
    \MoveEqLeft
    \big(\Gamma \comp_1 (\id{f'} \comp_0 u)\big) \comp_2
      \big(\alpha' \comp_1 (\id{f'} \circ \gamma)\big)
    \\
    & =
    \big(\Gamma \comp_1 (f' \comp_0 u)\big) \comp_2
    \big(\alpha' \comp_1 (f' \circ \gamma)\big)
    \\
    & =
    \big(\Gamma \comp_1 (f' \comp_0 u)\big) \comp_2
    \big(\alpha' \comp_1 (f' \comp_0 \gamma)\big)
    \\
    & =
    \Gamma \comp_2
    \big(\alpha' \comp_1 (f' \comp_0 \gamma)\big)
   \\
   & =
   \Gamma,
  \end{align*}
  ce qu'il s'agissait de vérifier.

  Soit maintenant
  \[
    \shorthandoff{;:}
    
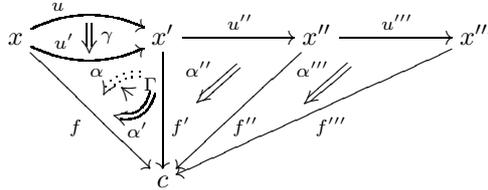
\xymatrix@C=3.5pc@R=3.5pc{
    x
    \ar@/^2ex/[r]^(.31){u}_{}="1"
    \ar@/_2ex/[r]^(.30){u'}_{}="0"_(.70){}="f"
    \ar[dr]_{}="g"_{\phantom{f''}f}
    \ar@2{<-}"0";"1"_{\,\gamma}
    &
    x' \ar[r]^{u''}_(.75){}="fp"
       \ar[d]_(.70){}="gp2"_(.20){}="gp"^(0.62){f'}
    &
    x'' \ar[dl]^{\!\!\!\!f''}
    \ar[r]^{u'''}_(.40){}="u4"
    &
    x''' \ar[dll]^{f'''}_(.65){}="f3"
    \\
    & c
    \ar@{}"g";"gp"_(.15){}="ff1"
    \ar@{}"g";"gp"_(.80){}="oo1"
    \ar@<-0.0ex>@/^1ex/@{<:}"ff1";"oo1"^(.35){\alpha}_(.40){}="h'"
    \ar@<-1.0ex>@/^-1.5ex/@2{<-}"ff1";"oo1"_(.46){\!\!\!\alpha'}_(.70){}="h"
    \ar@3"h";"h'"_(.20){\Gamma_{}}
    \ar@{}"gp2";"fp"_(.25){}="x2"
    \ar@{}"gp2";"fp"_(.75){}="y2"
    \ar@<0.4ex>@2{<-}"x2";"y2"^{\alpha''}
    \ar@{}"u4";"f3"_(.22){}="sa3"
    \ar@{}"u4";"f3"_(.78){}="ta3"
    \ar@2"sa3";"ta3"_{\alpha'''\!\!}
    }
  \]
  un diagramme dans $\GrayC$. Montrons qu'on a
  \[
    (u''', \alpha''') \comp \big((u'', \alpha'') \comp (\gamma,
    \Gamma)\big)
    =
    \big((u''',\alpha''')\,(u'', \alpha'')\big) \comp (\gamma, \Gamma).
  \]
  Comme précédemment, nous allons uniquement vérifier l'égalité des deux
  $3$-cellules des $3$-cônes associés. On a, en utilisant l'associativité de
  l'opération $\circ$ (proposition~\ref{prop:circ_assoc}) pour la première
  égalité,
  {
  \allowdisplaybreaks
  \begin{align*}
    \MoveEqLeft
    \Big[\Big(\big(\Gamma \comp_1 (\alpha'' \comp_0 u)\big) \comp_2 \big(\alpha'
    \comp_1 (\alpha'' \circ \gamma)\big)\Big)
    \comp_1 \Big(\alpha''' \comp_0 u'' \comp_0 u\Big)\Big]
    \\*
    \MoveEqLeft \quad
    \comp_2 \Big[\alpha' \comp_1 \big(\alpha'' \comp_0 u'\big)
    \comp_1 \big(\alpha''' \circ (u'' \comp_0 \gamma)\big)\Big]
    \\
    & =
    \Big[\Gamma \comp_1 \big(\alpha'' \comp_0 u\big) \comp_1 \big(\alpha''' \comp_0 u''
    \comp_0 u\big)\Big]
    \comp_2
    \Big[\alpha' \comp_1 \big(\alpha'' \circ \gamma\big) \comp_1
    \big(\alpha''' \comp_0 u'' \comp_0 u\big)\Big]
    \\*
    & \phantom{=1} \quad
    \comp_2
    \Big[\alpha' \comp_1 \big(\alpha'' \comp_0 u'\big)
    \comp_1 \big((\alpha''' \comp_0 u'') \circ \gamma\big)\Big]
    \\
    & =
    \Big[\Gamma \comp_1 \Big(\big(\alpha'' \comp_1 (\alpha''' \comp_0
    u'')\big) \comp_0 u\Big)\Big]
    \\*
    & \phantom{=1} \quad
    \comp_2
    \Big[ \alpha' \comp_1 \Big[\Big(\big(\alpha'' \circ \gamma\big) \comp_1
        \big(\alpha''' \comp_0 u'' \comp_0 u\big)\Big)
    \\*
    & \phantom{=1} \qquad\qquad\qquad\quad
        \comp_2
        \Big(\big(\alpha'' \comp_0 u'\big) \comp_1 \big((\alpha''' \comp_0 u'')
    \circ \gamma\big)\Big)\Big]\Big]
    \\
    & =
    \Big[\Gamma \comp_1 \Big(\big(\alpha'' \comp_1 (\alpha''' \comp_0
    u'')\big) \comp_0 u\Big)\Big]
    \comp_2
    \Big[ \alpha'
      \comp_1
      \Big[\big(\alpha'' \comp_1 (\alpha''' \comp_0 u'')\big) \circ
    \gamma\Big]\Big],
  \end{align*}
  }%
  la dernière égalité résultant de la compatibilité des contraintes de Gray
  à la composition $\comp_1$ (proposition~\ref{prop:circ_horiz}),
  ce qu'il fallait démontrer.

  Considérons maintenant
  \[
    \shorthandoff{;:}
    \xymatrix@C=3.5pc@R=3.5pc{
    x
    \ar@/^2ex/[r]^(.33){u}_{}="1"
    \ar@/_2ex/[r]^(.30){u}_{}="0"_(.70){}="f"
    \ar[dr]_{}="g"_{\phantom{f''}f}
    \ar@2{<-}"0";"1"_{\,\id{u}}
    &
    x' \ar[r]^{u'}_(.75){}="fp"
       \ar[d]_(.70){}="gp2"_(.20){}="gp"^(0.62){f'}
    &
    x'' \ar[dl]^{f''}
    \\
    & c
    \ar@{}"g";"gp"_(.15){}="ff1"
    \ar@{}"g";"gp"_(.80){}="oo1"
    \ar@<-0.0ex>@/^1ex/@{<:}"ff1";"oo1"^(.40){\alpha}_(.40){}="h'"
    \ar@<-1.0ex>@/^-1.5ex/@2{<-}"ff1";"oo1"_(.46){\!\!\!\alpha}_(.70){}="h"
    \ar@3"h";"h'"_(.25){\id{\alpha}\,}
    \ar@{}"gp2";"fp"_(.25){}="x2"
    \ar@{}"gp2";"fp"_(.75){}="y2"
    \ar@<0.4ex>@2{<-}"x2";"y2"^{\alpha'}
    }
  \]
  un diagramme dans $\GrayC$ et montrons qu'on a
  \[
    (u', \alpha') \comp \id{(u, \alpha)} = \id{(u', \alpha')(u, \alpha)}.
  \]
  Vérifions l'égalité des $3$-cellules associées. En utilisant la
  compatibilité des contraintes de Gray aux identités, on a
  \begin{align*}
    \MoveEqLeft
    \big(\id{\alpha} \comp_1 (\alpha' \comp_0 u)\big) \comp_2 \big(\alpha
    \comp_1 (\alpha' \circ \id{u})\big)
    \\
    & =
    \big(\id{\alpha} \comp_1 (\alpha' \comp_0 u)\big)
    \comp_2 \big(\alpha \comp_1 \id{\alpha' \comp_0 u}\big)
    \\
    & =
    \id{\alpha \comp_1 (\alpha' \comp_0 u)}
    \comp_2
    \id{\alpha \comp_1 (\alpha' \comp_0 u)}
    \\
    & =
    \id{\alpha \comp_1 (\alpha' \comp_0 u)},
  \end{align*}
  ce qu'il s'agissait de vérifier.

  Soit enfin
  \[
      \shorthandoff{;:!}
      
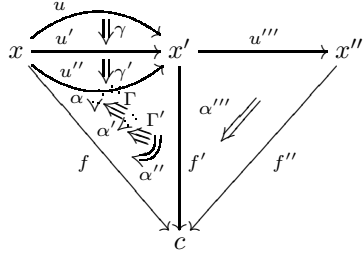
\xymatrix@R=5pc@C=4pc{
        x
        \ar[dr]_{\phantom{f''}f}_{}="f"
        \ar@/^3.5ex/[r]^(.30){u}_(.55){}="0"
        \ar[r]^(.30){u'}_(.55){}="1"
        \ar@/_3.5ex/[r]^(.29){u''}_(.55){}="2"
        \ar@2"0";"1"^\gamma
        \ar@2"1";"2"^{\gamma'}
        &
        x'
        \ar[d]^(.58){f'}_(.15){}="f1"_(.75){}="f12"
        \ar[r]^{u'''}_(.75){}="u3"
        &
        x''
        \ar[dl]^{f''}
        \\
        &
        c
        \ar@{}"f";"f1"_(.15){}="ff"
        \ar@{}"f";"f1"_(.55){}="oo"
        \ar@<2.0ex>@/^1.5ex/@{<:}"ff";"oo"^(.25)*+<.3em>{\labelstyle \alpha}_(.30){}="h''"
        \ar@<0ex>@/^0ex/@{<:}"ff";"oo"^(.0)*+<-.5em>{\labelstyle{\alpha'\!\!}}_(.30){}="h'"_(.70){}="h'2"
        \ar@<-2.0ex>@/^-1.5ex/@2{<-}"ff";"oo"_(.36){\alpha''}_(.80){}="h"
        \ar@3"h'2";"h''"_(.28){\Gamma}
        \ar@3"h";"h'"_(.20){\Gamma'_{}}
        \ar@{}"u3";"f12"_(.70){}="s"
        \ar@{}"u3";"f12"_(.30){}="t"
        \ar@<0.4ex>@2{<-}"s";"t"^{\alpha'''}
      }
  \]
  un diagramme dans $\GrayC$. Montrons qu'on a
  \[
    \big((u''', \alpha''') \comp (\gamma', \Gamma')\big)
       \, \big((u''', \alpha''') \comp (\gamma, \Gamma)\big)
    =
    (u''', \alpha''') \comp \big((\gamma', \Gamma')\,(\gamma, \Gamma)\big).
  \]
  On a
  {
  \allowdisplaybreaks
  \begin{align*}
    \MoveEqLeft
    \Big[\Gamma \comp_1 \big(\alpha''' \comp_0 u\big)\Big]
      \comp_2 
      \Big[\alpha' \comp_1 \big(\alpha''' \circ \gamma\big)\Big]
    \\*
    \MoveEqLeft \quad
    \comp_2
    \Big[\Big[
        \Big(\Gamma' \comp_1 \big(\alpha''' \comp_0 u'\big)\Big)
        \comp_2
        \Big(\alpha'' \comp_1 \big(\alpha''' \circ \gamma'\big)\Big)
    \Big] \comp_1 \Big[f'' \comp_0 u''' \comp_0 \gamma\Big]\Big]
    \\
    & =
    \Big[\Gamma \comp_1 \big(\alpha''' \comp_0 u\big)\Big]
      \comp_2
      \Big[\alpha' \comp_1 \big(\alpha''' \circ \gamma\big)\Big]
    \\*
    & \phantom{=1} \quad
    \comp_2
    \Big[
        \Gamma' \comp_1 \big(\alpha''' \comp_0 u'\big)
      \comp_1 \big(f'' \comp_0 u''' \comp_0 \gamma\big)
    \Big]
    \\*
    & \phantom{=1} \quad
        \comp_2
        \Big[\alpha'' \comp_1 \big(\alpha''' \circ \gamma'\big)
        \comp_1 \big(f'' \comp_0 u''' \comp_0 \gamma\big)\Big]
    \\
    & =
    \Big[\Gamma \comp_1 \big(\alpha''' \comp_0 u\big)\Big]
      \comp_2
      \Big[\Gamma' \comp_1 \big(f' \comp_0 \gamma\big) \comp_1 \big(\alpha'''
      \comp_0 u\big)\Big]
    \\*
    & \phantom{=1} \quad
    \comp_2
    \Big[
      \alpha'' \comp_1 \big(f' \comp_0 \gamma'\big) \comp_1 \big(\alpha'''
      \circ \gamma\big)\Big]
    \\*
    & \phantom{=1} \quad
        \comp_2
        \Big[\alpha'' \comp_1 \big(\alpha''' \circ \gamma'\big)
        \comp_1 \big(f'' \comp_0 u''' \comp_0 \gamma\big)\Big]
    \\*
    & \phantom{=1} \text{(en appliquant la loi d'échange aux termes
        centraux)}
    \\
    & =
    \Big[\Big(\Gamma \comp_2 \big(\Gamma' \comp_1 (f' \comp_0 \gamma)
        \big)\Big)
        \comp_1 \Big(\alpha''' \comp_0 u\Big)\Big]
    \\*
    & \phantom{=1} \quad
    \comp_2
    \Big[
      \alpha'' \comp_1 \Big[
        \Big(\big(f' \comp_0 \gamma'\big) \comp_1 \big(\alpha'''
        \circ \gamma\big)\Big)
        \comp_2
        \Big(\big(\alpha''' \circ \gamma'\big)
    \comp_1 \big(f'' \comp_0 u''' \comp_0 \gamma\big)\Big)\Big]\Big]
    \\
    & =
    \Big[\Big(\Gamma \comp_2 \big(\Gamma' \comp_1 (f' \comp_0 \gamma)
        \big)\Big)
        \comp_1 \Big(\alpha''' \comp_0 u\Big)\Big]
    \comp_2
    \Big[
      \alpha'' \comp_1 \Big(\alpha''' \circ \big(\gamma' \comp_1
    \gamma\big)\Big)\Big],
  \end{align*}
  }%
  la dernière égalité résultant de la compatibilité des contraintes de Gray
  à la composition $\comp_1$ (proposition~\ref{prop:circ_horiz}), ce
  qui achève la démonstration du théorème.
\end{proof}

\begin{remark}
  On peut vérifier que si $C$ est une \oo-catégorie et $c$ est un objet de
  $C$, alors la sesquicatégorie sous-jacente à la \oo-catégorie $\tr{C}{c}$ est la
  sesqui\-catégorie~$\tr{\GrayC}{c}$, où $\GrayC$ est la \oo-catégorie de Gray
  associée à $C$. Cela résulte des formules explicites définissant
  $\tr{C}{c}$ évoquées au début de la preuve de la proposition précédente et
  de la description des contraintes de Gray dans le cas strict donnée par la
  proposition~\ref{prop:contr_Gray_str}.
\end{remark}

\begin{paragraph}
  Si $\GrayC$ est une \oo-catégorie de Gray de sesquicatégorie sous-jacente
  $\SesquiC$ et $c$ est un objet de $\GrayC$, on définit un sesquifoncteur
  \[ \tr{\GrayC}{c} \to \SesquiC \]
  par
  \[
    \begin{split}
      (x, f) & \mapsto x \\
      (u, \alpha) & \mapsto u \\
      (\gamma, \Gamma) & \mapsto \gamma, \\
    \end{split}
  \]
  où les cellules de $\tr{\GrayC}{c}$ sont désignées selon les notations du
  paragraphe~\ref{paragr:def_tr_Gray}. Les formules décrivant la structure
  de $\tr{\GrayC}{c}$ données dans ce même paragraphe rendent évident le
  fait qu'on obtient bien ainsi un sesquifoncteur. On appellera ce
  sesquifoncteur le \ndef{sesquifoncteur d'oubli} de $\tr{\GrayC}{c}$ vers
  $\SesquiC$.
\end{paragraph}

\begin{paragraph}
  Soit $\GrayC$ une \oo-catégorie de Gray. On définit une \oo-catégorie de
  Gray $\oloc{\GrayC}$ de la manière suivante : les objets de $\oloc{\GrayC}$
  sont les mêmes que ceux de $\GrayC$ et, si $x$ et $y$ sont deux objets de
  $\oloc{\GrayC}$, on pose $\Homi_{\oloc{\GrayC}}(x, y) =
  \Homi_{\GrayC^{}}(x, y)^\o$. Il résulte immédiatement de la compatibilité
  du dual total au produit tensoriel (voir la
  proposition~\ref{prop:tens_dual}) qu'on obtient bien ainsi une
  \oo-catégorie de Gray. (La notation $\oloc{\GrayC}$ provient du fait que
  cette opération sur les \oo-catégories de Gray se décompose naturellement
  en deux opérations, voir~\cite[paragraphe~C.21]{AraMaltsiJoint}.) Lorsque
  $\GrayC$ provient d'une \oo-catégorie stricte~$C$, la \oo-catégorie de
  Gray $\oloc{\GrayC}$ provient de la \oo-catégorie stricte~$\oloc{C}$
  obtenue à partir de~$C$ en inversant le sens des $i$-cellules pour $i \ge
  2$. Notons également que si $\SesquiC$ est la sesquicatégorie sous-jacente
  à $\GrayC$, alors la sesquicatégorie sous-jacente à la \oo-catégorie de Gray
  $\oloc{\GrayC}$ est la sesquicatégorie $\oloc{\SesquiC}$ obtenue à partir
  de $\SesquiC$ en inversant le sens des $2$-cellules.
\end{paragraph}

\begin{paragraph}\label{paragr:def_trto}
  Soit $\GrayC$ une \oo-catégorie de Gray. On définit une sesquicatégorie
  $\trto{\GrayC}{c}$ en posant
  \[ \trto{\GrayC}{c} = \oloc{\big(\tr{\oloc{\GrayC}}{c}\big)}. \]
  Cette sesquicatégorie admet une description semblable à celle de
  $\tr{\GrayC}{c}$, l'orientation de certaines cellules étant renversée.
  Plus précisément, ses cellules se décrivent ainsi :
  \begin{itemize}[wide]
    \item Les objets de $\trto{\GrayC}{c}$ sont les diagrammes
      \[
        \xymatrix{
          y \ar[d]_g \\ c
        }
      \]
      dans $\GrayC$. On notera $(y, g)$ un tel objet.
    \item Les $1$-cellules sont les diagrammes
  \[
    \shorthandoff{;}
    
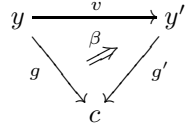
\xymatrix@C=1.5pc{
      y \ar[rr]^v \ar[dr]_(0.40){\phantom{g'}g}_(.60){}="g" & & y' \ar[dl]^(0.40){g'} \\
      & c
      \ar@{}"g";[ur]_(.15){}="gg"
      \ar@{}"g";[ur]_(.55){}="oo"
      \ar@<-0.0ex>@2"gg";"oo"^\beta
      &
    }
  \]
  dans $\GrayC$.
  On notera $(v, \beta)$ une telle $1$-cellule.  La source de $(v, \beta)$
  est $(y, g)$ et son but est $(y', g')$.
    \item Les $2$-cellules sont les diagrammes
   \[
      \shorthandoff{;:}
      
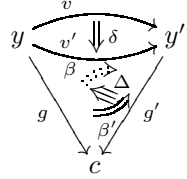
\xymatrix@C=1.5pc@R=3pc{
        y \ar@/^2ex/[rr]^(.33){v}_{}="1" \ar@/_2ex/[rr]^(.30){v'}_{}="0"
        \ar[dr]_{}="f"_{\phantom{g'}g}
        \ar@2"1";"0"^{\,\delta}
        & & y' \ar[dl]^{g'} \\
        & c
        \ar@{}"f";[ur]_(.15){}="ff"
        \ar@{}"f";[ur]_(.55){}="oo"
        \ar@<2.0ex>@/^1ex/@{<=}"oo";"ff"^(.70){\beta'}_(.30){}="h'"
        \ar@<0.5ex>@/^-1ex/@{<:}"oo";"ff"_(.75){\beta}_(.80){}="h"
        \ar@3"h'";"h"_(.20){\Delta_{}}
        &
        }
  \]
  dans $\GrayC$, où
  \[ \delta : v \tod v' \quadet \Delta : \beta' \tot (g' \comp_0 \delta)
  \comp_1 \beta. \]
  On notera $(\delta, \Delta)$ une telle $2$-cellule.  La source de
  $(\delta, \Delta)$ est $(v, \beta)$ et son but est~$(v', \beta')$.
  \end{itemize}

  Notons $\SesquiC$ la sesquicatégorie sous-jacente à $\GrayC$. Le
  sesquifoncteur d'oubli de $\tr{\oloc{\GrayC}}{c}$ vers la sesquicatégorie
  sous-jacente à $\oloc{\GrayC}$, qui n'est autre que
  $\oloc{\SesquiC}$, induit un sesquifoncteur de
  $\oloc{\big(\tr{\oloc{\GrayC}}{c}\big)}$ vers
  $\oloc{\big(\oloc{\SesquiC}\big)} = \SesquiC$. On dispose donc d'un
  sesquifoncteur
  \[ \trto{\GrayC}{c} \to \SesquiC \]
  qu'on appellera \ndef{sesquifoncteur d'oubli}.
\end{paragraph}

\begin{paragraph}\label{paragr:def_span}
  Soit $\GrayC$ une \oo-catégorie de Gray. Considérons la sesquicatégorie
  produit
  \[ \Span{\GrayC}{c}\ . \]
  Si $f : x \to c$ est une $1$-cellule de $\GrayC$, on définit
  un \ndef{sesquifoncteur d'inclusion}
  \[ {}^{}_f\iota : \trto{\GrayC}{c} \to \Span{\GrayC}{c} \]
  par le produit
  \[ (x, f) \times \id{} : \Dn{0} \times \trto{\GrayC}{c} \to
     \tr{\GrayC}{c} \times \trto{\GrayC}{c}\,, \]
  où $\Dn{0}$ désigne la sesquicatégorie terminale et $(x, f)$ le
  sesquifoncteur correspondant à l'objet $(x, f)$ de $\tr{\GrayC}{c}$.
  De même, si $g : y \to c$ est une $1$-cellule de $\GrayC$, on définit un
  \ndef{sesquifoncteur d'inclusion}
  \[ \iota^{}_g : \tr{\GrayC}{c} \to \Span{\GrayC}{c} \]
  par le produit
  \[ \id{} \times (y, g) : \tr{\GrayC}{c} \times \Dn{0} \to
     \tr{\GrayC}{c} \times \trto{\GrayC}{c}\,. \]
\end{paragraph}

\subsection{\pdfoo-catégories comma : sesquifonctorialités}\ 

\medskip

\emph{Dans cette sous-section, on fixe une \oo-catégorie $Z$.}

\begin{paragraph}
  On a défini dans la section~\ref{sec:comma_1} un foncteur
  \[
    \commaCfun : \SpanC \to \ooCat.
  \]
  Le but de cette sous-section est d'étendre ce foncteur en un
  sesquifoncteur
  \[
    \commaCfun : \Spanoo \to \ooCatOpLax,
  \]
  où $\tr{\ooCatOpLaxGray}{Z}$ et $\trto{\ooCatOpLaxGray}{Z}$ désignent les
  sesquicatégories décrites dans la sous-section précédente (et plus
  précisément aux paragraphes~\ref{paragr:def_tr_Gray} et
  \ref{paragr:def_trto}) dans le cas~$\GrayC = \ooCatOpLaxGray$ (voir
  l'exemple~\ref{ex:OpLaxGray}).
\end{paragraph}

\begin{paragraph}\label{paragr:def_sesqui_span}
  Explicitons la sesquicatégorie $\Spanoo$. Commençons par décrire ses
  cellules.
  \begin{itemize}[wide]
    \item Les objets sont les diagrammes
    \[
      \xymatrix{
        X \ar[r]^f & Z & Y \ar[l]_g
      }
    \]
    dans $\ooCat$. On notera $(X, f, g, Y)$ un tel objet.
      \item Les $1$-cellules sont les diagrammes
      \[
        \shorthandoff{;}
        \xymatrix@R=1pc@C=3pc{
          X \ar[dd]_u \ar[dr]^f_{}="f" & & Y \ar[dl]_g_{}="g" \ar[dd]^v \\
            & Z \\
          X' \ar[ur]_{f'} & & Y' \ar[ul]^{g'}
          \ar@{}[ll];"f"_(0.35){}="sa"_(0.85){}="ta"
          \ar@2"sa";"ta"^{\alpha}
          \ar@{}[];"g"_(0.35){}="tb"_(0.85){}="sb"
          \ar@2"sb";"tb"^{\beta}
        }
      \]
      dans $\ooCatOpLax$, où
      \[ \alpha : f' \comp_0 u \tod f \quadet \beta : g \tod g' \comp_0 v \]
      sont donc des transformations oplax. On notera $(u, \alpha, \beta, v)$
      un tel morphisme. La source de $(u, \alpha, \beta, v)$ est $(X, f, g,
      Y)$ et son but est~$(X', f', g', Y')$.
      \item Les $2$-cellules sont les diagrammes
      \[
        \shorthandoff{;:}
        \xymatrix@R=1.5pc@C=3.5pc{
          X \ar@/_2ex/[dd]_(0.62){\phantom{u'}u}_{}="u"
          \ar@/^2ex/[dd]_(0.65){u'\!}_{}="u'"
           \ar[dr]^f_{}="f" & &
          Y \ar@/_2ex/[dd]^(0.65){v'}_{}="v"
          \ar@/^2ex/[dd]^(0.65){v}_{}="v'"
           \ar[dl]_g_{}="g" \\
            & Z \\
          X' \ar[ur]_{f'} & & Y' \ar[ul]^{g'}
          \ar@2"u";"u'"^{\gamma\,\,}
          \ar@2"v'";"v"_{\,\,\delta}
          \ar@{}[ll];"f"_(0.40){}="sa"_(0.85){}="ta"
          \ar@<-1.5ex>@/_1ex/@2"sa";"ta"_(.70){\alpha'}_(0.40){}="a'"
          \ar@<0.0ex>@/^1ex/@{:>}"sa";"ta"^(.70){\alpha}_(0.60){}="a"
          \ar@3"a'";"a"_(.60){\Gamma_{}}
          \ar@{}[];"g"_(0.40){}="tb"_(0.85){}="sb"
          \ar@<-1.5ex>@/_1ex/@2"sb";"tb"_(.30){\beta'\!}_(0.40){}="b'"
          \ar@<0.0ex>@/^1ex/@{:>}"sb";"tb"^(.30){\!\beta}_(0.60){}="b"
          \ar@3"b'";"b"^(.40){\Delta_{}}
        }
      \]
      dans $\ooCatOpLax$, où
  \[ \gamma : u \tod u' \quadet \delta : v \tod v' \]
  sont des transformations oplax et
  \[
    \Gamma : \alpha' \comp_1 (f' \comp_0 \gamma) \tot \alpha
    \quad
    \Delta : \beta' \tot (g' \comp_0 \delta) \comp_1 \beta
  \]
  sont des $2$-transformations oplax. On notera $(\gamma, \Gamma, \Delta,
  \delta)$ une telle $2$-cellule. La source de $(\gamma, \Gamma, \Delta,
  \delta)$ est $(u, \alpha, \beta, v)$ et son but est $(u', \alpha', \beta',
  v')$.
  \end{itemize}

  Définissons maintenant les identités.
  \begin{itemize}[wide]
      \item L'identité d'un objet
    \[
      \xymatrix{
        X \ar[r]^f & Z & Y \ar[l]_g
      }
    \]
    est la $1$-cellule
      \[
        \shorthandoff{;}
        
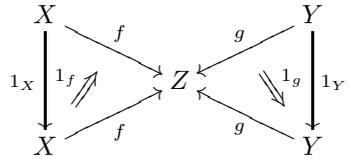
\xymatrix@R=1pc@C=3pc{
          X \ar[dd]_{\id{X}} \ar[dr]^f_{}="f" & & Y \ar[dl]_g_{}="g"
          \ar[dd]^{\id{Y}} \\
            & Z \\
          X \ar[ur]_{f} & & Y \ar[ul]^{g}
          \ar@{}[ll];"f"_(0.35){}="sa"_(0.85){}="ta"
          \ar@2"sa";"ta"^{\id{f}}
          \ar@{}[];"g"_(0.35){}="tb"_(0.85){}="sb"
          \ar@2"sb";"tb"^{\id{g}} \pbox{.}
        }
      \]
      \item L'identité d'une $1$-cellule
      \[
        \shorthandoff{;}
        
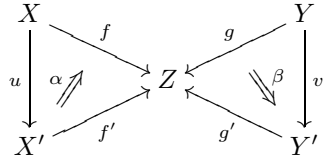
\xymatrix@R=1pc@C=3pc{
          X \ar[dd]_u \ar[dr]^f_{}="f" & & Y \ar[dl]_g_{}="g" \ar[dd]^v \\
            & Z \\
          X' \ar[ur]_{f'} & & Y' \ar[ul]^{g'}
          \ar@{}[ll];"f"_(0.35){}="sa"_(0.85){}="ta"
          \ar@2"sa";"ta"^{\alpha}
          \ar@{}[];"g"_(0.35){}="tb"_(0.85){}="sb"
          \ar@2"sb";"tb"^{\beta}
        }
      \]
      est la $2$-cellule
      \[
        \shorthandoff{;:}
        
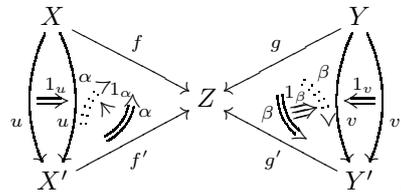
\xymatrix@R=1.5pc@C=3.5pc{
          X \ar@/_2ex/[dd]_(0.62){u}_{}="u"
          \ar@/^2ex/[dd]_(0.64){u\!}_{}="u'"
           \ar[dr]^f_{}="f" & &
          Y \ar@/_2ex/[dd]^(0.64){v}_{}="v"
          \ar@/^2ex/[dd]^(0.62){v}_{}="v'"
           \ar[dl]_g_{}="g" \\
            & Z \\
          X' \ar[ur]_{f'} & & Y' \ar[ul]^{g'}
          \ar@2"u";"u'"^{\,\,\id{u}}
          \ar@2"v'";"v"_{\,\,\id{v}}
          \ar@{}[ll];"f"_(0.40){}="sa"_(0.85){}="ta"
          \ar@<-1.5ex>@/_1ex/@2"sa";"ta"_(.70){\alpha}_(0.40){}="a'"
          \ar@<0.0ex>@/^1ex/@{:>}"sa";"ta"^(.70){\alpha}_(0.60){}="a"
          \ar@3"a'";"a"_(.75){\id{\alpha}}
          \ar@{}[];"g"_(0.40){}="tb"_(0.85){}="sb"
          \ar@<-1.5ex>@/_1ex/@2"sb";"tb"_(.30){\beta}_(0.40){}="b'"
          \ar@<0.0ex>@/^1ex/@{:>}"sb";"tb"^(.30){\!\beta}_(0.60){}="b"
          \ar@3"b'";"b"^(.40){\id{\beta}}
          \pbox{.}
        }
      \]
    \end{itemize}

    Enfin, définissons les compositions.
    \begin{itemize}[wide]
      \item Le composé de deux $1$-cellules composables
     \[
        \shorthandoff{;}
        
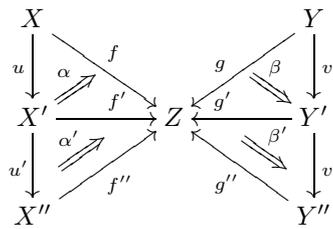
\xymatrix@R=2pc@C=3pc{
          X \ar[d]_u \ar[dr]^(0.50)f_{}="f" & & Y \ar[dl]_(0.60)g_{}="g"
          \ar[d]^v \\
          X' \ar[r]^(0.60){f'}_(.70){}="f'" \ar[d]_{u'\!}_(.70){}="u'"
          \ar@2[];"f"^(.68){\alpha\phantom{'}}
           & Z & Y' \ar[l]_(0.65){g'}_(.70){}="g'" \ar[d]^{v'}_(.70){}="v'"
          \ar@2"g";[]^(0.30){\beta} 
           \\
          X'' \ar[ur]_(0.51){f''}
          \ar@{}"u'";"f'"_(.20){}="sa'"^(.80){}="ta'"
          \ar@2"sa'";"ta'"^{\alpha'}
          & & Y'' \ar[ul]^(0.52){g''}_{}="g''"
          \ar@{}"g'";"v'"_(.20){}="sb'"^(.80){}="tb'"
          \ar@2"sb'";"tb'"^{\beta'}
        }
      \]
      est la $1$-cellule
    \[
    \shorthandoff{;}
    
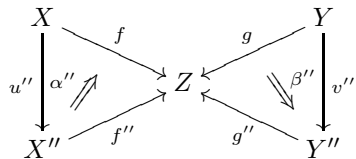
\xymatrix@R=1pc@C=3pc{
      X \ar[dd]_{u''} \ar[dr]^f_{}="f" & & Y \ar[dl]_g_{}="g" \ar[dd]^{v''} \\
        & Z \\
      X'' \ar[ur]_{f''} & & Y'' \ar[ul]^{g''}
      \ar@{}[ll];"f"_(0.35){}="sa"_(0.85){}="ta"
      \ar@2"sa";"ta"^{\alpha''}
      \ar@{}[];"g"_(0.35){}="tb"_(0.85){}="sb"
      \ar@2"sb";"tb"^{\beta''}
    }
    \]
    où
    \[
      u'' = u' \comp_0 u,
      \quad
      \alpha'' = \alpha \comp_1 (\alpha' \comp_0 u),
      \quad
      \beta'' = (\beta' \comp_0 v) \comp_1 \beta
      \quadet
      v'' = v' \comp_0 v.
    \]
    \item Le composé vertical de deux $2$-cellules
      \[
        \shorthandoff{;:}
        
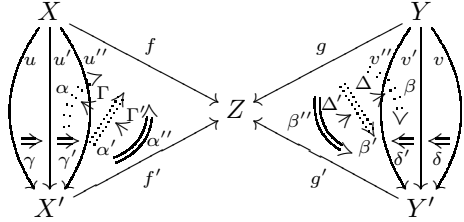
\xymatrix@R=2pc@C=4.5pc{
          X
          \ar@/_3.5ex/[dd]^(0.21){\mkern-6mu u}_(0.65){}="u"
          \ar[dd]^(0.22){\mkern-4mu u'}_(0.65){}="u'"
          \ar@/^3.5ex/[dd]^(0.25){\mkern-4mu u''}_(0.65){}="u''"
          \ar@2"u";"u'"_{\,\gamma\phantom{'}}
          \ar@2"u'";"u''"_{\gamma'}
          \ar[dr]^f_{}="f" & & Y \ar[dl]_g_{}="g"
          \ar@/_3.5ex/[dd]_(0.25){v''\mkern-9mu}_(0.65){}="v''"
          \ar[dd]_(0.22){v'\mkern-5mu}_(0.65){}="v'"
          \ar@/^3.5ex/[dd]_(0.21){v\mkern -4mu}_(0.65){}="v"
          \ar@2"v";"v'"^{\,\delta\phantom{'}}
          \ar@2"v'";"v''"^{\delta'}
            \\
            & Z \\
          X' \ar[ur]_{f'} & & Y' \ar[ul]^{g'}
          \ar@{}[ll];"f"_(0.40){}="sa"_(0.85){}="ta"
          \ar@<2ex>@/^1.5ex/@{:>}"sa";"ta"^(0.50){\alpha\!}_{}="a"
          \ar@{:>}"sa";"ta"_(0.20){\!\!\alpha'}_{}="a'"
          \ar@<-2ex>@/_1.5ex/@2"sa";"ta"_(0.60){\!\!\alpha''}_{}="a''"
          \ar@3"a'";"a"_{\,\Gamma}
          \ar@3"a''";"a'"_{\Gamma'}
          \ar@{}"g";[]_(0.10){}="sb"_(0.55){}="tb"
          \ar@<2ex>@/^1.5ex/@{:>}"sb";"tb"^(0.58){\beta}_{}="b"
          \ar@{:>}"sb";"tb"_(0.95)*+<-0.5em>{\labelstyle \beta'\mkern -2mu}_{}="b'"
          \ar@<-2ex>@/_1.5ex/@2"sb";"tb"_(0.30){\beta''\mkern-2mu}_{}="b''"
          \ar@3"b'";"b"^*+<0.2em>{\labelstyle \Delta}
          \ar@3"b''";"b'"^*+<0.25em>{\mkern17mu\labelstyle \Delta\mkern-2mu'}
        }
      \]
      est la $2$-cellule
      \[
        \shorthandoff{;:}
        
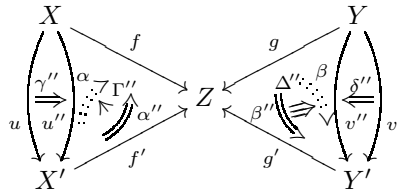
\xymatrix@R=1.5pc@C=3.5pc{
          X \ar@/_2ex/[dd]_(0.62){\phantom{u''}u}_{}="u"
          \ar@/^2ex/[dd]_(0.65){u''\!}_{}="u'"
           \ar[dr]^f_{}="f" & &
          Y \ar@/_2ex/[dd]^(0.65){v''}_{}="v"
          \ar@/^2ex/[dd]^(0.65){v}_{}="v'"
           \ar[dl]_g_{}="g" \\
            & Z \\
          X' \ar[ur]_{f'} & & Y' \ar[ul]^{g'}
          \ar@2"u";"u'"^{\gamma''\,\,}
          \ar@2"v'";"v"_{\,\,\delta''}
          \ar@{}[ll];"f"_(0.40){}="sa"_(0.85){}="ta"
          \ar@<-1.5ex>@/_1ex/@2"sa";"ta"_(.70){\alpha''}_(0.40){}="a'"
          \ar@<0.0ex>@/^1ex/@{:>}"sa";"ta"^(.70){\alpha}_(0.60){}="a"
          \ar@3"a'";"a"_(.60){\Gamma''_{}}
          \ar@{}[];"g"_(0.40){}="tb"_(0.85){}="sb"
          \ar@<-1.5ex>@/_1ex/@2"sb";"tb"_(.30){\beta''\!}_(0.40){}="b'"
          \ar@<0.0ex>@/^1ex/@{:>}"sb";"tb"^(.30){\!\beta}_(0.60){}="b"
          \ar@3"b'";"b"^(.40)*+<1.0em>{\labelstyle \mskip-5mu\Delta\!''}
        }
      \]
    où
    \[
    \renewcommand\quad{\hskip0.5em\relax}
     \gamma'' = \gamma' \comp_1 \gamma, \quad
     \Gamma'' = \Gamma \comp_2 \big(\Gamma' \comp_1 (f' \comp_0
     \gamma)\big), \quad
     \Delta '' = \big((g' \comp_0 \delta') \comp_1 \Delta\big)\comp_2
     \Delta'
     \quadet \delta'' = \delta' \comp_1 \delta.
    \]

  \item
    Le composé horizontal d'une $1$-cellule suivie d'une $2$-cellule
     \[
        \shorthandoff{;:}
        
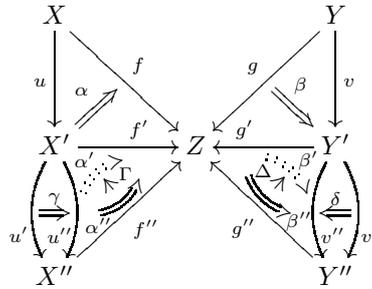
\xymatrix@R=3pc@C=3pc{
          X \ar[d]_u \ar[dr]^(0.50)f_{}="f" & & Y \ar[dl]_(0.50)g_{}="g"
          \ar[d]^v \\
          X' \ar[r]^(0.60){f'}_(.70){}="f'" \ar@/_2ex/[d]_(.64){u'\!}_{}="u'"
          \ar@/^2ex/[d]_(.70){u''\!\!}_{}="u''"
          \ar@2"u'";"u''"^{\gamma}
        \ar@2[];"f"^(.68){\alpha\phantom{'}}
          & Z & Y' \ar[l]_(0.65){g'}_(.70){}="g'" \ar@/^2ex/[d]^(.64){\!v'}_{}="v'"
          \ar@/_2ex/[d]^(0.70){\!v''}_{}="v''"
           \ar@2"v'";"v''"_{\delta}
          \ar@2"g";[]^(0.30){\beta}
           \\
          X'' \ar[ur]_(0.51){f''}
          & & Y'' \ar[ul]^(0.52){g''}_{}="g''"
          \ar@{}"u''";"f'"_(.10){}="sa'"_(.75){}="ta'"
          \ar@<0.6ex>@/^1ex/@{:>}"sa'";"ta'"_{}="tG"^(0.5)*+<-.3em>{\labelstyle \alpha'}
          \ar@<-0.9ex>@/_1ex/@2"sa'";"ta'"_{}="sG"_(0.05)*+<-.3em>{\labelstyle \alpha''}
          \ar@3"sG";"tG"_\Gamma
          \ar@{}"g'";"v''"_(.20){}="sb'"_(.90){}="tb'"
          \ar@<0.6ex>@/^1ex/@{:>}"sb'";"tb'"_{}="sD"^(0.6)*+<-.1em>{\!\labelstyle \beta'}
          \ar@<-0.9ex>@/_1ex/@2"sb'";"tb'"_{}="tD"_(0.95)*+<-.7em>{\,\,\,\,\labelstyle \beta''}
          \ar@3"tD";"sD"^*+<.1em>{\labelstyle \Delta\!}
        }
      \]
      est la $2$-cellule
      \[
        \shorthandoff{;:}
        
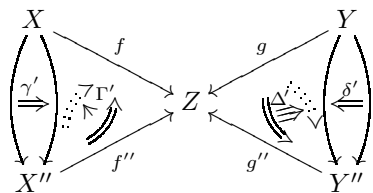
\xymatrix@R=1.5pc@C=3.5pc{
          X \ar@/_2ex/[dd]_(0.62){\phantom{}}_{}="u"
          \ar@/^2ex/[dd]_(0.65){}_{}="u'"
           \ar[dr]^f_{}="f" & &
          Y \ar@/_2ex/[dd]^(0.65){}_{}="v"
          \ar@/^2ex/[dd]^(0.65){}_{}="v'"
           \ar[dl]_g_{}="g" \\
            & Z \\
          X'' \ar[ur]_{f''} & & Y'' \ar[ul]^{g''}
          \ar@2"u";"u'"^{\gamma'\,\,}
          \ar@2"v'";"v"_{\,\,\delta'}
          \ar@{}[ll];"f"_(0.40){}="sa"_(0.85){}="ta"
          \ar@<-1.5ex>@/_1ex/@2"sa";"ta"_(.70){}_(0.40){}="a'"
          \ar@<0.0ex>@/^1ex/@{:>}"sa";"ta"^(.70){}_(0.60){}="a"
          \ar@3"a'";"a"_(.60){\Gamma'_{}}
          \ar@{}[];"g"_(0.40){}="tb"_(0.85){}="sb"
          \ar@<-1.6ex>@/_1ex/@2"sb";"tb"_(.30){}_(0.40){}="b'"
          \ar@<-0.0ex>@/^1ex/@{:>}"sb";"tb"^(.30){}_(0.60){}="b"
          \ar@3"b'";"b"^(.42)*+<0.1em>{\labelstyle \mskip-5mu\Delta\!'}
        }
      \]
    où
    \[
     \gamma' = \gamma \comp_0 u,
     \quad
     \Gamma' = \alpha \comp_1 (\Gamma \comp_0 u),
     \quad
     \Delta ' = (\Delta \comp_0 v) \comp_1 \beta
     \quadet
     \delta' = \delta \comp_0 v.
    \]
    \item Enfin, le composé horizontal d'une $2$-cellule suivie d'une
    $1$-cellule
      \[
        \shorthandoff{;:}
        
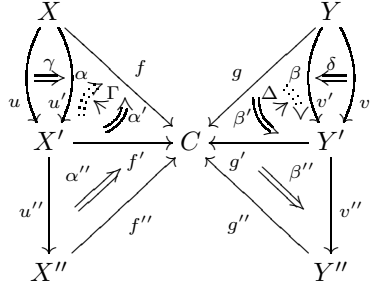
\xymatrix@R=3pc@C=3pc{
          X \ar@/_2ex/[d]_(0.67)u_{}="u" \ar@/^2ex/[d]_(0.70){u'\!\!}_{}="u'"
          \ar@2"u";"u'"^\gamma
          \ar[dr]^(0.55)f_{}="f" & & Y \ar[dl]_(0.58)g_{}="g"
          \ar@/^2ex/[d]^(.68)v_{}="v" \ar@/_2ex/[d]^(0.7){\!v'}_{}="v'"
          \ar@2"v";"v'"_\delta
          \\
          X' \ar[r]_(0.60){\,f'}_(.70){}="f'"_(.20){}="f'2" \ar[d]_{u''\!}_(.70){}="u''"
          & C & Y' \ar[l]^(0.65){g'\,}_(.70){}="g'"_(.20){}="g'2" \ar[d]^{v''}_(.70){}="v''"
          \ar@{}"f'2";"f"_(.25){}="sa"_(.90){}="ta"
          \ar@<0.6ex>@/^1ex/@{:>}"sa";"ta"_{}="tG"^(0.7)*+<.1em>{\labelstyle \alpha}
          \ar@<-0.9ex>@/_1ex/@2"sa";"ta"_{}="sG"_(0.65)*+<-.3em>{\labelstyle \alpha'}
          \ar@3"sG";"tG"_\Gamma
          \ar@{}"g";"g'2"_(.10){}="sb'"_(.80){}="tb'"
          \ar@<0.6ex>@/^1ex/@{:>}"sb'";"tb'"_{}="sD"^(0.3)*+<-.1em>{\!\labelstyle \beta}
          \ar@<-0.9ex>@/_1ex/@2"sb'";"tb'"_{}="tD"_(0.4)*+<-.7em>{\labelstyle
            \beta'\,\,}
          \ar@3"tD";"sD"^*+<.1em>{\labelstyle \Delta\!}
           \\
          X'' \ar[ur]_(0.51){f''}
          \ar@{}"u''";"f'"_(.20){}="sa'"^(.80){}="ta'"
          \ar@2"sa'";"ta'"^{\alpha''}
          & & Y'' \ar[ul]^(0.52){g''}_{}="g''"
          \ar@{}"g'";"v''"_(.20){}="sb'"^(.80){}="tb'"
          \ar@2"sb'";"tb'"^{\beta''}
        }
      \]
      est la $2$-cellule
      \[
        \shorthandoff{;:}
        
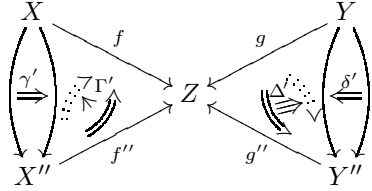
\xymatrix@R=1.5pc@C=3.5pc{
          X \ar@/_2ex/[dd]_(0.62){\phantom{}}_{}="u"
          \ar@/^2ex/[dd]_(0.65){}_{}="u'"
           \ar[dr]^f_{}="f" & &
          Y \ar@/_2ex/[dd]^(0.65){}_{}="v"
          \ar@/^2ex/[dd]^(0.65){}_{}="v'"
           \ar[dl]_g_{}="g" \\
            & Z \\
          X'' \ar[ur]_{f''} & & Y'' \ar[ul]^{g''}
          \ar@2"u";"u'"^{\gamma'\,\,}
          \ar@2"v'";"v"_{\,\,\delta'}
          \ar@{}[ll];"f"_(0.40){}="sa"_(0.85){}="ta"
          \ar@<-1.5ex>@/_1ex/@2"sa";"ta"_(.70){}_(0.40){}="a'"
          \ar@<0.0ex>@/^1ex/@{:>}"sa";"ta"^(.70){}_(0.60){}="a"
          \ar@3"a'";"a"_(.60){\Gamma'_{}}
          \ar@{}[];"g"_(0.40){}="tb"_(0.85){}="sb"
          \ar@<-1.6ex>@/_1ex/@2"sb";"tb"_(.30){}_(0.40){}="b'"
          \ar@<-0.0ex>@/^1ex/@{:>}"sb";"tb"^(.30){}_(0.60){}="b"
          \ar@3"b'";"b"^(.42)*+<0.1em>{\labelstyle \mskip-5mu\Delta\!'}
        }
      \]
    où
    \begin{align*}
     \gamma' = u'' \comp_0 \gamma,
     \quad
     \Gamma' & = (\Gamma \comp_1 (\alpha'' \comp_0 u)) \comp_2 (\alpha' \comp_1
    (\alpha'' \circ \gamma)),
     \\
     \Delta' & =
    ((\beta'' \circ \delta) \comp_1 \beta)
     \comp_2
    ((\beta'' \comp_0 v') \comp_1 \Delta),
     \quad
     \delta' = v'' \comp_0 \delta,
    \end{align*}
  \end{itemize}
  le symbole $\circ$ désignant la contrainte de Gray qui associe à deux
  transformations oplax composables horizontalement une $2$-transformation
  oplax (voir le paragraphe~\ref{paragr:def_contr_Gray} et la
  proposition~\ref{prop:s_t_Gray}).
\end{paragraph}

\begin{paragr}
  La description donnée au paragraphe précédent de la sesquicatégorie
  $\smash{\Spanoo}$ montre que, comme annoncé, sa catégorie sous-jacente est
  la catégorie $\smash{\SpanC}$ décrite au paragraphe~\ref{paragr:def_cat_span}.
  En particulier, le sesquifoncteur
  \[
    \commaCfun : \Spanoo \to \ooCatOpLax,
  \]
  qu'on cherche à définir est déjà défini sur les objets et les
  $1$-cellules (voir les paragraphes~\ref{paragr:def_comma} et
  \ref{paragr:comma_act_fonct}).
\end{paragr}

\begin{paragraph}\label{paragr:comma_pu_trans}
  Soit
  \[
    \xymatrix{
      X \ar[r]^f & Z & Y \ar[l]_g
    }
  \]
  un diagramme dans $\ooCat$ et soit $T$ une \oo-catégorie. En vertu du
  paragraphe~\ref{paragr:comma_pu}, la donnée d'une
  transformation oplax entre \oo-foncteurs de $T$ vers $f \comma g$,
  c'est-à-dire d'un \oo-foncteur $\Dn{1} \otimes T \to f \comma g$,
  correspond à celle d'un diagramme
  \[
    \shorthandoff{;}
    
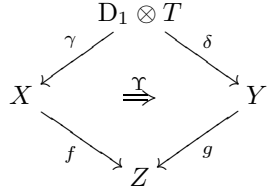
\xymatrix@C=1.5pc@R=1.5pc{
      & \Dn{1} \otimes T \ar[dl]_{\gamma} \ar[dr]^{\delta} \\
      X \ar[dr]_f
      \ar@{}[rr]_(.40){}="x"_(.60){}="y"
      \ar@2"x";"y"^{\Upsilon}
      & & Y \ar[dl]^g \\
        & Z
    }
  \]
  dans $\ooCatOpLax$. La transformation oplax $\Upsilon$ correspond par
  définition à un \oo-foncteur $\Dn{1} \otimes \Dn{1} \otimes T \to Z$ qui, à
  son tour, par adjonction, correspond à un \oo-foncteur
  $\Dn{1} \otimes \Dn{1} \to \HomOpLax(T, Z)$, ou encore à un « carré
  oplax » dans la \oo-catégorie $\HomOpLax(T, Z)$, c'est-à-dire à un diagramme
  \[
    \shorthandoff{;}
    
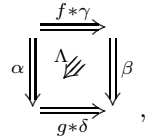
\xymatrix{
      \ar@2[d]_\alpha \ar@2[r]^{f \comp \gamma} & \ar@2[d]^\beta \\
      \ar@2[r]_{g \comp \delta} &
      \ar@{}[l];[u]_(.30){}="x"^(.70){}="y"
      \ar@3"y";"x"_{\Lambda\!}
      \pbox{,}
    }
  \]
  où $\alpha$ et $\beta$ sont des transformations oplax et $\Lambda$ est
  une $2$-transformation oplax. La donnée d'une transformation oplax entre
  \oo-foncteurs de $T$ vers~$f \comma g$ correspond donc exactement à celle
  d'un diagramme
  \[
    \shorthandoff{;}
    
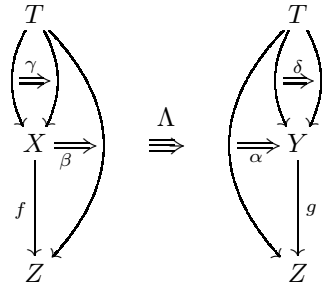
\xymatrix@R=3pc{
      T
      \ar@/_2ex/[d]_{}="sg" \ar@/^2ex/[d]_{}="tg"
      \ar@2"sg";"tg"^{\gamma\,\,}
      \ar@/^6ex/[dd]_{}="ta"
      \ar@2[d];"ta"_{\beta\,\,}
      & & &
      T
      \ar@/_2ex/[d]_{}="sd" \ar@/^2ex/[d]_{}="td"
      \ar@2"sd";"td"^{\delta}
      \ar@/_6ex/[dd]_{}="sb"
      \ar@2"sb";[d]_{\alpha\,\,\,\,}
      \\
      X \ar[d]_f
      & & &
      Y \ar[d]^g
      \\
      Z
      & & &
      Z
      \ar@{}"sb";"ta"_(0.30){}="sl"_(0.70){}="tl"
      \ar@3"tl";"sl"^{\raisebox{1ex}{$\Lambda$}}
    }
  \]
  dans $\ooCatOpLax$. On notera $(\gamma, \alpha, \Lambda, \beta, \delta)$
  la transformation oplax $\Dn{1} \otimes T \to f \comma g$ correspondant à
  un tel diagramme. La source et le but de cette transformation sont les
  \oo-foncteurs $T \to f \comma g$ correspondant respectivement aux
  diagrammes
  \[
    \shorthandoff{;}
    \xymatrix@C=1.5pc@R=1.5pc{
      & T \ar[dl]_{s(\gamma)} \ar[dr]^{s(\delta)} \\
      X \ar[dr]_f
      \ar@{}[rr]_(.35){}="x"_(.65){}="y"
      \ar@2"x";"y"^{\alpha}
      & & Y \ar[dl]^g \\
      & Z & \pbox{,}
    }
    \qquad\qquad
    \xymatrix@C=1.5pc@R=1.5pc{
      & T \ar[dl]_{t(\gamma)} \ar[dr]^{t(\delta)} \\
      X \ar[dr]_f
      \ar@{}[rr]_(.35){}="x"_(.65){}="y"
      \ar@2"x";"y"^{\beta}
      & & Y \ar[dl]^g \\
      & Z & \pbox{.}
    }
  \]
\end{paragraph}

\begin{paragraph}\label{paragr:comma_act_trans}
  Considérons
  \[
    \shorthandoff{;:}
    (\gamma, \Gamma, \Delta, \delta) =
    \raisebox{2.5pc}{
        $\xymatrix@R=1.5pc@C=3.5pc{
          X \ar@/_2ex/[dd]_(0.62){\phantom{u'}u}_{}="u"
          \ar@/^2ex/[dd]_(0.65){u'\!}_{}="u'"
           \ar[dr]^f_{}="f" & &
          Y \ar@/_2ex/[dd]^(0.65){v'}_{}="v"
          \ar@/^2ex/[dd]^(0.65){v}_{}="v'"
           \ar[dl]_g_{}="g" \\
            & Z \\
          X' \ar[ur]_{f'} & & Y' \ar[ul]^{g'}
          \ar@2"u";"u'"^{\gamma\,\,}
          \ar@2"v'";"v"_{\,\,\delta}
          \ar@{}[ll];"f"_(0.40){}="sa"_(0.85){}="ta"
          \ar@<-1.5ex>@/_1ex/@2"sa";"ta"_(.70){\alpha'}_(0.40){}="a'"
          \ar@<0.0ex>@/^1ex/@{:>}"sa";"ta"^(.70){\alpha}_(0.60){}="a"
          \ar@3"a'";"a"_(.60){\Gamma_{}}
          \ar@{}[];"g"_(0.40){}="tb"_(0.85){}="sb"
          \ar@<-1.5ex>@/_1ex/@2"sb";"tb"_(.30){\beta'\!}_(0.40){}="b'"
          \ar@<0.0ex>@/^1ex/@{:>}"sb";"tb"^(.30){\!\beta}_(0.60){}="b"
          \ar@3"b'";"b"^(.40){\Delta_{}}
    }$}
  \]
  une $2$-cellule de $\Spanoo$. On lui associe une transformation oplax
  \[
    \shorthandoff{;}
    \xymatrix@C=9pc{
      f \comma g
      \ar@/^4ex/[r]^{(u, \alpha) \comma (\beta, v)}_{}="0"
      \ar@/_4ex/[r]_{(u', \alpha') \comma (\beta', v')}_{}="1"
      \ar@{}"0";"1"_(0.10){}="00"_(0.90){}="11"
      \ar@2"00";"11"_{(\gamma, \Gamma) \comma (\Delta, \delta)\,}
      &
      f' \comma g' \pbox{,}
    }
  \]
  qu'on notera parfois également $(\gamma, \Gamma, \Delta, \delta)_\ast$, de
  la manière suivante. Soit $T$ une \oo-catégorie et soit $(x, \lambda,
  y) : T \to f \comma g$ un \oo-foncteur (voir le
  paragraphe~\ref{paragr:comma_pu}). En composant le diagramme
  \[
    \shorthandoff{;:}
        \xymatrix@R=1.5pc@C=3.5pc{
          & T \ar[dl]_x \ar[dr]^y
          \ar@{}[dl];[dr]_(0.40){}="sl"_(0.60){}="tl"
          \ar@2"sl";"tl"^{\lambda}
          \\
          X \ar@/_2ex/[dd]_(0.62){\phantom{u'}u}_{}="u"
          \ar@/^2ex/[dd]_(0.65){u'\!}_{}="u'"
           \ar[dr]^f_{}="f" & &
          Y \ar@/_2ex/[dd]^(0.65){v'}_{}="v"
          \ar@/^2ex/[dd]^(0.65){v}_{}="v'"
           \ar[dl]_g_{}="g" \\
            & Z \\
          X' \ar[ur]_{f'} & & Y' \ar[ul]^{g'}
          \ar@2"u";"u'"^{\gamma\,\,}
          \ar@2"v'";"v"_{\,\,\delta}
          \ar@{}[ll];"f"_(0.40){}="sa"_(0.85){}="ta"
          \ar@<-1.5ex>@/_1ex/@2"sa";"ta"_(.70){\alpha'}_(0.40){}="a'"
          \ar@<0.0ex>@/^1ex/@{:>}"sa";"ta"^(.70){\alpha}_(0.60){}="a"
          \ar@3"a'";"a"_(.60){\Gamma_{}}
          \ar@{}[];"g"_(0.40){}="tb"_(0.85){}="sb"
          \ar@<-1.5ex>@/_1ex/@2"sb";"tb"_(.30){\beta'\!}_(0.40){}="b'"
          \ar@<0.0ex>@/^1ex/@{:>}"sb";"tb"^(.30){\!\beta}_(0.60){}="b"
          \ar@3"b'";"b"^(.40){\Delta_{}}
          \pbox{,}
    }
  \]
  on obtient une $2$-transformation oplax
  \[
    \shorthandoff{;:}
        \xymatrix@R=1.5pc@C=3.5pc{
          & T \ar[dl]_x \ar[dr]^y
          \ar@{}[dl];[dr]_(0.40){}="sl"_(0.60){}="tl"
          \ar@2"sl";"tl"^{\lambda}
          \\
          X \ar@/_2ex/[dd]_(0.62){\phantom{u'}u}_{}="u"
          \ar@/^2ex/[dd]_(0.65){u'\!}_{}="u'"
           \ar[dr]^f_{}="f" & &
           Y \ar@/^2ex/[dd]^(0.62){v'}_{}="v'"
           \ar[dl]_g_{}="g" \\
            & Z \\
          X' \ar[ur]_{f'} & & Y' \ar[ul]^{g'}
          \ar@2"u";"u'"^{\gamma\,\,}
          \ar@{}[ll];"f"_(0.40){}="sa"_(0.85){}="ta"
          \ar@<-1ex>@2"sa";"ta"_(.70){\alpha'}_(0.40){}="a'"
          \ar@{}[];"g"_(0.33){}="tb"_(0.78){}="sb"
          \ar@<1ex>@2"sb";"tb"_(.38){\beta'\!}_(0.40){}="b'"
    }
    \raisebox{-5pc}{
    $\xymatrix@C=2pc{
      \ar@3[r]^{} &
    }$}
    \shorthandoff{;:}
        \xymatrix@R=1.5pc@C=3.5pc{
          & T \ar[dl]_x \ar[dr]^y
          \ar@{}[dl];[dr]_(0.40){}="sl"_(0.60){}="tl"
          \ar@2"sl";"tl"^{\lambda}
          \\
          X \ar@/_2ex/[dd]_(0.62){\phantom{u'}u}_{}="u"
           \ar[dr]^f_{}="f" & &
          Y \ar@/_2ex/[dd]^(0.67){v}_{}="v"
          \ar@/^2ex/[dd]^(0.62){v'}_{}="v'"
           \ar[dl]_g_{}="g" \\
            & Z \\
          X' \ar[ur]_{f'} & & Y' \ar[ul]^{g'}
          \ar@2"v";"v'"^{\delta}
          \ar@{}[ll];"f"_(0.40){}="sa"_(0.85){}="ta"
          \ar@2"sa";"ta"_(.50){\,\alpha}_(0.60){}="a"
          \ar@{}[];"g"_(0.40){}="tb"_(0.85){}="sb"
          \ar@2"sb";"tb"_(.45){\!\beta}_(0.60){}="b"
    }
  \]
  donnée par la formule
  $(\Delta \comp_0 y) \comp_1 \lambda \comp_1 (\Gamma \comp_0 x)$.
  Ainsi, en vertu du paragraphe précédent,
  \[
    \big(\gamma \comp_0 x,
    (\beta \comp_0 y) \comp_1 \lambda \comp_1 (\alpha \comp_0 x),
    (\Delta \comp_0 y) \comp_1 \lambda \comp_1 (\Gamma \comp_0 x),
    (\beta' \comp_0 y) \comp_1 \lambda \comp_1 (\alpha' \comp_0 x),
    \delta \comp_0 y\big)
  \]
  définit une transformation oplax qui, par adjonction, peut se représenter
  comme un \oo-foncteur $T \to \HomLax(\Dn{1}, f' \comma g')$. Par ailleurs,
  en vertu de ce même paragraphe, la source et le but de cette
  transformation oplax sont
  les \oo-foncteurs $T \to f' \comma g'$
  \[
  \big(u \comp_0 x, (\beta \comp_0 y) \comp_1 \lambda \comp_1 (\alpha
  \comp_0 x), v \comp_0 y\big)
  \quadet
  \big(u' \comp_0 x, (\beta' \comp_0 y) \comp_1 \lambda \comp_1
  (\alpha' \comp_0 x), v' \comp_0 y\big),
  \]
  selon la notation du paragraphe~\ref{paragr:comma_pu}. Il résulte de la
  fonctorialité de la composition horizontale par une $1$-cellule que
  l'application
  \[ \Hom_{\ooCat}(T, f \comma g) \to \Hom_{\ooCat}(T, \HomLax(\Dn{1}, f' \comma g')) \]
  que l'on vient de décrire est naturelle en $T$ et, en vertu du lemme de
  Yoneda, on a donc bien défini une transformation oplax $(\gamma, \Gamma)
  \comma (\Delta, \delta)$ entre \oo-foncteurs de $f \comma g$ vers $f'
  \comma g'$.  Par ailleurs, les formules donnant la source et le but de la
  transformation oplax $T \to \HomLax(\Dn{1}, f' \comma g')$ montrent que
  les source et but de la transformation oplax $(\gamma, \Gamma) \comma
  (\Delta, \delta)$ sont bien respectivement $(u, \alpha) \comma (\beta, v)$
  et $(u', \alpha') \comma (\beta', v')$ (voir le
  paragraphe~\ref{paragr:comma_act_fonct}).
\end{paragraph}

\begin{theorem}\label{thm:comma_fonct}
  Soit $Z$ une \oo-catégorie. Les applications
  \[
    \begin{split}
      (f, g) & \mapsto f \comma g \\
      (u, \alpha, \beta, v) & \mapsto (u, \alpha) \comma (\beta, v) \\
      (\gamma, \Gamma, \Delta, \delta) & \mapsto (\gamma, \Gamma) \comma (\Delta,
      \delta)
    \end{split}
  \]
  définissent un sesquifoncteur
    \[ \commaCfun : \Spanoo \to \ooCatOpLax. \]
\end{theorem}

\begin{proof}
  On a déjà montré la $1$-fonctorialité de la construction comma
  (proposition~\ref{prop:comma_fonct}) et il s'agit de montrer la
  compatibilité aux opérations mettant en jeu des $2$-cellules.
  Fixons
  \[
    (X, f, g, Y) =
    \xymatrix{
       X \ar[r]^f & Z & Y \ar[l]_g
    }
  \]
  un objet de $\Spanoo$, $T$ une \oo-catégorie et $(x, \lambda, y) : T \to f
  \comma g$ un \oo-foncteur, c'est-à-dire un diagramme
  \[
    \shorthandoff{;}
     \xymatrix@C=1.5pc@R=1.5pc{
      & T \ar[dl]_x \ar[dr]^y \\
      X \ar[dr]_f
      \ar@{}[rr]_(.35){}="x"_(.65){}="y"
      \ar@2"x";"y"^{\lambda}
      & & Y \ar[dl]^g \\
        & Z
  }
  \]
  dans $\ooCatOpLax$. On va vérifier la sesquifonctorialité de
  $\commaCfun$ en utilisant le lemme de Yoneda, c'est-à-dire en précomposant
  les égalités que l'on veut montrer par $(x, \lambda, y)$. Dans cette
  démonstration, on considérera toute transformation oplax entre
  \oo-foncteurs d'une \oo-caté\-gorie~$A$ vers une \oo-catégorie $B$ comme
  un \oo-foncteur de $A$ vers~$\HomLax(\Dn{1}, B)$.

  Considérons
  \[
    \shorthandoff{;}
    (u, \alpha, \beta, v) =
    \raisebox{2pc}{
    $\xymatrix@R=1pc@C=3pc{
      X \ar[dd]_u \ar[dr]^f_{}="f" & & Y \ar[dl]_g_{}="g" \ar[dd]^v \\
        & Z \\
      X' \ar[ur]_{f'} & & Y' \ar[ul]^{g'}
      \ar@{}[ll];"f"_(0.35){}="sa"_(0.85){}="ta"
      \ar@2"sa";"ta"^{\alpha}
      \ar@{}[];"g"_(0.35){}="tb"_(0.85){}="sb"
      \ar@2"sb";"tb"^{\beta}
    }$}
  \]
  une $1$-cellule de $\Spanoo$. Vérifions la compatibilité de $\commaCfun$ à
  l'identité de $(f, \alpha, \beta, g)$. Il suffit donc de vérifier
  l'égalité
  \[
    (\id{(f, \alpha, \beta, g)})_\ast (x, \lambda, y)
    = \id{(f, \alpha, \beta, g)_\ast} (x, \lambda, y),
  \]
  où on considère les transformations oplax $(\id{(f, \alpha, \beta,
  g)})_\ast$ et $\id{(f, \alpha, \beta, g)_\ast}$ comme des \oo-foncteurs de
  $f \comma g$ vers $\HomLax(\Dn{1}, f \comma g)$, comme mentionné
  précédemment. En vertu du paragraphe~\ref{paragr:comma_act_trans}, il
  s'agit donc de vérifier l'égalité entre deux $5$-uplets. Toutes les
  composantes de ces $5$-uplets à l'exception de la troisième étant
  déterminées par la source et le but de ces transformations oplax, par
  compatibilité de $\commaCfun$ aux sources et aux buts, il suffit de
  vérifier l'égalité des troisièmes composantes. Or ces composantes sont
  respectivement
  \[
    (\id{\beta} \comp_0 y) \comp_1 \lambda \comp_1 (\id{\alpha} \comp_0 x)
    \quadet
    \id{(\beta \comp_0 y) \comp_1 \lambda \comp_1 (\alpha \comp_0 x)}
  \]
  et sont donc bien égales.

  Soient maintenant
  \[
    \shorthandoff{;:}
    
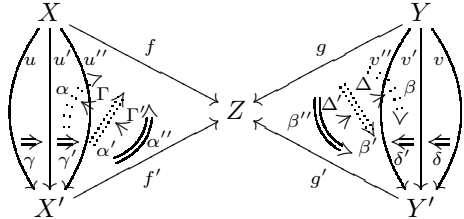
\xymatrix@R=2pc@C=4.5pc{
      X
      \ar@/_3.5ex/[dd]^(0.21){\mkern-6mu u}_(0.65){}="u"
      \ar[dd]^(0.22){\mkern-4mu u'}_(0.65){}="u'"
      \ar@/^3.5ex/[dd]^(0.25){\mkern-4mu u''}_(0.65){}="u''"
      \ar@2"u";"u'"_{\,\gamma\phantom{'}}
      \ar@2"u'";"u''"_{\gamma'}
      \ar[dr]^f_{}="f" & & Y \ar[dl]_g_{}="g"
      \ar@/_3.5ex/[dd]_(0.25){v''\mkern-9mu}_(0.65){}="v''"
      \ar[dd]_(0.22){v'\mkern-5mu}_(0.65){}="v'"
      \ar@/^3.5ex/[dd]_(0.21){v\mkern -4mu}_(0.65){}="v"
      \ar@2"v";"v'"^{\,\delta\phantom{'}}
      \ar@2"v'";"v''"^{\delta'}
        \\
        & Z \\
      X' \ar[ur]_{f'} & & Y' \ar[ul]^{g'}
      \ar@{}[ll];"f"_(0.40){}="sa"_(0.85){}="ta"
      \ar@<2ex>@/^1.5ex/@{:>}"sa";"ta"^(0.50){\alpha\!}_{}="a"
      \ar@{:>}"sa";"ta"_(0.20){\!\!\alpha'}_{}="a'"
      \ar@<-2ex>@/_1.5ex/@2"sa";"ta"_(0.60){\!\!\alpha''}_{}="a''"
      \ar@3"a'";"a"_{\,\Gamma}
      \ar@3"a''";"a'"_{\Gamma'}
      \ar@{}"g";[]_(0.10){}="sb"_(0.55){}="tb"
      \ar@<2ex>@/^1.5ex/@{:>}"sb";"tb"^(0.58){\beta}_{}="b"
      \ar@{:>}"sb";"tb"_(0.95)*+<-0.5em>{\labelstyle \beta'\mkern -2mu}_{}="b'"
      \ar@<-2ex>@/_1.5ex/@2"sb";"tb"_(0.30){\beta''\mkern-2mu}_{}="b''"
      \ar@3"b'";"b"^*+<0.2em>{\labelstyle \Delta}
      \ar@3"b''";"b'"^*+<0.25em>{\mkern17mu\labelstyle \Delta\mkern-2mu'}
    }
  \]
  deux $2$-cellules composables verticalement de $\Spanoo$. Vérifions la
  compatibilité de $\commaCfun$ à leur composition. Il s'agit de vérifier
  l'égalité
  \[
    \big((\gamma', \Gamma', \Delta', \delta')_\ast
    (\gamma, \Gamma, \Delta, \delta)_\ast\big)
    (x, \lambda, y)
    =
    \big((\gamma', \Gamma', \Delta', \delta')
    (\gamma, \Gamma, \Delta, \delta)\big)_\ast
    (x, \lambda, y).
  \]
  Pour les mêmes raisons que dans la vérification précédente (et le fait
  qu'on a déjà vérifié la $1$-fonctorialité), il suffit de vérifier
  l'égalité des troisièmes composantes des $5$-uplets décrivant les deux
  membres. Or, la troisième composante du membre de gauche est la
  $2$-transformation oplax composée du diagramme
  \[
    \shorthandoff{;}
    
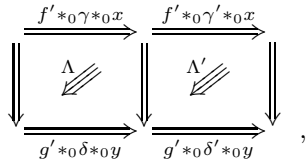
\xymatrix@C=3.5pc@R=2.5pc{
      \ar@2[d] \ar@2[r]^{f' \comp_0 \gamma \comp_0 x} & \ar@2[d]
      \ar@2[r]^{f' \comp_0 \gamma' \comp_0 x} & \ar@2[d] \\
      \ar@2[r]_{g' \comp_0 \delta \comp_0 y} &
      \ar@2[r]_{g' \comp_0 \delta' \comp_0 y} &
      \ar@{}[ll];[lu]_(.30){}="x"^(.70){}="y"
      \ar@3"y";"x"_{\Lambda}
      \ar@{}[l];[u]_(.30){}="x'"^(.70){}="y'"
      \ar@3"y'";"x'"_{\Lambda'\!}
      \pbox{,}
    }
  \]
  où les flèches verticales, qui ne joueront aucun rôle dans ce calcul, sont
  données par les formules du paragraphe~\ref{paragr:comma_act_trans}
  et où
  \[
    \Lambda = (\Delta \comp_0 y) \comp_1 \lambda \comp_1 (\Gamma \comp_0 x)
    \quadet
    \Lambda' = (\Delta' \comp_0 y) \comp_1 \lambda \comp_1 (\Gamma' \comp_0
    x),
  \]
  c'est-à-dire la $2$-transformation oplax
  \[
    \big[(g' \comp_0 \delta' \comp_0 y) \comp_1 (\Delta \comp_0 y) \comp_1 \lambda
    \comp_1 (\Gamma \comp_0 x)\big]
    \comp_2
    \big[(\Delta' \comp_0 y) \comp_1 \lambda \comp_1 (\Gamma' \comp_0 x)
      \comp_1 (f' \comp_0 \gamma \comp_0 x)\big].
  \]
  Par ailleurs, la troisième composante du membre de droite est
  \[ (\Delta'' \comp_0 y) \comp_1 \lambda \comp_1 (\Gamma'' \comp_0 x), \]
  où
  \[
    \Gamma '' = \Gamma \comp_2 (\Gamma' \comp_1 (f' \comp_0 \gamma))
    \quadet
    \Delta'' = ((g' \comp_0 \delta') \comp_1 \Delta) \comp_2 \Delta',
  \]
  c'est-à-dire
  \[
    \big[(((g' \comp_0 \delta') \comp_1 \Delta) \comp_2 \Delta') \comp_0 y\big]
    \comp_1 \lambda \comp_1
    \big[(\Gamma \comp_2 (\Gamma' \comp_1 (f' \comp_0 \gamma))) \comp_0
    x\big].
  \]
  Or, la fonctorialité de la composition par une $1$-cellule dans
  $\ooCatOpLaxGray$ et la loi de l'échange (dans $\HomOpLax(T, f' \comma
  g')$) entraînent que ces deux expressions sont égales, ce
  qui achève de montrer la compatibilité de $\commaCfun$ à la composition
  verticale des $2$-cellules.

  Considérons maintenant
  \[
    \shorthandoff{;:}
    
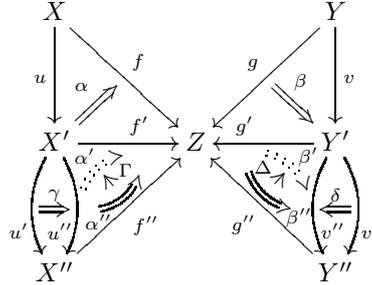
\xymatrix@R=3pc@C=3pc{
      X \ar[d]_u \ar[dr]^(0.50)f_{}="f" & & Y \ar[dl]_(0.50)g_{}="g"
      \ar[d]^v \\
      X' \ar[r]^(0.60){f'}_(.70){}="f'" \ar@/_2ex/[d]_(.64){u'\!}_{}="u'"
      \ar@/^2ex/[d]_(.70){u''\!\!}_{}="u''"
      \ar@2"u'";"u''"^{\gamma}
    \ar@2[];"f"^(.68){\alpha\phantom{'}}
      & Z & Y' \ar[l]_(0.65){g'}_(.70){}="g'" \ar@/^2ex/[d]^(.64){\!v'}_{}="v'"
      \ar@/_2ex/[d]^(0.70){\!v''}_{}="v''"
       \ar@2"v'";"v''"_{\delta}
      \ar@2"g";[]^(0.30){\beta}
       \\
      X'' \ar[ur]_(0.51){f''}
      & & Y'' \ar[ul]^(0.52){g''}_{}="g''"
      \ar@{}"u''";"f'"_(.10){}="sa'"_(.75){}="ta'"
      \ar@<0.6ex>@/^1ex/@{:>}"sa'";"ta'"_{}="tG"^(0.5)*+<-.3em>{\labelstyle \alpha'}
      \ar@<-0.9ex>@/_1ex/@2"sa'";"ta'"_{}="sG"_(0.05)*+<-.3em>{\labelstyle \alpha''}
      \ar@3"sG";"tG"_\Gamma
      \ar@{}"g'";"v''"_(.20){}="sb'"_(.90){}="tb'"
      \ar@<0.6ex>@/^1ex/@{:>}"sb'";"tb'"_{}="sD"^(0.6)*+<-.1em>{\!\labelstyle \beta'}
      \ar@<-0.9ex>@/_1ex/@2"sb'";"tb'"_{}="tD"_(0.95)*+<-.7em>{\,\,\,\,\labelstyle \beta''}
      \ar@3"tD";"sD"^*+<.1em>{\labelstyle \Delta\!}
    }
   \]
   une $1$-cellule suivie d'une $2$-cellule composables dans $\Spanoo$.
   Montrons la compatibilité de $\commaCfun$ à la composition de ces
   cellules. Il s'agit de montrer, comme dans les vérifications
   précédentes, l'égalité entre les troisièmes composantes des $5$\nbd-uplets
   \[
     \big((\gamma, \Gamma, \Delta, \delta)_\ast \comp (u, \alpha, \beta,
     v)_\ast\big) (x, \lambda, y)
     \quadet
     \big((\gamma, \Gamma, \Delta, \delta) \comp (u, \alpha, \beta,
     v)\big)_\ast (x, \lambda, y).
   \]
   Or
   \[
     \begin{split}
     \MoveEqLeft
     \big((\gamma, \Gamma, \Delta, \delta)_\ast \comp (u, \alpha, \beta,
     v)_\ast\big) (x, \lambda, y) \\
     & =
     (\gamma, \Gamma, \Delta, \delta)_\ast (u, \alpha, \beta, v)_\ast (x,
     \lambda, y)
     \\
     & =
     (\gamma, \Gamma, \Delta, \delta)_\ast (u \comp_0 x, (\beta \comp_0 y)
     \comp_1 \lambda \comp_1 (\alpha \comp_0 x), v \comp_0 y)
     \end{split}
   \]
   et la troisième composante de ce $5$-uplet est donc
   \[
     (\Delta \comp_0 v \comp_0 y) \comp_1 (\beta \comp_0 y)
     \comp_1 \lambda \comp_1
     (\alpha \comp_0 x) \comp_1 (\Gamma \comp_0 u \comp_0 x).
   \]
   D'autre part, la troisième composante de
   $\big((\gamma, \Gamma, \Delta, \delta) \comp (u, \alpha, \beta,
   v)\big)_\ast (x, \lambda, y)$ est
   \[
      (\Delta' \comp_0 y) \comp_1 \lambda \comp_1 (\Gamma' \comp_0 x)
   \]
   où
   \[
     \Gamma' = \alpha \comp_1 (\Gamma \comp_0 u)
     \quadet
     \Delta' = (\Delta \comp_0 v) \comp_1 \beta,
   \]
   c'est-à-dire
   \[
     \big(((\Delta \comp_0 v) \comp_1 \beta) \comp_0 y\big)
     \comp_1 \lambda \comp_1
     \big((\alpha \comp_1 (\Gamma \comp_0 u)) \comp_0 x\big).
   \]
   Or ces deux expressions sont bien égales en vertu de la fonctorialité de
   la composition par une $1$-cellule dans $\ooCatOpLaxGray$.

  Enfin, soient
  \[
    \shorthandoff{;:}
    
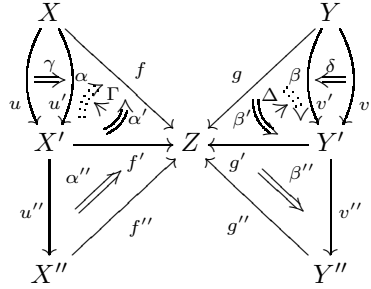
\xymatrix@R=3pc@C=3pc{
      X \ar@/_2ex/[d]_(0.67)u_{}="u" \ar@/^2ex/[d]_(0.70){u'\!\!}_{}="u'"
      \ar@2"u";"u'"^\gamma
      \ar[dr]^(0.55)f_{}="f" & & Y \ar[dl]_(0.58)g_{}="g"
      \ar@/^2ex/[d]^(.68)v_{}="v" \ar@/_2ex/[d]^(0.7){\!v'}_{}="v'"
      \ar@2"v";"v'"_\delta
      \\
      X' \ar[r]_(0.60){\,f'}_(.70){}="f'"_(.20){}="f'2" \ar[d]_{u''\!}_(.70){}="u''"
      & Z & Y' \ar[l]^(0.65){g'\,}_(.70){}="g'"_(.20){}="g'2" \ar[d]^{v''}_(.70){}="v''"
      \ar@{}"f'2";"f"_(.25){}="sa"_(.90){}="ta"
      \ar@<0.6ex>@/^1ex/@{:>}"sa";"ta"_{}="tG"^(0.7)*+<.1em>{\labelstyle \alpha}
      \ar@<-0.9ex>@/_1ex/@2"sa";"ta"_{}="sG"_(0.65)*+<-.3em>{\labelstyle \alpha'}
      \ar@3"sG";"tG"_\Gamma
      \ar@{}"g";"g'2"_(.10){}="sb'"_(.80){}="tb'"
      \ar@<0.6ex>@/^1ex/@{:>}"sb'";"tb'"_{}="sD"^(0.3)*+<-.1em>{\!\labelstyle \beta}
      \ar@<-0.9ex>@/_1ex/@2"sb'";"tb'"_{}="tD"_(0.4)*+<-.7em>{\labelstyle
        \beta'\,\,}
      \ar@3"tD";"sD"^*+<.1em>{\labelstyle \Delta\!}
       \\
      X'' \ar[ur]_(0.51){f''}
      \ar@{}"u''";"f'"_(.20){}="sa'"^(.80){}="ta'"
      \ar@2"sa'";"ta'"^{\alpha''}
      & & Y'' \ar[ul]^(0.52){g''}_{}="g''"
      \ar@{}"g'";"v''"_(.20){}="sb'"^(.80){}="tb'"
      \ar@2"sb'";"tb'"^{\beta''}
    }
  \]
  une $2$-cellule suivie d'une $1$-cellule composables dans $\Spanoo$.
  Vérifions la compatibilité de $\commaCfun$ à la composition de ces
  cellules. Comme précédemment, il s'agit de montrer l'égalité entre les
  troisièmes composantes des $5$-uplets
   \[
     \big((u'', \alpha'', \beta'', v'')_\ast \comp
     (\gamma, \Gamma, \Delta, \delta)_\ast\big)
       (x, \lambda, y)
     \quadet
     \big((u'', \alpha'', \beta'', v'') \comp (\gamma, \Gamma, \Delta, \delta)\big)_\ast
       (x, \lambda, y).
   \]
  Déterminons la troisième composante du $5$-uplet du membre de
  gauche. Par définition, ce $5$-uplet correspond au composé du diagramme
  \[
    \shorthandoff{;}
    
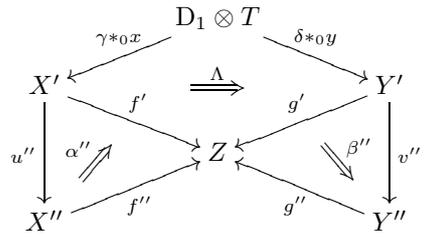
\xymatrix@R=1pc@C=3pc{
      & \Dn{1} \otimes T \ar[dl]_{\gamma \comp_0 x} \ar[dr]^{\delta \comp_0
    y} \\
      X' \ar[dd]_{u''} \ar[dr]^{f'}_{}="f" & & Y' \ar[dl]_{g'}_{}="g"
      \ar[dd]^{v''}
      \ar@{}[ll];[]_(0.40){}="x"_(0.60){}="y"
      \ar@2"x";"y"^{\Lambda}
       \\
        & Z \\
      X'' \ar[ur]_{f''} & & Y'' \ar[ul]^{g''}
      \ar@{}[ll];"f"_(0.35){}="sa"_(0.85){}="ta"
      \ar@2"sa";"ta"^{\alpha''\!}
      \ar@{}[];"g"_(0.35){}="tb"_(0.85){}="sb"
      \ar@2"sb";"tb"^{\beta''}
      \pbox{,}
    }
  \]
  où
  \[
    \Lambda = (\gamma, \Gamma, \Delta, \delta)_\ast (x, \lambda, y).
  \]
  On vérifie que si
  \[
    \shorthandoff{;}
    \xymatrix@C=2.5pc{
      \Dn{1} \otimes A \ar[r]^-{\phi} & B \ar@/^2ex/[r]_{}="s"
      \ar@/_2ex/[r]_{}="t" \ar@2"s";"t"_{\psi} & C
    }
  \]
  est un diagramme dans $\ooCatOpLax$, alors le composé
  \[
    \shorthandoff{;}
    \xymatrix@C=2.5pc{
      \Dn{1} \otimes A \ar@/^3ex/[r]_{}="s"
      \ar@/_3ex/[r]_{}="t" \ar@2"s";"t" & C
    }
  \]
  correspond, par adjonction, au carré commutatif à transformation oplax
  près donné par le \oo-foncteur
  \[
    \xymatrix@C=3pc{
      \Dn{1} \otimes \Dn{1} \ar[r]^-{\psi \otimes \phi} & \HomOpLax(B,
      C) \otimes \HomOpLax(A, B) \ar[r]^-{\circ_{C,B,A}} & \HomOpLax(A, C)
    }
  \]
  qui est décrit explicitement dans la proposition~ \ref{prop:s_t_Gray} et
  dont la $2$-cellule correspond à la contrainte de Gray $\psi \circ \phi$.
  Ainsi, le composé qu'on cherche à déterminer correspond au composé
  \[
    \shorthandoff{;}
    \xymatrix@C=3pc@R=3pc{
      \bullet
      \ar@2[r]^*+<0.4em>{\labelstyle f'' \comp_0 u'' \comp_0 \gamma \comp_0 x}
      \ar@2[d]_{\alpha'' \comp_0 u \comp_0 x}
      &
      \bullet
      \ar@2[d]^{\alpha'' \comp_0 u' \comp_0 x}
      \\
      \bullet
      \ar@2[r]
      \ar@2[d]_{(\beta \comp_0 y) \comp_1 \lambda \comp_1 (\alpha \comp_0
      x)}
      &
      \bullet
      \ar@2[d]^{(\beta' \comp_0 y) \comp_1 \lambda \comp_1 (\alpha' \comp_0
      x)}
      \\
      \bullet
      \ar@2[r]
      \ar@2[d]_{\beta'' \comp_0 v \comp_0 y}
      &
      \bullet
      \ar@2[d]^{\beta'' \comp_0 v' \comp_0 y}
      \\
      \bullet
      \ar@2[r]_*+<0.4em>{\labelstyle g'' \comp_0 v'' \comp_0 \delta \comp_0 y}
      &
      \bullet
      \ar@{}[uuu];[uul]_(.20){}="s1"_(.80){}="t1"
      \ar@3"s1";"t1"|-{\scriptscriptstyle \alpha'' \circ_{} (\gamma \comp_0 x)}
      \ar@{}[uu];[ul]_(.30){}="s2"_(.70){}="t2"
      \ar@3"s2";"t2"_{\Lambda'}
      \ar@{}[u];[l]_(.20){}="s3"_(.80){}="t3"
      \ar@3"s3";"t3"|-{\scriptscriptstyle \beta'' \circ_{} (\delta \comp_0 y)}
      \pbox{,}
    }
  \]
  où les deux flèches horizontales non décorées sont, de haut en bas,
  \[
    f' \comp_0 \gamma \comp_0 x
    \quadet
    g' \comp_0 \delta \comp_0 y,
  \]
  et où
  \[
    \Lambda' = (\Delta \comp_0 y) \comp_1 \lambda \comp_1 (\Gamma \comp_0
    x).
  \]
  Ainsi, la troisième composante recherchée est la $2$-transformation oplax
  \[
    \begin{split}
      &
      \big[ (\beta'' \circ (\delta \comp_0 y)) \comp_1 (\beta \comp_0 y)
        \comp_1 \lambda \comp_1 (\alpha \comp_0 x) \comp_1 (\alpha'' \comp_0
      u \comp_0 x) \big] \\
      & \quad \comp_2
      \big[ (\beta'' \comp_0 v' \comp_0 y) \comp_1 (\Delta \comp_0 y)
        \comp_1 \lambda \comp_1 (\Gamma \comp_0 x) \comp_1 (\alpha'' \comp_0
        u \comp_0 x) \big] \\
      & \quad \comp_2
        \big[ (\beta'' \comp_0 v' \comp_0 y) \comp_1 (\beta' \comp_0 y)
          \comp_1 \lambda \comp_1 (\alpha' \comp_0 x) \comp_1 (\alpha''
        \circ (\gamma \comp_0 x)) \big].
    \end{split}
  \]
  En vertu de la loi de l'échange pour les compositions $\comp_1$ et
  $\comp_2$, cette $2$-transformation oplax est égale à
  \[ \Delta'' \comp_1 \lambda \comp_1 \Gamma'' \]
  où $\Delta''$ est
  \[
    \big[(\beta'' \circ (\delta \comp_0 y)) \comp_1 (\beta \comp_0 y)\big] \comp_2
    \big[(\beta'' \comp_0 v' \comp_0 y) \comp_1 (\Delta \comp_0 y)\big] \comp_2
    \big[(\beta'' \comp_0 v' \comp_0 y) \comp_1 (\beta' \comp_0 y)\big]
  \]
  et $\Gamma''$ est
  \[
    \big[(\alpha \comp_0 x) \comp_1 (\alpha'' \comp_0 u \comp_0 x)\big]
    \comp_2
    \big[(\Gamma \comp_0 x) \comp_1 (\alpha'' \comp_0 u \comp_0 x)\big]
    \comp_2
    \big[(\alpha' \comp_0 x) \comp_1 (\alpha'' \circ (\gamma \comp_0
    x))\big].
  \]
  Or, il résulte du fait que $\ooCatOpLaxGray$ est une \oo-catégorie de Gray (et
  plus précisément des propositions~\ref{prop:circ_assoc} et
  \ref{prop:Gray_fonct_1-cell}) qu'on a
  \[
    \begin{split}
      \Delta''
      & =
    \big[(\beta'' \circ (\delta \comp_0 y)) \comp_1 (\beta \comp_0 y)\big] \comp_2
    \big[(\beta'' \comp_0 v' \comp_0 y) \comp_1 (\Delta \comp_0 y)\big] \\
    & =
    \big[((\beta'' \circ \delta) \comp_0 y) \comp_1 (\beta \comp_0 y)\big] \comp_2
    \big[(\beta'' \comp_0 v' \comp_0 y) \comp_1 (\Delta \comp_0 y)\big] \\
    & =
    \big[
    ((\beta'' \circ \delta) \comp_1 \beta) \comp_2
    ((\beta'' \comp_0 v') \comp_1 \Delta)\big] \comp_0 y \\
    \end{split}
  \]
  et, de même,
  \[ \Gamma'' = \big[(\Gamma \comp_1 (\alpha'' \comp_0 u)) \comp_2 (\alpha'
  \comp_1 (\alpha'' \circ \gamma))\big] \comp_0 x. \]
  Ainsi, la troisième composante du membre de gauche de l'égalité qu'on
  cherche à établir est égale à
   \[
      (\Delta' \comp_0 y) \comp_1 \lambda \comp_1 (\Gamma' \comp_0 x),
   \]
   où
  \begin{align*}
     \Gamma' & = (\Gamma \comp_1 (\alpha'' \comp_0 u)) \comp_2 (\alpha' \comp_1
    (\alpha'' \circ \gamma)),
     \\
     \Delta' & =
    ((\beta'' \circ \delta) \comp_1 \beta)
     \comp_2
    ((\beta'' \comp_0 v') \comp_1 \Delta).
  \end{align*}
  Or, c'est précisément la définition de la troisième composante du membre
  de droite. On obtient donc l'égalité recherchée, ce qui achève la
  démonstration.
\end{proof}

\begin{paragraph}\label{paragr:comma_var_fixe}
  Fixons $f : X \to Z$ un \oo-foncteur. En vertu du
  paragraphe~\ref{paragr:def_span}, on dispose d'un sesquifoncteur
  d'inclusion
  \[ {}^{}_f\iota : \trto{(\ooCatOpLaxGray)}{Z} \to \Spanoo \]
  et on notera
  \[ f \comma {-} : \trto{(\ooCatOpLaxGray)}{Z} \to \ooCatOpLax \]
  le sesquifoncteur composé
  \[
     \trto{(\ooCatOpLaxGray)}{Z} \xto{\,\,\,\,{}^{}_f\iota\,\,\,\,} \Spanoo
     \xto{\commaCfun} \ooCatOpLax.
  \]
  De même, si on fixe un \oo-foncteur $g : Y \to Z$, on dispose d'un
  sesquifoncteur d'inclusion
  \[ \iota_{g} : \tr{(\ooCatOpLaxGray)}{Z} \to \Spanoo \]
  et on notera
  \[ {-} \comma g : \tr{(\ooCatOpLaxGray)}{Z} \to \ooCatOpLax \]
  le sesquifoncteur obtenu par composition avec le sesquifoncteur $\commaCfun$.
\end{paragraph}

\begin{corollary}\label{coro:comma_retr_app}
  Soient
    \[
      \xymatrix{
        X \ar[r]^f & Z & Y \ar[l]_g
      }
    \]
  deux \oo-foncteurs.
  \begin{enumerate}
    \item Si $i : X' \to X$ est un rétracte par transformation oplax à gauche
      fort, alors il en est de même de
      \[ (i, \id{fi}) \comma g :  (fi) \comma g \to f \comma g. \]
    \item Si $j : Y' \to Y$ est un rétracte par transformation oplax à
      droite fort, alors il en est de même de
      \[ f \comma (\id{gj}, j) :  f \comma (gj) \to f \comma g. \]
  \end{enumerate}
\end{corollary}

\begin{proof}
  Démontrons la première assertion, la seconde se démontrant de manière
  analogue. Soit $(r, \alpha)$ une structure de rétracte par transformation
  oplax à gauche fort sur $i$.  Rappelons que cela signifie que $r : X \to
  X'$ est un \oo-foncteur vérifiant $ri = \id{X'}$ et que $\alpha : ir \tod
  \id{X}$ est une transformation oplax vérifiant $\alpha \comp i = \id{i}$.
  Par définition, $(i, \id{fi}) \comma g$ est l'image par le sesquifoncteur
  ${-} \comma g$ du triangle commutatif
  \[
    \shorthandoff{;}
    (i, \id{fi}) =
    \raisebox{1.5pc}{
    $\xymatrix@C=1.5pc{
      X' \ar[rr]^i \ar[dr]_(0.40){fi}_{}="f" & & X \ar[dl]^(0.40)f_{}="g" \\
      & Z
      \ar@{}"f";[ur]_(.15){}="ff"
      \ar@{}"f";[ur]_(.55){}="oo"
      \ar@{}"f";"g"^{\textstyle =}
      & \pbox{.}
    }$}
  \]
  Considérons le triangle
  \[
    \shorthandoff{;}
    (r, f \comp \alpha) =
    \raisebox{1.5pc}{
    $\xymatrix@C=1.5pc{
      X \ar[rr]^r \ar[dr]_(0.40){f}_(.60){}="f" & & X' \ar[dl]^(0.40){fi} \\
      & Z
      \ar@{}"f";[ur]_(.15){}="ff"
      \ar@{}"f";[ur]_(.55){}="oo"
      \ar@<-0.0ex>@2"oo";"ff"_{f \comp \alpha}
      & \pbox{.}
    }$}
  \]
  Le composé
  \[
    \shorthandoff{;}
    \xymatrix@C=2pc{
      X' \ar[r]^i \ar[dr]_(0.50){}="se"_(.40){fi}
      & X \ar[r]^{r}_(.75){}="fp" \ar[d]_(.70){}="gp"_(0.50){}="te"_(.56){f} & X'
      \ar[dl]_{}="gpp"^(.40){fi} \\
      & Z
      \ar@{}"se";"te"^{\textstyle =}
      \ar@{}"gp";"fp"_(.25){}="x2"
      \ar@{}"gp";"fp"_(.75){}="y2"
      \ar@<0.0ex>@{<=}"x2";"y2"^(0.40){f \comp \alpha\!\!\!}
    }
  \]
  est égal à
  \[
  (r, f \comp \alpha)(i, \id{fi}) = (ri, f \comp \alpha \comp i) =
  (ri, f \comp \id{i}) = (\id{X'}, \id{fi}) = \id{(X', fi)}.
  \]
  Ainsi, par fonctorialité de ${-} \comma g$, le \oo-foncteur $(r, f \comp
  \alpha) \comma g$ est une rétraction de~$(i, \id{fi}) \comma g$. Par
  ailleurs, le composé
  \[
    \shorthandoff{;}
    \xymatrix@C=2pc{
      X \ar[r]^r \ar[dr]_(.40){f}_{}="g"
      & X' \ar[r]^{i}_(.75){}="fp"
      \ar[d]_(.70){}="gp"_(0.50){}="se"_(.56){fi} & X
      \ar[dl]_{}="gpp"^(.40){f}_{}="te" \\
      & Z
      \ar@{}"g";[u]_(0.10){}="x"
      \ar@{}"g";[u]_(.85){}="y"
      \ar@<-0.1ex>@{<=}"x";"y"^(.30){f \comp \alpha\!\!}
      \ar@{}"se";"te"^{\textstyle =}
    }
  \]
  est égal à
  \[
  (i, \id{i})(r, f \comp \alpha) = (ir, f \comp \alpha)
  \]
  et le cône commutatif
  \[
      \shorthandoff{;:}
      (\alpha, \id{f \comp \alpha}) =
      \raisebox{1.5pc}{
      $\xymatrix@C=1.5pc@R=3pc{
        X \ar@/^2ex/[rr]^(.33){ir}_{}="1" \ar@{=}@/_2ex/[rr]_{}="0"
        \ar[dr]_f_{}="f"
        \ar@2"1";"0"_{\alpha\,}
        & & X \ar[dl]^f \\
        & Z
        \ar@{}"f";[ur]_(.10){}="ff"
        \ar@{}"f";[ur]_(.50){}="oo"
        \ar@<-0.5ex>@/^1ex/@{:>}"ff";"oo"^(.25){f \comp \alpha\!\!}_(.30){}="h'"
        \ar@<-2.0ex>@/^-1ex/@{=}"ff";"oo"_(.80){}="h"
        \ar@3{-}"h";"h'" &
        }$}
  \]
  définit une $2$-cellule de ce composé vers $\id{(X, f)}$ dans
  $\tr{\ooCatOpLaxGray}{Z}$. Ainsi, par sesquifonctorialité de ${-} \comma
  g$, on dispose d'une transformation oplax
  \[ (\alpha, \id{f \comp \alpha}) \comma g :
    \big((i, \id{fi}) \comma g\big)\big((r, f \comp \alpha) \comma g)\big)
    \tod
    \id{(X, f) \comma g}.
  \]
  Enfin, le composé
  \[
      \shorthandoff{;:}
      \xymatrix@C=3.5pc@R=3.5pc{
      X' \ar[r]^i \ar[dr]_{}="g"_{fi}_(0.40){}="se" &
      X \ar@/^2ex/[r]^(.33){ir}_{}="1"
      \ar@/_2ex/@{=}[r]_{}="0"_(.70){}="fp"
      \ar[d]_(.50){}="gp2"_(.20){}="gp"_(0.62){f}_(0.40){}="te" &
      X \ar[dl]^{f} &
      \ar@2{<-}"0";"1"^{\alpha} \\
        & Z
      \ar@{}"se";"te"|{\textstyle =}
      \ar@{}"gp2";"fp"_(.10){}="ff2"
      \ar@{}"gp2";"fp"_(.55){}="oo2"
      \ar@<+0.5ex>@/^1ex/@{<:}"ff2";"oo2"^*+<-.2em>{\labelstyle f \comp \alpha\!\!}_(.40){}="h'''"
      \ar@<-0.5ex>@/^-1.5ex/@{=}"ff2";"oo2"_(.70){}="h''"
      \ar@3{-}"h''";"h'''"
      }
  \]
  étant égal à
  \[
  (\alpha, \id{f \comp \alpha}) \comp (i, \id{fi})
   = (\alpha \comp i, \id{f \comp \alpha \comp i})
   = (\id{i}, \id{\id{fi}})
   = \id{(i, \id{fi})},
  \]
  on a, par sesquifonctorialité de ${-} \comma g$,
  \[
     \big((\alpha, \id{f \comp \alpha}) \comma g\big) \comp
     \big((i, \id{fi}) \comma g\big) = \id{(i, \id{fi}) \comma g},
   \]
   ce qui achève de montrer que $\big((r, \alpha) \comma g, (\alpha, \id{f
   \comp \alpha}) \comma g)\big)$ est une structure de rétracte par
   transformation oplax à gauche fort sur $(i, \id{fi}) \comma g$.
\end{proof}

\begin{proposition}
  Soit
  \[
    \shorthandoff{;:}
    (\gamma, \Gamma, \Delta, \delta) =
    \raisebox{2.5pc}{
        $\xymatrix@R=1.5pc@C=3.5pc{
          X \ar@/_2ex/[dd]_(0.62){\phantom{u'}u}_{}="u"
          \ar@/^2ex/[dd]_(0.65){u'\!}_{}="u'"
           \ar[dr]^f_{}="f" & &
          Y \ar@/_2ex/[dd]^(0.65){v'}_{}="v"
          \ar@/^2ex/[dd]^(0.65){v}_{}="v'"
           \ar[dl]_g_{}="g" \\
            & Z \\
          X' \ar[ur]_{f'} & & Y' \ar[ul]^{g'}
          \ar@2"u";"u'"^{\gamma\,\,}
          \ar@2"v'";"v"_{\,\,\delta}
          \ar@{}[ll];"f"_(0.40){}="sa"_(0.85){}="ta"
          \ar@<-1.5ex>@/_1ex/@2"sa";"ta"_(.70){\alpha'}_(0.40){}="a'"
          \ar@<0.0ex>@/^1ex/@{:>}"sa";"ta"^(.70){\alpha}_(0.60){}="a"
          \ar@3"a'";"a"_(.60){\Gamma_{}}
          \ar@{}[];"g"_(0.40){}="tb"_(0.85){}="sb"
          \ar@<-1.5ex>@/_1ex/@2"sb";"tb"_(.30){\beta'\!}_(0.40){}="b'"
          \ar@<0.0ex>@/^1ex/@{:>}"sb";"tb"^(.30){\!\beta}_(0.60){}="b"
          \ar@3"b'";"b"^(.40){\Delta_{}}
    }$}
  \]
  une $2$-cellule de $\Spanoo$. Alors le diagramme
  \[
    \shorthandoff{;:}
    \xymatrix@C=5pc@R=3pc{
      f \comma g
      \ar[d]_p
      \ar@/^3ex/[r]^(.50){(u, \alpha) \comma (\beta, v)}_{}="1"
      \ar@/_3ex/[r]_(.50){(u', \alpha') \comma (\beta', v')}_{}="0"
      \ar@2"1";"0"|{(\gamma, \Gamma) \comma (\Delta, \delta)\,}
      &
      f' \comma g'
      \ar[d]^p
      \\
      X \times Y
      \ar@/^3ex/[r]^(.50){u \times v}_{}="2"
      \ar@/_3ex/[r]_(.50){u' \times
      v'}_{}="1"
      \ar@2"2";"1"_{\gamma \times \delta\,}
      &
      X' \times Y'
      }
  \]
  est commutatif au sens où on a l'égalité
  \[ p \comp ((\gamma, \Gamma) \comma (\Delta, \delta)) = (\gamma \times
  \delta) \comp p. \]
\end{proposition}

\begin{proof}
  Démontrons l'égalité par le lemme de Yoneda. Soient donc $T$ une
  \oo-catégorie et $(x, \lambda, y) : T \to f \comma g$ un \oo-foncteur. En
  vertu des paragraphes~\ref{paragr:comma_pu_trans} et
  \ref{paragr:comma_act_trans}, la transformation oplax $(\gamma, \Gamma)
  \comma (\Delta, \delta)$ associe à $(x, \lambda, y)$ le \oo-foncteur
  $(\gamma \comp_0 x, \Upsilon, \delta \comp_0 y) : \Dn{1} \otimes T \to f'
  \comma g'$, pour $\Upsilon$ une certaine transformation oplax. En
  postcomposant par la projection $p : f' \comma g' \to X' \times Y'$, on
  obtient donc le \oo-foncteur $(\gamma \comp_0 x, \delta \comp_0 y) :
  \Dn{1} \otimes T \to X \times Y$ qui est bien le \oo-foncteur associé
  à $(x, \lambda, y)$ par la transformation oplax $(\gamma \times \delta)
  \comp p$, d'où le résultat.
\end{proof}

\section{Comparaison avec la preuve simpliciale}
\label{app:thmAI}

\begin{paragraph}
  Dans \cite{AraMaltsiThmAI}, nous donnons une preuve alternative, de nature
  simpliciale, de notre théorème~A \oo-catégorique. Le point central des
  deux démonstrations est le même : il s'agit de montrer que si $v : A \to
  C$ est un \oo-foncteur et $c$ est un $m$\nbd-simplexe de $N(C)$, alors le
  morphisme simplicial $r : \cotr{N(A)}{c} \to \cotr{N(A)}{c_m}$ du
  paragraphe~\ref{paragr:def_thmA} est une équivalence faible, c'est-à-dire,
  dans la terminologie du présent texte, que l'objet cosimplicial $\cO :
  \cDelta \to \ooCat$ donné par les orientaux permet un théorème~A. Pour ce
  faire, dans \cite{AraMaltsiThmAI}, nous définissons par des formules
  explicites une section $s : \cotr{N(A)}{c_m} \to \cotr{N(A)}{c}$ de~$r$
  (au paragraphe 6.2) et une homotopie simpliciale~$h$ de $sr$
  vers~$\id{\cotr{N(A)}{c}}$ (au paragraphe 6.6).

  Le but de cet appendice est de montrer que le rétracte par déformation
  simplicial~$s$, la rétraction $r$ et l'homotopie simpliciale $h$ sont les
  nerfs respectifs du rétracte par transformation oplax de la
  proposition~\ref{prop:permet_thmA} du présent texte, et de la rétraction
  et de la transformation oplax produite par cette même proposition.

  On a déjà observé, dans la preuve de la
  proposition~\ref{prop:def_equiv_permet_A}, que le morphisme $r$ est bien
  le nerf du \oo-foncteur $m^\ast : \cotr{A}{c} \to \cotr{A}{c_m}$ de la
  proposition~\ref{prop:permet_thmA}. Rappelons que par définition ce
  \oo-foncteur est égal à $m^\ast \times_C A$, où cette fois $m^\ast :
  \cotr{C}{c} \to \cotr{C}{c_m}$ désigne le \oo-foncteur associé en vertu du
  paragraphe~\ref{paragr:fonct_tr_sur_1} au triangle commutatif
  \[
    \xymatrix@C=1.5pc{
      \On{0} \ar[dr]_{c_m} \ar[rr]^{m} & & \On{m} \ar[dl]^c \\
                                     & C & \pbox{.}
    }
  \]
  Il nous reste donc à traiter les cas de $s$ et $h$.
\end{paragraph}

\begin{remark}
  Nous avons choisi de rendre cet appendice, dont le but n'est pas
  d'établir un résultat mais de justifier les formules
  \forlang{ad hoc} de \cite[section~6]{AraMaltsiThmAI}, raisonnablement
  concise. C'est pourquoi on ne rappellera pas les définitions de $s$ et $h$
  (le lecteur devra donc parfois se référer à \cite[section
  6]{AraMaltsiThmAI}) et on laissera plus de vérifications au lecteur que
  dans le reste du texte.
\end{remark}

\begin{paragraph}\label{paragr:r'_h'}
  La section du \oo-foncteur $m^\ast \times_C A : \cotr{A}{c} \to \cotr{A}{c_m}$
  produite par la proposition~\ref{prop:permet_thmA} est le \oo-foncteur
  $(r', h', c)^\ast \times_C A : \cotr{A}{c_m} \to \cotr{A}{c}$, où $(r',
  h', c)^\ast : \cotr{C}{c_m} \to \cotr{C}{c}$ est le \oo-foncteur associé
  par le paragraphe~\ref{paragr:fonct_tri} au triangle
  \[
    \shorthandoff{;}
    \xymatrix@C=1.5pc{
      \cn(\Deltan{m}) \ar[rr]^{r'} \ar@{=}[dr]_{}="f" & &
      \cn(\Deltan{0}) \ar[dl]^(0.42){m} \\
       & \cn(\Deltan{m})
      \ar@{}"f";[ur]_(.15){}="ff"
      \ar@{}"f";[ur]_(.55){}="oo"
      \ar@<-0.5ex>@2"ff";"oo"^{h'}
      & \pbox{,}
    }
  \]
  où $r'$ et $h'$ sont le morphisme $r$ et l'antihomotopie $h$ de la preuve de
  la proposition~\ref{prop:Deltan_contr} (qu'on a décorés d'un « $'$ » pour
  ne pas les confondre avec les morphismes simpliciaux en jeu dans cette
  appendice). Rappelons que cette antihomotopie $h'$ est définie sur la base
  de $\cn(\Deltan{m})$ (voir le paragraphe~\ref{paragr:base_cn}) par
  \[ h'(i_0, \dots, i_p) = (i_0, \dots, i_p, m), \]
  en convenant que cette expression est nulle lorsque $i_p = m$.
\end{paragraph}

Nous allons montrer qu'on a $N((r', h', c)^\ast \times_C A) = s$. Pour cela,
nous avons besoin de compléments sur la construction du
paragraphe~\ref{paragr:fonct_tri}.

\begin{paragraph}\label{paragr:desc_fonct_tri}
  Soit
  \[
    \shorthandoff{;}
    \xymatrix@C=1.5pc{
      K \ar[rr]^f \ar[dr]_{g}_{}="f" & & K' \ar[dl]^(0.42){g'} \\
      & L
      \ar@{}"f";[ur]_(.15){}="ff"
      \ar@{}"f";[ur]_(.55){}="oo"
      \ar@<-0.5ex>@2"ff";"oo"^{k}
      & \pbox{,}
    }
  \]
  un diagramme de complexes de Steiner forts, avec $g'$ une inclusion rigide
  ordonnée et $k$ une antihomotopie de $g$ vers $g'f$, et soit $C$ une
  \oo-catégorie munie d'un \oo-foncteur $b : \nu(L) \to C$. Posons $c =
  b\nu(g)$ et $c' = b\nu(g')$ et considérons le \oo-foncteur
  \[ (f, k, b)^\ast : \cotr{C}{c'} \to \cotr{C}{c} \]
  du paragraphe~\ref{paragr:fonct_tri}. Soit $T$ un complexe de Steiner
  fort. D'après \cite[remarque 11.2.3]{AraMaltsiJoint}, on peut décrire
  l'application
  \[
    \Hom_{\ooCat}(\nu(T), \cotr{C}{c'}) \to \Hom_{\ooCat}(\nu(T), \cotr{C}{c})
  \]
  induite par $(f, k, b)^\ast$ de la manière suivante. On a des bijections
  naturelles
  \[
     \begin{split}
     \Hom_{\ooCat}(\nu(T), \cotr{C}{c'})
     & \simeq \Hom_{\cotr{\ooCat}{\nu(L)}}((\nu(L \amalg_{K'} (K' \joint T)),
     \nu(j_1)), (C, b)) \\
     & \subset \Hom_{\ooCat}(\nu(L \amalg_{K'} (K' \joint T)), C), \\
     \Hom_{\ooCat}(\nu(T), \cotr{C}{c})
     & \simeq \Hom_{\cotr{\ooCat}{\nu(K)}}((\nu(K \joint T), \nu(\iota_1)),
     (C, c)) \\
     & \subset \Hom_{\ooCat}(\nu(K \joint T), C),
     \end{split}
   \]
  où $j_1$ désigne la première inclusion canonique,
  et l'application
  \[
    \Hom_{\ooCat}(\nu(T), \cotr{C}{c'}) \to \Hom_{\ooCat}(\nu(T), \cotr{C}{c})
  \]
  est induite par le morphisme
  \[
    \psi : K \joint T \to L \amalg_{K'} (K' \joint T)
  \]
  défini par
  \[
     \psi(x \joint y) =
       \begin{cases}
         g(x) & \text{si $y = \vide$,} \\
         f(x) \joint y + e(y)k(x) & \text{sinon,}
       \end{cases}
  \]
  où on convient que $f(\vide) = \vide$, $k(\vide) = 0$ et $e(y) = 0$ si $y$
  n'est pas de degré $0$.

  Notons que lorsque $T = \cn(\Deltan{n})$, de sorte qu'on a $\nu(T) =
  \On{n}$, l'application que l'on vient de décrire n'est autre que
  \[
    N((f, k, b)^\ast)_n : N(\cotr{C}{c'})_n \to N(\cotr{C}{c})_n.
  \]
\end{paragraph}

\begin{paragraph}
  Explicitons le morphisme $\psi$ du paragraphe précédent dans le cas qui
  nous intéresse, à savoir celui du triangle
  \[
    \shorthandoff{;}
    \xymatrix@C=1.5pc{
      \cn(\Deltan{m}) \ar[rr]^{r'} \ar@{=}[dr]_{}="f" & &
      \cn(\Deltan{0}) \ar[dl]^(0.42){m} \\
       & \cn(\Deltan{m})
      \ar@{}"f";[ur]_(.15){}="ff"
      \ar@{}"f";[ur]_(.55){}="oo"
      \ar@<-0.5ex>@2"ff";"oo"^{h'}
    }
  \]
  du paragraphe~\ref{paragr:r'_h'}
  et de $T = \cn(\Deltan{n})$.
  Le morphisme
  \[ \psi : \cn(\Deltan{m}) \joint \cn(\Deltan{n}) \to \cn(\Deltan{m})
  \amalg_{\cn(\Deltan{0})} (\cn(\Deltan{0}) \joint \cn(\Deltan{n}))  \]
  est donné sur la base de $\cn(\Deltan{m}) \joint \cn(\Deltan{n})$ (voir
  le paragraphe~\ref{paragr:base_cn} et la
  proposition~\ref{prop:joint_Steiner}) par
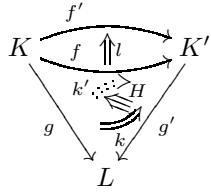
  \[
    \psi((i_0, \dots, i_p) \joint (j_0, \dots, j_q)) =
    \begin{cases}
      (i_0, \dots, i_p) & \text{si $q = -1$,} \\
      \vide \joint (j_0, \dots, j_q) & \text{si $p = -1$,} \\
      (m) \joint (j_0) + (i_0, m) & \text{si $p = 0$ et $q = 0$,} \\
      (m) \joint (j_0, \dots, j_q) & \text{si $p = 0$ et $q > 0$,} \\
      (i_0, \dots, i_p, m) & \text{si $p > 0$ et $q = 0$,} \\
      0 & \text{si $p > 0$ et $q > 0$,}
    \end{cases}
  \]
  où, d'une part, on a convenu que $(i_0, \dots, i_p)$ et $(j_0, \dots,
  j_q)$ sont égaux à $\vide$ pour $p = -1$ et $q = -1$ respectivement et,
  d'autre part, on a noté $(m)$ la base de $\cn(\Deltan{0})$ de sorte
  qu'on puisse considérer le morphisme $m : \cn(\Deltan{0}) \to
  \cn(\Deltan{n})$ comme une inclusion.

  Ainsi, pour $n \ge 0$, l'application
  \[
    N((r', h', c)^\ast)_n : \Hom_{\ooCat}(\On{n}, \cotr{C}{c_m}) \to
      \Hom_{\ooCat}(\On{n}, \cotr{C}{c})
  \]
  est induite par le morphisme $\psi$ décrit ci-dessus. Or, ce morphisme
  $\psi$ coïncide avec le morphisme $f_n$ défini au paragraphe 6.2 de
  \cite{AraMaltsiThmAI} pour construire le morphisme simplicial~$s$. On en
  déduit que $N((r', h', c)^\ast) = s$ dans le cas où $A = C$ et $v = \id{C}$
  et donc que $N((r', h', c)^\ast \times_C A) = s$ dans le cas général
  puisque le nerf commute aux produits fibrés et que $s$ est
  défini comme un produit fibré.
\end{paragraph}

\begin{paragraph}
  La transformation oplax de $((r', h', c)^\ast m^\ast) \times_C A$ vers
  $\id{\cotr{A}{c}}$ produite par la proposition~\ref{prop:permet_thmA} est
  le changement de base le long de $v : A \to C$ de la transformation oplax
  $(h', \id{h'}, c)^\ast : (r', h', c)^\ast m^\ast \to \id{\cotr{C}{c}}$ associée
  par le paragraphe~\ref{paragr:fonct_cone} au cône
  \[
     \shorthandoff{;:}
     \xymatrix@C=1pc@R=4pc{
       \cn(\Deltan{m}) \ar@/^2ex/[rr]^(.33){mr'}_{}="1" \ar@/_2ex/@{=}[rr]_{}="0"
       \ar@{=}[dr]_{}="f"
       \ar@2"0";"1"_{\,h'}
       & & \cn(\Deltan{m}) \ar@{=}[dl] \\
       & \cn(\Deltan{m})
       \ar@{}"f";[ur]_(.15){}="ff"
       \ar@{}"f";[ur]_(.55){}="oo"
       \ar@<-0.0ex>@/^1ex/@{:>}"ff";"oo"^(.18){h'\!\!}_(.30){}="h'"
       \ar@<-3.0ex>@/^-1ex/@{=}"ff";"oo"_(.80){}="h"
       \ar@3"h";"h'"_(.20){\id{h'}} & \pbox{,}
       }
  \]
  où $r'$ et $h'$ sont le morphisme et l'antihomotopie du
  paragraphe~\ref{paragr:r'_h'}.
\end{paragraph}

Nous allons montrer qu'on a $N((h', \id{h'}, c)^\ast \times_C A) = h$. Pour cela,
nous avons besoin de compléments sur la construction du
paragraphe~\ref{paragr:fonct_cone}.

\begin{paragraph}\label{paragr:desc_fonct_cone}
  Soit
   \[
      \shorthandoff{;:}
      \xymatrix@C=1.5pc@R=3pc{
        K \ar@/^2ex/[rr]^(.33){f'}_{}="1" \ar@/_2ex/[rr]^(.30)f_{}="0"
        \ar[dr]_{}="f"_{\phantom{g'}g}
        \ar@2"0";"1"_l
        & & K' \ar[dl]^{g'} \\
        & L
        \ar@{}"f";[ur]_(.15){}="ff"
        \ar@{}"f";[ur]_(.55){}="oo"
        \ar@<-0.5ex>@/^1ex/@{:>}"ff";"oo"^(.18){k'\!\!}_(.30){}="h'"
        \ar@<-2.0ex>@/^-1ex/@2"ff";"oo"_(.36){k}_(.80){}="h"
        \ar@3"h";"h'"_(.20){H_{}}
        }
  \]
  un diagramme de complexes de Steiner forts, avec $g'$ une inclusion rigide
  ordonnée, $k$, $k'$ et $l$ des antihomotopies de $g$ vers $g'f$, de
  $g$ vers $g'f'$ et de $f$ vers $f'$ respectivement et $H$ une
  $2$-antihomotopie de $g'l + k$ vers $k'$, et soit $C$ une \oo-catégorie
  munie d'un \oo-foncteur $b : \nu(L) \to C$. Posons $c = b\nu(g)$ et $c' =
  b\nu(g')$ et considérons la transformation oplax
  \[ (l, H, b)^\ast : (f', k', b)^\ast \to (f, k, b)^\ast \]
  du paragraphe~\ref{paragr:fonct_cone}. Par adjonction, cette
  transformation correspond à un \oo-foncteur
  \[
    \cotr{C}{c'} \to \HomLax(\Dn{1}, \cotr{C}{c}).
  \]
  Soit $T$ un complexe de Steiner fort. D'après \cite[remarque
  11.4.3]{AraMaltsiJoint}, on peut décrire l'application induite
  \[
    \Hom_{\ooCat}(\nu(T), \cotr{C}{c'}) \to \Hom_{\ooCat}(\nu(T),
    \HomLax(\Dn{1}, \cotr{C}{c}))
  \]
  de la manière suivante. On a des bijections naturelles
  \[
     \Hom_{\ooCat}(\nu(T), \cotr{C}{c'})
     \simeq \Hom_{\cotr{\ooCat}{\nu(L)}}((\nu(L \amalg_{K'} (K' \joint T)),
     \nu(j_1)), (C, b)),
  \]
  où $j_1$ désigne la première inclusion canonique, et
   \[
     \begin{split}
     \MoveEqLeft \Hom_{\ooCat}(\nu(T), \HomLax(\Dn{1}, \cotr{C}{c})) \\
     & \simeq \Hom_{\cotr{\ooCat}{\nu(K)}}((\nu(K \joint (\cn(\Deltan{1})
     \otimes T)), \nu(\iota_1)), (C, c)),
     \end{split}
   \]
  et l'application
  \[
    \Hom_{\ooCat}(\nu(T), \cotr{C}{c'}) \to \Hom_{\ooCat}(\nu(T),
    \HomLax(\Dn{1}, \cotr{C}{c}))
  \]
  est induite par le morphisme
  \[
    \chi : K \joint (\cn(\Deltan{1}) \otimes T) \to L \amalg_{K'} (K' \joint T)
  \]
  défini par
  \[
    \begin{split}
      \chi(x \joint \vide) & = g(x) \\
      \chi(x \joint ((0) \otimes y)) & = f'(x) \joint y + e(y)k'(x) \\
      \chi(x \joint ((1) \otimes y)) & = f(x) \joint y + e(y)k(x) \\
      \chi(x \joint ((01) \otimes y)) & = l(x) \joint y + e(y)H(x),
    \end{split}
  \]
  où on ajoute aux conventions utilisées dans la définition de $\psi$ au
  paragraphe~\ref{paragr:desc_fonct_tri} les conventions $l(\vide) = 0$ et
  $H(\vide) = 0$.
\end{paragraph}

On va voir que l'application $\chi$, pour $T = \cn(\Deltan{n})$, permet de
décrire $N((l, H, b)^\ast)$ en termes de complexes dirigés augmentés. Pour
ce faire, nous avons besoin d'une description alternative du nerf d'une
transformation oplax.

\begin{paragraph}
  Soit $\alpha$ une transformation oplax entre \oo-foncteurs de source
  $C$ et de but $D$. Par adjonction, la transformation oplax $\alpha$
  correspond à un \oo-foncteur $C \to \HomLax(\Dn{1}, D)$ qu'on notera
  $k_\alpha$. L'homotopie simpliciale $N(\alpha)$ peut se décrire à partir de
  $k_\alpha$ de la manière suivante. Soit $(\phi, x) :
  \Deltan{n} \to \Deltan{1} \times N(C)$ un $n$-simplexe de
  $\Deltan{1} \times N(C)$. Le $n$-simplexe $N(\alpha)(\phi, x)$ de $N(D)$ est
  l'image de $x : \On{n} \to C$ par le composé
  \[
    \xymatrix@C=1.5pc{
      \Hom_{\ooCat}(\On{n}, C) \ar[r] & \Hom_{\ooCat}(\On{n},
      \HomLax(\Dn{1}, D)) \ar[d]_*[@]{\sim}  \\
      & \Hom_{\ooCat}(\Dn{1} \otimes \On{n}, D) \ar[r]
      & \Hom_{\ooCat}(\On{n}, D) \pbox{,}
    }
  \]
  où les flèches horizontales du haut et du bas sont induites respectivement
  par $k_\alpha$ et le \oo-foncteur $\nu(g_\phi)$ du
  paragraphe~\ref{paragr:g_phi}.
\end{paragraph}

\begin{paragraph}
  Explicitons le paragraphe précédent dans le cas où $\alpha = (k, H,
  b)^\ast$ est la transformation du paragraphe~\ref{paragr:desc_fonct_cone}.
  On a des isomorphismes canoniques
  \[
     \begin{split}
       \Hom_{\ooCat}(\On{n}, \cotr{C}{c'})
     & \simeq \Hom_{\cotr{\ooCat}{\nu(L)}}((\nu(L \amalg_{K'} (K' \joint
       \cn(\Deltan{n}))), \nu(j_1)), (C, b)) \\
     \Hom_{\ooCat}(\On{n}, \cotr{C}{c})
     & \simeq \Hom_{\cotr{\ooCat}{\nu(K)}}((\nu(K \joint \cn(\Deltan{n})),
     \nu(\iota_1)), (C, c)),
     \end{split}
  \]
  et, si $\phi : \Deltan{n} \to \Deltan{1}$ est un $n$-simplexe de
  $\Deltan{1}$, l'application
  \[
    \Hom_{\ooCat}(\On{n}, \cotr{C}{c'}) \to \Hom_{\ooCat}(\On{n}, \cotr{C}{c})
  \]
  du paragraphe précédent est induite par le composé
  \[
    K \joint \cn(\Deltan{n}) \xto{K \joint g_\phi}
    K \joint (\cn(\Deltan{1}) \otimes \cn(\Deltan{n})) \xto{\quad\chi\quad}
    L \amalg_{K'} (K' \joint \cn(\Deltan{n})),
  \]
  où $\chi$ est le morphisme du paragraphe~\ref{paragr:desc_fonct_cone}
  pour $T = \cn(\Deltan{n})$. On notera $\chi_\phi$ ce composé.
\end{paragraph}

\begin{paragraph}
  Explicitons les morphismes $\chi$ et $\chi_\phi$ du
  paragraphe~\ref{paragr:desc_fonct_cone} et du paragraphe précédent dans le
  cas qui nous intéresse, à savoir celui du cône
  \[
     \shorthandoff{;:}
     \xymatrix@C=1pc@R=4pc{
       \cn(\Deltan{m}) \ar@/^2ex/[rr]^(.33){mr'}_{}="1" \ar@/_2ex/@{=}[rr]_{}="0"
       \ar@{=}[dr]_{}="f"
       \ar@2"0";"1"_{\,h'}
       & & \cn(\Deltan{m}) \ar@{=}[dl] \\
       & \cn(\Deltan{m})
       \ar@{}"f";[ur]_(.15){}="ff"
       \ar@{}"f";[ur]_(.55){}="oo"
       \ar@<-0.0ex>@/^1ex/@{:>}"ff";"oo"^(.18){h'\!\!}_(.30){}="h'"
       \ar@<-3.0ex>@/^-1ex/@{=}"ff";"oo"_(.80){}="h"
       \ar@3"h";"h'"_(.20){\id{h'}}
       }
  \]
  et de $T = \cn(\Deltan{n})$. Le but de $\chi$ est $\cn(\Deltan{m})
  \amalg_{\cn(\Deltan{m})} (\cn(\Deltan{m}) \joint c(\Deltan{n}))$ qui est
  canoniquement isomorphe à~$\cn(\Deltan{m}) \joint \cn(\Deltan{n})$ et on
  considérera $\chi$ comme un morphisme
  \[
    \chi : \cn(\Deltan{m}) \joint (\cn(\Deltan{1}) \otimes
    \cn(\Deltan{n})) \to \cn(\Deltan{m}) \joint \cn(\Deltan{n}).
  \]
  De même, on considérera $\chi_\phi$ comme un morphisme
  \[
    \chi_\phi : \cn(\Deltan{m}) \joint \cn(\Deltan{n}) \to 
    \cn(\Deltan{m}) \joint \cn(\Deltan{n}).
  \]
  En explicitant les formules du paragraphe~\ref{paragr:desc_fonct_cone}
  pour le cône ci-dessus, on obtient que $\chi$ est donné par
  \[
    \begin{split}
      (i_0, \dots, i_p) \joint \vide & \mapsto (i_0, \dots, i_p) \joint
      \vide \\
      (i_0, \dots, i_p) \joint ((0) \otimes (j_0, \dots, j_q))
      & \mapsto
        \begin{cases}
          \vide \joint (j_0, \dots, j_q) & \text{si $p = -1$,} \\
          (m) \joint (j_0) + (i_0, m) \joint \vide & \text{si $p = 0$ et $q = 0$,} \\
          (m) \joint (j_0, \dots, j_q) & \text{si $p = 0$ et $q > 0$,} \\
          (i_0, \dots, i_p, m) \joint \vide & \text{si $p > 0$ et $q = 0$,} \\
          0 & \text{si $p > 0$ et $q > 0$,} \\
        \end{cases} \\
      (i_0, \dots, i_p) \joint ((1) \otimes (j_0, \dots, j_q))
      & \mapsto (i_0, \dots, i_p) \joint (j_0, \dots, j_q) \\
      (i_0, \dots, i_p) \joint ((01) \otimes (j_0, \dots, j_q))
      & \mapsto
        \begin{cases}
           0 & \text{si $p = -1$,} \\
           (i_0, \dots, i_p, m) \joint (j_0, \dots, j_q) & \text{si $p \ge 0$.}
        \end{cases}
    \end{split}
  \]

  Décrivons maintenant l'endomorphisme
  \[
    \chi_\phi : \cn(\Deltan{m}) \joint \cn(\Deltan{n}) \to 
    \cn(\Deltan{m}) \joint \cn(\Deltan{n}).
  \]
  Soit $(i_0, \dots, i_p) \joint (j_0, \dots, j_q)$ un élément de la base de
  $\cn(\Deltan{m}) \joint \cn(\Deltan{n})$. Notons $r$ le nombre de $0$
  parmi $\phi(j_0), \dots, \phi(j_q)$. Alors l'endomorphisme $\chi_\phi$
  envoie l'élément~$(i_0, \dots, i_p) \joint (j_0, \dots, j_q)$ sur
  \[
  \begin{cases}
    (i_0, \dots, i_p) \joint (j_0, \dots, j_q) & \text{si $r = 0$,} \\
    \vide \joint (j_0, \dots, j_q) & \text{si $r = 1$ et $p = -1$,} \\
    (m) \joint (j_0) + (i_0, m) \joint \vide &
      \text{si $r = 1$, $p = 0$ et $q = 0$,} \\
    (m) \joint (j_0, \dots, j_q) + (i_0, m) \joint (j_1, \dots, j_q) &
      \text{si $r = 1$, $p = 0$ et $q > 0$,} \\
    (i_0, \dots, i_p, m) \joint \vide &
      \text{si $r = 1$, $p > 0$ et $q = 0$,} \\
    (i_0, \dots, i_p, m) \joint (j_1, \dots, j_q) &
      \text{si $r = 1$, $p > 0$ et $q > 0$,} \\
    \vide \joint (j_0, \dots, j_q) & \text{si $r \ge 2$ et $p = -1$,} \\
    (m) \joint (j_0) + (i_0, m) \joint \vide & \text{si $r \ge 2$, $p = 0$ et $q = 0$,} \\
    (m) \joint (j_0, \dots, j_q) &
      \text{si $r \ge 2$, $p = 0$ et $q > 0$,} \\
    (i_0, \dots, i_p, m) \joint \vide &
      \text{si $r = 2$, $p > 0$ et $q = 0$,} \\
    0 &
      \text{si $r \ge 2$, $p > 0$ et $q > 0$.} \\
  \end{cases}
  \]
  On vérifie que ce morphisme $\chi_\phi$ s'identifie à travers
  l'isomorphisme canonique $\cn(\Deltan{m}) \joint \cn(\Deltan{n}) \simeq
  \cn(\Deltan{m + 1 + n})$ au morphisme $f_\phi$ défini au paragraphe 6.6 de
  \cite{AraMaltsiThmAI} pour construire l'homotopie simpliciale $h$. On en
  déduit que $N(\alpha) = h$ dans le cas où $A = C$ et $v = \id{C}$ et donc
  que $N(\alpha \times_C A) = h$ dans le cas général puisque le nerf est
  compatible aux changements de base des transformations oplax (voir le
  corollaire~\ref{coro:nerf_trans_prod_fib}) et que $h$ est définie par un
  changement de base. Ceci achève de montrer que
  \[
   s = N((r', h', c)^\ast \times_C A),
   \quad
   r = N(m^\ast \times_C A)
   \quadet
    h = N((h', \id{h'}, c)^\ast \times_C A).
  \]
\end{paragraph}

\bibliography{biblio}
\bibliographystyle{mysmfplain}

\end{document}